\documentclass{article}

\usepackage{graphicx} 
\usepackage{subcaption}

\usepackage{csquotes}
\title{Sparser Abelian High Dimensional Expanders}
\author{Yotam Dikstein \thanks{Institute for Advanced Study. This material is based upon work supported by the National Science Foundation under Grant No. DMS-1926686. Email: yotam.dikstein@gmail.com.}
\and Siqi Liu \thanks{Institute for Advanced Study. Supported by Minerva Research Foundation Member Fund. This work was partially conducted at the Center for Discrete Mathematics and Theoretical Computer Science (DIMACS) with support from DIMACS and Rutgers University. Email: siqiliu@ias.edu.}
\and Avi Wigderson \thanks{Institute for Advanced Study. Supported by NSF grant CCF-1900460. Email: avi@ias.edu.}
}

\date{\today}

\def\showauthornotes{0}

\def\showkeys{0}
\def\showdraftbox{0}
\def\showcolorlinks{1}
\def\usemicrotype{1}
\def\showfixme{0}

\usepackage[
backend=biber,
maxcitenames=50,
maxbibnames=99,
style=alphabetic,
sorting=nyt 
]{biblatex}
\providecommand{\RR}{\mathbb{R}}
\providecommand{\NN}{\mathbb{N}}

\usepackage[l2tabu, orthodox]{nag}

\usepackage{xspace,enumerate}

\usepackage[dvipsnames]{xcolor}

\usepackage[T1]{fontenc}
\usepackage[full]{textcomp}

\usepackage[american]{babel}

\usepackage{mathtools}

\usepackage{amsthm}
\usepackage{amsmath}

\newtheorem{theorem}{Theorem}[section]
\newtheorem*{theorem*}{Theorem}

\newtheorem{proposition}[theorem]{Proposition}
\newtheorem*{proposition*}{Proposition}
\newtheorem{lemma}[theorem]{Lemma}
\newtheorem*{lemma*}{Lemma}
\newtheorem{corollary}[theorem]{Corollary}
\newtheorem*{corollary*}{Corollary}
\newtheorem*{conjecture*}{Conjecture}

\newtheorem*{fact*}{Fact}

\newtheorem*{hypothesis*}{Hypothesis}

\theoremstyle{definition}
\newtheorem{definition}[theorem]{Definition}
\newtheorem*{definition*}{Definition}

\newtheorem{example}[theorem]{Example}

\theoremstyle{remark}
\newtheorem{claim}[theorem]{Claim}
\newtheorem*{claim*}{Claim}
\newtheorem{remark}[theorem]{Remark}
\newtheorem*{remark*}{Remark}
\newtheorem{observation}[theorem]{Observation}
\newtheorem*{observation*}{Observation}

 \usepackage[letterpaper,
 top=1in,
 bottom=1in,
 left=1in,
 right=1in]{geometry}

\usepackage[varg]{pxfonts} 

\usepackage[sc,osf]{mathpazo}
\usepackage{lmodern}

\ifnum\showkeys=1
\usepackage[color]{showkeys}
\fi

\ifnum\showcolorlinks=1
\usepackage[
colorlinks=true,
urlcolor=blue,
linkcolor=blue,
citecolor=OliveGreen,
]{hyperref}
\fi

\ifnum\showcolorlinks=0
\usepackage[
colorlinks=false,
pdfborder={0 0 0}
]{hyperref}
\fi

\usepackage{prettyref}

\newcommand{\savehyperref}[2]{\texorpdfstring{\hyperref[#1]{#2}}{#2}}

\newrefformat{eq}{\savehyperref{#1}{\textup{(\ref*{#1})}}}
\newrefformat{lem}{\savehyperref{#1}{Lemma~\ref*{#1}}}
\newrefformat{def}{\savehyperref{#1}{Definition~\ref*{#1}}}
\newrefformat{thm}{\savehyperref{#1}{Theorem~\ref*{#1}}}
\newrefformat{cor}{\savehyperref{#1}{Corollary~\ref*{#1}}}
\newrefformat{cha}{\savehyperref{#1}{Chapter~\ref*{#1}}}
\newrefformat{sec}{\savehyperref{#1}{Section~\ref*{#1}}}
\newrefformat{app}{\savehyperref{#1}{Appendix~\ref*{#1}}}
\newrefformat{tab}{\savehyperref{#1}{Table~\ref*{#1}}}
\newrefformat{fig}{\savehyperref{#1}{Figure~\ref*{#1}}}
\newrefformat{hyp}{\savehyperref{#1}{Hypothesis~\ref*{#1}}}
\newrefformat{alg}{\savehyperref{#1}{Algorithm~\ref*{#1}}}
\newrefformat{rem}{\savehyperref{#1}{Remark~\ref*{#1}}}
\newrefformat{item}{\savehyperref{#1}{Item~\ref*{#1}}}
\newrefformat{obs}{\savehyperref{#1}{Observation~\ref*{#1}}}
\newrefformat{step}{\savehyperref{#1}{step~\ref*{#1}}}
\newrefformat{conj}{\savehyperref{#1}{Conjecture~\ref*{#1}}}
\newrefformat{fact}{\savehyperref{#1}{Fact~\ref*{#1}}}
\newrefformat{prop}{\savehyperref{#1}{Proposition~\ref*{#1}}}
\newrefformat{prob}{\savehyperref{#1}{Problem~\ref*{#1}}}
\newrefformat{claim}{\savehyperref{#1}{Claim~\ref*{#1}}}
\newrefformat{relax}{\savehyperref{#1}{Relaxation~\ref*{#1}}}
\newrefformat{red}{\savehyperref{#1}{Reduction~\ref*{#1}}}
\newrefformat{part}{\savehyperref{#1}{Part~\ref*{#1}}}
\newrefformat{ex}{\savehyperref{#1}{Example~\ref*{#1}}}
\newrefformat{para}{\savehyperref{#1}{Paragraph~\ref*{#1}}}
\newrefformat{ass}{\savehyperref{#1}{Assumption~\ref*{#1}}}
\newrefformat{question}{\savehyperref{#1}{Question~\ref*{#1}}}

\newcommand{\Sref}[1]{\hyperref[#1]{\S\ref*{#1}}}

\usepackage{nicefrac}

\ifnum\usemicrotype=1
\usepackage{microtype}
\fi

\ifnum\showauthornotes=1
\newcommand{\Authornote}[2]{{\sffamily\small\color{red}{[#1: #2]}}}
\newcommand{\Authornotecolored}[3]{{\sffamily\small\color{#1}{[#2: #3]}}}
\newcommand{\Authorcomment}[2]{{\sffamily\small\color{gray}{[#1: #2]}}}
\newcommand{\Authorstartcomment}[1]{\sffamily\small\color{gray}[#1: }

\newcommand{\Authorfnote}[2]{\footnote{\color{red}{#1: #2}}}
\newcommand{\Authorfixme}[1]{\Authornote{#1}{\textbf{??}}}
\newcommand{\Authormarginmark}[1]{\marginpar{\textcolor{red}{\fbox{\Large #1:!}}}}
\else
\newcommand{\Authornote}[2]{}
\newcommand{\Authornotecolored}[3]{}
\newcommand{\Authorcomment}[2]{}
\newcommand{\Authorstartcomment}[1]{}

\newcommand{\Authorfnote}[2]{}
\newcommand{\Authorfixme}[1]{}
\newcommand{\Authormarginmark}[1]{}
\fi

\newcommand{\Ynote}{\Authornotecolored{Purple}{Y}}

\newcommand{\Snote}{\Authornotecolored{Green}{S}}
\newcommand{\snote}{\Snote}

\ifnum\showfixme=0

\fi

\usepackage{boxedminipage}

\newcommand{\Brac}[1]{\left[#1\right]}

\newcommand{\abs}[1]{\lvert#1\rvert}
\newcommand{\Abs}[1]{\left\lvert#1\right\rvert}

\newcommand{\card}[1]{\lvert#1\rvert}

\newcommand\sett[2]{\left\{ #1 \left| \; \vphantom{#1 #2} \right. #2  \right\}}
\newcommand{\set}[1]{\{#1\}}

\newcommand{\norm}[1]{\lVert#1\rVert}

\newcommand{\snorm}[1]{\norm{#1}^2}

\newcommand{\iprod}[1]{\langle#1\rangle}

\def\dim{\mathrm{ dim}}
\def\sp{\mathrm{ span}}

\newcommand{\F}{\mathbb{F}}
\newcommand{\Esymb}{\mathbb{E}}
\newcommand{\Psymb}{\mathbb{P}}

\DeclareMathOperator*{\E}{\Esymb}

\DeclareMathOperator*{\ProbOp}{\Psymb}

\renewcommand{\Pr}{\ProbOp}
\def\one{{\mathbf{1}}}
\def\wt{{\mathrm{wt}}}

\newcommand{\prob}[1]{\Pr \left[ {#1} \right] }
\newcommand{\Prob}[2][]{\Pr_{{#1}}\left[#2\right]} 
\newcommand{\cProb}[3]{\Pr_{{#1}}\left[ #2 \left| \; \vphantom{#2 #3} \right. #3  \right]} 

\newcommand{\Ex}[2][]{\E_{{#1}}\Brac{#2}}

\newcommand{\ve}{\;\hbox{and}\;}

 \usepackage{dsfont}
\usepackage{mathrsfs}

\newcommand{\textparen}[1]{\text{(#1)}}

\ifx\because\undefined
\newcommand{\because}[1]{\textparen{because #1}}
\else
\renewcommand{\because}[1]{\textparen{because #1}}
\fi

\newcommand\bdot\bullet

\DeclareMathOperator{\supp}{supp}
\DeclareMathOperator{\dist}{dist}
\DeclareMathOperator{\sign}{sign}

\DeclareMathOperator{\rank}{rank}

\DeclareMathOperator{\row}{row}
\DeclareMathOperator{\col}{col}

\newcommand{\R}{\mathbb R}

\def\cont{f}

\renewcommand{\leq}{\leqslant}
\renewcommand{\le}{\leqslant}
\renewcommand{\geq}{\geqslant}
\renewcommand{\ge}{\geqslant}

\ifnum\showdraftbox=1

\else

\fi

\let\epsilon=\varepsilon

\numberwithin{equation}{section}

\newcommand{\MYstore}[2]{%
  \global\expandafter \def \csname MYMEMORY #1 \endcsname{#2}%
}

\newcommand{\MYload}[1]{%
  \csname MYMEMORY #1 \endcsname%
}

\newcommand{\MYnewlabel}[1]{%
  \newcommand\MYcurrentlabel{#1}%
  \MYoldlabel{#1}%
}

\newcommand{\MYdummylabel}[1]{}

\newcommand{\torestate}[1]{%
  \let\MYoldlabel\label%
  \let\label\MYnewlabel%
  #1%
  \MYstore{\MYcurrentlabel}{#1}%
  \let\label\MYoldlabel%
}

\newcommand{\restatetheorem}[1]{%
  \let\MYoldlabel\label
  \let\label\MYdummylabel
  \begin{theorem*}[Restatement of \prettyref{#1}]
    \MYload{#1}
  \end{theorem*}
  \let\label\MYoldlabel
}

\newcommand{\restatelemma}[1]{%
  \let\MYoldlabel\label
  \let\label\MYdummylabel
  \begin{lemma*}[Restatement of \prettyref{#1}]
    \MYload{#1}
  \end{lemma*}
  \let\label\MYoldlabel
}

\newcommand{\restateprop}[1]{%
  \let\MYoldlabel\label
  \let\label\MYdummylabel
  \begin{proposition*}[Restatement of \prettyref{#1}]
    \MYload{#1}
  \end{proposition*}
  \let\label\MYoldlabel
}

\newcommand{\restateclaim}[1]{%
  \let\MYoldlabel\label
  \let\label\MYdummylabel
  \begin{claim*}[Restatement of \prettyref{#1}]
    \MYload{#1}
  \end{claim*}
  \let\label\MYoldlabel
}

\newcommand{\restatecorollary}[1]{%
  \let\MYoldlabel\label
  \let\label\MYdummylabel
  \begin{corollary*}[Restatement of \prettyref{#1}]
    \MYload{#1}
  \end{corollary*}
  \let\label\MYoldlabel
}

\newcommand{\restatefact}[1]{%
  \let\MYoldlabel\label
  \let\label\MYdummylabel
  \begin{fact*}[Restatement of \prettyref{#1}]
    \MYload{#1}
  \end{fact*}
  \let\label\MYoldlabel
}

\newcommand{\restatedefinition}[1]{
\let\MYoldlabel\label 
\let\label\MYdummylabel 
\begin{definition*}[Restatement of \prettyref{#1}] 
    \MYload{#1} 
\end{definition*} 
\let\label\MYoldlabel 
} 

\newcommand{\restate}[1]{%
  \let\MYoldlabel\label
  \let\label\MYdummylabel
  \MYload{#1}
  \let\label\MYoldlabel
}

\let\origparagraph\paragraph
\renewcommand{\paragraph}[1]{\origparagraph{#1.}}

\allowdisplaybreaks

\newcommand{\dunion}{\mathbin{\mathaccent\cdot\cup}}

\let\pref=\prettyref

\newcommand{\dir}[1]{\overset{\to}{#1}}
\newcommand{\bproof}[1]{\begin{proof}[Proof of \pref{#1}]}
\newcommand{\eproof}{\end{proof}}
\newcommand{\coboundary}{{\delta}}
\newcommand{\boundary}{\partial}
\newcommand{\Img}{{\textrm {Im}}}

\def\S{X}

\def\Gr{\mathcal{G}}

\def\Y{Y}
\def\SL{\mathcal{L}}
\def\T{T}
\def\B{\mathcal{H}}

\def\diml{d}
\def\dimlin{{\diml}}
\def\dimlout{\diml'}
\def\dimloutfull{{\diml'}}

\def\lambdai{\lambda}
\def\lambdao{\lambda'}

\def\dimm{\ell}

\def\degree{D}
\def\M{\mathcal{SP}}
\def\cont{S}
\def\Mgras{Y^{all}}
\def\MP{\mathcal{M}}

\def\Meet{\textup{Meet}}

\newcommand{\fl}[1]{\lfloor #1 \rfloor}
\newcommand{\ceil}[1]{\lceil #1 \rceil}
\usepackage{anyfontsize} 

\begin{document}

\maketitle
\pagenumbering{roman}
\begin{abstract}
We present two new explicit constructions of Cayley high dimensional expanders (HDXs) over the abelian group $\F_2^n$. Our expansion proofs use only linear algebra and combinatorial arguments.

The first construction gives local spectral HDXs of any constant dimension
and subpolynomial degree $\exp(n^\epsilon)$  for every $\epsilon >0$, improving on
a construction by Golowich \cite{Golowich2023} which achieves $\epsilon =1/2$. \cite{Golowich2023} derives these HDXs by sparsifying the \emph{complete} Grassmann poset
of subspaces. The novelty in our construction is the ability
to sparsify \emph{any} expanding Grassmannian posets, leading to
iterated sparsification and much smaller degrees. The sparse Grassmannian (which is of independent interest in the theory of HDXs) serves as the generating set of the Cayley graph.

Our second construction gives a 2-dimensional HDXs of any polynomial degree  $\exp(\epsilon n$) for any constant $\epsilon > 0$, which is simultaneously a spectral expander and a coboundary expander.\footnote{A notion of combinatorial expansion which in high dimension is not equivalent to the spectral one.} To the best of our knowledge, this is the first such non-trivial construction. We name it the Johnson complex, as it is derived from the classical Johnson scheme, whose vertices serve as the generating set of this Cayley graph. This construction may be viewed as a derandomization of the recent random geometric complexes of \cite{LiuMSY2023}. Establishing coboundary expansion through Gromov's ``cone method'' and the associated isoperimetric inequalities is the most intricate aspect of this construction.

While these two constructions are quite different, we show that they both share a common structure, resembling the intersection patterns of vectors in the Hadamard code. We propose a general framework of such ``Hadamard-like'' constructions in the hope that it will yield new HDXs.
\end{abstract}
\Ynote{Before submission turn off comments}

\clearpage
\tableofcontents
\clearpage
\pagenumbering{arabic}
\setcounter{page}{1}
\section{Introduction}
Expander graphs are of central importance in many diverse parts of computer science and mathematics. Their structure and applications have been well studied for half a century, resulting in a rich theory (see e.g. \cite{hoory2006expander,lubotzky2012expander}). First, different definitions of expansion (spectral, combinatorial, probabilistic) are all known to be essentially equivalent. Second, we know that expanders abound; a random sparse graph is almost surely an expander \cite{pinsker1973complexity}. Finally, there is a variety of methods for constructing and analyzing explicit expander graphs. Initially, these relied on deep algebraic and number theoretic methods \cite{margulis1973explicit,gabber1981explicit,lubotzky1988ramanujan}. However, the quest for elementary (e.g. combinatorial and linear-algebraic) constructions has been extremely fertile, and resulted in major breakthroughs in computer science and math; the \emph{zigzag} method of \cite{reingold2000entropy} lead to \cite{dinur2007pcp},\cite{Reingold2008},\cite{TaShma2017explicit}, and the \emph{lifting} method of \cite{bilu2004constructing} lead to \cite{marcus2015interlacing,MarkusSS2015}.

In contrast, the expansion of sparse high dimensional complexes (or simply hypergraphs) is still a young and growing research topic. The importance of high dimensional expanders (HDXs) in computer science and mathematics is consistently unfolding, with a number of recent breakthroughs which essentially depend on them, such as locally testable codes \cite{dinur2022good,PanteleevK22}, quantum coding \cite{anshu2023nlts} and Markov chains \cite{AnariLOV2019} as a partial list. High dimensional expanders differ from expander graphs in several important ways. First, different definitions of expansion are known to be non-equivalent \cite{GundertW2016}. Second, they seem rare, and no natural distributions of complexes are known which almost surely produce bounded-degree expanding complexes. Finally, so far the few existing explicit constructions of bounded degree HDXs use algebraic methods \cite{Ballantine2000,CartwrightSZ2003,Li2004,LubotzkySV2005,KaufmanO2021,ODonnellP2022,Dikstein2023}. No elementary constructions
are known where high dimensional expansion follows from a clear (combinatorial or linear algebraic) argument. This demands further understanding of HDXs and motivates much prior work as well as this one.

In this work we take a step towards answering this question by constructing sparser HDXs by elementary means. We construct sparse HDXs with two central notions of high dimensional expansion: local spectral expansion and coboundary expansion.

To better discuss our results, we define some notions of simplicial complexes and expansion. For simplicity, we only define them in the \(2\)-dimensional case, though many of our results extend to arbitrary constant dimensions. For the full definitions see \pref{sec:prelims}. A 2-dimensional simplicial complex is a hypergraph \(X=(V,E,T)\) with vertices \(V\), edges \(E\) and triangles \(T\) (sets of size \(3\)). We require that if \(t \in T\) then all its edges are in \(E\), and all its vertices are in \(V\). We denote by \(N\) the number of vertices. Given a vertex \(u \in V\), its degree is the number of edges adjacent to it.  An essential notion for HDXs are the \emph{local} properties of its links. A \emph{link} of a vertex \(u\) is the graph \(G_u=(V_u,E_u)\) whose set of vertices is the neighborhood of \(u\), 
\[V_u = \sett{v \in V}{\set{u,v} \in E}.\]
The edges correspond to the triangles incident on \(u\):
\[E_u = \sett{\set{v,w}}{\set{u,v,w} \in T}.\]

In this paper we focus on a particularly simple class of complexes, called Cayley complexes over \(\F_2^n\). These are simplicial complexes whose underlying graph \((V,E)\) is a Cayley graph over \(\mathbb{F}_2^n\), and so have \(N=2^n\) vertices. 

In the rest of the introduction we will explain our constructions. In \pref{sec:intro-spec-exp} we present our construction of local spectral Cayley complexes with arbitrarily good spectral expansion and a subpolynomial degree. In \pref{sec:intro-cob-exp} we present our family of Cayley complexes that are both coboundary expanders and non-trivial local spectral expanders, with arbitrarily small polynomial degree. In \pref{sec:intro-common-struct} we discuss the common structure of our two constructions that relies on the Hadamard code. In \pref{sec:intro-open-questions} we present some open questions.

\subsection{Local spectral expanders} \label{sec:intro-spec-exp}
We start with definitions and motivation, review past results and then state ours. The first notion of high dimensional expansion we consider is \emph{local spectral expansion}. The spectral expansion of a graph \(G\) is the second largest eigenvalue of its random walk matrix in absolute value, denoted \(\lambda(G)\). Local spectral expansion of complexes is the spectral expansion of every link.

\begin{definition}[Local spectral expansion]
    Let \(\lambda \geq 0\). A \(2\)-dimensional simplicial complex \(X=(V,E,T)\) is a \(\lambda\)-local spectral expander if the underlying graph $(V,E)$ is connected and for every \(u \in V\), \(\lambda(X_u) \leq \lambda\). 
\end{definition}
Local spectral expansion is the most well known definition of high dimensional expansion. It was first introduced implicitly by \cite{Garland1973} for studying the real cohomology of simplicial complexes. A few decades later, a few other works observed that properties of specific local spectral expanders are of interest to the computer science community \cite{LubotzkySV2005,KaufmanKL2016,KaufmanM2017}. The definition above was given in \cite{DinurK2017}, for the purpose of constructing agreement tests. These are strong property tests that are a staple in PCPs (see further discussion in \pref{sec:intro-cob-exp}). Since then they have found many more applications that we now discuss.

\paragraph{Applications of local spectral expanders} One important application of local spectral expansion is in approximate sampling and Glauber dynamics. Local spectral expansion implies optimal spectral bounds on random walks defined on the sets of the complex. These are known as the Glauber dynamics or the up-down walks. This property has lead to a lot of results in approximate sampling and counting. These include works on Matroid sampling \cite{AnariLOV2019}, random graph coloring \cite{chen2021rapid} and mixing of Glauber dynamics on the Ising model \cite{AnariJKPV2021}. 

Local spectral expanders have also found applications in error correcting codes. Their sampling properties give rise to construction of good codes with efficient list-decoding algorithms \cite{AlevJT2019,DinurHKLT2021,jeronimo2021near}. One can also use the to construct locally testable codes that sparsify the Reed-Muller codes \cite{DinurLZ2023}.

The list decoding algorithms in \cite{AlevJT2019,jeronimo2021near} rely on their work on constraint satisfaction problems on local spectral expanders. These works generalize classical algorithms for solving CSPs on dense hypergraphs, to algorithms that solve CSPs on local spectral expanders. These works build on the `dense-like' properties of these complexes, and prove that CSPs on local spectral expanders are \emph{easy}. It is interesting to note that local spectral expanders have rich enough structure so that one can also construct CSPs over them that are also \emph{hard} for Sum-of-Squares semidefinite programs \cite{DinurFHT2020,HopkinsL2022}\footnote{In fact, \cite{AlevJT2019} uses Sum-of-Squares to solve CSPs on local spectral expanders. It seems contradictory, but the hardness results in \cite{DinurFHT2020,HopkinsL2022} use CSPs where the variables are the edges of the HDX. The CSPs in \cite{AlevJT2019} use its vertices as variables.}.

Some of the most important applications of local spectral expansion are constructions of derandomized agreement tests \cite{DinurK2017,DiksteinD2019,GotlibK2022,BafnaM2024,DiksteinDL2024} and even PCPs \cite{bafna2024quasi}. Most of these works also involve coboundary expansion so we discuss these in \pref{sec:intro-cob-exp}. 

Local spectral expanders have other applications in combinatorics as well. They are especially useful in sparsifying central objects like the complete uniform hypergraph and the boolean hypercube. The HDX based sparsifications maintain many of their important properties such as spectral gap, Chernoff bounds and hypercontractivity (see \cite{KaufmanM2017,DiksteinDFH2018,kaufman2020high,GurLL2022,BafnaHKL2022,dikstein2024chernoff} as a partial list).

\paragraph{Previous constructions}Attempts to construct bounded degree local spectral expanders with sufficient expansion for applications have limited success besides the algebraic constructions mentioned above \cite{Ballantine2000,CartwrightSZ2003,Li2004,LubotzkySV2005,KaufmanO2021,ODonnellP2022,Dikstein2023}. Random complexes in the high dimensional Erdős–Rényi model defined by \cite{LinialM2006} are \emph{not} local spectral expanders with overwhelming probability when the average degree is below \(N^{1/2}\) and no other bounded degree model is known. The current state of the art is the random geometric graphs on the sphere \cite{LiuMSY2023}, that achieve non-trivial local spectral expansion and arbitrarily small polynomial degree (but their expansion is also bounded from below, see discussion in \pref{sec:intro-cob-exp}). Even allowing an unbounded degree, the only other known construction that is non-trivially sparse complexes is the one in \cite{Golowich2023} mentioned in the abstract.

\paragraph{Our result}
The first construction we present is a family of Cayley local spectral expanders. Namely,
\begin{theorem} \torestate{\label{thm:main-recursive-hdx-intro}
    For every \(\lambda, \varepsilon > 0\) and integer \(\diml \geq 2\) there exists an infinite family of \(\diml\)-dimensional Cayley \(\lambda\)-local spectral expanders over the vertices $\F_2^n$ with degree at most \(2^{n^{\varepsilon}}\).}
\end{theorem}
This construction builds on an ingenious construction by Golowich \cite{Golowich2023}, which proved this theorem for \(\varepsilon = \frac{1}{2}\). As far as we know, this is the sparsest construction that achieves an arbitrarily small \(\lambda\)-local spectral expansion other than the algebraic constructions.


\paragraph{Proof overview} Our construction is based on the well studied Grassmann poset, the partially ordered set of subspaces of \(\F_2^n\), ordered by containment. This object is the vector space equivalent of high dimensional expanders, and was previously studied in \cite{DiksteinDFH2018,kaufmanT2021garland,GaitondeHKLZ2022,Golowich2023}. To understand our construction, we first describe the one in \cite{Golowich2023}. Golowich begins by sparsifying the Grassmann poset, to obtain a subposet. The generators of the Cayley complex in \cite{Golowich2023} are all (non-zero vectors in) the one-dimensional subspaces in the subposet. The analysis of expansion in \cite{Golowich2023} depend on the expansion properties of the subposet. The degree is just the number of one-dimensional vector spaces in the subposet, which is \(2^{n^{1/2}}\).

Our construction takes a modular approach to this idea. We modify the poset sparsification in \cite{Golowich2023} so that instead of the entire Grassmann poset, we can plug in \emph{any} subposet \(\Y\) of the Grassmann, and obtain a sparsification \(\Y'\) whose size depends on \(\Y\) (not the complete Grassmann).

Having this flexible sparsification step, we \emph{iterate}. We start with \(Y_0\), the complete Grassmann poset, we obtain a sequence \(\Y_1,\Y_2,\dots,\Y_m\) of sparser and sparser subposets. The vectors in \(\Y_i\) are low-rank matrices whose rows and columns are vectors in subspaces of \(\Y_{i+1}\).

We comment that while this is a powerful paradigm, it does have a drawback. The dimension of \(\Y_{i}\) is logarithmic in the dimension of \(\Y_{i-1}\) (the dimension of \(Y\) is the maximal dimension of a space inside \(\Y\)). Thus for a given complex \(\Y_0\) we can only perform this sparsification a constant number of steps keeping \(\varepsilon\) constant. Our infinite family is generated by taking an infinite family of \(\Y_0\)'s, and using our sparsification procedure on every one of them separately.

This construction of sparsified posets produces local spectral expanders of any constant dimension, as observed by \cite{Golowich2023}. The higher dimensional sets in the resulting simplicial complex correspond to the higher dimensional subspaces of the sparsification we obtain. We refer the reader to \pref{sec:construction-subpoly-HDX} for more details.

The analysis of the local spectral expansion is delicate, so we will not describe it here. One property that is necessary for it is the fact that every rank-\(r\) matrix has many decompositions into sums of rank-\(1\) matrices. Therefore, to upper bound spectral expansion we study graphs that arise from decompositions of matrices.

In addition, we wish to highlight a local-to-global argument in our analysis that uses a theorem by \cite{Madras2002}. Given a graph \(G = (V,E)\) that we wish to analyze, we find a set of expanding subgraphs \( \set{H_1, H_2, \dots H_m}\) of \(G\) (that are allowed to overlap). We consider the \emph{decomposition} graph whose vertices are \(V\) and whose edges are all \(\set{v,v'}\) such that there exists an \(H_i\) such that \(v,v' \in H_i\). If every \(H_i\) is an expander and the decomposition graph is also an expander, then \(G\) itself is an expander \cite{Madras2002}.
This decomposition is particularly useful in our setup. The graphs we need to analyze in order to prove the expansion of our complexes, can be decomposed into smaller subgraphs \(H_i\). These \(H_i\)'s are isomorphic to graphs coming from the original construction of \cite{Golowich2023}. Using this decomposition we are able to reduce the expansion of the links to the expansion of some well studied containment graphs in the original Grassmann we sparsify.

\begin{remark}[Degree lower bounds for Cayley local spectral expanders]
The best known lower bound on the degree of Cayley $\lambda$-local spectral expanders over \(\F_2^n\) is $\Omega \left (\frac{n}{\lambda^2 \log (1/\lambda)} \right)$ \cite{alon1992simple}. This bound is obtained simply because the underlying graph of a Cayley \(\lambda\)-local spectral expander is an \(O(\lambda)\)-spectral expander itself \cite{oppenheim2018local}. In \pref{app:degree-lower-bound} we investigate this question further, and provide some additional lower bounds on the degree of Cayley complexes based on their link structure. We prove that if the Cayley complex has a connected link then its degree is lower bounded by \(2n-1\) (compared to \(n\) if the link is not connected). If the link mildly expands, we provide a bound that grows with the degree of a vertex \emph{inside the link}. However, these bounds do not rule out Cayley \(\lambda\)-local spectral expanders whose generator set has size $O \left (\frac{n}{\lambda^2 \log (1/\lambda)} \right)$. The best trade-off between local spectral expansion and the degree of the complex is still open.
\end{remark}

\subsection{Our coboundary expanders} \label{sec:intro-cob-exp}

Our next main result is an explicit construction of a family of $2$-dimensional complexes with an arbitrarily small polynomial degree, that have both non-trivial local spectral expansion and coboundary expansion. To the best of our knowledge, this is the first such nontrivial construction. Again, these are Cayley complexes of \(\F_2^n\).

Coboundary expansion is a less understood notion compared to local spectral expansion. Therefore, before presenting our result we build intuition for it slowly by introducing a testing-based definition of coboundary expanders. After the definition, we will discuss the motivation behind coboundary expansion and applications of known coboundary expanders.

As mentioned before, there are several nonequivalent definitions of expansion for higher-dimensional complexes. In particular, coboundary expansion generalizes the notion of edge expansion in graphs to simplicial complexes. 
Coboundary expansion was defined independently by \cite{LinialM2006},\cite{MeshulamW2009} and \cite{Gromov2010}. The original motivation for this definition came from discrete topology; more connections to property testing were discovered later. We will expand on these connections after giving the definition. This definition we give is equivalent to the standard definition, but is described in the language of property testing rather than cohomology.

The more general and conventional definition for arbitrary groups $\Gamma$ can be found in \pref{sec:prelim-coboundary}. 

%
    
%
\begin{definition}
    Let $X = (V,E,T)$ be a $2$-dimensional simplicial complex. Consider the code $B^1 \subseteq \F_2^E$ whose generating matrix is the vertex-edge incidence matrix $A_X\in \F_2^{E\times V}$, and the local test $\mathcal{T}$ of membership in this code for $B^1$ given by first sampling a uniformly random triangle $\{u,v,w\}\in T$ and then accepting $f\in \F_2^E$ if and only if $ f(\{u, v\})+f(\{v, w\})+f(\{u,w\}) = 0$ over $\F_2$. 
    
    Let \(\beta > 0\). $X$ is a $\beta$-coboundary expander over $\F_2$ if 
    \[\forall f \in \F_2^E,~\Pr[\mathcal{T} \text{ rejects }f] \ge \beta\cdot dist(f,B^1),\]
    where $dist(f,B^1) = \min_{c\in B^1}\Pr_{e\sim \mathrm{Unif}(E)}[f(e)\neq c(e)]$ is the relative distance from $f$ to the closest codeword in $B^1$. 
\end{definition}
Observe that the triangles are parity checks of \(B^1\), and so if the input \(f\) is a codeword, the test always accepts. A coboundary expander's local test should reject inputs $f$ which are ``far'' from the code with probability proportional to their distance from it. The proportionality constant \(\beta\) captures the quality of the tester (and coboundary expansion). We note that although the triangles are \emph{some} parity checks of the code $B^1$, they do not necessarily span \emph{all} parity checks. In such cases, \(X\) is not a coboundary expander for any \(\beta > 0\). 

 

\begin{definition}
    A set of $2$-dimensional simplicial complexes $\{X_n\}$ is a family of coboundary expanders if there exists some $\beta >0$ such that for all $n$, $X_n$ is a $\beta$-coboundary expander. 
\end{definition}

\paragraph{Applications of coboundary expansion}We mention three different motivations to the study of coboundary expansion: from agreement testing, discrete geometry, and algebraic topology.

Coboundary expanders are useful for constructing derandomized agreement tests for the low-soundness (list-decoding) regime. As we mentioned before, agreement tests (also known as derandomized direct product tests) are a strong property test that arise naturally in many PCP constructions \cite{dinur2007pcp, Raz-parrep,GolSaf97,ImpagliazzoKW2012} and low degree tests \cite{RuSu96,ArSu,RazS1997}. They were introduced by \cite{GolSaf97}. In these tests one gets oracle access to `local' partial functions on small subsets of a variable set, and is supposed to determine by as few queries as possible, whether these functions are correlated with a `global' function on the whole set of variables. The `small soundness regime', is the regime of tests where one wants to argue about closeness to a `global' function even when the `local' partial functions pass the test with very small success probability (see \cite{DG08} for a more formal definition). Works such as \cite{DG08,ImpagliazzoKW2012,DL17} studied this regime over complete simplicial complexes and Grassmanians, but until recently these were essentially the only known examples. Works by \cite{GotlibK2022,BafnaM2023,DiksteinD2023agr} reduced the agreement testing problems to unique games constraint satisfaction problems, and showed existence of a solution via coboundary expansion (over some symmetric groups), following an idea by \cite{DinurM2019}. This lead to derandomized tests that come from bounded degree complexes \cite{DiksteinDL2024,BafnaM2024}.

In discrete geometry, a classical result asserts that for any $n$ points in the plane, there exists a point that is contained in at least a constant fraction of the $\binom{n}{3}$ triangles spanned by the $n$ points \cite{Barany1982} \footnote{This can also be generalized to any dimension by replacing triangles with simplices \cite{BorosF1984}.}. In the language of complexes, for every embedding of the vertices of the \emph{complete} $2$-dimensional complex to the plane, such a \emph{heavily covered point} exists. One can ask whether such a point exists even when one allows the edges in the embedding to be arbitrary continuous curves. If for every embedding of a complex \(X\) to the plane (where the edges are continuous curves), there exists a point that lays inside a constant fraction of the triangles of \(X\), we say \(X\) has the \emph{topological overlapping property}. The celebrated result by Gromov states that the complete complex indeed has this property \cite{Gromov2010}. It is more challenging to prove this property for a given complex.

In \cite{Gromov2010}, Gromov asked whether there exists a family of bounded-degree complexes with the topological overlapping property. Towards finding an answer, Gromov proved that \emph{every} coboundary expander has this property. This question has been extremely difficult. An important progress is made in \cite{DotterrerKW2018}, where they defined \emph{cosystolic expansion}, a relaxation of coboundary expansion, and proved that this weaker expansion also implies the topological overlapping property. The problem was eventually resolved in \cite{KaufmanKL2016} for dimension $2$ and in \cite{EvraK2016} for dimension $>2$ where the authors show that the \cite{LubotzkySV2005} Ramanujan complexes are bounded-degree cosystolic expanders. 

Coboundary expansion has other applications in algebraic topology as well. Linial, Meshulam, and Wallach defined coboundary expansion independently, initiating a study on high dimensional connectivity and cohomologies of random complexes \cite{LinialM2006,MeshulamW2009}. Lower bounding coboundary expansion turned out to be a powerful method to argue about the vanishing of cohomology in high dimensional random complexes \cite{DotterrerK2012,HoffmanKP2017}.

%
%

\paragraph{Known coboundary expanders}
Most of the applications mentioned above call for complexes that are simultaneously coboundary expanders and non-trivial local spectral expanders. So far there are no known constructions of such complexes with arbitrarily small polynomial degree even in dimension $2$ \footnote{By coboundary expanders, we are referring to complexes that have coboundary expansion over \textit{every} group, not just \(\F_2\).}. Here we summarize some known results.

Spherical buildings of type \(A_d\) are dense complexes that appear as the links of the Ramanujan complexes in \cite{LubotzkySV2005}. The vertex set of a spherical building consists of all proper subspaces of some ambient vector space $\F_q^d$. \cite{Gromov2010} proved that spherical buildings are coboundary expanders (see also \cite{LubotzkyMM2016}).
Geometric lattices are simplicial complexes that generalize spherical buildings. They were first defined in \cite{KozlovM2019}.
Most recently, \cite{DiksteinD2023cbdry} show that geometric lattices are coboundary expanders.
\cite{KaufmanO2021} show that the $2$-dimensional vertex links of the coset geometry complexes from \cite{kaufmanO2018} are coboundary expanders.

If we restrict our interest to coboundary expanders over a single finite group instead of all groups, bounded degree coboundary expanders that are local spectral expanders are known. \cite{KaufmanKL2016} solved this for the \(\F_2\) case conditioned on a conjecture by Serre. \cite{ChapmanL2023} constructed such complexes for the \(\F_2\) unconditionally, and their idea was extended by \cite{DiksteinDL2024,BafnaM2024} to every \emph{fixed} finite group, which lead to the aforementioned agreement soundness result. Unfortunately, their approach cannot be leveraged to constructing a coboundary expander with respect to all groups simultaneously.

 We note that the coboundary expanders mentioned above are also local spectral expanders and that if we do not enforce the local spectral expansion property, coboundary expanders are trivial yet less useful.
 
 \Ynote{Please make sure my footnote is correct}To the best of our knowledge, all known constructions of $2$-dimensional coboundary expanders (over all groups) that are also non-trivial local spectral expanders have vertex degree at least $N^{c}$ for some fixed $c>0$, where $N$ is the number of vertices (see e.g.\ \cite{LubotzkySV2005,kaufmanO2018}).\footnote{To the best of our knowledge, this constant is \(c=0.5\).} The diameter of all of those complexes is at most a fixed constant, which implies this lower bound on the maximal degree. In this work, we give the first such construction with arbitrarily small polynomial degree.

\paragraph{Our result} We now state our main result in this subsection. For every integer $k>0$, and constant $\varepsilon \in (0,\frac{1}{2}]$ we construct a family of $k$-dimensional simplicial complexes $\{X_{\varepsilon,n}\}_n$ called the Johnson complexes. For any $k$-dimensional simplicial complex $X$, we use the notation $(X)^{\le 2}$ (the $2$-skeleton of $X$) to denote the $2$-dimensional simplicial complex consists of the vertex, edge, and triangle sets of $X$.
\begin{theorem}[Informal version of \pref{thm:coboundary-Johnson-complex}, \pref{lem:Johnson-complex-spectral-expansion}, and \pref{lem:cone-for-link-X}] \label{thm:inf-Johnson-complex-expansions}
For any integer $k\ge 2$, $\varepsilon \in (0,\frac{1}{2}]$, the $2$-skeletons of the Johnson complexes $X_{\varepsilon,n}$ are $\left(\frac{1}{2} - \frac{\varepsilon}{2}\right)$-local spectral expanders and $\Omega(\varepsilon)$-coboundary expanders (over every group $\Gamma$).

Moreover, if $k\ge 3$, the $2$-skeletons of $X_{\varepsilon,n}$s' vertex links are $\frac{1}{85}$-coboundary expanders . 
\end{theorem}

 Fix $\varepsilon \in (0,\frac{1}{2}]$. For every $n$ satisfying $4\;|\;\varepsilon n$, we briefly describe the $2$-skeletons of $X_{\varepsilon,n}$. The underlying graph of $X_{\varepsilon,n}$ is simply the Cayley graph $Cay(H; S_\varepsilon)$ where $H\subset \F_2^n$ consists of vectors with even Hamming weight, and $S_\varepsilon$ is the set of vectors with Hamming weight $\varepsilon n$. Thus the number of vertices in the graph is $2^{n-1}$ while the vertex degree is $\binom{n}{\varepsilon n} < 2^{nh(\varepsilon)}$ \footnote{Here $h$ is the binary entropy function.}. The triangles are given by:
\[X_{\varepsilon,n}(2) = \sett{ \{x, x+s_1, x+s_2\} }{x\in \F_2^n, s_1,s_2, s_1+s_2 \in S_\varepsilon}.\]

Observe that the link of every vertex is isomorphic to the classically studied Johnson graph \(J(n,m,m/2)\) of \(m\)-subsets of \([n]\), that are connected if their intersection is \(m/2\) (for \(m=\varepsilon n\)). We will use some of the known properties of this graph in our analysis.

We additionally show that when $2^k \;|\; \varepsilon n$, we can extend the above construction to get $k$-dimensional simplicial complexes $X_{\varepsilon,n}$ with the exact same vertex set, edge set, and triangle set as defined above. Moreover we show that for every integer $0\le m \le k-2$, the link of an $m$-face in the complex is also a non-trivial local spectral expander (\pref{lem:Johnson-complex-spectral-expansion}).

\begin{remark}
    Let $N = 2^{n-1}$.
    We note that all vertices in $X_{\varepsilon,n}$ are in $\mathrm{poly}(N)$ edges and all vertices in the links of $X_{\varepsilon, n}$ are in $\mathrm{poly}(N)$ edges. 
    We remark that using a graph sparsification result due to \cite{ChungH07}, we can randomly subsample the triangles in $X_{\varepsilon,n}$ to obtain a random subcomplex which is still a local spectral expander with high probability but whose links have vertex degree $\mathrm{polylog}(N)$. A detailed discussion can be found in \pref{app:sparsify-JC}.
\end{remark}

Before giving a high-level overview of the proof for \pref{thm:inf-Johnson-complex-expansions}, we describe a general approach for showing coboundary expansion called the \emph{cone method}. Appearing implicitly in \cite{Gromov2010}, it was used in a variety of works (see \cite{LubotzkyMM2016,KozlovM2019,KaufmanO2021} as a partial list). We also take this approach to show that the Johnson complexes are coboundary expanders.

Even though most of our results on coboundary expansion are for the Johnson complexes, we can also use the methods in this paper to prove that the vertex links of the \cite{Golowich2023} Cayley local spectral expanders are coboundary expanders. This implies that these Cayley local spectral expanders are cosystolic expanders (see \cite{KaufmanKL2016,DiksteinD2023cbdry}), the relaxation of coboundary expansion mentioned above. This makes them useful for constructing agreement tests and complexes with the topological overlapping property. We note that \cite{Golowich2023} observed that these complexes are not coboundary expanders so this relaxation is necessary. A definition of cosystolic expansion and a detailed proof can be found in \pref{sec:coboundary-Golowich}.

\paragraph{Cones and isoperimetry}
Recall first the definition of $\beta$-coboundary expansion above. We start with the code \(B^1\) of edge functions $f$ which arise from vertex labelings by elements of some group $\Gamma$. This implies that composing (in order) the values of $f\in B^1$ along the edges of any cycle will give the identity element $id \in \Gamma$. This holds in particular for triangles, which are our ``parity checks''.

We digress to discuss an analogous situation in geometric group theory, of the \emph{word problem} for finitely presented groups. In this context, we are given a word in the generators of a group and need to check if it simplifies to the identity under the given relations. In our context the given word is the labeling of $f$ along some cycle, and a sequence of relations (here, triangles) is applied to successively ``contract'' this word to the identity. The tight upper bound on the number of such contractions in terms of the length of the word (called the \emph{Dehn function} of the presentation), captures important properties of the group, e.g. whether or not it is hyperbolic. This ratio between the contraction size and the word length is key to the cone method. 

One convenient way of capturing the number of contractions is the so-called ``van Kampen diagram'', which simply ``tiles'' the cycle with triangles in the plane (all edges are consistently labeled by group elements). The van Kampen lemma ensures that if a given word can be tiled with $t$ tiles, then there is a sequence of $t$ contractions that reduce it to the identity \cite{VanKampen1933}. The value in this viewpoint is that tiling is a static object, which can be drawn on the plane, and allows one to forget about the ordering of the contractions.
We will make use of this in our arguments, and for completeness derive the van Kampen lemma in our context. Note that a bound on the minimum number of tiles (a notion of area) needed to cover any cycle of a given length (a notion of boundary length) can be easily seen as an \emph{isoperimetric inequality}! 
The cone method will require such isoperimetric inequality, and the Dehn function gives an upper bound on it.

Back to $\beta$-coboundary expansion. Here the given $f$ may not be in $B^1$, and we need to prove that if it is ``far'' from $B^1$, then the proportion of triangles which do not compose to the identity on $f$ will be at least $\beta \dist(f,B^1)$. The cone method localizes this task. To use this method, one needs to specify a family of cycles (also called a cone) in the underlying graph of the complex. Gromov observed that the complex's coboundary expansion has a lower bound that is inverse proportional to the number of triangles (in the complex) needed to tile a cycle from this cone \cite{Gromov2010}. This number is also referred to as the cone area.
Thus, bounding the coboundary expansion of the complex reduces to computing the cone area of some cone. Needless to say, an upper bound on the Dehn function, which gives the worst case area for \emph{all} cycles, certainly suffices for bounding the cone area of any cone.\footnote{We note that in a sense the converse is also true - computing the cone area for cones with cycles of ``minimal'' length suffices for computing the Dehn function. We indeed give such a strong bound.}

\paragraph{Proof overview}It remains to construct a cone in $X_{\varepsilon,n}$ and upper bound its cone area. We now provide a high level intuition for our approach. 
First observe that the diameter of the complex $X=X_{\epsilon,n}$ is $\Theta(1/\epsilon)$.
The proof then proceeds as follows. 
\begin{enumerate}
    \item  We move to a denser $2$-dimensional complex $X'$ whose underlying graph is the Cayley graph $Cay(H; S_{\le \varepsilon})$ where $H\subset \F_2^n$ consists of vectors with even Hamming weight, and $S_{\le \varepsilon}$ is the set of vectors with even Hamming weight $\le \varepsilon n$. The triangle set consists of all $3$-cliques in the Cayley graph. We note that $X'$ has the same vertex set as $X$ but has more edges and triangles than the Johnson complex.
    
    \item Then we carefully construct a cone $\mathcal{A}'$ in $X'$ with cone area $\Theta(1/\epsilon)$.

    \item We show that every edge in $X'$ translates to a length-$2$ path in $X$, and every $3$-cycle (boundary of a triangle) in $X'$ translates to a $6$-cycle in $X$ which can be tiled by $O(1)$ triangles from $X$. 

    \item We translate the cone $\mathcal{A}'$ in $X'$ into a cone $\mathcal{A}$ in $X$ by replacing every edge in the cycles of $\mathcal{A}'$ with a certain length-$2$ path in $X$. Thus every cycle $C \in \mathcal{A}$ can be tiled by first tiling its corresponding cycle $C' \in \mathcal{A}'$ with $O(1/\epsilon)$ $X'$ triangles and then tiling each of the $X'$ triangles with $O(1)$ $X$ triangles. Thereby we conclude that $\mathcal{A}$ has cone area $\Theta(1/\epsilon)$.
\end{enumerate}

\paragraph{Local spectral expansion and comparison to random geometric complexes}
These Johnson complexes also have non-trivial local spectral expansion. While they do not achieve arbitrarily small \(\lambda\)-local spectral expansion, they do pass the \(\lambda < \frac{1}{2}\) barrier. From a combinatorial point of view, \(\lambda < \frac{1}{2}\) is an important threshold. For \(\lambda \geq \frac{1}{2}\), there are elementary bounded-degree constructions of \(\lambda\)-local spectral expanders \cite{Conlon2019,ChapmanLP2020,ConlonTZ2020,LiuMY2020,Golowich2021} but these fail to satisfy any of the local-to-global properties in HDXs. For \(\lambda < \frac{1}{2}\) all these constructions break. For this regime, all bounded degree constructions rely on algebraic tools. This is not by accident; below \(\frac{1}{2}\) there are a number of high dimensional global properties suddenly begin to hold. For instance, a theorem by \cite{Garland1973} implies that when \(\lambda < \frac{1}{2}\) all real cohomologies of the complex vanish. Another example is the `Trickle-Down' theorem by \cite{oppenheim2018local} which says that the expansion of the skeleton \((V,E)\) is non-trivial whenever \(\lambda < \frac{1}{2}\). These are strong properties, and we do not know how produce them in elementary constructions. Therefore, when constructions of local spectral expanders are considered non-trivial only when \(\lambda < \frac{1}{2}\).

To show local spectral expansion, we first note that the Johnson complex is symmetric for each vertex. Hence, it suffices to focus on the link of $0 \in \F_2^n$. Because the triangle set $T$ is well-structured, we can show that the link graph of $0$ is isomorphic to tensor products of Johnson graphs (\pref{prop:link-structure}). This allows us to use the theory of association schemes to bound their eigenvalues. 

We also note that since all boolean vectors in $\{0,1\}^n$ lie on a sphere in $\R^n$ centered at $[\frac{1}{2},\frac{1}{2},\dots,\frac{1}{2}]$, the Johnson complex $X_{\varepsilon,n}$ can be viewed as a geometric complex whose vertices are associated with points over a sphere and two vertices form an edge if and only if their corresponding points' $L_2$ distance is $\sqrt{\varepsilon n}$. Previously \cite{LiuMSY2023} prove that randomized $2$-dimensional geometric complexes are local spectral expanders. Comparing the two constructions, we conclude that Johnson complexes are sparser local spectral expanders than random geometric complexes.

\subsection{The common structure between the two constructions} \label{sec:intro-common-struct}
Taking a step back, we wish to highlight that the two constructions in our paper share a common structure. Both of these constructions can be described via `induced Hadamard encoding'. For a (not necessarily linear) function \(s:\F_2^d \to \F_2^n\), the `induced' Hadamard encoding is the linear map \(\widehat{s}:\F_2^d \to \F_2^n\) given by
\[\widehat{s}(x) \coloneqq \sum_{v \in \F_2^d} \iprod{x,v} s(v),\]
where \(\iprod{x,v} = \sum_{i=1}^d x_i v_i\) and the sum is over \(\F_2\). Our two constructions of Cayley HDXs take the generating set of the Cayley graph to be all bases of spaces \(\Img(\widehat{s}) \subseteq \F_2^n\), for carefully chosen functions \(s\) as above; the choice determines the construction. We note that the ``orthogonality'' of vectors of the Hadamard code manifests itself very differently in the two constructions: in the Johnson complexes it can be viewed via the Hamming weight metric, while in the other construction it may be viewed via the matrix rank metric. We connect the structure of links in our constructions, to the restrictions of these induced Hadamard encodings to affine subspaces of \(\F_2^n\). Using this connection we show that one can decompose the link into a tensor product of simpler graphs, which are amenable to our analysis. A special case of this observation was also used in the analysis of the complexes constructed by \cite{Golowich2023}.  

\subsection{Open questions} \label{sec:intro-open-questions}
\paragraph{Local spectral expanders} As mentioned above, we do not know how sparse can a Cayley local spectral expander be. To what extent can we close the gaps between the lower and upper bounds for various types of abelian Cayley HDXs? In particular, can we show that any nontrivial abelian Cayley expanders \snote{over $\F_2^n$ or $\F^n$?} must have degree $\omega(n)$? Conversely, can we construct Cayley HDXs where the degree is upper bounded by some polynomial in \(n\)? Limits on our iterated sparsification technique would also be of interest.

So far, the approach of constructing Cayley local spectral expanders over $\F_2^n$ by sparsifying Grassmann posets yields the sparsest known such complexes. In contrast to the success of this approach, we have limited understanding of its full power. To this end, we propose several questions: what codes can be used in place of the Hadamard codes so that the sparsification step preserves local spectral expansion?  Could this approach be generalized to obtain local spectral expanders over other abelian groups? As mentioned above, we know that the approach we use can not give a complex of polynomial degree in \(n\) without introducing a new idea.

\paragraph{Coboundary expanders}Another fundamental question regards the isoperimetric inequalities we describe above. In this paper and many others, one uses the approach pioneered by Gromov \cite{Gromov2010} that applies isoperimetric inequalities to lower bound coboundary expansion. A natural question is whether an isoperimetric inequality is also \emph{necessary} to obtain coboundary expansion. An equivalence between coboundary expansion and isoperimetry will give us a simple alternative description for coboundary expansion. A counterexample to such a statement would motivate finding alternative approaches for showing coboundary expansion.

We call our family of $2$-dimensional simplicial complexes that are both local spectral expanders and $1$-dimensional coboundary expanders Johnson complexes. This construction can be generalized to yield families of $k$-dimensional Johnson complexes that are still local spectral expanders. However, it remains open whether they are also coboundary expanders in dimension $>1$.

\subsection{Organization of this paper}
In \pref{sec:prelims}, We give background material on simplicial complexes and local spectral expansion. We also define Grassmann posets - the vector space analogue for HDXs - which we use in our constructions Other elementary background on graphs and expansion is given there as well.

In \pref{sec:construction-subpoly-HDX}, we construct the subpolynomial degree Cayley complexes, proving \pref{thm:main-recursive-hdx-intro}. Most of this section discusses Grassmann posets and not Cayley HDXs, but we describe the connection between the two given by \cite{Golowich2023} in this section too. We deffer some of the more technical expansion upper bounds to \pref{app:proofs-of-expansion-matrix-graphs}.

In \pref{sec:spectral-Johnson} we construct the Johnson complexes which are both coboundary expanders and local spectral expanders. In this section we prove they are local spectral expanders, leaving coboundary expansion to be the focus of the next two sections. We also give a detailed comparison of this construction to the one in \cite{LiuMSY2023}, and discuss how to to further sparsify our complexes by random subsampling.

In \pref{sec:vk-lemma} we take a detour to formally define coboundary expansion and discuss its connection to isoperimetric inequalities. In this section we also give our version of the van-Kampen lemma, generalized to the setting of coboundary expansion. This may be of independent interest, as the van-Kampen lemma simplifies many proofs of coboundary expansion.

In \pref{sec:cob-exp-johnson} we prove that the Johnson complexes are coboundary expanders (\pref{thm:coboundary-Johnson-complex}), and local coboundary expanders (\pref{thm:coboundary-Johnson-link}). We also prove that links in the construction of \cite{Golowich2023} are coboundary expanders.

In \pref{sec:induced-Grassmann} we show that both the Johnson complex and the Matrix complexes are special cases of a more general construction. In this section we show that the two complexes have a similar link structure, which is necessary to analyze the local spectral expansion in both Cayley HDXs.

The appendices contain some of the more technical claims for ease of reading. In \pref{app:degree-lower-bound} we also give a lower bound on the degree of Cayley local spectral expanders, based on the degree of the link.

\subsection{Acknowledgments}
We thank Louis Golowich for helpful discussions and Gil Melnik for assistance with the figures.
\section{Preliminaries} \label{sec:prelims}
We denote by \([k]=\set{1,2,\dots,k}\). 

\subsection{Graphs and expansion} \label{sec:graph-expansion}
Let \(G=(V,E,\mu)\) be a weighted graph where \(\mu:E \to \R_{\geq 0}\) are the weights. In this paper we assume that \(\mu\) is a probability distribution \(\mu:E \to (0,1]\) that gives every edge a positive measure. When \(mu\) is clear from context we do not specify it. We extend \(\mu\) to a measure on vertices where \(\mu(v) = \sum_{u \sim v}\mu(uv)\). When \(G\) is connected this is the stationary distribution of the graph. The adjacency operator of the graph \(A\) sends \(f:V \to \R\) to \(Af: V \to \R\) with \(Af(v) = \Ex[u \sim v]{f(u)} = \sum_{u \sim v} \frac{\mu(u)}{\mu(v)} f(u)\). We denote by \(\iprod{f,g} = \Ex[v \sim \mu]{f(v)g(v)}\) the usual inner product, and recall the classical fact that \(A\) is self adjoint with respect to this inner product. We denote by \(\lambda(A)\) its second largest eigenvalue in absolute value. Sometimes we also write \(\lambda(G)\) instead of \(\lambda(A)\). It is well known that \(\lambda(G) \leq 1\).

\begin{definition}[Expander graph]
    Let \(\lambda \geq 0\). A graph \(G\) is a \(\lambda\)-expander if \(\lambda(G)\leq \lambda)\).
\end{definition}

This definition is usually referred to as \(\lambda\)-two-sided spectral expansion. We note that the one-sided definition where one only gets an upper bound on the eigenvalues of \(G\) is also interesting, but is out of our paper's scope (see e.g.\ \cite{hoory2006expander}).

When \(G=(L,R,E)\) is a bipartite graph, we also have inner products on each side defined by \(f,g:L \to \R\), \(\iprod{f,g}=\Ex[v \in L]{f(v)g(v)}\) with respect to \(\mu\) restricted to \(L\) (and similarly to real valued functions on \(R\)). The left-to-right bipartite graph operator \(A'\) that sends \(f:L \to \R\) to \(A'f:R \to \R\) with \(A'f(u) = \Ex[v \sim u]{f(v)}\) (resp.\ right-to-left bipartite operator). It is well known that \(\norm{A'|_{f \bot \one}}\) is equal to the second largest eigenvalue of the non bipartite adjacency operator \(A\) defined above (\emph{not} in absolute value). We also denote this quantity \(\lambda(G) \coloneqq \norm{A'|_{f \bot \one}}\) (it will be clear that whenever \(G\) is bipartite this is the expansion parameter in question).

\begin{definition}[Bipartite expander graph]
    Let \(\lambda \geq 0\). A \emph{bipartite} graph \(G\) is a \(\lambda\)-bipartite expander if \(\lambda(G) \leq \lambda\).
\end{definition}

We use the following standard auxiliary statements.
\begin{claim} \label{claim:close-distance}
    Let \(A,B\) be adjacency matrices of two graphs over the same vertex set that have the same stationary distribution. Let \(\varepsilon = \max_{v} \norm{A_v - B_v}_1 = \sum_{u \in V} \abs{A(v,u) - B(v,u)}\). Then \(\lambda(A) \leq \lambda(B) + \varepsilon\). \(\qed\)
\end{claim}


We will usually instantiate this claim with \(B\) being the complete graph with self cycles, i.e. the edge distribution is just two independent choices of vertices. This \(B\) has \(\lambda(B) = 0\) so we get this corollary.
\begin{corollary} \label{cor:closeness-to-uniform}
    Let \(A\) be an adjacency matrix of a weighted graph \(G=(V,E,\mu)\).
    \begin{enumerate}
        \item \(\lambda(A) \leq \max_{v \in V} \norm{A_v - \mu}_1 = \sum_{u \in V} \abs{A(v,u) - \mu(u)}\).
        \item In particular, if \(A\) is uniform over a set \(V' \subseteq V\) with \(\mu(V') \geq 1-p\) then \(\lambda(A) \leq 2p\). \(\qed\)
    \end{enumerate}
\end{corollary}

We also define a tensor product of graphs.
\begin{definition}
    Let \(G_1 = (V_1,E_1),G_2 = (V_2,E_2)\) be graphs. The tensor product \(G_1 \otimes G_2\) is a graph with vertices \(V_1 \times V_2\) and edges \((v_1,v_2) \sim (u_1,u_2)\) if \(\set{v_1,u_1} \in E_1\) and \(\set{v_2,u_2} \in E_2\). The measure on \(G_1\otimes G_2\) is the product of measures.
\end{definition}
We record the following well known fact.
\begin{equation} \label{eq:tensors}
    \lambda(G_1 \otimes G_2) = \max \set{\lambda(G_1), \lambda(G_2)}.
\end{equation}
and that in fact, if \(\lambda_1 \geq \lambda_2 \geq \dots \geq \lambda_n\) are the eigenvalues of \(G_1\) and \(\eta_1 \geq \eta_2 \geq \dots \geq \eta_m\) are the eigenvalues of \(G_2\), then the eigenvalues of \(G_1 \otimes G_2\) are the multiset \(\sett{\lambda_i \eta_j}{i \in [n], j \in [m]}\).

We also need the notion of a double cover.
\begin{definition}[Double cover] \label{def:double-cover}
    Let \(G = (V,E)\) be a graph and \(K_{1,1}\) the graph that contains a single edge. The double cover of \(G\) is the graph \(G \otimes K_{1,1}\). Explicitly, its vertex set is \(V \otimes \set{0,1}\) and we connect \(\set{(v,i),(u,j)}\) if \(\set{v,u} \in E\) and \(i \ne j\).
\end{definition}
The following observation is well known.
\begin{observation} \label{obs:double-cover-expansion}
    A graph \(G\) is a \(\lambda\)-expander if and only if its double cover is a \(\lambda\)-bipartite expander. \(\qed\)
\end{observation}

\subsubsection{A local to global lemma} \label{sec:local-to-global}
We will use the following `local-to-global' proposition due to \cite[Theorem 1.1]{Madras2002}. We prove it in \pref{app:loc-to-glob} to stay self contained. This lemma is an application of the Garland method \cite{Garland1973}, which is quite common in the high dimensional expander literature (see e.g.\ \cite{oppenheim2018local}). Let us set it up as follows.

Let \(G = (V,E,\mu)\) be a graph where \(\mu\) is the distribution over edges. A local-decomposition of \(G\) is a pair \(\tau = (\T,\nu_{\T})\) as follows. The set \(\T = \set{G_t=(V_t,E_t,\mu_t)}_{t \in \T}\) be such that \((V_t,E_t)\) are a subgraph of \(G\) and \(\mu_t\) is a weight distribution over \(E_t\). The distribution \(\nu_\T\) be a distribution over \(\T\). Assume that \(\set{v,u} \sim \mu\) is sampled according to the following process:
\begin{enumerate}
    \item Sample \(t \sim \mu_\T\).
    \item Sample \(\set{v,u} \sim \mu_t\).
\end{enumerate}

The local-to-global graph \(\B_\tau\) is the bipartite \(V\) and \(\T\) and whose edge distribution is 
\begin{enumerate}
    \item Sample \(t \sim \mu_\T\).
    \item Sample a vertex \(v \sim \mu_t\)\footnote{i.e. \(v\) is sampled with probability \(\mu_t(v) = \frac{1}{2}\Prob[e \sim \mu_t]{v \in e}\).}.
    \item Output \(\set{v,t}\).
\end{enumerate}
Note that the probability of \(v \in V\) under \(G\) and under \(\B_\tau\) (conditioned on the \(V\) side) is the same.

\begin{lemma}[{\cite[Theorem 1.1]{Madras2002}}]\label{lem:local-to-global}
    Let \(G,\tau,\B_\tau\) be as above. Let \(\lambda_2 = \lambda_2(\B_\tau)\) and \(\gamma = \max \sett{\lambda(G_t)}{t\in \T}\). Then \(\lambda(G) \leq \gamma + \lambda_2^2(1-\gamma)\). 
\end{lemma}

We will also use the one-sided version of \cite[Theorem 1.1]{Madras2002}. For a formal proof see \cite[Lemma 4.14]{DiksteinD2019}. We formalize it in our language.
\begin{lemma}\label{lem:bipartite-loc-to-glob}
    Let \(G = (L,R,E)\) be a bipartite graph. Let \(\tau,\T\) be a local decomposition. Let \(\B_{\tau,L}\) be the restriction of \(\B_{\tau}\) to the vertices of \((L,\T)\) (resp. \(\B_{\tau,R}\)). Then
    \[\lambda_2(G) \leq \max \sett{\lambda_2(G_t)}{t \in \T} + \lambda_2(B_{\T,L})\lambda_2(B_{\T,R}).\footnote{A tighter analysis following the proof of \pref{lem:local-to-global} will give an upper bound of \((1-\lambda(B_{\T,L})\lambda(B_{\T,R}))\max \sett{\lambda_2(G_t)}{t \in \T} + \lambda(B_{\T,L})\lambda(B_{\T,R})\). We do not make use of this fact and so we leave the statement as is.}\]
\end{lemma}


\subsection{Vector spaces} \label{sec:vec-spaces-defs}
Unless stated otherwise, all vector spaces are over \(\mathbb{F}_2\). For a vector space \(V\) and subspaces \(u_1,u_2 \subseteq V\) we say their intersection is trivial if it is equal to \(\set{0}\). We denote the sum \(u_1 + u_2 = \sett{x_1+x_2}{x_1 \in u_1, x_2 \in u_2}\). We write \(u_1 \oplus u_2\) to denote the sum of vectors \emph{and} to indicate that the sum is direct, or in other words, their intersection is trivial. We need the following claim.
    \begin{claim}\label{claim:basic-three-spaces}
        Let \(V\) be a vector space. Let \(u_1,u_2,u_3 \subseteq V\) be three subspaces such that \((u_1 \oplus u_2) \cap u_3 = \set{0}\). Then \(u_2 \cap (u_1 \oplus u_3) = \set{0}\).
    \end{claim}
    \begin{proof}
        Let \(x_2 = x_1 + x_3 \in u_2 \cap (u_1 \oplus u_3)\) where \(x_i \in u_i\). Thus \(x_3 = x_2-x_1 \in (u_1 \oplus u_2) \cap u_3\) which by assumption means that \(x_3=0\) and \(x_2=x_1\). As \(u_1 \cap u_2 = \set{0}\) it follows that \(x_2=x_1=0\) and the intersection is trivial. 
    \end{proof}
    
\subsection{Grassmann posets} \label{sec:prelims-Grassmann}
In this paper all vector spaces are with respect to \(\mathbb{F}_2\) unless otherwise stated. Let \(V\) be a vector space. The Grassmann partial ordered set (poset) \(\Gr(V)\) is the poset of subspaces of \(V\) ordered by containment. When \(V\) is clear we sometimes just write \(\Gr\).
\begin{definition}[Grassman Subposet] \label{def:grass-sbpt}
    Let \(n > 0\) \(V\) be an \(n\)-dimensional vector space and \(\ell \leq n\). An \(\ell\)-dimensional Grassmann subposet \(\Y\) is a subposet of \(\Gr(\mathbb{F}_2^n)\) such that.
\begin{enumerate}
    \item Every maximal subspace in \(\Y\) has dimension \(\ell\).
    \item If \(w \in \Y\) and \(w' \subseteq w\) then \(w' \in \Y\).
\end{enumerate}
\end{definition}
We sometimes write \(\dim(\Y) = \ell\). The \(i\)-dimensional subspaces in \(\Y\) are denoted \(\Y(i)\). The vector space \(V\) is sometimes called the ambient space of \(\Y\) (Formally there could be many such spaces, but in this paper we usually assume that \(V\) is spanned by the subspaces in \(\Y\)).

A measured subposet is a pair \((\Y,\mu)\) where \(\Y\) is an \(\ell\)-dimensional Grassmann subposet and \(\mu\) is a probability distribution over \(\ell\)-dimmensional subspaces such that \(\supp(\mu) = \Y(\ell)\). We extend \(\mu\) to a distribution on \emph{flags} of subspaces by sampling \(v_\ell \sim \mu\) and then sampling a flag \((v_1 \subsetneq v_2 \subsetneq \dots v_\ell)\) uniformly at random. Some times we abuse notation and write \(v_i \sim \mu\) where \(v_i \in \Y(i)\) is the marginal over said flags. We sometimes say that \(\Y\) is measured, without explicitly annotating \(\mu\) for brevity. When we write \(\Gr\) we mean that the distribution is uniform.

The containment graph is a bipartite graph with sides \((\Y(i),\Y(j))\) and whose edges \(\set{w,w'}\) are such that \(w \subseteq w'\). There is a natural distribution on edges of this graph induced by the marginal of \((v_i,v_j)\) of the flag above.

Let \(\Y\) be an \(\ell\)-dimensional Grassmann poset. For \(i < \ell\) the Grassmann poset \(\Y^{\leq i} = \bigcup_{j=0}^i \Y(j)\) is the poset that contains every subspace in \(\Y\) of dimension \(\leq i\).

A link of a subspace \(w \in \Y(i)\) is denoted \(\Y_w\). It is the poset containing \(\Y_w = \sett{w' \supseteq w}{w' \in \Y}\). If \((\Y,\mu)\) is measured the distribution over \(\Y_w\) is the distribution \(\mu\) conditioned on the above flag sampled \(v_i = w\). We note that if \(V\) is the ambient space of \(\Y\) then \(\Y_w\) is (isomorphic to) a Grassmann subposet with ambient space \(V/w\) by the usual quotient map.

For and \(\ell\)-dimensional subposet \(\Y\), \(i \leq \ell-2\) and \(w \in \Y(i)\), the link \(\Y_w = \sett{w' \supseteq w}{w' \in \Y}\) and we denote \(\Y_w(j) = \sett{w' \supseteq w}{w' \in \Y(j)}\). The underlying graph of the link is the following graph (which we also denote by \(\Y_w\) when there is no confusion). The vertices are
\[U_w = \sett{v \in V}{\sp(v) + w \in \Y(i+1)}\]
and the edges are
\[E_w = \sett{\set{v,v'}}{\sp(v,v') + w \in \Y(i+2)}.\]

We note that sometimes the link graph is defined by taking the vertices to be \(\Y_w(i+1)\) and the edges to be all \(\set{w',w''}\) such that \(w'+w'' \in \Y(i+2)\). These two graphs are closely related: 
\begin{claim} \label{claim:smaller-vs-bigger-link-defs}
    Let \(\Y\) be Grassmann poset. \(w \in \Y\) and let \(\Y_w=(U_w,E_w)\) be the link as defined above. Let \(G = (V,E)\) be the graph whose vertices are \(V = \Y_w(i+1)\) and whose edges are \(\set{w',w''}\) such that \(w'+w'' \in \Y(i+2)\). Then \(\lambda(Y_w) = \lambda(G)\).
\end{claim}

\begin{proof}
    Let \(K_{2^i}\) be the complete graph over \(2^i\) vertices including self loops. We prove that \(G \cong Y_w \otimes K_{2^i}\). As \(\lambda(K_{2^i}) = 0\) it will follow by \eqref{eq:tensors} that \(\lambda(G) = \max \set{\lambda(Y_w),0} = \lambda(Y_w)\). For every \(w' \in Y_w(i+1)\) we order the vectors of \(w' \setminus w\) according to some arbitrary ordering \(w' \setminus w = \set{v_1,v_2,\dots,v_{2^i}}\). Our isomorphism from \(G\) to \(Y_w \otimes K_{2^i}\) sends \(v_j \mapsto (w',j)\) where \(w'= w+\sp (v_j)\) and the \(j\) in the subscript is according to the ordering of \(w' \setminus w\). It is direct to check that this is a bijection.

    Moreover, \(v \sim v'\) in \(G\) if and only if \(w + \sp(v,v') \in \Y(i+2)\). This span is equal to \(w + \sp(v,v') = (w+ \sp(v)) + (w+ \sp(v'))\), therefore \(v \sim v'\) if and only if the left coordinates of the images \((w+\sp(v),j_1)\) and \((w+\sp(v'),j_2)\) are connected in \(Y_w\). As \(j_1 \sim j_2\) in the complete graph with self loops, it follows that \(v \sim v'\) in \(G\) if and only if the corresponding images are connected in \(Y_w \otimes K_{2^i}\). The claim follows.
\end{proof}
    

Thus we henceforth use our definition of a link.

We say that such a subposet is a \(\lambda\)-expander if for every \(i \leq \ell-2\) and \(w \in \Y(i)\) the \(\Y_w\) is a \(\lambda\)-two sided spectral expander. We also state the following claim, whose proof is quite standard.

\begin{claim} \label{claim:link-in-Grassmann-expands}
    The poset \(\Y = \Gr(\F_2^n)^{\leq d}\) is \(\frac{1}{2^{n-d}-1}\)-expanding.
\end{claim}
\begin{proof}[Sketch]
    Let \(w \in \Y(i)\). By \pref{claim:smaller-vs-bigger-link-defs}, to analyze the expansion of the link of \(w\), it is enough to consider the graph whose vertices are the subspaces \(w' \in \Y_{w}(i+1)\) and whose edges are \(w' \sim w''\) if \(w' + w'' \in \Y_w (i+2)\). In this case, it is easy to see that this graph is just the complete graph (without self loops) with \(2^{n-i}-1\) vertices. The claim follows by the expansion of the complete complex.
\end{proof}

The following claim was essentially proven in \cite{DiksteinDFH2018}.
\begin{claim}[\cite{DiksteinDFH2018}] \label{claim:loc-to-glob-grassmann}
    Let \(\Y\) be a Grassmann poset that is \(\lambda\)-expander. Then for every \(i \leq j\) and \(w \in \Y\), the containment graph between \(\Y_w(i)\) and \(\Y_w(j)\) is a \(\left (0.61 + j\lambda \right )^{\frac{1}{2}(j-i)}\)-bipartite expander.
\end{claim}
This claim was proven \cite{DiksteinDFH2018} without calculating the explicit constant \(0.61\) hence we reprove it in \pref{app:grassmann-containment} just to stay self contained.

\subsection{The matrix domination poset} \label{sec:def-of-domination-relation}
In this section we develop the necessary preliminaries on the matrix domination poset. This poset is essential for proving that the construction given in \pref{sec:construction-subpoly-HDX} sufficiently expands. These posets were a crucial component in the proof of \cite{Golowich2023}. However, we believe that they are also interesting on their own right, perhaps being another interesting partially ordered set that one could study as in \cite{DiksteinDFH2018,kaufmanT2021garland}.

Let \(\MP\) be \(n \times n\) matrices over \(\F_2\) for some fixed \(n\). Let \(\MP(i) = \sett{A \in \MP}{\rank(A) = i}\). The partial order on this set is the domination order.
\begin{definition}[Matrix domination partial order]
    Let \(A,B \in \MP\). We say that \(A\) \emph{dominates} \(B\) if \(\rank(A) = \rank(B) + \rank(A-B)\). In this case we write \(B \leq A\).
\end{definition}

Another important notation is the direct sum.
\begin{definition}[Matrix direct sum]
    We say that two matrices \(A\) and \(B\) have \emph{a direct sum} if \(\rank(A) + \rank(B) = \rank(A+B)\). In this case we write \(A \oplus B\).
\end{definition}
We note that this definition is equivalent to saying that \(A,B \leq A+B\).
We also clarify that when we write \(A \oplus B\), this refers both to the object \(A+B\), and is a statement that the sum is direct \(\rank(A) + \rank(B) = \rank(A+B)\). This is similar to the use of \(\oplus\) when describing a sum of subspaces \(U_1,U_2 \subseteq \F_2^n\); the term \(U_1 \oplus U_2\) refers both to the sum of subspaces, and is a statement that \(U_1 \cap U_2 = \set{0}\). We will see below that \(A \oplus B\) is equivalent to \(\row(A) \cap \row(B) = \set{0}, \col(A) \cap \col(B)= \set{0}\), so we view this notation as a natural analogue of what happens in subspaces.

In \pref{sec:construction-subpoly-HDX} we define and analyze expansion in graphs coming from this the domination order Hasse diagram. These graphs appear in the analysis of the construction in \pref{sec:construction-subpoly-HDX}. Before doing so, let us state some preliminary properties of this partial order, all of which are proven in \pref{app:removing-lower-matrix}. More specifically, we will show that this is indeed a partial ordering (as already observed in \cite{Golowich2023}), discuss the properties of pairs of matrices that have a direct sum, and study structures of \emph{intervals} in this partially ordered set (see definition below). 

\begin{claim} \label{claim:domination-is-partial-order}
    The pair \((\MP,\leq)\) is a partial ordered set. 
\end{claim}

We move on to characterize when for \(A\) and \(B\), \(A \oplus B\), that is, when the sum of \(A\) and \(B\) is direct. This will be important in the subsequent analyses.
\begin{claim} \label{claim:trivially-intersecting spaces}
    For any matrices \(A,B\) the following are equivalent:
    \begin{enumerate}
        \item \(\rank(A+B) = \rank(A) + \rank(B)\), i.e.\ \(A \oplus B\).
        \item The spaces \(\row(A) \cap \row(B)\) and \(\col(A) \cap \col(B)\) are trivial.
        \item The spaces \(\col(A+B) = \col(A) \oplus \col(B)\) and \(\row(A+B) = \row(A) \oplus \row(B)\).
        \item There exists a decomposition \(A = \sum_{j=1}^{\rank(A)} e_j \otimes f_j\) and \(B = \sum_{j=1}^{\rank(B)}h_j \otimes k_j\) such that \(\set{e_j} \dunion \set{h_j}\) and \(\set{f_j} \dunion \set{k_j}\) are independent sets. Indeed, spanning \(row(A+B)\) and \(col(A+B)\).
    \end{enumerate}
\end{claim}

We will also use this elementary (but important) observation throughout.
\begin{claim}[Direct sum is associative] \label{claim:three-term-direct-sum}
    Let \(A,B,C \in \MP\). If \(A\) and \(B\) have a direct sum, and \(C\) has a direct sum with \(A \oplus B\), then \(A\) has a direct sum with \(C \oplus B\). Stated differently, if \(A \oplus B\) and \(C \oplus (A \oplus B)\), then \(A \oplus (C \oplus B)\). The same holds for any other permutation of \(A,B\) and \(C\).
\end{claim}
Henceforth, we just write \(A_1 \oplus A_2 \oplus A_3\) to indicate that \(A_i \oplus (A_j \oplus A_k)\) for any three distinct \(i,j,k\).

Finally, we will analyze sub-posets of \(\MP\) that correspond to local components, the same way that we have links in simplicial complexes and Grassmann posets.
\begin{definition}[Intervals]
    Let \(M_1 \leq M_2\) be two matrices. The \emph{interval} between \(M_1\) and \(M_2\) is the sub-poset of \(\MP\) induced by the set 
    \[\MP_{M_1}^{M_2} = \sett{A}{M_1 \leq A \leq M_2}.\]
    Similar to the above, we also denote by \(\MP_{M_1}^{M_2}(i)\) the matrices in \(\MP_{M_1}^{M_2}\) of rank \(i\). When \(M_1 = 0\) we just write \(\MP^{M_2}\).
\end{definition}

The main observation we need on these intervals is this.
\begin{proposition} \label{prop:iso-of-posets}
    Let \(M_1 \leq M_2\) be of ranks \(m_1 \leq m_2\). The sub-posets \(\MP_{M_1}^{M_2}\) and \(\MP^{M_2-M_1}\) are isomorphic by \(A \overset{\psi}{\mapsto} A' = A - M_1\). This isomorphism has the property that
        \[rank(A') = rank(A) - m_1.\]
\end{proposition}

Another sub-poset we will be interested in is the subposet of matrices that are ``disjoint'' from some \(A \in \MP\). That is,
\begin{definition}
    Let \(A \in \MP\). The sub-poset \(\MP_{\cap A}\) is the sub-poset of \(\MP\) the consists of all \(\sett{B \in \MP}{B \oplus A}\).
\end{definition}
We note that by \pref{claim:trivially-intersecting spaces} \(\MP_{\cap A} = \sett{B \in \MP}{\row(B) \cap \row(A), \col(B) \cap \col(A) = \set{0}}\).
\begin{claim} \label{claim:iso-second-type-of-lower-loc}
    Let \(A \in \MP(i)\). Then \(\phi: \MP_{A} \to \MP_{\cap A}, \; \phi(B)=B-A\) is an order preserving isomorphism.
\end{claim}

\subsection{Simplicial complexes and high dimensional expanders}
A pure \(d\)-dimensional simplicial complex \(X\) is a hypergraph that consists of an arbitrary collection of sets of size \((d+1)\) together with all their subsets. The $i$-faces are sets of size \(i+1\) in \(X\), denoted by \(X(i)\). We identify the vertices of \(X\) with its singletons \(X(0)\) and for good measure always think of \(X(-1) = \set{\emptyset}\). The \(k\)-skeleton of \(X\) is \(X^{\leq k} = \bigcup_{i=-1}^k X(i)\). In particular the \(1\)-skeleton of \(X\) is a graph. We denote by \(diam(X)\) the \emph{diameter} of the graph underlying $X$.

%

\subsubsection{Probability over simplicial complexes}
Let \(X\) be a simplicial complex and let \(\Pr_d:X(d)\to (0,1]\) be a density function on \(X(d)\) (that is, \(\sum_{s \in X(d)}\Pr_d(s)=1\)). This density function induces densities on lower level faces \(\Pr_k:X(k)\to (0,1]\) by \(\Pr_k(t) = \frac{1}{\binom{d+1}{k+1}}\sum_{s \in X(d),s \supset t} \Pr_d(s)\). 
When clear from the context, we omit the level of the faces, and just write \(\Pr[T]\) or \(\Prob[t \in X(k)]{T}\) for a set \(T \subseteq X(k)\).

\subsubsection{Links and local spectral expansion} \label{sec:prelims-lse}
Let \(X\) be a \(d\)-dimensional simplicial complex and let \(s \in X\) be a face. The link of \(s\) is the \(d'=d-|s|\)-dimensional complex
\[X_s = \sett{t \setminus s}{t \in X, t \supseteq s}.\] For a simplicial complex \(X\) with a measure \(\Pr_d:X(d) \to (0,1]\), the induced measure on \(\Pr_{d',X_s}:X_s(d-|s|)\to (0,1]\) is \[\Pr_{d',X_s}(t \setminus s) = \frac{\Pr_d(t)}{\sum_{t' \supseteq s} \Pr_d(t')}.\] 

\begin{definition}[Local spectral expander]
    Let \(X\) be a \(d\)-dimensional simplicial complex and let \(\lambda \geq 0\). We say that \(X\) 
    is a \emph{\(\lambda\)-local spectral expander} if for every \(s \in X\), the link of \(s\) is connected, and for every \(s \in X(d-2)\) it holds that \(\lambda(X_s) \leq \lambda\).
\end{definition}
We stress that this definition includes connectivity of the underlying graph of \(X\), because of the face \(s= \emptyset\). In addition, in some of the literature the definition of local spectral expanders includes a spectral bound \(\lambda(X_s) \leq \lambda\) for all \(i \leq d-2\) and \(s \in X(i)\). Even though the definition we use only requires a spectral bound for links of faces \(s \in X(d-2)\), it is well known that when \(\lambda\) is sufficiently small, this actually implies bounds on \(\lambda(X_s)\) for all \(i \leq d-2\) and \(s \in X(i)\) (see \cite{oppenheim2018local} for the theorem and more discussion).

\paragraph{Cayley Graphs and Cayley Local Spectral Expanders}
Let \(S \subseteq \mathbb{F}_2^n\). A Cayley graph over \(\mathbb{F}_2^n\) with generating set \(S\), is a graph whose vertices are \(\mathbb{F}_2^n\) and whose edges are \(E = \sett{\set{x,x+s}}{x \in \mathbb{F}_2^n, s \in S}\). This graph is denoted by \(Cay(\mathbb{F}_2^n,S)\). Note that we can label every edge in \(E\) by its corresponding generator. A \emph{Cayley complex} is a simplicial complex \(X\) such that \(X^{\leq 1}\) is a Cayley graph (as a set) and such that all edges with the same label have the same weight. We say that \(X\) is a \(\lambda\)-two sided (one sided) Cayley local spectral expander if it is a Cayley complex and a \(\lambda\)-two sided (one sided) local spectral expander.
\section{Subpolynomial degree Cayley local spectral expanders}
\label{sec:construction-subpoly-HDX}
%
In this section we construct Cayley local spectral expander with a subpolynomial degree, improving a construction by \cite{Golowich2023}.

\begin{theorem} \label{thm:main-recursive-hdx}
    For every \(\lambda, \varepsilon > 0\) and integer \(\dimlout \geq 2\) there exists 
    an infinite family \(\set{\S_k}_{k=1}^\infty\) of \(\dimlout\)-dimensional \(\lambda\)-local spectral expanding Cayley complexes with \(\S(0)=\mathbb{F}_2^{n_k}\). The degree is upper bounded by \(2^{ (n_k)^{\varepsilon}}\).
\end{theorem}
Golowich proved this theorem for \(\varepsilon = \frac{1}{2}\).

As in Golowich's construction, we proceed as follows. We first construct an expanding family of Grassmann posets \(\set{\Y_k}_{k=1}^\infty\) over \(\mathbb{F}_2\) (see \pref{sec:prelims-Grassmann}), and then we use a proposition by \cite{Golowich2023} (\pref{prop:basification} below) to construct said Cayley local spectral expanders. The main novelty in this part of the paper is showing that this idea can be iterated, leading to a successive reduction of the degree.

%
The following notation will be useful later. 
\begin{definition}
    Let \(\diml, n, \degree \) be integers and \(\lambda > 0\). We say a Grassmann poset is a \([\diml,n,\degree,\lambda]\)-poset if \(Y\) is a \(\diml\)-dimensional \(\lambda\)-local spectral expander whose ambient space is \(\F_2^n\), with \(|\Y(1)| \leq 2^\degree\). 
\end{definition}

\medskip

\cite{Golowich2023} introduces a general approach that turns a Grassmann poset $Y$ into a simplicial complex \(\beta(Y)\) called the basification of $Y$. The formal definition of this complex is as follows.

\begin{definition}\label{def:basification}
    Given a $d$-dimensional Grassmann poset $Y$. Its basification $\beta(Y)$ is a $(d-1)$-dimensional simplicial complex such that for every $0\le i \le d-1$
    \[ \beta(Y)(i) = \sett{ \set{v_0,\dots,v_i}}{\sp (v_0,\dots,v_i) \in Y(i+1)}.\]
\end{definition}

In \cite{Golowich2023} it is observed that the local spectral expansion of \(Y\) is identical to that of \(\beta(Y)\). Golowich constructs a Cayley complex whose vertex links are (copies of) the basification. The construction is given below.

\begin{definition}\label{def:abelian-complex-from-basification}
    Given the basification $\beta(Y)$ of a $d$-dimensional Grassmann complex $Y$ over the ambient space $\F_2^n$, define the $d$-dimensional abelian complex $\S = Cay(\F_2^n, \beta(Y))$ such that $\S(0) = \F_2^n$ and for $0<i\le d$
    \[\S(i) = \sett{\set{x,x+v_1,x+v_2,\dots,x+v_i}}{\set{v_1,v_2,\dots,v_i} \in \beta(Y)(i)}.\]
\end{definition}

     We observe that the underlying graph is \(Cay(\mathbb{F}_2^n,Y(1))\) and therefore its degree is \(2^\degree\).
\begin{proposition}[\cite{Golowich2023}] \label{prop:basification}
    Let \(\Y\) be a \(\diml\)-dimensional Grassmann poset over \(\mathbb{F}_2^n\) that is \(\lambda\)-expanding. Let \(\S\) be its basification. Then \(\S\) is a \(\lambda\)-local spectral expander. In addition, the underlying graph of \(X\) is a \(\frac{\lambda}{1-\lambda}\)-expander.
\end{proposition}
%
Thus our new goal is to prove the following theorem about expanding posets.
\begin{theorem} \label{thm:main-recursive-Grassmann}
    For every \(\varepsilon, \lambdao >0\), and \(\dimlout\) there exists infinite family of integers \(n_k'\) and of \([\dimlout,n_k',n_k'^\varepsilon,\lambdao]\)-posets \(\Y_k\).
\end{theorem}
This is sufficient to prove \pref{thm:main-recursive-hdx}:
\begin{proof}[Proof of \pref{thm:main-recursive-hdx}, assuming \pref{thm:main-recursive-Grassmann}]
    Let \(\Y_k\) be an infinite family of \([\diml,n_k,n_k^\varepsilon,\lambda']\) in \pref{thm:main-recursive-Grassmann} for \(\lambda'\) such that \(\lambdao = \frac{\lambda'}{1-\lambda'}\). Let \(\S_k\) be their basifications. By \pref{prop:basification}. The complexes \(\S_k\) have ambient subspace \(\mathbb{F}_2^{n_k}\), degree \(|\Y(1)| \leq 2^{(n_k)^\varepsilon}\), and expansion \(\lambdao\).
\end{proof}

\subsection{Expanding Grassmann posets from matrices} \label{sec:expanding-Grassmann-from-ME}

Let \(\Y\) be a \([\dimlin,n,\degree,\lambdai]\)-poset. We now describe how to construct from it a \([\dimlout, n', \degree', \lambdao]\)-poset \(\Y'\) where \(\dimlout = \Omega(\log \dimlin)\), \(n'=n^2\), \(\degree' = \dimlin \cdot \degree\) and \(\lambdao = 2^{-\Omega(\dimlout)}\) when \(\lambdai \leq \frac{1}{4\dimlin}\). In other words, we describe a transformation
\[[\dimlin,n,\degree,\lambdai]\text{-poset } \Y \; \; \; \longmapsto\; \; \; [\dimlout,n',\degree',\lambdao]-\text{poset } \Y'.\]
After this, we will iterate this construction, this time taking \(Y'\) to be the input, producing an even sparser poset \(Y''\) etc., as we illustrate further below. 

Without loss of generality we identify the ambient space of \(\Y\) with \(\mathbb{F}_2^{n}\) and set the ambient space of \(\Y'\) to be \(n \times n\) matrices (\(\cong \mathbb{F}_2^{n^2}\)). The \(1\)-dimensional subspaces of \(\Y'\) are
\begin{equation} \label{eq:dim-1-ss}
    \Y'(1) = \sett{M}{\rank(M) = \frac{\dimlin}{2} \ve \row(M), \col(M) \in \Y \left (\frac{\dimlin }{2} \right)}.
\end{equation}
In words, these are all matrices whose rank is \(\dimlin/2\) so that their row and column spaces are in \(\Y\). As we can see, \(n' = n^2\) and \(|\Y'(1)| \leq |\Y(1)|^\dimlin\) and therefore \(\degree'=\log |\Y'(1)| \leq \dimlin \degree\). We postpone the construction higher dimensional faces (which determine the dimension \(\dimlin '\)) and the intuition behind the expansion of \(\Y'\) for later. Let us first show how an iteration of this construction will construct sparser Grassmann posets.

\medskip
%


This single step can be iterated any constant number of times, starting from a high enough initial dimension \(\dimlin_0\). Let us denote by \(\M(\Y)\) the operator that takes as input one Grassmann poset \(\Y\) and outputs the sparser one \(\Y'\) as above. Fix \(m\) and fix a target dimension \(\dimlin_m\). Let \(\dimlin_0 \gg \dimlin_m\) and let \(n_0\) be sufficiently large. We set \(\Y_0,\Y_1,\dots,\Y_m\) to be such that \(\Y_0 = \Gr(\mathbb{F}_2^{n_0})^{\leq \dimlin_0}\) (which is a $[\dimlin_0,n_0,\degree_0,\lambda_0]$-poset for \(\degree_0=n_0\) and \(\lambda_0\) that goes to \(0\) with \(n_0\)) and \(\Y_{i+1} = \M(\Y_i)\). The final poset \(\Y_m\) is a 
\[ [\diml_m ,n_m,\degree_m,\lambda_m]\text{-poset.}\]
Here \(\diml_m = \Omega_m(\underbrace{\log \circ \log \circ \dots \circ \log }_{m \text{ times}} \dimlin_0)\), \(n_m=n^{2^m}\), \(\degree_m = \Tilde{O}(\dimlin) \cdot n\) and \(\lambda_m = 2^{-\Omega(\diml_m)}\) when \(\lambda_0 = 2^{-\Omega(d_0)}\).


The above construction works for any large enough \(n_0\). This gives us an infinite family of Grassmann complexes whose ambient dimension is \(n_m=n_0^{2^m}\) and whose degree is \(D_m = n_m^{\varepsilon}\) (where \(m=\log \frac{1}{\varepsilon}\)).  This infinite family proves \pref{thm:main-recursive-Grassmann}.

\subsubsection[The construction of the Grassmann posets]{The construction of the Grassmann poset \(\M(\Y)\)}
We are ready to give the formal definition of the poset. We construct the top level subspaces, and then take their downward closure to be our Grassmann poset. Every top level face is going to be an image of a `non-standard' Hadamard encoding. The `standard' Hadamard could be described as follows. Let \(\F_2^i\) be a vector space, and let \(\F_2^{2^i}\) be indexed by the vectors in \(\F_2^i\). That is \(\F_2^{2^i} \cong \sp(\set{e_v}_{v \in \F_2^i})\) where \(e_v\) is the the standard basis vector that is \(1\) in the \(v\)-th coordinate and \(0\) elsewhere. With this notation the `standard' Hadamard encoding maps 
\[x \mapsto \sum_{v \in \F_2^i}\iprod{x,v} e_v\]
where \(\iprod{x,v} = \sum_{i=1}^n x_i v_i \; \text{(mod 2)}\) is the standard bilinear form. A `non-standard' Hadamard encodings in this context is the same as above, only replacing the standard vectors \(e_v\) with another basis for a subspace, represented by \(2^{i}\) \(n \times n\) matrices. It is defined by a set \(\set{s(v)}_{v \in \F_2^i}\) satisfying some properties described formally below. The encoding is the linear map \(\widehat{s}:\F_2^i \to \F_2^{n \times n}\),
\[\widehat{s}(x) = \sum_{v \in \F_2^i}\iprod{x,v}s(v)\]
similar to the formula above.
We continue with a formal definition.

We start with \(Y\), a \([\dimlin,n,\degree,\lambdai]\)-poset (and assume without loss of generality that \(\dimlin\) is a power of \(2\)) and let \(\dimloutfull = \log \dimlin\)\footnote{Although we define this poset to the highest dimension possible. Our construction will eventually take a lower dimensional skeleton of it.}.
\begin{definition}[Admissible function] \label{def:admissible-func}
Let \(i > 0\), a (not necessarily linear) function \(s: \F_2^{i} \setminus \set{0} \to \F_2^{n \times n}\) is \emph{admissible} if:
    \begin{enumerate}
        \item For every \(v \in \F_2^i \setminus \set{0}\), \(\rank(s(v)) = 2^{\dimloutfull-i}\).
        \item The sum matrix \(M_s = \bigoplus_{v \in \F_2^i \setminus \set{0}}s(v)\) is \emph{a direct sum} of the \(s(v)\)'s, that is, \[\rank(M_s) = \sum_{v \in \F_2^i \setminus \set{0}} \rank(s(v)) = 2^{\dimloutfull}-2^{\dimloutfull-i}.\]
        \item The spaces \(\col(M_s), \row(M_s) \in \Y(2^{\dimloutfull}-2^{\dimloutfull-i})\). 
    \end{enumerate}
\end{definition}
We remark that for the construction we only need to consider admissible functions for \(i=\dimloutfull\), that is, functions \(s: \F_2^{\dimloutfull} \setminus \set{0} \to \F_2^{n \times n}\), where \(\rank(s(v))=1\), \(\rank(M_s) = 2^{\dimloutfull}-1\) and \(\col(M_s),\row(M_s) \in \Y(2^{\dimloutfull}-1)\). The other admissible functions will be used later in the analysis (see \pref{sec:induced-Grassmann}).

We denote by \(\mathcal{S}_i\) the set of admissible functions \(s: \F_2^{i} \setminus \set{0} \to \F_2^{n \times n}\). The \emph{Hadamard encoding} of a function \(s: \F_2^{i} \setminus \set{0} \to \F_2^{n \times n}\) is the linear function \(\widehat{s}:\F_2^i \to \F_2^{n \times n}\) given by
\[\widehat{s}(x) = \sum_{v \in \F_2^i \setminus \set{0}} \iprod{x,v} s(v),\]
Observe that \(\widehat{s}\) is an isomorphism so the image of \(\widehat{s}\) is an \(i\)-dimensional subspace. This of course means that every subspace \(u \subseteq \F_2^{i}\) is mapped isomorphically to a subspace of the same dimension \(Im(\widehat{s}|_u)\). Thus, the Grassmann poset \(\mathcal{G}_{s} \coloneqq \sett{Im(\widehat{s}|_u)}{u \subseteq \F_2^i, \; u \text{ is a subspace}}\) is isomorphic to the complete Grassmann \(\mathcal{G}(\F_2^i)\). 

\begin{definition}[Matrix poset] \label{def:matrix-poset}
    The poset \(\Y' = \M(\Y)\) is a \(\dimloutfull\)-dimensional poset 
    \(\Y' = \bigcup_{s \in \mathcal{S}_\dimloutfull} \mathcal{G}_{s}\).
    We denote by \(\M(\Y)\) the operator that takes \(\Y\) and produces \(\Y'\) as above. 
\end{definition}
In other words, \(\Y'\) is clearly a Grassmann poset and it is a union of \(\mathcal{G}_s\) for all admissible mappings \(s \in \mathcal{S}_{\dimloutfull}\).

The admissible functions on level \(i\) give us a natural description for \(Y'(i)\):
\begin{claim} \torestate{\label{claim:poset-i-space-description}
    Let \(\Y'\) be as in \pref{def:matrix-poset}. Then
    \[\Y'(i) = \sett{Im(\widehat{s})}{s \in \mathcal{S}_i}.\]
}\end{claim}
We prove this claim in \pref{app:admissible-functions-structure-proof}. 

This poset also comes with a natural distribution over \(\dimlout\)-subspaces is the follows.
\begin{enumerate}
    \item Sample two \((\dimlin-1)\)-dimensional subspaces \(w_1,w_2 \in \Y(\dimlin-1)\) independently according to the distribution of \(\Y\). 
    \item Sample uniformly an admissible function \(s \in \mathcal{S}_{\dimloutfull}\) such that the matrix \(M_s = \bigoplus_{v \in \F_2^\dimloutfull\setminus \set{0}} s(v)\) has \(\row(M_s) = w_1\) and \(\col(M_s) = w_2\).
    \item Output \(Im(\widehat{s})\).
\end{enumerate}

We remark that we can generalize this construction by changing the definition of ``admissible'' to other settings. One such generalization that replaces the rank function with the hamming weight results in the Johnson poset which we describe in \pref{sec:spectral-Johnson}. We explore this more generally in \pref{sec:induced-Grassmann}.

We finally reach the main technical theorem of this section, \pref{thm:one-step-of-recursion}. This theorem proves that if \(Y\) is an expanding poset, then so is a skeleton of \(\M(Y)\). Using this theorem inductively is the key to proving \pref{thm:main-recursive-Grassmann}.

\begin{theorem}\label{thm:one-step-of-recursion}
    There exists a constant \(K > 0\) such that the following holds. Let \(\lambdao \in (0,1)\), \(D > 0\) and \(\dimlout >0\) be an integer. Let \(\dimlin = \max \set{2^{2\dimlout},\left (K \log \frac{1}{\lambdao} \right )^2}\) and let \(\lambdai \leq \frac{1}{4\dimlin}\). If \(\Y\) is a \([\dimlin, n, \degree, \lambdai]\)-poset then the \(\dimlout\)-skeleton of \(\Y' = \M(\Y)\), namely \(\Y'^{\leq \dimlout}\), is a \([\dimlout, n', \degree ' , \lambdao]\)-poset where \(n' = n^2\) and \(\degree' = \dimlin \cdot \degree\).
\end{theorem}

\begin{proof}[Proof of \pref{thm:main-recursive-Grassmann}, assuming \pref{thm:one-step-of-recursion}]
    Fix \(\varepsilon\) and without loss of generality \(\varepsilon = \frac{1}{2^m}\) for some integer \(m \geq 0\). Changing variables, it is enough to construct a family \(\set{Y_k}\) of \([\dimlout,n_k^{2^m},O(n_k),\lambdao]\)-posets for an infinite sequence of integers \(n_k\), such that the constant in the big \(O\) only depends on \(\dimlout,\lambdao\) and \(m\) but not on \(n_k\). Then by setting \(n_k' = n_k^{2^m}\) gives us the result \footnote{Technically this only gives a family \([\dimlout,n_k',O(n_k'^\varepsilon),\lambdao]\)-posets. But \(\varepsilon > 0\) can be arbitrarily small. Thus for any \(\varepsilon > 0\), constructing \([\dimlout,n',C \cdot n'^{\varepsilon/2},\lambdao]\)-posets is also a construction of \([\dimlout,n_k',n_k'^{\varepsilon},\lambdao]\)-posets, since eventually \(n_k'^{\varepsilon} > C \cdot n_k'^{\varepsilon/2}\) for any constant \(C\) independent of \(n_k'\)}.
    
    We prove this by induction on \(m\) for all \(\dimlout\) and \(\lambdao\) simultaneously. For \(m=0\) we take the complete Grassmanns, namely \(n_k = k\) and the family \(\Y =\Y_k = \Gr(\mathbb{F}_2^k)^{\leq \dimlout}\) that are \([\dimlout,k,k,\lambdao]\)-posets for sufficiently large \(k\) (the first three parameters are clear. By \pref{claim:link-in-Grassmann-expands} the links expansion of the complete Grassmann is \(\leq \frac{1}{2^{(k-\dimlout)}-1}\) which goes to \(0\) with \(k\), hence for every \(\lambdao > 0\) and large enough \(k\) they are \(\lambdao\)-expanders).
    
    Let us assume the induction hypothesis holds for \(m\) and prove it for \(m+1\). Fix \(\lambdao > 0\) and \(\dimlout\). By \pref{thm:one-step-of-recursion}, there exists some \(\dimlin\) and \(\lambdai\) such that if we can find a family of \([\dimlin, n_k^{2^m},O(n_k),\lambdai]\)-posets \(\set{\Y_k}_{k=1}^\infty\), then the posets \(\Y_k' \coloneqq \M(Y_k)^{\leq \dimlout}\) are \([\dimlout, n_k^{2^{m+1}},O(n_k \cdot \dimlin),\lambdao]\)-posets. By the induction hypothesis (applied for \(\dimlin\) and \(\lambdai\)) such a family \(\set{\Y_k}_{k=1}^\infty\) indeed exists. The theorem is proven.
\end{proof}

\subsection[A Proof roadmap for a single iteration]{A Proof Roadmap for a single iteration of \(\M\) - \pref{thm:one-step-of-recursion}} \label{sec:proof-roadmap}
The rest of this section is devoted to proving \pref{thm:one-step-of-recursion}, and mainly to bounding the expansion of links in \(\Y'\). Recall the notation \(\MP\) to be the poset of matrices with the domination relation.

We fix the parameters \(\lambdao\) and \(\dimlout\). Without loss of generality we take \(\dimlout\) to be large enough so that \(\dimlin = 2^{2\dimlout}\) (this requires that \(\dimlout \geq \log K + \log \log \frac{1}{\lambdao}\) for the constant \(K > 0\) in \pref{thm:one-step-of-recursion}). 

Recall the partial order defined in \pref{sec:def-of-domination-relation} on matrices where we write \(A \leq M\) for two matrices \(A,M \in \F_2^{n \times n}\) if \(\rank(M) = \rank(A) + \rank(M-A)\). 

\medskip

Fix \(\Y\), and let \(\Y' = \M(\Y)\). Let \(W \in \Y'(\dimm)\) be a subspace (where \(\dimm \leq \dimlout\)).  
In the next two subsections we describe the two steps of the proof: decomposition and upper bounding expansion.
\subsubsection{Step I: decomposition}
We will show that a link of a subspace \(W \in \Y'(\dimm)\) decomposes to a tensor product of simpler components. 

For this we must define two types of graphs, relying on the domination relation of matrices and the direct sum.
\def\under{\mathcal{T}}
\def\disjointgraph{\mathcal{H}}

\begin{definition}[Subposet graph] \label{def:subposet-graph}
    Let \(m > 0\) and let \(U \in \MP(4m)\). The graph \(\under^U = (V,E)\) has vertices \(V = \MP^U(2m)\) and edges \(\set{A,B}\) such that there exists \(C \in \MP(m)\) such that \(C \oplus (A-C) \oplus (B-C) \in \MP^{U}(3m)\).
\end{definition}
For any two \(U_1,U_2 \in \MP(4m)\) of rank \(4m\), the graphs \(\under^{U_1} \cong \under^{U_2}\) so we denote this graph by \(\under^m\) when we do not care about the specific matrix.

\begin{definition}[Disjoint graph] \label{def:disjoint-graph}
    Let \(m,i > 0\) and let \(D \in \MP(i)\). The graph \(\disjointgraph_{D,m} = (V,E)\) has vertices \(V = \MP_{\cap D}(2m)\) and edges \(\set{A,B}\) such that there exists \(C \in \MP(m)\) such that we have \(C \oplus (A-C) \oplus (B-C) \in \MP_{\cap D}(3m)\).
\end{definition}
For any \(n\) and two \(D_1,D_2 \in \MP(i)\), the graphs \(\disjointgraph_{D_1,m} \cong \disjointgraph_{D_2,m}\) so we denote this graph by \(\disjointgraph_{i,m}\) when we do not care about the specific matrix.

\medskip

The first component in the decomposition in the link of \(W\) is the subposet graph.
The second component is the sparsified link graph. This graph depends on \(\Y\), unlike the previous component. In particular, the matrices in this component will have row and column spaces in \(\Y\). For the rest of this section we fix \(m = \frac{\dimlin}{2^{\dimm+2}}\). Let \(Z_W\) be any matrix whose row span and column span are \(\row(Z_W) = \sum_{M \in W}\row(M)\) and \(\col(Z_W) = \sum_{M \in W}\col(M)\). Note that both $\row(Z_W)$ and $\col(Z_W)$ have dimension $(1-\frac{1}{2^{\ell}})d$ (that is, $d-4m$). The following definition will only depend on these row and column spaces, but it will be easier in the analysis to give the definition using this matrix \(Z_W\).
\begin{definition}[Sparsified link graph] \label{def:sparsified-link-graph}
   The graph \(\SL_{\Y,W}\) is the graph with vertex set
    \[ V = \sett{A\in \MP_{\cap Z_W}(2m)}{\row(A) \oplus \row(Z_W), \; \col(A) \oplus \col(Z_W)  \in \Y\left(\dimlin-2m \right)}.\]
    The edges are all \(\set{A_1,A_2}\) such that there exists a matrix \(U \in \MP_{\cap Z_W}(4m)\) such that \(\row(U) \oplus \row(Z_W), \; \col(M) \oplus \col(Z_W) \in \Y(\dimlin)\), and \(\set{A_1, A_2}\) are an edge in \(\under^U\).
%
\end{definition}
Formally, the distribution over edges in the sparsified link graph \(\SL_{\Y,W}\) is the following. 
\begin{enumerate}
    \item Sample \(w_1 \in Y_{\row(Z_W)}(\dimlin)\) and \(w_2 \in Y_{\col(Z_W)}(\dimlin)\) according to their respective distributions.
    \item Uniformly sample a matrix \(U\) such that \(\row(Z_W) \oplus \row(U) = w_1\) and \(\col(Z_W) \oplus \col(U) = w_2\) .
    \item Sample an edge in \(\under^U\).
\end{enumerate}
%
We can now state our decomposition lemma:
\begin{lemma}\label{lem:grassmannian-link-decomposition}~
  \(\Y'_W \cong \underbrace{(\under^m \otimes \under^m \otimes \dots \otimes \under^m)}_{2^\dimm -1 \text{ times}} \otimes \SL_{\Y,W}\).
\end{lemma}
This lemma is proven in \pref{sec:proof-of-decomp-lemma}.
%

\subsubsection{Step II: upper bounding expansion}
It remains to show that \(\under^m\) and \(\SL_{\Y,W}\) are expander graphs. We prove the following two lemmas.

\begin{lemma} \label{lem:complete-part-analysis}
There exists an absolute constant \(\gamma \in (0,1)\) such that 
    \(\lambda(\under^m) \leq \gamma^{m}\) for all \(m \leq \dimlin\).
    
    In particular, there is an absolute constant \(K_1 > 0\) such that the following holds: Let \(\lambdao > 0\). If \(\dimlout \geq K_1 + \log \log \frac{1}{\lambdao}\) and \(\dimlin \geq 2^{2 \dimlout}\), then
    \[\lambda(\under^m) \leq \lambdao.\]
\end{lemma}

\begin{lemma} \label{lem:loc-to-glob-decomp-of-link}
There exists a constant \(\gamma  \in (0,1)\) such that the following holds. Let \(\Y\) be a \([\dimlin,n,\degree,\lambdai]\)-poset for \(\lambdai < \frac{1}{4\dimlin}\), and let \(W \in \Y'(\dimm)\) where \(\Y'=\M(Y)^{\leq \dimlout}\) and \(\dimm \leq \dimlout\). Then
    \(\lambda(\SL_{\Y,W}) \leq \gamma^{m}\). 
    
    In particular, there is an absolute constant \(K_2 >0\) such that the following holds: Let \(\lambdao > 0\). If \(\dimlout \geq K_2 + \log \log \frac{1}{\lambdao}\), \(\dimlin \geq 2^{2\dimlout}\) and \(\lambdai \leq \frac{1}{4\dimlin}\), then
    \[\lambda(\SL_{\Y,W}) \leq \lambdao.\]
\end{lemma}

The constant \(K\) in \pref{thm:one-step-of-recursion} is chosen to be \(2^{\max(K_1,K_2)}\). We now give an high level explanation of the proofs of \pref{lem:complete-part-analysis} and \pref{lem:loc-to-glob-decomp-of-link}, and afterwards we use them to prove \pref{thm:one-step-of-recursion}.

The graph \(\under^m\) does not depend on \(Y\) and the proof of \pref{lem:complete-part-analysis} relies only on the properties of the matrix domination poset. Therefore it is deferred to \pref{app:proofs-of-expansion-matrix-graphs}. In a bird's-eye view, we use \pref{lem:bipartite-loc-to-glob}, the decomposition lemma, to decompose the graph into smaller subgraphs, and show that every subgraph is close to a uniform subgraph. 

We prove \pref{lem:loc-to-glob-decomp-of-link} in \pref{sec:lemma-sparsified-links-proof}. Its proof also revolves around \pref{lem:bipartite-loc-to-glob}, only this time the subgraphs in the decomposition come from (pairs of) subspaces in \(Y\). Observe that in the first step in sampling an edge in \pref{def:sparsified-link-graph}, one first samples a pair \((w_1,w_2) \in \Y_{\row(Z_W)}(\dimlin)\times \Y_{\col(Z_W)}(\dimlin)\). The subgraph induced by fixing \(w_1\) and \(w_2\) in the first step is isomorphic to a graph that contains \emph{all} rank \(2m\) matrices in \(\F_2^{n'}\) for some \(n'\), that are a direct sum with some matrix \(Z_W\) (independent of \(\Y\)). Thus every such subgraph is analyzed using the same tools and ideas as in the proof of \pref{lem:complete-part-analysis}. 

In order to apply \pref{lem:bipartite-loc-to-glob}, we also need to show that the decomposition graph \(\B_\tau\) expands (see its definition in \pref{sec:local-to-global}). This is the bipartite graph where one side is the vertices in \(\SL_{\Y,W}\), and the other side are the pairs of subspaces \(\Y_{\row(Z_W)}(\dimlin)\times \Y_{\col(Z_W)}(\dimlin)\). A matrix \(A \in \SL_{\Y,W}\) is connected to \((w_1,w_2)\) if \(\row(A) \subseteq w_1\) and \(\col(A) \subseteq w_2\).

To analyze this graph, we reduce its expansion to the expansion of containment graphs in \(\Y\). We make the observation that we only care about \(\row(A)\) and \(\col(A)\), since if \(A\) and \(A'\) have the same rowspan and colspan, then they also have the same neighbors. To say this in more detail, \(\B_\tau \cong \mathcal{B} \otimes \mathcal{Q}\) where \(\mathcal{Q}\) is some complete bipartite graph, and \(\mathcal{B}\) replaces the matrices in \(\B_\tau\) with \(\Y_{\row(Z_W)}(\dimlin-2m) \times \Y_{\col(Z_W)}(\dimlin-2m)\) (here every matrix \(A\) is projected to \((\row(A),\col(A))\) in \(\mathcal{B}\)).
The graph \(B'\) itself is a tensor product of the containment graph of \(Y_{\row(Z_W)}\) and \(Y_{\col(Z_W)}\). Fortunately, as we discussed in \pref{sec:prelims-Grassmann}, this graph is an excellent expander, which concludes the proof.

Given these components we prove \pref{thm:one-step-of-recursion}.
\begin{proof}[Proof of \pref{thm:one-step-of-recursion}]
    Let us verify all parameters \([\dimlout,n',\degree', \lambdao]\). The dimension \(\dimlout\) is clear from the construction. 
    
    The size \(|\Y'(1)| \leq |\Y(1)|^{2 \cdot \frac{\dimlin}{2}}\) since it is enough to specify \(\dimlin\) vectors to construct a rank \(\frac{\dimlin}{2}\) matrix \(M = \sum_{j=1}^{\frac{\dimlin}{2}} e_j \otimes f_j\) (of course we have also double counted since many decompositions correspond to the same matrix). Thus \(\log |\Y'(1)| \leq \dimlin \cdot \degree =\degree'\).

    Let us verify that the ambient space is \(n'=n^2\) dimensional, or in other words that \(\Y'(1)\) really spans all \(n \times n\) matrices (recall that we assumed without loss of generality that \(\Y\)'s ambient space is \(\mathbb{F}_2^n\)). For this it is enough to show that there are \(n^2\) linearly independent matrices inside \(\Y'(1)\). Let \(e_1,e_2\dots,e_n \in \Y(1)\) be \(n\) linearly independent vectors. The ambient space of \(Y\) is \(\F_2^n\). Let us show that \(e_i \otimes e_j\) is spanned by the matrices inside \(\Y'(1)\). Indeed, choose some subspaces \(w_i,w_j \in \Y(\dimlin)\) such that \(e_i \in w_i, e_j \in w_j\). Then $\Y'(1)$ contains all rank \(\frac{\dimlin}{2}\) matrices in $\F_2^{n^2}$ whose row and column spans are inside \(w_i,w_j\). These matrices span all matrices whose row and column spans are in \(w_i,w_j\) and in particularly they span \(e_i \otimes e_j\). 
       
    Finally let us show that all links are \(\lambdao\)-expanders. Fix \(W \in \Y' (\dimm)\). By \pref{lem:grassmannian-link-decomposition}
    \[\Y'_W \cong (\under^m)^{\otimes 2^{\dimm}-1} \otimes \SL_{\Y,W}\]
    so \(\lambda(\Y'_W) \leq \max \set{\lambda(\under^m), \lambda(\SL_{\Y,W})}\).
    Both are \(\lambda\)-expanders by \pref{lem:complete-part-analysis} and \pref{lem:loc-to-glob-decomp-of-link} when the conditions of the theorem are met.

    To sum up, \(\Y'\) is a \([\dimlout,n',\degree', \lambdao]\)-expanding poset.
\end{proof}
\begin{remark}
    For the readers who are familiar with \cite{Golowich2023}, we remark that contrary to that paper, we do not use the strong trickling down machinery developed in \cite{kaufmanT2021garland}. \cite{kaufmanT2021garland} show that expansion in the top-dimensional links trickles down to the links of lower dimension (with a worse constant, but still sufficient for our theorem). A priori, one may think that using \cite{kaufmanT2021garland} could simplify the proof. However, in our poset all links have the same abstract structure, so we might as well analyze all dimensions at once. Moreover, one observes that in many cases the bound on the expansion of these posets actually \emph{improves} for lower dimensions.
\end{remark}

\subsection{Proof of the decomposition lemma - \pref{lem:grassmannian-link-decomposition}} \label{sec:proof-of-decomp-lemma}
The proof of \pref{lem:grassmannian-link-decomposition} uses the following claim proven in \cite{Golowich2023}. 

\begin{claim}[{\cite[Proposition 55]{Golowich2023}}] 
    Let \(\Mgras = \M(\Gr)^{\leq \dimlout}\) and let \(W \in \Mgras(\dimm)\). Then there exists matrices of \(Z_1,Z_2,\dots,Z_{2^{\dimm-1}}\) of rank \(\frac{\dimlin}{2^\dimm}\), whose direct sum is denoted by \(Z_W = \bigoplus_{i=1}^{2^\dimm - 1}Z_i\), such that the link graph 
    \[\Mgras_W \cong  \left ( \bigotimes_{j=1}^{2^\dimm -1} \under^{Z_j} \right ) \otimes \disjointgraph_{Z_W,m}.\]
\end{claim}
This is not just an abstract isomorphism. The details are as follows. First, \cite{Golowich2023} observes that given \(W = Im(\widehat{s})\), for some admissible \(s\), the set \(Im(s) \coloneqq \sett{Z_j}{\exists x \in \F_2^\dimm \setminus \set{0}, s(x)=Z_j}\) does not depend on \(s\), only on \(U\). We denote it by \(Q_W = \set{Z_j}_{j=1}^{2^\dimm-1}\).

Intuitively, the following proposition follows from the 'Hadamard-like' construction of the Grassmann, and the intersection patterns of Hadamard code.

\begin{proposition}[{\cite[Lemmas 59,61,62]{Golowich2023}}] \label{prop:louis-decomposition}
    Let \(\Mgras = \M(\Gr)^{\leq \dimlout}\), let \(W \in \Mgras(\dimm)\), let \(Q_U = \set{Z_1,Z_2,\dots,Z_{2^{\dimm}-1}}\) be as above, with \(Z_W = \bigoplus_{j=1}^{2^\dimm -1} Z_j\). For every matrix \(N\) such that \(W + \sp(N) \in \Mgras_U(\dimm+1)\) there exists unique matrices \(A_1,A_2,\dots,A_{2^{\dimm}}\) such that
    \begin{enumerate}
            \item \(\rank(A_j) = \frac{\dimlin}{2^{\dimm+1}}\).
            \item \(N = \sum_{j=1}^{2^\dimm} A_{j}\).
            \item For \(j=1,2,\dots,2^\dimm-1\) the matrices \(A_j \in \under^{Z_j}\). That is, \(A_j \leq Z_j\).
            \item The matrix \(A_{2^{\dimm}} \in \disjointgraph_{Z_W,m}\). That is, \(\row(A_{2^\dimm})\) (resp. \(\col(A_{2^\dimm})\)) intersects trivially with \(\row(Z_W)\) (resp. \(\col(Z_W)\)).
            \item The function \(N \mapsto (A_1,A_2,\dots,A_{2^\dimm})\) is an isomorphism 
            \[\Mgras_W \cong \left ( \bigotimes_{j=1}^{2^\dimm -1} \under^{Z_j} \right ) \otimes \disjointgraph_{Z_W,m}\]
    \end{enumerate}
\end{proposition}

We restate and reprove this proposition in \pref{sec:induced-Grassmann}. The proof our decomposition lemma is short given the one from \cite{Golowich2023}. For this (and this only) we need to slightly alter \pref{def:matrix-poset} and \pref{def:sparsified-link-graph} so that we take into account two posets instead of one. Let \(\Y_1,\Y_2\) be two posets of the same dimension \(\dimlin\).
\begin{definition}[Two-poset admissible function, generalizing \pref{def:admissible-func}]
    \emph{An admissible function} with respect to \(\Y_1,\Y_2\) to be a function \(s: \F_2^i \setminus \set{0} \to \F_2^{n \times n}\) as in \pref{def:admissible-func} with the distinction that the row and column spaces of the matrix \(M_s = \bigoplus_{v \in \F_2^i \setminus \set{0}}s(v)\) are 
    \[\row(M_s) \in \Y_1(2^{\dimloutfull} - 2^{\dimloutfull - i}), \;\col(M_s) \in \Y_2(2^{\dimloutfull} - 2^{\dimloutfull - i}).\]
\end{definition}
\begin{definition}[Two-poset matrix poset, generalizing \pref{def:matrix-poset}]
    We denote by \(\M(\Y_1,\Y_2)\) the matrix poset defined in \pref{def:matrix-poset} such that the admissible functions are with respect to \(\Y_1\) and \(\Y_2\).
\end{definition}

\begin{definition}[Two-poset sparsified link graph, generalizing \pref{def:sparsified-link-graph}]
    For \(W \in \M(\Y_1,\Y_2)\) we denote by \(\SL_{\Y_1,\Y_2,W}\) the graph in \pref{def:sparsified-link-graph}, with the distinction that all sums of row spaces are required to be in \(\Y_1\) and all column spaces are required to be in \(\Y_2\) (instead of them all being in the same poset \(\Y\)).
\end{definition}
 
\begin{proof}[Proof of \pref{lem:grassmannian-link-decomposition}]
    Fix \(W\) and its corresponding \(Q_W = \set{Z_1,Z_2,\dots,Z_{2^{\dimm}-1}}\) and \(Z_W\). Conditioned on \(W\), the first step of the sampling distribution in \pref{def:matrix-poset} is the same as sampling \(w_1 \supset \row(Z_W)\) and \(w_2 \supset \col(Z_W)\) (where \(w_1,w_2 \in \Y(\dimlin)\)). Fixing the pair \(w_1,w_2\) and conditioning on all three \(W,w_1,w_2\), sampling an edge in the link \(\Y'_U\) is the same as sampling an edge in the link of \(W\) in the complex \(\M(\Gr(w_1),\Gr(w_2))\). As both \(\Gr(w_1)\) and \(\Gr(w_2)\) are isomorphic, \(\M(\Gr(w_1),\Gr(w_2)) \cong \M(\Gr(w_1),\Gr(w_1)) \cong \M(\Gr(w_1))\)\footnote{Recall that \(\Gr(w_i)\) is the Grassmann poset containing all subspaces of \(w_i\).}. We are back to the complete Grassmann case and are allowed to use \pref{prop:louis-decomposition}.
    Let \(\mathcal{B}(W,w_1,w_2)\) be the subgraph obtained when fixing and conditioning on \(W,w_1,w_2\). By \pref{prop:louis-decomposition}, we can decompose this graph to a tensor product  
    \[\mathcal{B}(W,w_1,w_2) \cong \left ( \bigotimes_{j=1}^{2^\dimm-1} \under^{Z_i} \right )\otimes \SL_{\Gr(w_1),\Gr(w_2),W}.\] 
    The first \(2^{\dimm}-1\) components are independent of \(w_1\) and \(w_2\); they are constant among all conditionals (and are all isomorphic to \(\under^{m}\)). Hence this is isomorphic to \( (\under^{m})^{\otimes 2^{\dimm}-1} \otimes \mathcal{R}\) where \(\mathcal{R}\) is the graph corresponding to a subgraph of \(\SL_{\Y,W}\) obtained by conditioning the edge-sampling process of \(\SL_{\Y,W}\) to sampling \(w_1\) and \(w_2\) in the first step. Thus the graph \(\mathcal{R} = \SL_{\Gr(w_1),\Gr(w_2),W}\) and the lemma is proven.
\end{proof}

\subsection{Expansion of the sparsified link graph} \label{sec:lemma-sparsified-links-proof}
In this section we use the decomposition lemma to reduce the analysis of the sparsified link graph to the case of the complete Grassmann and prove \pref{lem:loc-to-glob-decomp-of-link}.

We use the following claim on the complete Grassmann.
\begin{claim} \label{claim:expansion-no-maximal-matrix-graph}
    Let \(\Mgras=\M(\Gr)^{\leq \dimlout}\) be the poset of rank \(\frac{\dimlin}{2}\) matrices for a sufficiently large \(\dimlout\). Let \(W \in \Mgras(\dimm)\) and let \(w_1,w_2 \in \Gr(\dimlin)\) be such that \(\row(Z_W) \subseteq w_1, \col(Z_W) \subseteq w_2\). Then \(\lambda(\SL_{\Gr(w_1),\Gr(w_2),W)}) \leq \gamma^{m}\) for some universal constant \(\gamma \in (0,1)\). 
\end{claim}
We prove this claim in the next section.

\begin{proof}[Proof of \pref{lem:loc-to-glob-decomp-of-link}]
    Fix the subspace \(W \in \Y'(\dimm)\) and recall the notation \(w_R = \row(Z_W), w_C = \col(Z_W)\). We intend to use \pref{lem:local-to-global}. The local decomposition we use is the following \(\tau = (\T,\mu_\T)\). For every pair \((w_1,w_2) \in \Y_{w_R}(\dimlin) \times \Y_{w_C}(\dimlin)\) we define the graph \(G_{(w_1,w_2)}\) which is the subgraph obtained by conditioning on \((w_1,w_2)\) being sampled in the first step of \pref{def:sparsified-link-graph}. The distribution \(\mu_\T\) is the distribution that samples \(w_1,w_2 \in \Y_W(\dimlin)\) independently. By definition of \(\SL_{\Y,W}\) this is indeed a local-decomposition.
    
    We observe that the subgraph of \(\SL_{\Y,W}\) obtained by fixing \(w_1\) and \(w_2\) in the first step is just \(G_{w_1,w_2} = \SL_{\Gr(w_1),\Gr(w_2),W}\) and hence by \pref{claim:expansion-no-maximal-matrix-graph} it is a \(\gamma^{ m}\)-expander. 
    
    We describe the local-to-global graph \(\B_\tau\) in \pref{lem:local-to-global}.
    \begin{enumerate}
        \item \(\T = \Y_{w_R}(\dimlin) \times \Y_{w_C}(\dimlin)\),
        \item \(V = \sett{A}{\rank(A) = 2m, \row(A) \oplus \row(Z_W) \in \Y, \col(A) \oplus \col(Z_W) \in \Y \left ( \dimlin - 2m \right )}\),
    \end{enumerate}
    and is an edge between \((w_1,w_2) \in \T\) and \(A \in V\) if \(\row(A) \oplus \row(Z_W) \subseteq w_1\) and \(\col(A) \oplus \col(Z_W) \subseteq w_2\).
    
    Observe that we only care about \(\tau(A) = (\row(A) \oplus \row(Z_W),\col(A) \oplus \col(Z_W))\), and any two \(A_1,A_2\) with \(\tau(A_1) = \tau(A_2)\), have the same neighbor set. In other words, this graph is isomorphic to the bipartite tensor product \(\mathcal{B}_R \otimes \mathcal{B}_C \otimes \mathcal{Q}_{1,k}\) where \(\mathcal{B}_R\) is the containment graph between \(\Y_{w_R}(\dimlin)\) and \(\Y_{w_R}(\dimlin-2m)\) (resp. \(\mathcal{B}_C\) and \(w_C\)), and \(\mathcal{Q}_{1,k}\) is the complete bipartite graph with \(k\) vertices on one side and a single vertex on the other side. The number \(k\) is the number of matrices \(A\) such that  \((\row(A) \oplus \row(Z_W),\col(A) \oplus \col(Z_W)) = (u_1,u_2)\), for any pair of fixed subspaces \(u_1\) and \(u_2\). By \pref{claim:loc-to-glob-grassmann} and the fact that \(\Y\) is \(\lambdai\)-expanding, the graph \(G\) is a \(\left ( 0.61 + \dimlin \lambdai\right )^{\frac{1}{2} (\dimlin-2m)}\leq \gamma^{m} \), where \(\gamma < 1\) is some constant and the last inequality is because \(\lambdai \leq \frac{1}{4d}\). 

    By \pref{lem:local-to-global}, \(\lambda(\SL_{\Y,U}) \leq \lambda(G)^2 + O(\gamma^{ m}) = \gamma^{m}\).
\end{proof}

\newcommand{\0}[0]{\mathbf{0}}
\newcommand{\ot}[0]{\cap}
\newcommand{\co}[1]{C_{#1}} 
\newcommand{\sym}[1]{T_{#1}} 
\newcommand{\oco}[1]{\overline{C_{#1}}} 
\newcommand{\ol}[1]{\overline{#1}}
\newcommand{\cd}[1]{c_{#1}}
\newcommand{\dc}[1]{T_{#1}}

\section{Johnson complexes and local spectral expansion}
\label{sec:spectral-Johnson}
In this section we construct a family of simplicial complexes which we call Johnson complexes whose degree is polynomial in their size. 
The underlying $1$-skeletons of Johnson complexes are abelian Cayley graphs $\mathrm{Cay}(\mathbb{F}_2^n, \binom{[n]}{\varepsilon n})$. 
We prove in this section that these complexes are local spectral expanders and in particular the links of all faces except for the trivial face have spectral expansion strictly smaller than $\frac{1}{2}$.
We then compare the Johnson complexes with the $2$-dimensional local spectral expanders constructed from random geometric graphs \cite{LiuMSY2023} and show that the former complexes can be viewed as a derandomization of the latter random complexes.
On the way we also define Johnson Grassmann posets $\set{Y_{\varepsilon,n}}_n$, analogous to the matrix posets given in \pref{sec:expanding-Grassmann-from-ME}, and prove in \pref{sec:JGposet} that Johnson complexes are exactly $\S(\F_2^n, \beta(Y_{\varepsilon,n}))$, which are Cayley complexes whose vertex links are basifications of $Y_{\varepsilon,n}$s (\pref{def:abelian-complex-from-basification}). 
\begin{remark}
    Later in \pref{sec:induced-Grassmann}, we shall define the induced Grassmann posets which generalize both the Johnson Grassmann posets and the matrix Grassmann posets, and we shall provide a unified approach to show local spectral expansion for induced Grassmann posets. However in this section, we will present a more direct proof of local expansion that is specific to Johnson complexes.
\end{remark}
We name our construction Johnson complexes due to their connection to Johnson graphs. In particular, all the vertex links in a Johnson complex are generalization of Johnson graphs to simplicial complexes. So we start by giving the definition of Johnson graphs and then provide the construction of Johnson complexes.

\begin{definition}[Johnson graph] \label{def:johnson-graph}
    Let \(n>k>\ell\) be integers. The adjacency matrix Johnson graph \(J(n,k,\ell)\) is the (uniformly weighted) graph whose vertices are all sets of size \(k\) inside \(\set{1,2,\dots,n}\) and two sets \(t_1,t_2\) are connected if \(\abs{t_1 \cap t_2} = \ell\).
\end{definition}

Next we define Johnson complexes. We shall use $wt(x)$ to denote the Hamming weight of a binary string $x$.
\begin{definition}[Johnson complex $X_{\varepsilon,n}$]\label{def:Johnson-complex}
Let \(k\) and  be any integer. Let \(n\in\mathbb{N}\) and let \(\varepsilon \in (0,\frac{1}{2}]\) be such that \(2^k \; | \; \varepsilon n\). Let \(S_{\varepsilon,n}  = \sett{x \in \mathbb{F}_2^n}{wt(x) = \varepsilon n}\) (abbreviate as \(S_{\varepsilon}\) when \(n\) is clear from context). The \(\varepsilon\)-weight Johnson complex \(X = X_{\varepsilon,n}\) is the \(k\)-dimensional simplicial complex defined as follows.
\[X(0) = \sett{x \in \mathbb{F}_2^n}{ wt(x) \text{ is even}},\]
\[X(1) = \sett{\set{x,x+s}}{x \in X(0), s \in S_{\varepsilon}},\]
\[X(k) = \sett{\set{x,x+ s_1,\dots,x+ s_k}}{x \in X(0),\forall i\in[k]\,s_i\in S_{\varepsilon}, \forall\,\emptyset\subsetneq T \subseteq [k]~wt\left(\sum_{i \in T}s_i\right) = \varepsilon n}.\]
\end{definition}
For the rest of this work, we fix \(k,n,\varepsilon\) as above.

It turns out that the weight constraints on $k$-faces in a Johnson complex gives rise to a Hadamard-like intersection pattern for elements in the link faces. Any $(k-1)$-face in the link of $x$ can be written as $\set{x+s_1,\dots,x+s_k}$. If we identify each $s_i$ with the set of nonzero coordinates in it, these $k$ sets always satisfy that any two distinct sets intersect at exactly $\varepsilon n/2$ coordinates, and in general any $\ell \le k$ distinct $s_i$s intersect at $\varepsilon n/2^{i-1}$ coordinates. We note that this is similar to the intersection pattern of $k$ distinct codewords in a Hadamard code.


The main result of this section is proving that all links of this complex are good expanders.
\begin{lemma}[Local spectral expansion of Johnson complexes]\label{lem:Johnson-complex-spectral-expansion}
    For any $0\le m \le k-2$, the link graph of any $m$-face in the $k$-dimensional Johnson complex $X_{\varepsilon,n}$ is a two-sided $ \left(\frac{1}{2}-\frac{\varepsilon}{2^{m+1}}\right)$-spectral expander. In particular,  $X_{\varepsilon,n}$ is a $k$-dimensional two-sided $\left(\frac{1}{2}-\frac{\varepsilon}{2^{k-1}}\right)$-local spectral expander. 
\end{lemma}

The key observation (\pref{prop:link-structure}) towards showing this lemma is that the link graphs of $m$-faces are the tensor product of Johnson graphs.

\snote{How does it sounds now?} \Ynote{I removed last sentence, because the definition made it obsolete. I also changed the sentence beforehand because we don't have Oppenheim's theorem in sec 2.1.1. but otherwise cool.}
Once this decomposition lemma is proven, we can use the second eigenvalue bounds for Johnson graphs to show that the links are $\left(\frac{1}{2}-\frac{\varepsilon}{2^{m+1}}\right)$-spectral expanders. We note that this construction breaks the barrier $\frac{1}{2}$ link expansion barrier of known elementary constructions of local spectral HDXs (discussed in the introduction). Because our construction breaks this barrier, using the Oppenheim's trickle down theorem \cite{oppenheim2018local}, we 
 directly get that the $1$-skeleton of the Johnson complex is also a $(1-2\varepsilon)$-spectral expander. 

We now outline the rest of this section. First, we prove the link decomposition lemma \pref{prop:link-structure}. Then using this lemma we show that Johnson complexes are basifications of Johnson Grassmann posets. Afterwards, we apply this lemma to prove the main result. Finally, we conclude this section by comparing the density parameters of the Johnson complexes with those of spectrally expanding random geometric complexes.


\subsection{Link decomposition}
In this part we prove the link decomposition result.
Note that by construction the links of all vertices $X_{\varepsilon,n}$ are isomorphic. So it suffices to focus on the link of $\0\in \mathbb{F}_2^n$ (denoted by $X_{\0}$) to show local spectral expansion for $X_{\varepsilon,n}$. 
The link decomposition is a consequence of the Hadmard-like intersection patterns of the faces in $X_{\0}$ as explained earlier. 

For the rest of the section, we use $\co{i}\subset [n]$ to denote the set of nonzero coordinates in a vector $x_i$ and use $M$ to denote \(2^m\). We now state the face intersection pattern lemma.
%


\begin{lemma}\label{lem:face-decomposition}
Let $m$ be an integer. 
Given a set $\{x_1,\dots,x_m\} \subset \F_2^n$ and a subset $B \subset [m]$, define 
\[\co{B} =\bigcap_{i\in B} \co{i} \bigcap_{j\not\in B} \oco{j}\, \text{ and }\cd{B} = |\co{B}| ~.\]
Then $\{x_1,\dots,x_m\}\in X_\0(m-1)$  if and only if for any subset $B \subset [m]$,
\[ \cd{B} = \begin{cases}\frac{2\varepsilon n}{M} ~~~&B \neq \emptyset \\
 n - \frac{M-1}{M}\cdot2\varepsilon n&\text{otherwise}\end{cases}\, .\]
\end{lemma}
\Ynote{In sec 6 we really identify between \(x_i\) and \(C_i\). Can we do this here maybe?}
\snote{You mean use $x_i$ in place of $C_i$? This would make certain steps in the proof confusing such as \pref{eq:wt-to-cd}.}
\begin{proof}
    We first verify that every set $\{x_1,\dots,x_m\}$ satisfying the $M$ cardinality constraints above is in $X_\0(m-1)$. Recall that by definition $\{x_1,\dots,x_m\}\in X_\0(m-1)$ if and only if they satisfy the $M-1$ weight equations \[\forall S\subset[m], S\neq \emptyset,~wt\left(\sum_{i\in S} x_i\right) = \varepsilon n.\] Observe that the set of nonzero coordinates in $\sum_{i\in S} x_i$ are precisely $j\in[n]$ that are nonzero in odd number of vectors in $\sett{x_i}{i\in S}$
    \[\text{Set}\left(\sum_{i\in S} x_i\right) = \bigoplus_{i\in S} \co{i} = \bigcup_{\substack{B\subseteq [m], \\\abs{B\cap S}\text{ odd}}}\co{B}~,\]
    and that by definition the set $\{\co{B}\}_{B\subset [m]}$ is a partition of $[n]$. Therefore \Ynote{"the union above is disjoint and "}the weights can be rewritten as
    \begin{equation}\label{eq:wt-to-cd}
        \forall S\subset [m],\,S\neq \emptyset,~~  wt\left(\sum_{i\in S} x_i\right) = \sum_{\substack{B\subseteq [m], \\\abs{B\cap S}\text{ odd}}} \cd{B} ~.
    \end{equation}
    Furthermore, for any nonempty subset $S$ there are $M/2$ nonempty sets $B\subset [m]$ whose intersection with $S$ has odd cardinality: let $t = |S|$, then the number of  $S$'s subsets of odd cardinality is  \[\sum_{i\text{ odd}} \binom{t}{i} = \frac{1}{2}\left(\sum_{i\le t} \binom{t}{i} - \left(\sum_{i\text{ even}} \binom{t}{i} - \sum_{i\text{ odd}} \binom{t}{i}\right) \right) =  \frac{1}{2} ((1+1)^t-(1-1)^t) = 2^{t-1}.\] 
    So the total number of such $B$ is $2^{t-1}\cdot 2^{m-t} = M/2$.
    Therefore if $\cd{B}  = \frac{2\varepsilon n}{M}$ when $B\neq \emptyset$, then $wt(\sum_{i\in S} x_i) = M/2\cdot 2\varepsilon n/M =
 \varepsilon n$. Thus under this condition $\{x_1,\dots,x_m\}\in X_\0(m-1)$.
    
    Next we prove by induction that any $\{x_1,\dots,x_m\} \in X_\0(m-1)$ satisfies the conditions on $\cd{B}$. The statement trivially holds when $m = 0$. Suppose the same statement holds for all $i$-faces where $i\le m-2$. 

    By \eqref{eq:wt-to-cd}, we can deduce that every $\{x_1,\dots,x_{m-1}\} \in X_\0(m-2)$ satisfies:

    \begin{equation}\label{eq:m-2-constraint}
        \forall S\subset[m-1],\,S\neq \emptyset,~~ \sum_{B: |B\cap S| \text{ odd}} \cd{B}  = \varepsilon n ~.
    \end{equation}

    Then by the induction hypothesis we know that 
    $\cd{B} = \frac{4\varepsilon n }{M}$ for $B \neq \emptyset $ is the unique solution to the set of constraints above.

    Now consider the set of constraints for the $(m-1)$-face $\{x_1,\dots,x_m\}$. 
    There are in total $M$ variables and we denote them $\{\cd{B}',\cd{B\cup\{m\}}'\}_{B\subset[m-1]}$ where $\cd{B}' =  \card{\co{B}\cap \oco{m}}$ and $\cd{B\cup\{m\}}' = \card{\co{B}\cap \co{m}}$. 
    The new variables are related to variables in $\{\cd{B}\}_{B\subset[m-1]}$ via the identity $\cd{B} = \cd{B}'+\cd{B\cup\{m\}}'$ . 
    So \eqref{eq:m-2-constraint} can be expanded as:
    \begin{equation}\label{eq:m-1-constraint-1}
        \forall S\subset[m-1],\,S\neq \emptyset,~~  \sum_{B: |B\cap S| \text{ odd}} \cd{B} = \sum_{B: |B\cap S| \text{ odd}} \cd{B}' + \sum_{B: |B\cap S| \text{ odd}} \cd{B\cup\{m\}}' =  \varepsilon n~.
    \end{equation}

    Additionally, via \eqref{eq:wt-to-cd} the weight constraints $wt\left(x_m + \sum_{i\in S} x_i\right) = \varepsilon n$ can be rewritten as:
    \begin{equation}\label{eq:m-1-constraint-2}
        \forall S\subset[m-1],\,S\neq \emptyset,~~ \sum_{B: |B\cap S| \text{ odd}} \cd{B}' + \sum_{B: |B\cap S| \text{ even}} \cd{B\cup\{m\}}' = \varepsilon n ~.
    \end{equation}

    Finally using the constraint $wt(x_m) = \varepsilon n$, we obtain equations:
    \begin{equation}\label{eq:m-1-constraint-3}
        \forall S\subset[m-1],\,S\neq \emptyset,~~\sum_{B: |B\cap S| \text{ odd}} \cd{B\cup\{m\}}' + \sum_{B: |B\cap S| \text{ even}} \cd{B\cup\{m\}}' = \card{\co{m}} = \varepsilon n ~.
    \end{equation}

    From \eqref{eq:m-1-constraint-1},\eqref{eq:m-1-constraint-2},and \eqref{eq:m-1-constraint-1} we can deduce that 
    \begin{equation}\label{eq:m-1-induction}
    \forall S\subset[m-1],\,S\neq \emptyset,~~ \sum_{B: |B\cap S| \text{ odd}} \cd{B\cup\{m\}}' = \varepsilon n/2\,.
    \end{equation}

    The $M/2-1$ constraints in \eqref{eq:m-1-induction} are exactly the same as the set of constraints in \eqref{eq:m-2-constraint} except that the RHS values are $\varepsilon n/2$ instead of $\varepsilon n$. Therefore by the induction hypothesis, the unique solution is $\forall B \neq \emptyset,\, \cd{B\cup\{m\}}'= \frac{2\varepsilon n}{M}$. As a result $\cd{B}'= \cd{B}-\cd{B\cup\{m\}}' = \frac{2\varepsilon n}{M}$. Finally we also have $\cd{m}' = \varepsilon n - (M/2-1)\cdot \frac{2\varepsilon n}{M} = \frac{2\varepsilon n}{M}$ and $\cd{\emptyset}' = \cd{\emptyset}-\cd{m}' = n - \frac{M-1}{M}\cdot 2\varepsilon n$. Thus we complete the induction step.
\end{proof}

\begin{corollary} \label{cor:face-decomposition} Let $m$ be an integer and $M = 2^m$.
For any $\{x_1,\dots,x_{m}\}\in X_\0(m-1)$, $x_m$'s nonzero coordinate set $\co{m}$ can be decomposed into the disjoint union of $M/2$ sets each of size $2\varepsilon n/M$ as follows,
\[\co{m} =  \dunion_{B\subseteq [m-1]} \co{B\cup\{m\}} \text{, where }  \co{B\cup\{m\}} = \bigcap_{i\in B\cup\{m\}} \co{i}  \bigcap_{j\not\in B\cup\{m\}} \oco{j}  ~.\]
\end{corollary}

Applying this coordinate decomposition to vertices and edges in the link graph of any face in $X_\0$, we obtain the following characterization of the graphs.
\begin{proposition} \label{prop:link-structure}
    Let \(0 \leq m \leq k-2\) be an integer and let \(M=2^m\). For any $m$-face in $X_{\varepsilon, n}$, the underlying graph of its link is isomorphic to
    \(J(\frac{2\varepsilon}{M}n,\frac{\varepsilon}{M}n,\frac{\varepsilon}{2M}n)^{M-1} \otimes J(cn,\frac{\varepsilon}{M}n,\frac{\varepsilon}{2M}n)\) where
    \(c = 1- \frac{M-1}{M}\cdot 2\varepsilon\).
\end{proposition}

\begin{proof}
    Give an $m$-face $\{x,x+x_1,\dots,x+x_m\}$ in $X_{\varepsilon, n}$, we note that by construction its link is isomorphic to the link of $\{0,x_1,\dots,x_m\}$. So without loss of generality we only consider $m$-faces of the form above and we denote such a face by $t$.
    The by definition the vertex set in the link of $t$ can be written as 
    \[X_t(0) =\sett{x_{m+1}}{\forall \emptyset\subsetneq S \subset[m+1], ~wt\left(\sum_{i\in S} x_i\right) = \varepsilon n}~. \]
    By the decomposition lemma \pref{cor:face-decomposition}, we have the following isomorphism between the link vertices and collections of disjoints sets in $\binom{n}{\varepsilon n/M}$. First recall the notation $C_B = \bigcap_{i\in B} \co{i}  \bigcap_{j\in [m]\setminus B}\oco{j}$. Then the isomorphism $\phi: X_t(0) \to \prod_{B\subseteq [m]} \binom{C_B}{\varepsilon n/M}$ is given by 
    
    \[ \phi(x_{m+1}) = \{ C_{m+1}\cap C_B \}_{B\subset[m]} ~.\]
    We use $\phi(x_{m+1})[B]$ to denote the set $C_{m+1}\cap C_B$.
    
    Similarly, using \pref{cor:face-decomposition} we can write the edge set as:
    \[X_{t}(1) =\sett{ \{x_{m+1},x'_{m+1}\} }{ x_{m+1},x'_{m+1}  \in X_\0(0),~ \forall B\subset [m],~ \card{\phi(x_{m+1})[B]\cap \phi(x'_{m+1})[B]} = \frac{\varepsilon n}{2M} }     \,.\]
    
    Therefore the adjacency matrix of $t$'s link graph has entries:
    \[A_{t}\left[ x_{m+1},x'_{m+1} \right] = \prod_{B \subset [m]} \mathbb{1}\left[\card{\phi(x_{m+1})[B]\cap \phi(x'_{m+1})[B]}  = \frac{\varepsilon n}{2M}  \right] \,.\]
    The product structure of $A_{t}$ gives rise to a natural decomposition into tensor products of matrices $H_B \in \{0,1\}^{ \binom{\co{B}}{\varepsilon n / M} \times \binom{\co{B}}{\varepsilon n / M} }$ where 
    \[ \forall~ U,V \in \binom{\co{B}}{\varepsilon n / M}, \,H_B[U,V] = \mathbb{1}\left[ \card{U\cap V} = \frac{\varepsilon n}{2M}   \right] \,.\]

    Furthermore since every nonempty $B\subset [m]$ has $\card{\co{B}} =\cd{B} = \frac{2\varepsilon}{M} n$ (\pref{lem:face-decomposition}), $H_B$ is the adjacency matrix of the Johnson graph $J(\frac{2\varepsilon}{M}n,\frac{\varepsilon}{M}n,\frac{\varepsilon}{2M}n)$. 
    In the case $B = \emptyset$, $\card{\co{B}} = cn$ (\pref{lem:face-decomposition}), and so $H_\emptyset$ is the adjacency matrix of $J(cn,\frac{\varepsilon}{M}n,\frac{\varepsilon}{2M}n)$. In conclusion, the link graph of $t$ is isomorphic to the tensor product \(J(\frac{2\varepsilon}{M}n,\frac{\varepsilon}{M}n,\frac{\varepsilon}{2M}n)^{M-1} \otimes J(cn,\frac{\varepsilon}{M}n,\frac{\varepsilon}{2M}n)\).
\end{proof}



\subsection{Local spectral expansion}
We now finish the local spectral expansion proof for Johnson complexes using \pref{prop:link-structure} together with the following Johnson graph eigenvalue bounds and the trickle-down theorem.

\begin{theorem}{\cite{DelsarteL1998}}
\label{thm:Johnson-graph-eigenvalue}
Let $A$ be the unnormalized adjacency matrix of \(J(n,k,\ell)\). The eigenvalues of $A$ are $\lambda_0,\dots,\lambda_k$ where 
\[\lambda_t = \sum_{i=0}^t (-1)^{t-i} \binom{k-i}{\ell-i} \binom{n-k+i-t}{k-\ell+i-t}\binom{t}{i} \,. \]
\end{theorem}
We note that in the theorem above $\lambda_i$ are unnormalized eigenvalues of the Johnson graphs. In the rest of the paper when we talk about spectral expansion, we consider normalized eigenvalues whose values are always in $[-1,1]$ . 

We specifically care about the second eigenvalues of the Johnson graphs with $n \ge 2k$. The case of $n = 2k$ is proved in \cite{Koshelev2023}.
\begin{lemma}{\cite{Koshelev2023}}\label{lem:Johnson-eigval-2k}
    The Johnson graph $J(2k,k,k/2)$'s eigenvalues satisfy $\max_{t=1,\dots,k} \frac{\abs{\lambda_t}}{\lambda_0}\le \frac{\abs{\lambda_2}}{\lambda_0} = \frac{1}{2k-2}$. 
\end{lemma}
In other words, the graph $J(2k,k,k/2)$ is a two-sided $\frac{1}{2k-2}$-spectral expander. 

For the case $n> (2+\eta)k$, we show the following lemma (proof in \pref{app:proofs-of-Johnson-lambda2}):
\begin{lemma}\label{lem:Johnson-eigval-g2k}
     For any constant $\eta > 0$, the Johnson graph $J((2+\eta)k,k,k/2)$'s eigenvalues satisfy that $\max_{t=1,\dots,k}\frac{\abs{\lambda_i}}{\lambda_0} \le  \frac{\eta}{2(1+\eta)}
$. In other words the graph is a $\frac{\eta}{2(1+\eta)}$-spectral expander. 
\end{lemma}

The trickle-down theorem states that local spectral expansion implies expansion of all links. 
\begin{theorem}[\cite{oppenheim2018local}] \label{thm:trickle-down}
    Let \(\lambda\geq 0\). Let \(X\) be a \(2\)-dimensional connected simplicial complex and assume that for any vertex \(v \in X(0)\), the link \(X_v\) is a two-sided \(\lambda\)-spectral expander, and that the $1$-skeleton of $X$ is connected. Then the $1$-skeleton of \(X\) is a two-sided \(\frac{\lambda}{1-\lambda}\)-spectral expander.
    In particular, \(X\) is a two-sided \(\frac{\lambda}{1-\lambda}\) local spectral expander.
\end{theorem}
Now we are ready to prove the main result of this section.
\begin{proof}[Proof of \pref{lem:Johnson-complex-spectral-expansion}]
    Let $t$ be any $m$-face in the complex. By \pref{prop:link-structure}, the link graph of $t$ is \(G_t = J(\frac{2\varepsilon}{M}n,\frac{\varepsilon}{M}n,\frac{\varepsilon}{2M}n)^{M-1} \otimes J(cn,\frac{\varepsilon}{M}n,\frac{\varepsilon}{2M}n)\).
    Since eigenvalues of $A_1\otimes A_2$ are $\sett{ \lambda_1\cdot\lambda_2 }{\lambda_1 \in  \mathsf{Spec}(A_1),~ \lambda_2 \in  \mathsf{Spec}(A_2)}$, 
    by \pref{eq:tensors} the second largest eigenvalues satisfy  $\lambda(A_1\otimes A_2) = \max (\lambda(A_1), \lambda(A_2))$. Therefore the link graph $G_t$ of $t$ has $\abs{\lambda}_2(G_t)  \le \max\left(\frac{1}{2\varepsilon n/M - 2}, \frac{1-2\varepsilon}{2(1-(2-M^{-1})\varepsilon)}\right) =\frac{1}{2}\cdot \frac{1-2\varepsilon}{1-(2-M^{-1})\varepsilon} \le \frac{1}{2}\cdot\left(1-\frac{\varepsilon}{M}\right)$ (via \pref{lem:Johnson-eigval-2k} and \pref{lem:Johnson-eigval-g2k}). Therefore the $k$-dimensional Johnnson complex $X_{\varepsilon,n}$ is a two-sided $\left(\frac{1}{2}-\frac{\varepsilon}{2^{k-1}}\right)$-local spectral expander.

 Furthermore, applying the trickle-down theorem to the $2$-skeleton $X^{\le 2}_{\varepsilon,n}$ we get that the underlying graph of the complex is a two-sided $\frac{(1-2\varepsilon)/(2-2\varepsilon)}{1 - (1-2\varepsilon)/(2-2\varepsilon)} = (1-2\varepsilon)$-spectral expander. 
\end{proof}


\subsection{Comparison to local spectral expanders from random geometric complexes}
In this part, we compare the structure of the Johnson complexes to that of the random geometric complexes from \cite{LiuMSY2023}. 
We observe that the former can be viewed as a derandomization of the latter.


The random $2$-dimensional geometric complex (RGC) $\mathrm{Geo}_d(N,p)$ is defined by sampling $N$ i.i.d. points $\{x_i\}_{i\in[N]}$ from the uniform distribution over the $d$-dimensional unit sphere, adding an edge between $ij$ if $\iprod{x_i,x_j} \ge \tau$ where $\tau$ is picked so that the marginal probability of an edge is $p$, and adding a $2$-face if the three points are pairwise connected. In the following parameter regime the complexes have expanding vertex links.

\begin{theorem}[Rephrase of Theorem 1.6 in \cite{LiuMSY2023}]
For every $\delta < \frac{1}{2}$, there exists $C_\delta$ and $c = \Theta\left(\frac{1}{\log (1/\delta)}\right)$ such that when $H \sim \mathrm{Geo}_d(N,N^{-1+c})$ for $d = C_\delta \log N$, with high probability every vertex link of $H$ is a two-sided $(1-2\delta)$-local spectral expander whose vertex links are $\left(\frac{1}{2} - \delta\right)$-spectral expanders.
\end{theorem}

The Johnson complex $X:=X_{\varepsilon,n}$ with $N = \card{X(0)}$ can be viewed as an approximate derandomization of the RGCs in $\R^{n}$ over $N$ vertices. Note that the points in \(X(0)\) (which are boolean vectors) embed in the natural way to \(\R^{n}\). Furthermore, these points in $\R^{n}$ all lie on a sphere of radius $\frac{\sqrt{n}}{2}$ centered at $[\frac{1}{2}, \dots,\frac{1}{2}]$. Two points are connected in $X$ if the $L_2$ distance between their images in $\R^{n}$ is exactly $\sqrt{\varepsilon n}$. The distinction from edges in RGCs is that in Johnson complexes points of distance $<\sqrt{\varepsilon n}$ are not connected. Albeit the analysis in \cite{LiuMSY2023} shows that most neighbors of a vertex have distance \(\approx \sqrt{\varepsilon n}\) so this is not a significant distinction. Similar to RGCs, three points in $X$ form a $2$-face if they are pairwise connected. By \pref{lem:Johnson-complex-spectral-expansion} the resulting $2$-dimensional Johnson complex $X$ has degree $N^{h(\varepsilon)}$\footnote{$h$ is the binary entropy function $h(p) = p\cdot\log\frac{1}{p} + (1-p)\log\frac{1}{1-p}$. } and its vertex links are $\left(\frac{1}{2} - \frac{\varepsilon}{2}\right)$-spectral expanders. 

\begin{remark}
    We further observe that while both constructions yield local spectral expanders with arbitrarily small polynomial vertex degree, the link graphs of the Johnson complexes are much denser. Use $N$ to denote the number of vertices in the complexes. Then in the Johnson complexes the link graphs have vertex degree $\mathsf{poly}(N)$ while in the RGCs the link graphs have average vertex degree $\mathsf{polylog}(N)$.
    We remark that one can use a graph sparsification result due to \cite{ChungH07} to randomly subsample the $2$-face in $X_{\varepsilon,n}$ while preserving their link expansion. As a result we obtain a family of random subcomplexes of the Johnson complexes that are local spectral expanders with similar vertex degree and link degree as the random geometric complexes. Detailed descriptions and proofs are given in \pref{app:sparsify-JC}.
\end{remark}



\subsection{Johnson Grassmann posets}
\label{sec:JGposet}
In this part, we take a small digression to define the Johnson Grassmann posets $Y_{\varepsilon,n}$ and show that their basifications are exactly Johnson complexes $X_{\varepsilon, n}$.

\begin{definition}[Johnson Grassmann posets $Y_{\varepsilon,n}$]\label{def:Johnson-admissible-sets}
Let $\diml = \log 2\varepsilon n$. For every $i\in[\diml]$, define the admissible function set as \[\mathcal{S}_i  = \sett{s: \F_2^i\setminus \set{0} \to \F_2^n}{\forall v\in \F_2^i~ wt(s(v)) = 2^{\diml - i},~ wt\left(\sum_{v\in \F_2^i \setminus \set{0}} s(v)\right) = 2^{\diml} - 2^{\diml - i}}.\] 
Define the Hadamard encoding $\widehat{s}:\F_2^i \to \F_2^n$ of $s$ as 
\[\widehat{s}(y) = \sum_{v \in \F_2^i\setminus \set{0}} \iprod{y,v}s(v).\]
For each admissible function $s\in \mathcal{S}_i$, use $\mathcal{G}_{s}$ to denote the Grassmann poset $ \sett{Im(\widehat{s}|_U)}{U \subseteq \F_2^i, \; U \text{ is a subspace}}$

Then the Johnson Grassmann poset is $Y_{\varepsilon,n} = \bigcup_{s\in \mathcal{S}_d} \mathcal{G}_{s}$.
\end{definition}

We first verify that these posets are well-defined. 
\begin{claim} For all $1\le i\le d$,
    $Y_{\varepsilon,n}(i) = \sett{Im(\widehat{s})}{s\in  \mathcal{S}_i}$ . 
\end{claim}
 
\begin{proof}
The claim trivially holds for $i = d$. For $i < d$, we need to prove that (1) for every $t \in \mathcal{S}_i$, there exist $s\in \mathcal{S}_d$ and $i$-dimensional subspace $U \subseteq \F^d_2$ such that  $Im(\widehat{t})= Im(\widehat{s}|_U)$ and (2) for every $s\in \mathcal{S}_d$, any subspace $W \subset Im(\widehat{s})$ is in $Y_{\varepsilon,n}$.

We first prove (1). Given any admissible function $t\in \mathcal{S}_i$, for every $v\in \F_2^i\setminus\set{0}$ and $u\in\F_2^{d-i}$, use $e_{t(v),u}\in\F^n_2$ to denote the vector that is nonzero only at the $u$-th nonzero coordinate of $t(v)$. Here we identify the vector $u$ in a natural way with a number in $[2^{d-i}]$. Use $z$ to denote the sum $z = \sum_{v\in \F_2^i \setminus \set{0}} t(v)$. Set $U$ to be the subspace that spans the first $i$ coordinates in $\F_2^d$, and set $s:\F_2^d\setminus \set{0} \to  \F_2^n$ as follows 
\[\forall v\in \F_2^i, u \in \F_2^{d-i},~s(v,u)  = \begin{cases}
     e_{t(v),u} ~ &v\neq 0,\\
     e_{\Vec{1} - z, u} ~&\text{o.w.} 
\end{cases}\]
Then it is straightforward to verify that $s \in  \mathcal{S}_d$ and that $Im(\widehat{t})=Im(\widehat{s}|_U)$.

Next we check (2). Note that the codespace $Im(\widehat{s})$ is a $d$-dimensional subspace in $\F_2^n$ whose nonzero elements have weight $2^{\diml}/2 = \varepsilon n$. Furthermore, consider any linear subspace $W\subset Im(\widehat{s})$ of dimension $j<d$. Since $\widehat{s}$ is an isomorphism, there exists a $j$-dimensional linear subspace $U\subset \F_2^d$ such that $W = Im(\widehat{s}|_U)$ which is the image of $\widehat{s}$ restricted to the domain $U$. Let $\psi\in \F_2^{d\times j}$ be a linear map from $\F_2^j$ to $U$, and $\phi \in \F_2^{j\times d}$ be such that $\phi\psi = Id$. Define the function $s_U:\F_2^j\setminus\set{0} \to \F_2^n$ as follows
\[ s_U(v) = \sum_{v'\in U^{\perp}} s(\psi v+v').\]
By construction, $s_U$ is admissible and furthermore for every $y\in \F_2^j$ 
\[\widehat{s}_U(y) = \sum_{v \in \F_2^j\setminus \set{0}} \iprod{y,v}\cdot\sum_{v'\in U^{\perp}} s(\psi v+v') = \sum_{v''\in \F_2^d\setminus \set{0}} \iprod{y,\phi v''} s(v'') =\sum_{v''\in \F_2^d\setminus \set{0}} \iprod{\phi^\top y,v''} s(v'') =  \widehat{s}(\phi^\top y).\]

We derive that $Im(\widehat{s}_U) = Im(\widehat{s}|_U) = W$. Thus $W \in Y_{\varepsilon,n}(j)$ and so the claim is proven.
\end{proof}

\begin{remark}
    We observe that $Y_{\varepsilon,n}$ is a subposet of the complete Grassmann poset $\Gr(\F_2^n)$, and is obtained by applying the sparsification operator $\M$ to  $\Gr(\F_2^n)^{\le \varepsilon n}$.
\end{remark}

We now prove that $X_{\varepsilon,n}$ is the basification of $Y_{\varepsilon,n}$ using the link decomposition result.

\begin{claim}\label{claim:Johnson-basification}
    The $k$-dimensional Johnson complex $X_{\varepsilon,n}$ is the $k$-skeleton of the basification of $Y_{\varepsilon,n}$.
\end{claim} 

\begin{proof}
    To prove the claim it suffices to show that any $\sp(x_1, x_2, \dots, x_k)\in Y_{\varepsilon,n}(k)$ if and only if all non-empty subset $T \subseteq [k]$ has $wt(\sum_{i \in T}x_i) = \varepsilon n$. In this proof we abuse the notation a bit and identify a vector in $\F_2^n$ with the set of its nonzero coordinates.
    
    For the if direction, given $\set{x_1,x_2,\dots,x_k}$, construct $s:\F_2^{k}\setminus \set{0} \to \F_2^n$ such that $s(v) = \cap_{i=1}^{k} x_i^{v_i}$ where $x_i^{v_i}$ is the set of nonzero coordinates of $x_i$ if $v_i = 1$, and the set of zero coordinates of $x_i$ if $v_i = 0$. By \pref{lem:face-decomposition}, all such $s(v)$'s have weights exactly $\varepsilon n/2^{k-1}$. By construction their nonzero coordinates are disjoint from each other's. So $s$ is an admissible function. Furthermore, let $\{e_1, e_2,\dots,e_k\}$ be elementary vectors in $\F_2^k$,
    \begin{align*}
        \widehat{s}(e_i) &= \bigcup_{v \in \F_2^k\setminus \set{0}} \iprod{e_i, v} s(v) \\
        &= \bigcup_{v, \text{ s.t.} v_i=1} x_i \cap_{j\neq i} x_j^{v_j} \\
        &= x_i \cap \bigcup_{v_{-i}\in \F_2^{k-1}} \cap_{j\neq i}x_j^{v_j} \\
        &= x_i \cap [n] = x_i
    \end{align*}
    Therefore $\sp(x_1,x_2,\dots,x_k) = Im(\widehat{s}) \in Y_{\varepsilon,n}(k)$.

    For the only if direction, we can easily check that any $Im(\widehat{s})\in Y_{\varepsilon,n}(k)$ contains only vectors of weight $\varepsilon n$. So given any basis $\set{x_1,x_2,\dots,x_k}$ of $Im(\widehat{s})$, their nonzero linear combinations all have weight $\varepsilon n$. Now the proof is complete.
\end{proof}

\section{Coboundary expansion and a van Kampen lemma} \label{sec:vk-lemma}
In this section we give the necessary preliminaries on coboundary expansion and also introduce a version of the famous van Kampen Lemma \cite{VanKampen1933} that shall be used to prove coboundary expansion in the next section. 
\subsection{Preliminaries on coboundary expansion}
\label{sec:prelim-coboundary}
We define coboundary expansion for general groups, closely following the definitions and notation in \cite{DiksteinD2023cbdry}. For a more thorough discussion on the connection between coboundary expansion and topology, we refer the reader to \cite{DiksteinD2023cbdry}. In the introduction we defined coboundary expansion as a notion of local testability of a code (which was a subspace of the functions on the edges). Instead of defining coboundary expansion for only $\F_2$-valued functions, we give the most general version of coboundary expansion, which uses arbitrary groups.


Let \(X\) be a \(d\)-dimensional simplicial complex for \(d \geq 2\) and let \(\Gamma\) be any group. Let $\dir{X}(1)$ be the directed edges in $X$ and $\dir{X}(2)$ be the directed $2$-faces. For \(i=-1,0\) let 
\(C^i(X,\Gamma) = \set{f:X(i) \to \Gamma}\). We sometimes identify \(C^{-1}(X,\Gamma) \cong \Gamma\). For \(i=1,2\) let
\[C^1(X,\Gamma) = \sett{f:\dir{X}(1) \to \Gamma}{f(u,v)=f(v,u)^{-1}}\]
and
\[C^2(X,\Gamma) = \sett{f:\dir{X}(2) \to \Gamma}{\forall \pi \in Sym(3), (v_0,v_1,v_2) \in \dir{X}(2) \; f(v_{\pi(0)},v_{\pi(1)},v_{\pi(2)}) = f(v_0,v_1,v_2)^{\sign(\pi)}}.\]
be the spaces of so-called \emph{anti symmetric} functions on edges and triangles. For \(i=-1,0,1\) we define operators \(\coboundary_i : C^i(X,\Gamma) \to C^{i+1}(X,\Gamma)\) by
\begin{enumerate}
    \item \(\coboundary_{-1}:C^{-1}(X,\Gamma)\to C^{0}(X,\Gamma)\) is \(\coboundary_{-1} h (v) = h(\emptyset)\).
    \item \(\coboundary_{0}:C^{0}(X,\Gamma)\to C^{1}(X,\Gamma)\) is \(\coboundary_{0} h (vu) = h(v)h(u)^{-1}\).
    \item \(\coboundary_{1}:C^{1}(X,\Gamma)\to C^{2}(X,\Gamma)\) is \(\coboundary_{1} h (vuw) = h(vu)h(uw)h(wv)\). 
\end{enumerate}
Let \(Id = Id_i \in C^i(X,\Gamma)\) be the function that always outputs the identity element. It is easy to check that \(\coboundary_{i+1} \circ \coboundary_i h \equiv Id_{i+2}\) for all \(i=-1,0\) and \(h \in C^{i}(X,\Gamma)\). Thus we denote by
\[Z^i(X,\Gamma) = \ker \coboundary_{i} \subseteq C^i(X,\Gamma),\footnote{$\ker \coboundary_{i}$ \text{ consists of all functions such that } $\coboundary_i f = Id$.}\]
\[B^i(X,\Gamma) = \Img \coboundary_{i-1} \subseteq C^i(X,\Gamma),\]
and have that \(B^i(X,\Gamma) \subseteq Z^i(X,\Gamma)\). 

Henceforth, when the dimension $i$ of the function $f$ is clear from the context we denote $\coboundary_i f$ by $\coboundary f$.

Coboundary expansion is a property testing notion so for this we need a notion of distance. Let \(f_1,f_2 \in C^i(X,\Gamma)\). Then
\begin{equation} \label{eq:def-of-dist}
    \dist(f_1,f_2) = \Prob[s \in \dir{X}(i)]{f_1(s) \ne f_2(s)}.
\end{equation}
We also denote the weight of the function \(\wt(f) = \dist(f,Id)\).

\begin{definition}[$1$-dimensional coboundary expander] \label{def:def-of-cob-exp}
    Let \(X\) be a \(d\)-dimensional simplicial complex for \(d \geq 2\). Let \(\beta >0\). We say that \(X\) is a $1$-dimensional \(\beta\)-coboundary expander if for every group \(\Gamma\), and every \(f \in C^1(X,\Gamma)\) there exists some \(g \in C^0(X,\Gamma)\) such that
    \begin{equation} \label{eq:def-of-cob-exp}
        \beta \dist(f,\coboundary g) \leq \wt(\coboundary f).
    \end{equation}
    In this case we denote \(h^1(X) \geq \beta\).    
\end{definition}
This definition in particular implies that \(B^1=Z^1\) for any coefficient group.

\paragraph{Coboundary expansion generalizes edge expansion} We note that $1$-dimensional coboundary expanders are generalizations of edge expanders to simplicial complexes of dimension $\ge 2$. A graph $G = (V,E)$ is a $1$-dimensional simplicial complex. For $\Gamma = \F_2$ we can analogously define function spaces $C^{-1}(G,\F_2), C^{0}(G,\F_2)$, $C^{1}(G,\F_2)$, and operators $\coboundary_{-1}$ and $\coboundary_{0}$ over $G$. 

Observe that $B^0(G,\F_2)$ consists of constant functions over the vertices. 
Using the notations above we can write the definition of a $0$-dimensional coboundary expander as follows. 
    \begin{definition}[$0$-dimensional coboundary expander in $\F_2$ ]
        A $1$-dimensional simplicial complex $G=(V,E)$ is a $0$-dimensional $\beta$-coboundary expander in $\F_2$ if for every $f \in C^0(G,\F_2)$ there exists some $g \in C^{-1}(G,\F_2)$ such that 
\begin{equation}\label{eq:edge-expansion}
        \beta \dist(f,\coboundary_{-1} g) \leq \wt(\coboundary_0 f).
    \end{equation}
    
    \end{definition}
    Let us check that this is equivalent to the definition of $\beta$-edge expanders in terms of graph conductance. 
    Viewing every function $f$
    as a partition of $V$ into two sets $V_f$ and $\overline{V_f}$ and viewing $\coboundary_0 f$ as the indicator function of whether an edge is in the cut between $V_f$ and $\overline{V_f}$, we derive that 
    \[ \min_{g\in C^{-1}(G,\F_2)}\dist(f,\coboundary_{-1} g) =  \min\left(\frac{|V_f|}{|V|},\frac{|\overline{V_f}|}{|V|}\right), ~\wt(\coboundary_0 f) = \frac{|E(V_f,\overline{V_f})|}{|E|}.\]
    Then the condition \pref{eq:edge-expansion} is equivalent to the statement that for every vertex partition $(V_f,\overline{V_f})$ of the graph $G=(V,E)$,    
    \[\beta \min\left(\frac{|V_f|}{|V|},\frac{|\overline{V_f}|}{|V|}\right) \leq \frac{|E(V_f,\overline{V_f})|}{|E|}.\]
    This is the standard definition of a $\beta$-edge expander.

\paragraph{Coboundary Expansion and Local Testability}One should view coboundary expansion as a property testing notion. We call \(B^1(X,\Gamma) \subseteq C^1(X,\Gamma)\) a code (as a subset, no linearity is assumed in the general case). Given a function in the ambient space \(f \in C^1(X,\Gamma)\) one cares about its Hamming distance to the code, that is, how many edges of \(f\) need to be changed in order to get a word in \(B^1(X,\Gamma)\). Denote this quantity \(\dist(f,B^1(X,\Gamma))\). Coboundary expansion relates this quantity to another quantity, the fraction of triangles \(\set{u,v,w}\) such that \(\coboundary f(uvw) \ne Id\), namely \(\wt(\coboundary f)\).

We note that if \(f =\coboundary g \in B^1(X,\Gamma)\), namely $f$ satisfies the property of being in the function space, then one can verify that \(\coboundary \coboundary g = Id\), or in other words, that \(\wt(\coboundary f)=0\). In a \(\beta\)-coboundary expander we get a robust inverse of this fact. That is, that if \(f\) is \(\varepsilon\)-far from any function in \(B^1(X,\Gamma)\), then its coboundary is not equal to the identity on at least a \(\beta \varepsilon\)-fraction of the triangles.

This suggests the following natural test for this property.
\begin{enumerate}
    \item Sample a random triangle \(\set{u,v,w} \in X(2)\).
    \item Accept if and only if \(f(uv) = f(uw) \cdot f(wv)\) (or equivalently \(\coboundary f (uvw) = Id\)).
\end{enumerate} Note that $ \wt(\coboundary f (uvw))$ is the probability that $f$ fails the test. Saying that a complex is a coboundary expander is saying that this test is sound, namely \(X\) is a $\beta$-coboundary expander if and only if for every \(f :\dir{X}(1) \to \Gamma\)
\begin{equation}\label{eq:test-coboundary}
    \dist(f,B^1(X,\Gamma)) \leq \frac{1}{\beta} \Prob[uvw]{f \text{ fails the test}}.\footnote{This distance is defined as $\min_{f'\in B^1(X,\Gamma)} \dist(f,f')$.}
\end{equation}

\snote{This paragraph doesn't mention the Gromov paper since it is talking about the testing view of coboundary expansion not the isoperimetric view.} Over \(\mathbb{F}_2\), this view of coboundary expansion was first observed by \cite{KaufmanL2014}, who used it for testing whether matrices over \(\mathbb{F}_2\) have low rank. This view also appeared in \cite{DinurM2019} for a topological property testing notion called cover-stability. In later works, this test was as a component in the analysis of agreement tests \cite{GotlibK2022, DiksteinD2023agr,BafnaM2023}. We also refer the reader to \cite{BafnaM2023,DiksteinD2023swap} for additional connections of coboundary expansion and local testing in the context of unique games problems and to \cite{ChapmanL2023stability} for its connection to group (te)stability.

\subsection{Contractions, cycles and cones} \label{sec:cones}
One general approach to proving coboundary expansion is the cone method invented by \cite{Gromov2010} and used in many subsequent works to bound coboundary expansion of various complexes \cite{LubotzkyMM2016,KaufmanM2018,KozlovM2019,KaufmanO2021,DiksteinD2023cbdry}. The following adaptation of the cone method to functions over non-abelian groups is due to \cite{DiksteinD2023swap}. \Ynote{Just to note: I am ok with doing things in this order but Avi will possibly want to first give the intuition and then explain? So maybe we can at least ref to what we said about this in the intro?} \snote{We provide intuition after the definition. It would be hard to explain them without defining cones and recall def of coboundary expander. I am for keeping it as is.} We adopt their definition of cones, and after presenting the definition, we will explain the motivation behind it as well as its connection to isoperimetric inequalities of the complex.

We use ordered sequences of vertices to denote paths in $X$, and $\circ$ to denote the concatenation operator \footnote{If $P= (v_0,v_1,\dots,v_m)$ and $Q = (u_0 = v_m,u_1,\dots,u_n)$ then their concatenation is $P\circ Q = (v_0,v_1,\dots,v_m,u_1,\dots,u_n)$.}, and for any path \(P \), we use $P^{-1}$ to denote the reverse path. 
Fixing a simplicial complex \(X\), we define two symmetric relations on paths in \(X\).
\begin{enumerate}
    \item[(BT)]~ We say that \(P_0 \overset{(BT)}{\sim} P_1\) if \(P_i = Q \circ (u,v,u) \circ Q'\) and \(P_{1-i} = Q \circ (u) \circ Q'\) where \(i\in\{0,1\}\).
    \item[(TR)]~ We say that \(P_0 \overset{(TR)}{\sim} P_1\) if \(P_{i} = Q \circ (u,v) \circ Q'\) and \(P_{1-i} = Q \circ (u,w,v) \circ Q'\) for some triangle \(uwv \in X(2)\) and \(i=0,1\).
\end{enumerate}

Let \(\sim\) be the smallest equivalence relation that contains the above relations (i.e. the transitive closure of two relations). We comment that these are also the relations defining the fundamental group \(\pi_1(X,v_0)\) of the complex $X$ (see e.g. \cite{Surowski1984}).

We also want to keep track of the number of \((TR)\) equivalences used to get from one path to another. Towards this we denote by \(P \sim_1 P'\) if there is a sequence of paths \((P=P_0,P_1,...,P_m=P')\) and \(j \in [m-1]\) such that:
\begin{enumerate}
    \item \(P_j \overset{(TR)}{\sim} P_{j+1}\) and
    \item For every \(j' \ne j\), \(P_{j'} \overset{(BT)}{\sim} P_{j'+1}\).
\end{enumerate}
\snote{notation general note check number correspondence and put defined things first.}
I.e. we can get from \(P\) to \(P'\) by a sequence of equivalences, where exactly one equivalence is by \((TR)\).

Similarly, for \(i >1\) we denote by \(P \sim_i P'\) if we get from \(P\) to \(P'\) using \(i\) triangles. 

A contraction is a sequence of equivalent \textit{cycles} that ends with a single point. More formally,
\begin{definition}
    Let \(C = (v_0,v_1,\dots,v_t,v_0)\) be a cycle around \(v_0\). A contraction \(\cont\) is a sequence of cycles \(\cont = (C_0 = C,C_1,...,C_m = (v_0))\) such that \(C_i \sim_1 C_{i+1}\) for every \(i=0,1,\dots,m-1\). If the contraction contains \(m\) cycles, we say it uses \(m\)-triangles, or simply that it is an \(m\)-contraction.
\end{definition}
A \(m\)-contraction can be viewed as a witness to \(C_0 \sim_m (v_0)\).

Before giving the definition of cones which are used to prove coboundary expansion, we provide some intuition behind the connection between coboundary expansion and cones.
\Ynote{I think maybe not adding words such as `codespace' and `message'? I don't feel too strongly about this though.}In order to prove \pref{eq:test-coboundary} for a complex $X$, we need to upper bound the distance from any function $f\in C^1(X,\Gamma)$ to the codespace $B^1(X,\Gamma)$ (with $\coboundary_0$ being the encoding function). This is done by constructing a message $g\in  C^0(X,\Gamma)$ from $f$ and $\mathcal{P}$ a family of paths from one vertex $v_0$ to every other vertex. 
Note that every edge together with the two paths in $\mathcal{P}$ from $v_0$ to the edge's endpoints form a cycle.
The ratio between the probability that $f$ fails the test and the distance from $f$ to the codespac is bounded by these cycles' length of contraction.  
The vertex $v_0$, the family of paths, and the contractions of the induced cycles constitutes a cone. 
The connection between the cone and coboundary expansion of $X$ is captured by \pref{thm:group-and-cones} and a detailed proof can be found in \cite{DiksteinD2023swap}.

\begin{definition}[cone]
    A cone is a triple \(\mathcal{A}=(v_0,\set{P_u}_{u \in X(0)}, \set{\cont_{uw}}_{uw \in X(1)})\) such that
\begin{enumerate}
    \item \(v_0 \in X(0)\).
    \item For every \( u \in X(0)\setminus\set{v_0}\), \(P_{u}\) is a walk from \(v_0\) to \(u\). For \(u = v_0\), we take \(P_{v_0} = (v_0)\) to be the cycle with no edges from \(v_0\) to itself.
    \item For every \(uw \in X(1)\), \(\cont_{uw}\) is a contraction of \(P_{u} \circ (u,w) \circ P_w^{-1}\).
\end{enumerate}
\end{definition}
The sequence \(\cont_{uw}\) depends on the direction of the edge \(uw\). We take as a convention that \(\cont_{wu}\) just reverses all cycles of \(\cont_{uw}\) (which is also a contraction Of the reverse cycle). Thus for each edge it is enough to define one of \(\cont_{uw},\cont_{wu}\). 

The maximum length of contraction of \(S_{uw}\) among all edges \(uw\) is called the cone area.
\begin{definition}[Cone area]
    Let \(\mathcal{A}\) be a cone. The \emph{area} of the cone is
        \[Area(\mathcal{A}) = \max_{uw \in X(1)} \abs{\cont_{uw}}-1.\]
\end{definition}

We note that in previous works this was known as the \emph{diameter} or \emph{radius} of a cone \cite{KaufmanO2021,DiksteinD2023swap}. We chose the term ``area'' here since it aligns with the geometric intuition behind triangulation.

Gromov is the first to use  the area of the cone to prove coboundary expansion for abelian groups. 
The following theorem which generalizes Gromov's approach to general groups is by \cite{DiksteinD2023swap}.
\begin{theorem}[{\cite[Lemma 1.6]{DiksteinD2023swap}}] \label{thm:group-and-cones}
    Let \(X\) be a simplicial complex such that \(Aut(X)\) is transitive on \(k\)-faces. Suppose that there exists a cone \(\mathcal{A}\) with area \(R\). Then \(X\) is a \(\frac{1}{\binom{k+1}{3}\cdot R}\)-coboundary expander for any coefficient group \(\Gamma\).
\end{theorem}


\subsection{Contraction via van Kampen diagrams}\label{sec:van-Kampen}
In this subsection we prove a version of van Kampen's lemma for our `robust' context. Previously, we defined the area of a cycle in the cone to be the length of its contraction sequence. The van Kampen lemma reduces the problem of bounding the contraction length of a cycle to counting the number of faces in a planar graph. 

The van Kampen diagrams are an essential tool in geometric group theorem \cite[Chapter V]{LyndonS1977}. They were introduced by van Kampen in 1933 in the context of the \emph{word problem} in groups defined by presentations \cite{VanKampen1933}, and is a powerful visual tool for exhibiting equivalences of words of some given group presentation. They are also used to bound the number of relations required to prove that a word is equal to identity (such a bound is called the \emph{Dehn function} of a presentation). 
The van Kampen diagrams are plane graphs that embed into the presentation complex of a group. 
\Ynote{I don't understand the following sentence:}\snote{Better now?} The van Kampen lemma states that a word is equivalent to the identity if and only if it forms a cycle in the presentation complex which can be tiled by faces in a van Kampen diagram. Furthermore the number of inner faces of the diagram provides an upper bound on the relations needed to prove that the word equals to identity. 
The advantage of these diagrams is that they bound the length of a sequential process by the size of static objects.

We need a similar tool in the context of abstract complexes, which is a simple generalization of the above setup. Our main motivation for this statement is give a simple and intuitive way to bound the area of cones.
Previous works that use the cones, give sequences of contractions, explaining which triangle goes after another (see e.g.\ \cite{DiksteinD2023swap}). 
This is sufficient for the proof. 
However, in all known examples, the ordering is just a formality, and the proofs were really describing a tiling 
(i.e. the contraction was really just a set of faces that formed a disk whose boundary is the cycle we needed to contract). So in this section we give justification to just drawing out the tiling via a van Kampen diagram instead of giving the full sequence of contraction.

A plane graph \(H=(V,E)\) is a planar graph together with an embedding. We will denote by \(H\) both the abstract graph and the embedding.

The van Kampen lemma has a few components. First we have our simplicial complex \(X\) and a cycle \(C_0\) in \(X\). In addition, we have a \emph{diagram}. This \emph{diagram} contains:
\begin{enumerate}
    \item A plane graph \(H = (V,E)\) that is \(2\)-vertex connected.
    \item A labeling \(\psi: V \to X(0)\) such that for every edge \(\set{u,v} \in E\), either \(\psi(u)=\psi(v)\), or \(\set{\psi(u),\psi(v)} \in X(1)\).
\end{enumerate}
It is easy to see that a walk in \(H\) is mapped by \(\psi\) to a walk in \(X\) (possibly with repeated vertices).

Let \(F = \set{f_0,f_1,\dots,f_\ell}\) be the set of faces of \(H\) with \(f_0\) being the \emph{outer} face. For every face let \(\tilde{C}_i\) be the cycle bounding \(f_i\), and let \(C_i = \psi(\tilde{C}_i)\). To be completely formal one needs to choose a starting point and an orientation to the cycle, but we ignore this point since the choice is not important to our purposes.

\begin{figure}[h]
     \centering
     \begin{subfigure}[b]{0.4\textwidth}
         \centering
\includegraphics[width=\textwidth]{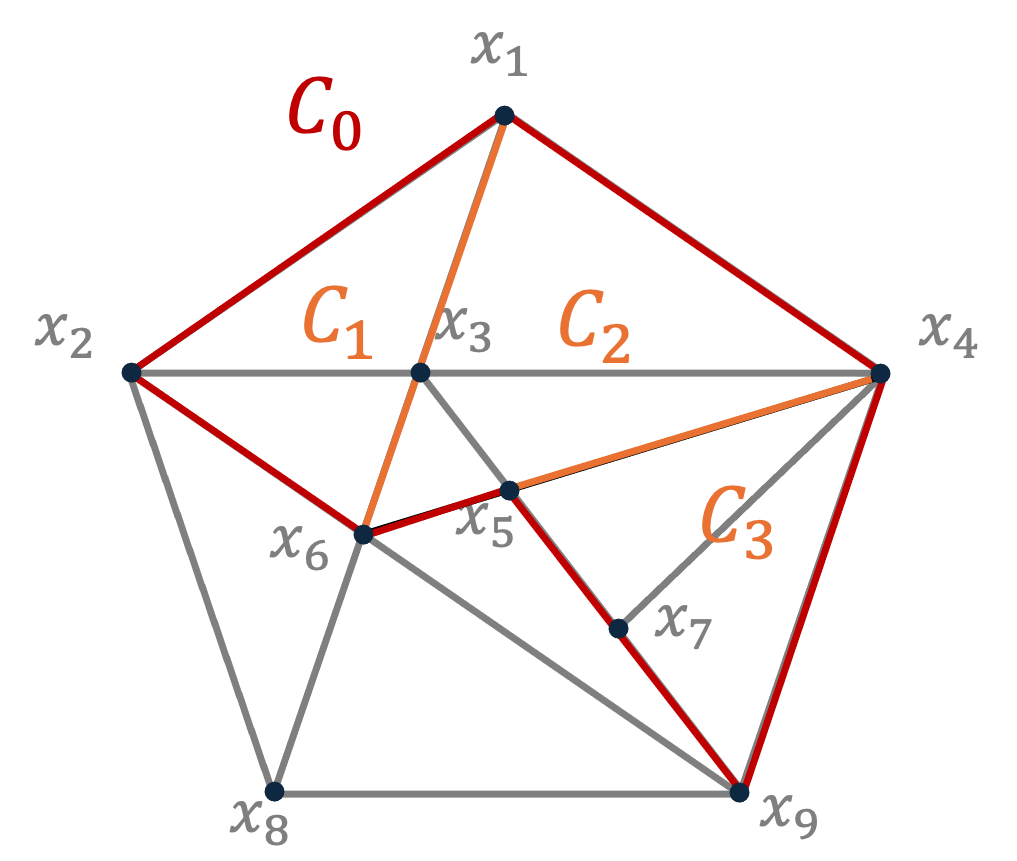}
    \vspace{1em}
         \caption{A $2$-dimensional complex $X$}
         \label{fig:vK-X}
     \end{subfigure}
    \hspace{40pt}
     \begin{subfigure}[b]{0.4\textwidth}
         \centering   \includegraphics[width=\textwidth]{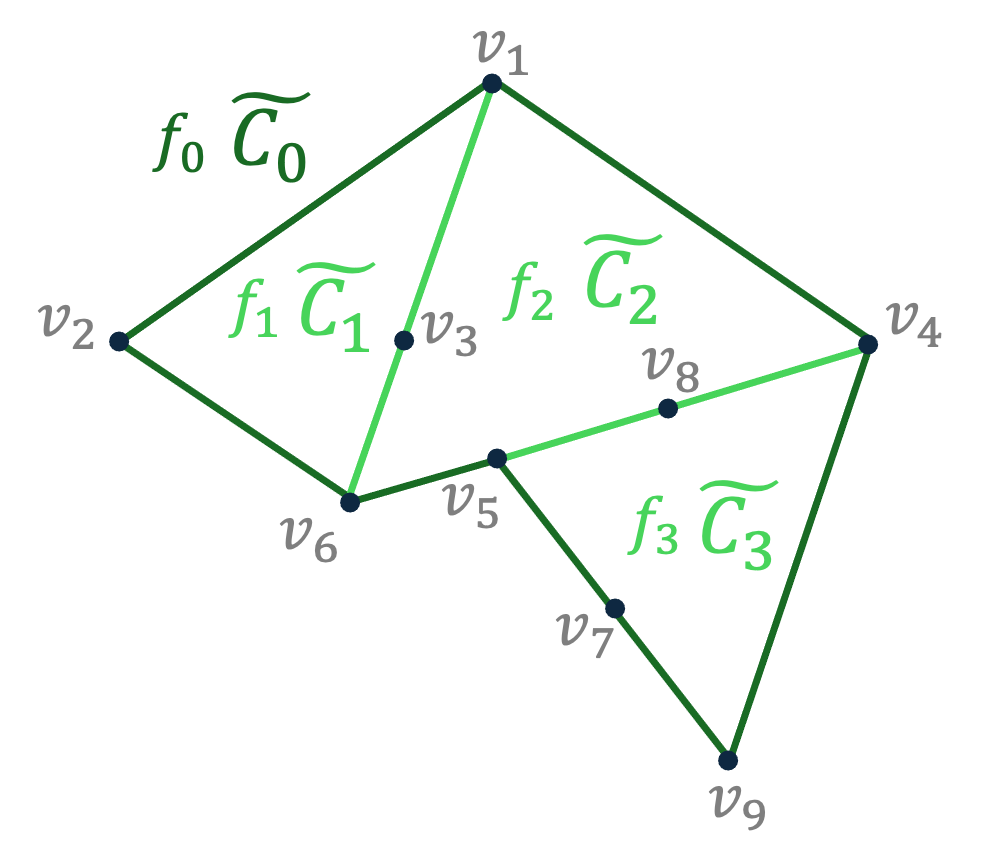}
         \vspace{5pt}
         \caption{A van Kampen diagram $H, F$ of $X$}
         \label{fig:vK-H}
     \end{subfigure}
     \caption{A van Kampen diagram for $X,C_0$}
\end{figure}

We are now ready to state our van Kampen lemma.
\begin{lemma}\torestate{\label{lem:van-kampen}
    Let \(X,C_0,H,\) and \(F\) be as above. Assume that for every \(i=1,2,\dots,\ell\) the cycle \(C_i\) has an \(m_i\)-triangle contraction. Then there exists an \(m_0\)-triangle contraction for \(C_0\) where \(m_0 = \sum_{i=1}^\ell m_i\).}
\end{lemma}

A proof of the lemma can be found in \pref{app:van-kampen}. Here we give a simply example to demonstrate how to use this lemma. 
Consider the $2$-dimensional complex $X$ in \pref{fig:vK-X} with $9$ vertices and $10$ $2$-faces (all triangles in the underlying graph are $2$-faces in $X$). Let $C_0 = (x_1,x_2,x_6,x_5,x_7,x_9,x_4,x_1)$ be a cycle in $X$. 

A van Kampen diagram $H,F$ for $X,C_0$ is given in \pref{fig:vK-H}. $H$ has $9$ vertices and $4$ faces $F = \{f_0,f_1,f_2,f_3\}$. We can easily check that $H$ is $2$-vertex connected and has the natural labeling $\psi(v_i) = x_i$ for $i\in [9]\setminus \set{8}$ and $\psi(v_8) = x_4$. The labeling satisfies that $\psi(\tilde{C}_i) = C_i$ for $i = 0,1,2,3$. 

Furthermore, its clear from the figure that $C_1$, $C_2$, and $C_3$ have $2$-triangle, $3$-triangle, and $2$-triangle contractions respectively. So by the van Kampen lemma, $C_0$ has a $2+3+2 = 7$-triangle contraction.


\subsection{Cone areas and isoperimetric inequalities} \label{sec:cones-and-area}
\Ynote{Cob exp (cones) is connected to word problem in groups there we simplify words via relations. A standard tool is to make it }
In this subsection we explain the connection between cone areas and isoperimetric inequality. In a \(2\)-dimensional simplicial complex, an isoperimetric inequality is an upper bound on the area of a cycle in terms of its length.
 In our context this is the following.
\begin{definition}[Isoperimetric inequality]
    Let \(X\) be a \(2\)-dimensional simplicial complex and let \(R > 0\). We say that \(X\) admits an \(R\)-isoperimetric inequality if for every cycle \(C\) with \(\ell\) edges there exists an \(R\ell\)-triangle contraction 
    \[Area(C) \le R\cdot Len(C) .\]
\end{definition}

Isoperimetric inequalities are particularly useful in presentation complexes of Cayley graphs in geometric and combinatorial group theory \cite{LyndonS1977}. Cones are closely related to isoperimetric inequalities as was already observed by \cite{Gromov2010}.

We formalize the connection with this proposition.
\begin{proposition} \label{prop:cone-like-isoperimetric-inequality}
    Let \(X\) be a simplicial complex with diameter \(D\) and let \(R>0\). Then
    \begin{enumerate}
        \item If \(X\) admits an \(R\)-isoperimetric inequality, then \(X\) has a cone whose area is at most \((2D+1)R\).
        \item If \(X\) has a cone with area \(a\) then \(X\) admits an \(a\)-isoperimetric inequality.
    \end{enumerate}
\end{proposition}
Here the diameter of a complex is the diameter of the underlying graph.

\begin{proof}[Sketch]
    For the first item, we construct a cone as follows. We fix an arbitrary vertex \(v_0 \in X(0)\), and as for the paths \(P_u\), we take arbitrary shortest paths from \(v_0\) to \(u\). Any cycle of the form \(P_u \circ (u,w) \circ P_w^{-1}\) (where \(\set{u,w} \in X(1)\)) has at most \(2D+1\) edges. We take the smallest contraction of this cycle as $\cont_{uw}$. By the isoperimetric inequality, \(\cont_{uw}\) uses at most \((2D+1)R\) triangles.
    
    As for the second item, let \(\mathcal{A}\) be a cone with area \(a\) and let \(C\) be a cycle with \(m+1\) edges \((w_0,w_1,\dots,w_m,w_0)\). Let \(P_0,P_1,\dots,P_m\) be the corresponding paths from $v_0$ to each of the vertices in \(w_0,w_1,\dots,w_m\). By \pref{lem:van-kampen}, we can contract $C$ via the van Kampen diagram in \pref{fig:cone-isoperimetric-inequality}. We note that the boundary cycle of every inner face in the diagram corresponds to \(P_{w_i} \circ (w_i,w_{i+1})\circ P_{w_{i+1}}\) (where $i \in \mathbb{Z}/(m+1)$) therefore its area is at most \(a\). We deduce the isoperimetric inequality $Area(C)\le a\cdot(m+1) = a\cdot Len(C)$.
    \begin{figure}
        \centering
        \includegraphics[scale=0.5]{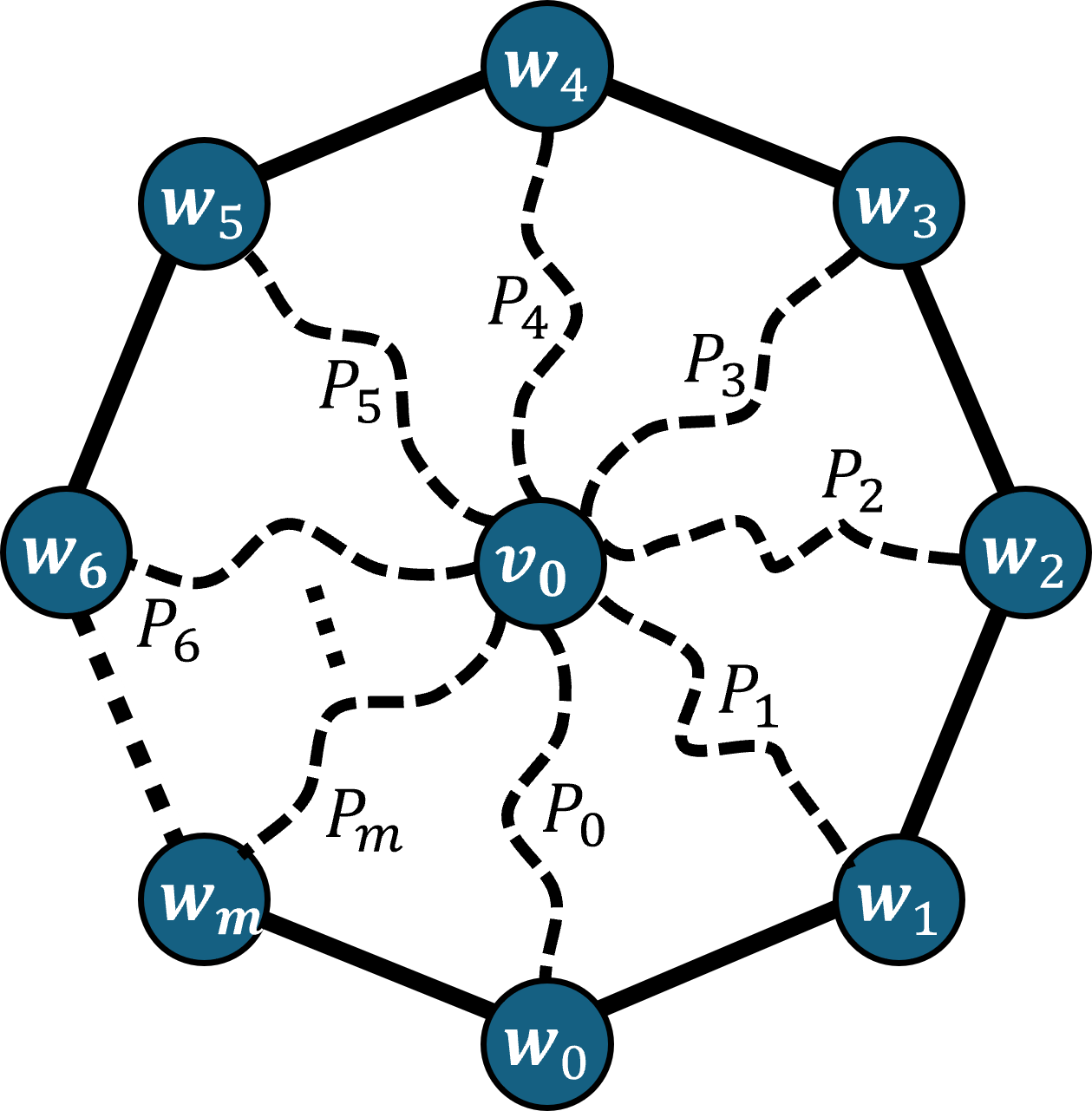}
        \caption{The second item of \pref{prop:cone-like-isoperimetric-inequality}}
        \label{fig:cone-isoperimetric-inequality}
    \end{figure}
\end{proof}

\subsection{Substitution and composition}

Later in \pref{sec:cone-for-X}, we will reduce the question of bounding the cone area of some complex $X$ to bounding that of $X'$ which contains $X$ as a subcomplex. The relevant definitions and auxiliary proposition for this approach are given below.

\begin{definition}
    Let \(X'\) be any complex and let \(X\) be a subcomplex of \(X'\) with $X(0) = X'(0)$. A substitution is a set \(\mathcal{P} = \sett{P_{uv}}{uv \in X'(1)}\) of paths \(P_{uv}\) from $u$ to $v$ in \(X\) such that $P_{uv} = P^{-1}_{vu}$ (so it suffices to give one path for each pair of vertices). We say that \(\mathcal{P}\) is an \(m\)-substitution if all paths have length at most \(m\).
\end{definition}
We note that the underlying complexes $X',X$ of a substitution are implicit in the notation, but they will always be clear in context. Via the substitution, every path in $X'$ has a corresponding path in $X$:
\begin{definition}
    Let \(X',X\) be as above and let \(\mathcal{P} = \sett{P_{uv}}{uv \in X'(1)}\) be a substitution. Let \(Q=(v_0,v_1,\dots,v_m)\) be a path in \(X'\). The composition \(\tilde{Q}\) is the path
    \[\tilde{Q} = P_{v_0v_1} \circ P_{v_1 v_2} \dots \circ P_{v_{m-1}v_m} \text{ in }X.\]
\end{definition}

The following proposition captures the key property of this substitution.
\begin{proposition} \label{prop:cone-composition}
    Let \(X'\) be any complex and let \(X\) be a subcomplex of \(X'\) containing all vertices of $X'$. Let \(\mathcal{P}\) be a substitution. Assume that for every triangle \(t=(u,v,w) \in X'(2)\), the composition \(\tilde{C_t}\) of the \(3\)-cycle \(C_t=(u,v,w,u)\) has an \(m\)-triangle contraction in \(X\). Then for every cycle \(C\) in \(X'\) that has an \(\ell\)-triangle contraction in \(X'\), the composition cycle \(\tilde{C}\) in \(X\) has an \(\ell m\)-triangle contraction .
\end{proposition}

To prove this proposition, we need the following claim about the contraction property of two cycles with shared vertices.

\begin{claim} \label{claim:shift-between-paths}
    Let \(C_s = P_1 \circ Q^{-1}\) and \( C_t = P_2 \circ Q^{-1}\) be two cycles in $X$ such that $Q$ is nonempty. If their difference cycle $C_\Delta = P_1 \circ P_2^{-1}$ has an \(m\)-contraction, then there is an \(m\)-contraction from \(C_s\) to \(C_t \).
\end{claim}

We delay the proof of this claim and prove the proposition assuming this claim.

\begin{proof}[Proof of \pref{prop:cone-composition}]
    Let \(\cont = (C_0 = C, C_1,\dots,C_\ell)\) be the contraction in $X'$. Let \(\tilde{C}_0,\tilde{C}_1,\dots,\tilde{C}_k\) be their compositions in $X$. If \(C_i \sim_1 C_{i+1}\) via a TR relation \((u,v) \leftrightarrow (u,w,v)\) and some BT relations in $X'$, by \pref{claim:shift-between-paths} \(\tilde{C}_i \sim_m \tilde{C}_{i+1}\) via the $m$-triangle contraction of \(P_{uv} \circ P_{uw} \circ P_{wv}\) in $X$. 
    Since there are $\ell$ \((TR)\) steps in the contraction of $C$, we have $\tilde{C} \sim_{\ell m} \tilde{C}_\ell$. Lastly, since \(C_\ell\) is a point, by definition of composition so is \(\tilde{C}_\ell\). So the proposition is proven.
\end{proof}

\begin{proof}[Proof of \pref{claim:shift-between-paths}]
    We first contract \(C_s\) to the intermediate cycle \(P_1 \circ P_2^{-1} \circ P_2 \circ Q^{-1} = C_\Delta \circ C_t\) via \((BT)\) relations. Then we contract \(C_\Delta \) to a point without touching the part \(C_t\). This procedure gives a $m$-contraction from $C_s$ to $C_t$.
\end{proof}

\section{Coboundary expansion} \label{sec:cob-exp-johnson}
Our main result in this section is proving that the Johnson complexes $X_{\varepsilon,n}$ and their vertex links are \Ynote{make sure we say just cob-exp} coboundary expanders via  \pref{thm:group-and-cones}. In addition, in \pref{sec:coboundary-Golowich}, we also show that the vertex links of the complexes constructed by \cite{Golowich2023} are coboundary expanders.
\begin{theorem}\label{thm:coboundary-Johnson-complex}
    Let $n\in \mathbb{N}$ and $\varepsilon>0$ satisfy that $4\, |\, \varepsilon n$. Then the $2$-dimensional Johnson complex $X_{\varepsilon}$ is a $1$-dimensional $\Omega(\varepsilon)$-coboundary expander.
    
\end{theorem}
To show the statement via \pref{thm:group-and-cones}, it boils down to show that $Aut(X_{\varepsilon})$ is transitive on $2$-faces and that there exists a cone $\mathcal{A}$ with area $\Theta(\frac{1}{\varepsilon})$.
The first claim is straightforward. The automorphism group $Aut(X_{\varepsilon})$ contains all permutations of coordinates and all translations by even-weight vectors. In particular, it is transitive on all triangles. So it is enough to bound the cone area of $X_\varepsilon$ using the following lemma. 
\begin{lemma} \label{lem:cone-for-X}
      The Johnson complex \(X_\varepsilon\) has a cone $\mathcal{A}$ with area \(\Theta(\frac{1}{\varepsilon})\).
\end{lemma}
We note that by \pref{prop:cone-like-isoperimetric-inequality}, this immediately implies the following corollary.
\begin{corollary} \label{cor:isoperimetric-inequality}
    The complex \(X_\varepsilon\) admits an \(\Theta(\frac{1}{\varepsilon})\)-isoperimetric inequality. \(\qed\)
\end{corollary}

In addition, we also show that every vertex link of $X_\varepsilon$ is a coboundary expander.
\begin{theorem}\label{thm:coboundary-Johnson-link}
    Let $n\in\mathbb{N}$ and $\frac{1}{4}\ge\varepsilon > 0$ satisfies that $8\,|\,\varepsilon n$. In the $3$-dimensional Johnson complex $X_{\varepsilon}$, every vertex link $(X_{\varepsilon})_v$ is a $1$-dimensional $\frac{1}{85}$-coboundary expander.
\end{theorem}
\begin{remark}
    Experts may recall the local-to-global theorem in \cite{EvraK2016} which states that if a complex is both a good local-spectral expander and its vertex links are $1$-dim coboundary expanders then the complex is also a $1$-dim coboundary expander. So one may suspect that the coboundary expansion theorem \pref{thm:coboundary-Johnson-complex} can be derived from coboundary expansion of links (\pref{thm:coboundary-Johnson-link}) and local-spectral expansion of $X_\varepsilon$ (\pref{lem:Johnson-complex-spectral-expansion}). However, this approach fails since the local-spectral expansion of $X_\varepsilon$ is not good enough for applying the local-to-global theorem. So instead, we prove coboundary expansion of $X_\varepsilon$ by directly using the cone method.
\end{remark}

By link isomorphism, we can focus on the link of $\0$ without loss of generality. Then the automorphism group $Aut(X_\0)$ contains all permutations of coordinates and so is transitive over  $X_\0(2)$ (\pref{lem:face-decomposition}). By \pref{thm:group-and-cones} it is enough to show that there exists a cone with area $85$.
\begin{lemma} \label{lem:cone-for-link-X}
    Let \(0<\varepsilon \leq \frac{1}{4}\). Any vertex link \((X_{\varepsilon})_v\) has a cone with area at most \(85\).
\end{lemma}

As above, by \pref{prop:cone-like-isoperimetric-inequality}, this implies:
\begin{corollary} \label{cor:isoperimetric-inequality-link}
    A link of a vertex \(v \in X_\varepsilon\) admits an \(85\)-isoperimetric inequality. \(\qed\)
\end{corollary}

%
\subsection[Constructing a bounded area cone for the Johnson complex]{Constructing a bounded area cone for \(X_\varepsilon\)}\label{sec:cone-for-X}
In this subsection we give a proof of \pref{lem:cone-for-X}. We take a two-step approach to showing that $X_\varepsilon$ has a cone of bounded area. 
First we define a related complex $X_\varepsilon'$ which contains $X_\varepsilon$ as a subcomplex, and construct a cone for $X_\varepsilon'$ with bounded area. Next, we replace every edge in $X_\varepsilon'$ with a path in $X_\varepsilon$ between the same two endpoints. Then effectively the edges forming a triangle in $X_\varepsilon'$ become a cycle in $X_\varepsilon$. The second step is then to find a bounded-size tiling for every such cycle inside $X_\varepsilon$.

The key proposition used in this approach is as follows. 
\begin{proposition}[Restatement of \pref{prop:cone-composition} ] \label{lem:cone-area-composition}
    Let \(X'\) be any complex and let \(X\) be a subcomplex of \(X'\) containing all vertices. Let \(\mathcal{P}\) be a substitution. Assume that for every triangle \(t=\set{u,v,w} \in X'(2)\), the composition \(\tilde{C_t}\) of the \(3\)-cycle \(C_t=(u,v,w)\) has an \(m\)-triangle contraction in \(X\). 
    Then for every cycle $C$ in $X'$ that has an $\ell$-triangle contraction in $X'$, the composition cycle $\tilde{C}$ in $X$ has an $\ell m$-triangle contraction.
\end{proposition}

This proposition implies that if 
for the boundary cycle $C_t$ of every triangle $t \in X'(2)$, its composition $\tilde{C_t}$ has a $\Theta(1)$-triangle contraction in $X$, then the cone area of $X$ is asymptotically the same as that of $X'$.
This lemma suggests a two-step strategy towards showing \pref{lem:cone-for-X}. First we construct a $X_\varepsilon'$ which has a cone of area $\Theta(1/\varepsilon)$ (\pref{claim:X'-good-cone} is proved in \pref{sec:cone-area-X'}), and then we show that the boundary of $2$-faces in $X_\varepsilon'$ has $\Theta(1)$-contraction in $X_\varepsilon$ (\pref{claim:contract-X'(2)} is proved in \pref{sec:contracting-X(2)}).

\begin{claim} \label{claim:X'-good-cone}
    $X_\varepsilon'$ has a cone of area \(\le 16/\varepsilon \).
\end{claim}

\begin{claim}\label{claim:contract-X'(2)}
    For every triangle \(t=\set{u,v,w} \in X_\varepsilon'(2)\) whose corresponding boundary cycle is \(C_t=(u,v,w,u)\), we can contract the composition \(\tilde{C}_t\) in $X_\varepsilon$ using \(O(1)\) triangles.
\end{claim}

Postponing the proofs of these lemma and claims till later subsections, we are ready to derive \pref{lem:cone-for-X}. 

\begin{proof}[Proof of \pref{lem:cone-for-X}]
    \pref{claim:X'-good-cone} implies that the complex $X_\varepsilon'$ has cone area $O(\frac{1}{\varepsilon})$, while \pref{claim:contract-X'(2)} implies that for every triangle boundary $C$ in $X_\varepsilon'$, its composition $\tilde{C}$ has a $O(1)$-triangle contraction in $X_\varepsilon$. By  \pref{lem:cone-area-composition} we conclude that he cone area of $X_\varepsilon$ is $O(\frac{1}{\varepsilon})$.
\end{proof}



%

Before moving on, we consider two ``patterns'' of contractions shown in \pref{fig:mid-v} and \pref{fig:mid-p}. They are going to be used extensively in the later proofs. 

\begin{definition}[Middle vertex contraction.] \label{def:middle-vertex}
    A cycle $C$ has a middle vertex contraction in a simplicial complex \(X\) if there is a vertex \(z\) so that for every edge \(\set{u,w} \in C\), \(\set{z,u,w} \in X(2)\).
\end{definition}

\begin{claim}[Middle vertex pattern] \label{claim:middle-vertex}
    Let \(C\) be a length \(m\) cycle in a simplicial complex $X$. If $C$ has a middle vertex contraction, then \(C\) has an \(m\)-triangle contraction.
\end{claim}

\begin{proof}
    Given such a cycle $C$ and a center vertex $z$, we can construct a van-Kampen diagram $(H, \psi)$ for the triangles between $z$ and edges in $C$, where $H$ is the plane graph given by \pref{fig:mid-v} and $\psi$ is the identity map. Note that $H$ has $m$ inner faces each corresponding to a triangle and an outer face whose boundary is $C$. Then by \pref{lem:van-kampen}, $C$ has an $m$-triangle contraction.
\end{proof}

\begin{definition}[Middle path contraction.] \label{def:middle-path}
    A $4$-cycle \(C=(v,u,w,u')\) has a length $m$ middle path contraction in a simplicial complex \(X\) if there is a length $m$ walk \((u=u_1,u_2,\dots,u_m=u')\) so that for every edge \(\set{u_i,u_{i+1}}\), the triangles \(\set{u_i,u_{i+1},v}, \set{u_i,u_{i+1},w} \in X(2)\). 
\end{definition}

\begin{claim}[Middle path pattern] \label{claim:middle-path}
    Let \(C=(v,u,w,u')\) be a \(4\)-cycle in a simplicial complex $X$. If $C$ has a length $m$ middle path contraction, then \(C\) has a \(2m\)-contraction.
\end{claim}

\begin{proof}
    We again construct a van-Kampen diagram $(H, \psi)$ for the set of triangles \(\set{u_i,u_{i+1},v}, \set{u_i,u_{i+1},w} \in X(2)\), where $H$ is the plane graph given by \pref{fig:mid-p} and $\psi$ is the identity map. Note that $H$ has $2m$ inner faces each corresponding to a triangle and an outer face whose boundary is $C$. Then by \pref{lem:van-kampen}, $C$ has an $2m$-triangle contraction.
\end{proof}
\begin{figure}
     \centering
     \begin{subfigure}[b]{0.3\textwidth}
         \centering
    \includegraphics[width=\textwidth]{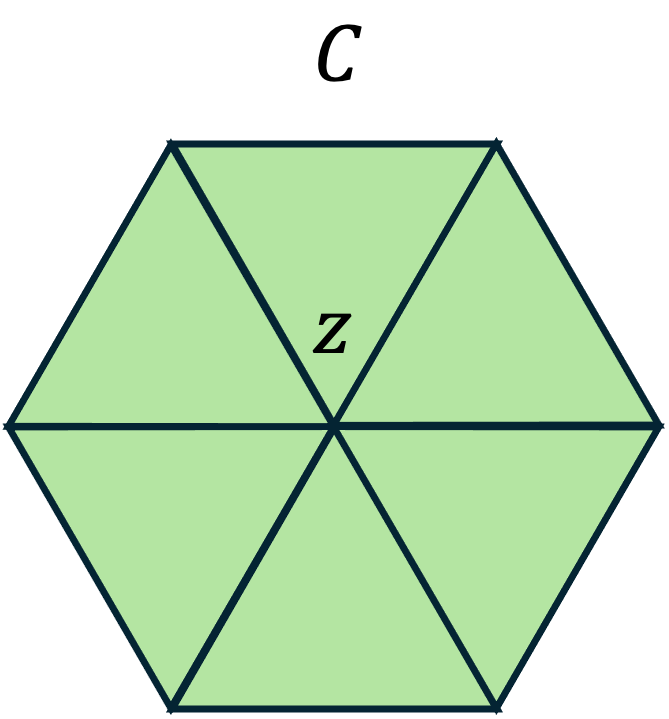}
    \vspace{1em}
         \caption{Middle vertex contraction}
         \label{fig:mid-v}
     \end{subfigure}
     \hspace{40pt}
     \begin{subfigure}[b]{0.4\textwidth}
         \centering
    \includegraphics[width=\textwidth]{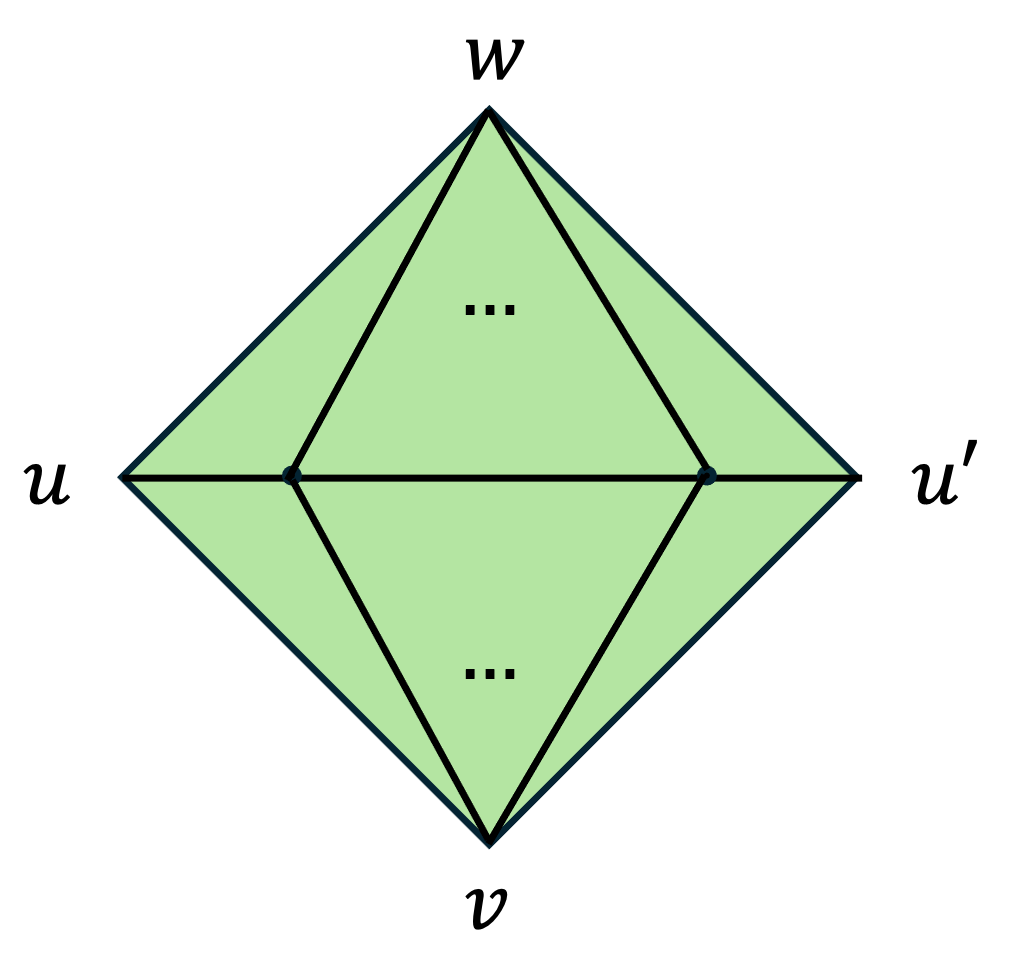}
         \caption{Middle path contraction}
         \label{fig:mid-p}
     \end{subfigure}
     \caption{Two contraction patterns}
\end{figure}

\subsubsection[X' and its cone A]{$X_\varepsilon'$ and its cone $\mathcal{A}'$}
In this part, we construct a $2$-dimensional simplicial complex $X_\varepsilon'$ such that $X_\varepsilon'$ has a $2$-substitution in $X_\varepsilon$.
\begin{definition}
    Let \(S_{\leq \varepsilon} = \sett{x \in \mathbb{F}_2^n}{wt(x) \leq \varepsilon n, wt(x) \text{ is even}}\). The $2$-dimensional complex $X'_\varepsilon$ is given by
\[X_\varepsilon'(0) = X(0),\]
\[X_\varepsilon'(1) = \sett{\set{x,x+s}}{x \in X(0), s \in S_{\leq \varepsilon}},\]
\[X_\varepsilon'(2) = \sett{\set{x,x+s_1,x+s_2}}{x \in X(0), s_1,s_2,s_1+s_2 \in S_{\leq \varepsilon}}.\]
\end{definition}

\begin{claim}\label{claim:X-2-sub}
    $X_\varepsilon'$ has a $2$-substitution in $X_\varepsilon$.
\end{claim}
\begin{proof}
    For any edge $(x, x+s)\in X_\varepsilon'$ (where $s \in S_{\le \varepsilon}$), find $s_1,s_2 \in S_\varepsilon$ such that $s = s_1 + s_2$. Such $s_1$ and $s_2$ can be constructed as follows: first split \(s=t_1+t_2\) into two vectors of equal weight. Then take any vector \(r\) of weight \(\varepsilon n - \frac{|s|}{2}\) with disjoint support from \(s\), and set \(s_i=r+t_i\). Now by setting $P_{x,x+s} = (x,x+s_1, x+s_1+s_2 = x+s)$, we obtain a $2$-substitution for $X_\varepsilon'(1)$ in $X_\varepsilon$. 
\end{proof}

To bound the cone area of \(X_\varepsilon'\), we start by constructing a cone $\mathcal{A}'$ for $X_\varepsilon'$. In the rest of the section we drop the subscript $\varepsilon$ and simply write $X$ and $X'$. Now we introduce a few definitions in order to describe $\mathcal{A}'$ concisely.

\begin{definition}[Monotone path]
    A path \(P=(v_0,v_1,v_2,\dots,v_m)\) in $X'$ is called \emph{monotone} if \(v_i \subseteq v_{i+1}\), where we identify the string with the set of coordinates that are equal to \(1\).  
\end{definition}

Analogously we can define monotone cycles.
\begin{definition}[Monotone cycle]
    A cycle \(C\) is monotone if it can be decomposed to \(C = P_1 \circ P_2^{-1}\) where \(P_1,P_2\) are monotone paths of \textit{equal} length.
\end{definition}

\begin{definition}[Lexicographic path]
    Fix an ordering on the coordinates of vectors in $X(0)$. A path \(P=(v_0,v_1,v_2,\dots,v_m)\) in $X'$ is called \emph{lexicographic} if the path is monotone \textit{and} every element in $v_{i+1}\setminus v_i$ is larger than all elements in $v_i$ under the given ordering.  
\end{definition}


We also define half-step lexicographic path.
\begin{definition}[Half-step lexicographic path]
    A path \(P=(v_0,v_1,v_2,\dots,v_m)\) in $X'$ is called a half-step lexicographic path if it is a lexicographic path from $v_0$ to $v_m$ such that all but the last edge \(\set{v_i,v_{i+1}}\) in the path satisfy that \(\abs{v_{i+1} \setminus v_i} = \frac{\varepsilon n}{2}\), while the last edge satisfies \(\abs{v_{m} \setminus v_{m-1}} \le \frac{\varepsilon n}{2}\).  
\end{definition}
Note that if \(v_0\) and \(v_m\) are connected then there is a unique half-step lexicographic path between them.  

Analogously we can define half-step lexicographic cycles.
\begin{definition}[Half-step lexicographic cycle]
    A cycle \(C\) in $X'$ is half-step lexicographic if it can be decomposed to \(C = P_1 \circ (x,a) \circ (a,y) \circ P_2^{-1}\) where \(P_1,P_2\) are half-step lexicographic paths of \textit{equal} length, and that $a \supseteq x \cup y$.
\end{definition}
This is saying that in the half-step lexicographic cycle $C$, every edge $\{v_i,v_{i+1}\}$ in $P_1$ and $P_2$ satisfies that  \(\abs{v_{i+1} \setminus v_i} = \frac{\varepsilon n}{2}\), while the endpoints of the two paths satisfy that $\abs{a \setminus x} \le \varepsilon n$ and $\abs{a \setminus y} \le \varepsilon n$. 


We consider the following cone $\mathcal{A}'$: 
\begin{definition}[Cone of $X'$]\label{def:cone-X'}
    Fix an ordering on the coordinates of vectors in $X(0)$. Set the base vertex \(v_0 = 0\). For every \(v \in X(0) \setminus\{v_0\}\) we take \(P_v\) to be the half-step lexicographic path from $v_0$ to $v$ in $X'$.  
\end{definition}




\subsubsection[Cone area of X']{Bounding the area of $\mathcal{A}'$: proof of \pref{claim:X'-good-cone}} \label{sec:cone-area-X'}

We prove that for every $uv\in X'(1)$, the cycle $C = P_u\circ uv \circ P_v$ has a $\Theta(\frac{1}{\varepsilon})$-triangle contraction. 
The strategy is to decompose $C$ into a set of cycles in $X'$ and compose these cycles' contractions together to obtain a contraction for $C$. 
The cycle decomposition is done by a recursive use of \pref{lem:van-kampen}.

In the rest of the section, we will use both $+,-$ and set operations $\setminus, \cup, \cap, \overline{\phantom{v}}$ (exclusion, union, intersection, and complement) over boolean strings in $X'(0) \subset \{0,1\}^n$. $+,-$ are entry-wise XOR for the boolean strings, while set operations are applied to the support of a string (i.e. the set of nonzero entries in the string). 


The base case in the recursive argument is a monotone four-cycle, and we first show that all such cycles have $\Theta(1)$-contraction. 

\begin{claim}\label{claim:four-cycle-in-X-prime}
    Let \(C=(v_0,v_1,v_2,v_1',v_0)\) be any monotone four-cycle in \(X'\). Then \(C\) has a \(\le 4\)-triangle contraction.
\end{claim}
\begin{proof}
    By definition of monotone cycles, \(v_2\) has size \(\leq 2 \varepsilon n\) and that \(v_1,v_1' \subseteq v_2\). If \(\abs{v_2} \leq \varepsilon n\) we can contract \(C\) since \(\set{v_0,v_1,v_2},\set{v_0,v_1',v_2} \in X'(2)\).
    
    Otherwise \(\varepsilon n <\abs{v_2} \leq 2 \varepsilon n\), we intend to show that  $C$  has a length $\le 2$ middle path contraction. For these vertices we need to find a path between \(v_1,v_1'\) so that for every edge \(e\) in that path \(e \cup \set{v_0}, e \cup \set{v_2} \in X'(2)\). Let \(a = v_1 + v_1'\). We split to cases:
    \begin{enumerate}
        \item If \(\abs{a} \leq \varepsilon n\) then we just take the path to be \((v_1,v_1')\). 
        \item Otherwise \( \varepsilon n < \abs{a} \le 2\varepsilon n\). Then we split \(v_1 \setminus v_1' = b_1 \dunion b_2\) and \(v_1' \setminus v_1 = c_1 \dunion c_2\) both into equal-size parts as illustrated in \pref{fig:4-cyc-partition}. So both \(a_1 = b_1 + c_1\) and \(a_2 = b_2 + c_2\) have size at most $\varepsilon n$. We take the path \((v_1, v_1 + a_1, v_1' = v_1+a_1+a_2)\). Since the middle vertex \(v_1 + a_1\) is contained in \(v_2\) and has size between that of \(v_1\) and \(v_1'\), it is connected to \(v_2\) and to \(0\). 
    \end{enumerate}
    By \pref{claim:middle-path} there is a \(\le 4\)-triangle contraction of the cycle.
\end{proof}

\begin{figure}
         \centering
    \includegraphics[width=0.3\textwidth]{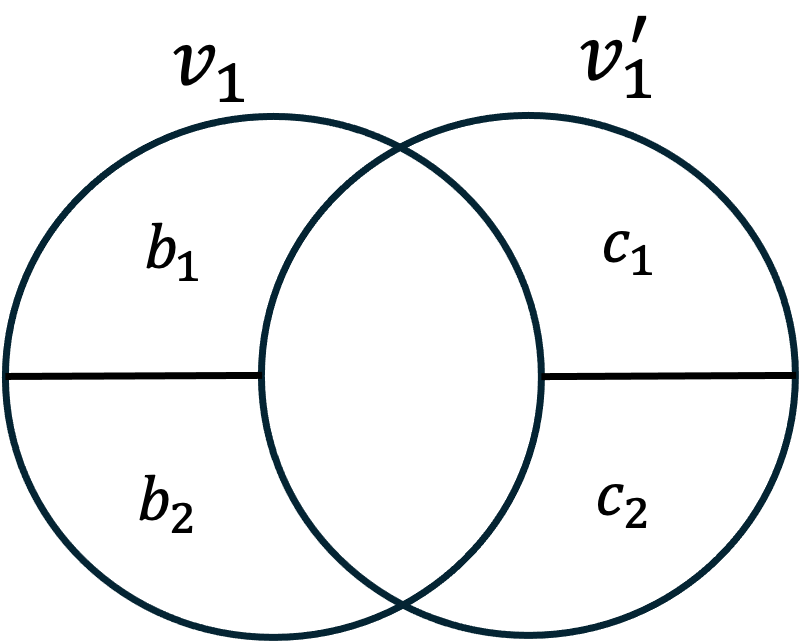}
         \caption{Partition of $v_1 \cup v_1'$ in the proof of \pref{claim:four-cycle-in-X-prime}}
         \label{fig:4-cyc-partition}

\end{figure}

The general statement for all half-step lexicographic cycles is as follows.  

\begin{lemma} \label{lem:monotone-path-contraction}
    Let \(m \geq 2\) and let \(C = (0,b_1,\dots,b_m,a,c_m,\dots,c_1,0)\) be a half-step lexicographic cycle of length $2(m+1)$. Then \(C\) has an \((8m-2)\)-contraction.
\end{lemma}

\begin{figure}
     \centering
    \includegraphics[width=0.3\textwidth]{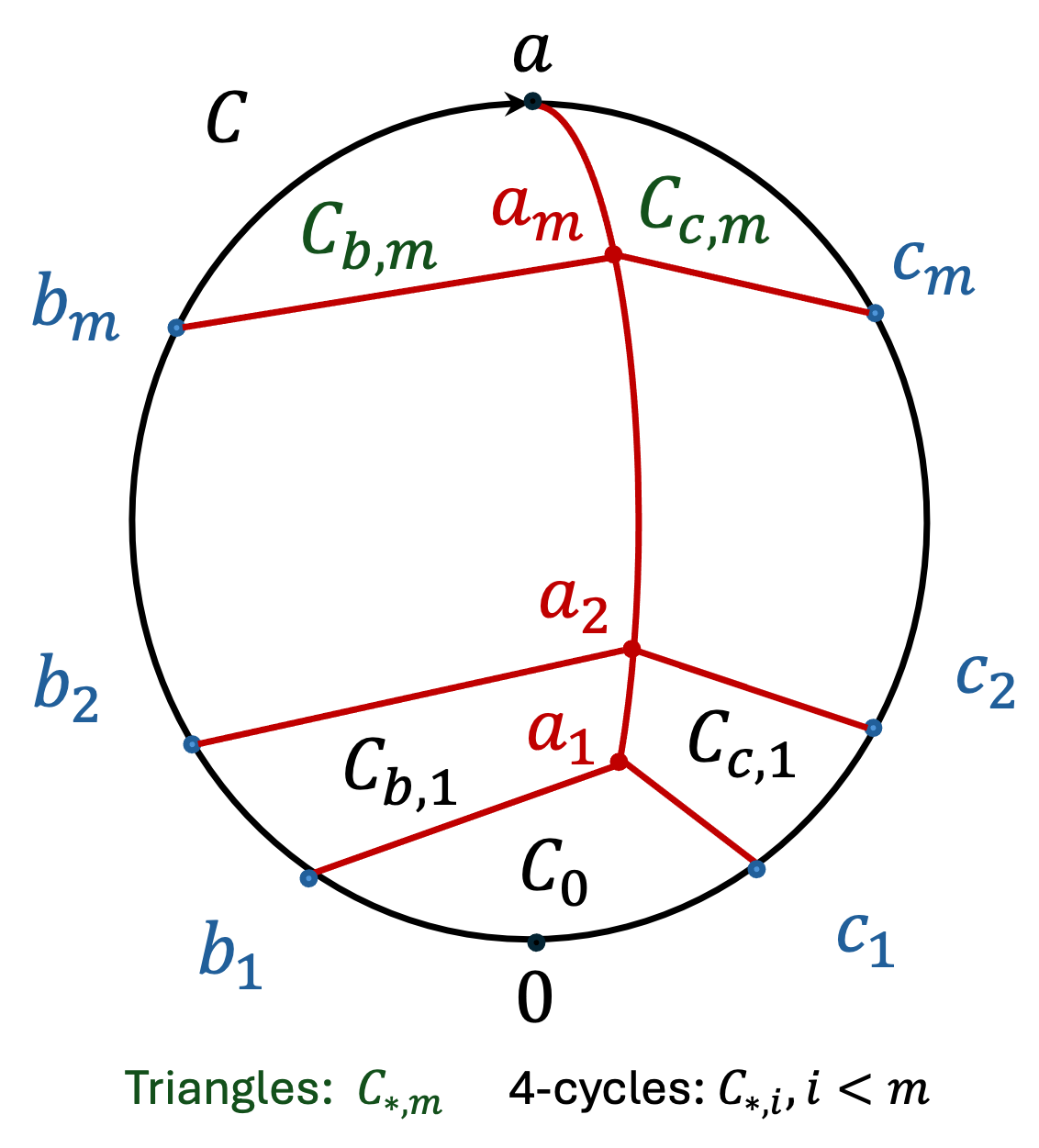}
         \caption{Half-step lexicographic cycle decomposition in \pref{lem:monotone-path-contraction}}
         \label{fig:mon-cyc-dp}
\end{figure}
\begin{proof}
We shall decompose $C$ into $2m-1$ monotone four-cycle and $2$ triangles in $X'$ as illustrated in \pref{fig:mon-cyc-dp}. Then applying \pref{claim:four-cycle-in-X-prime}, we deduce that $C$ has a $4\cdot (2m -1 )+ 2 = 8m-2$-contraction. To describe the decomposition, we define auxiliary vertices $a_i = b_i\cup c_i$ for $i=1,\dots,m$ and $a_{m+1} = a$. We next argue that $a_i$ is connected to $b_i, c_i$ and $a_{i+1}$ in $X'$. 

First to show that for all $i\in[m]$ $a_i$ is connected to $b_i$ and $c_i$, note that $\abs{b_i} = \abs{c_i} = i\cdot \varepsilon n/2$. Thus if $\abs{b_i + c_i} \le 2\varepsilon n$, then $\abs{b_i\setminus c_i} = \abs{c_i\setminus b_i} \le \varepsilon n$ and $a_i=b_i\cup c_i$ would be connected to both $b_i$ and $c_i$. To show $\abs{b_i + c_i} \le 2\varepsilon n$, we observe that $\abs{c_m \setminus b_m} \le \abs{a\setminus b_m} \le \varepsilon n$ and analogously $\abs{b_m \setminus c_m} \le \varepsilon n$. Now, assume for contradiction that for some $i< m$, $\abs{b_i + c_i} > 2\varepsilon n$. Similarly we deduce that $\abs{c_i \setminus b_i} , \abs{b_i \setminus c_i} > \varepsilon n$. Suppose without loss of generality the largest element in $a_i$ comes from $b_i$ then $c_i\setminus b_i$ is a subset of $c_m \setminus b_m$ because any element added to \(b_m\) after the $i$-th step is larger than all elements in \(c_i\). Thus $\abs{c_m \setminus b_m}> \varepsilon n$. We arrive at a contradiction. So indeed $\abs{b_i + c_i} \le 2\varepsilon n$, and $b_i$ and $c_i$ are both connected to $a_i$ for $i\in[m]$.

To show that for all $i\in[m-1]$ $a_i$ is connected to $a_{i+1}$, note that $a_{i+1} \supseteq a_i$ and $\abs{a_{i+1}\setminus a_i} \le \abs{b_{i+1}\setminus b_i} + \abs{c_{i+1}\setminus c_i} \le \varepsilon n$. 
To see that $a_m$ is connected to $a$, note that $a \supseteq a_m \supseteq b_m$ and so $\abs{a\setminus a_m} \le \abs{a\setminus b_m}\le \varepsilon n$. 

Let $H$ be the plane graph in \pref{fig:mon-cyc-dp} whose vertices are $\{0,a_1,\dots,a_m,b_1,\dots, b_m, c_1,\dots, c_m\}$. By construction $H$ is a van-Kampen diagram for $X'$ with $2m+2$ inner faces. 

Furthermore, we note that the cycle boundaries $C_0 = (0,b_1,a_1,c_1,0)$, $C_{b,i} = (b_i,b_{i+1},a_{i+1},a_i,b_i)$ and $C_{c,i} = (c_i,c_{i+1},a_{i+1},a_i,c_i)$ for $i\in [m-1]$ are all monotone four cycles. $C_{b,m} =(b_m,a,a_m,b_m)$ and $C_{c,m} = (c_m,a,a_m,c_m)$ are two triangles. Applying \pref{lem:van-kampen} we deduce that $C$ has a $4\cdot (1+2(m-1)) + 2 = 8m - 2$-triangle contraction. Thus we conclude the proof.

\end{proof}

\begin{figure}
     \centering
     \begin{subfigure}[b]{0.49\textwidth}
         \centering
    \includegraphics[width=0.6\textwidth]{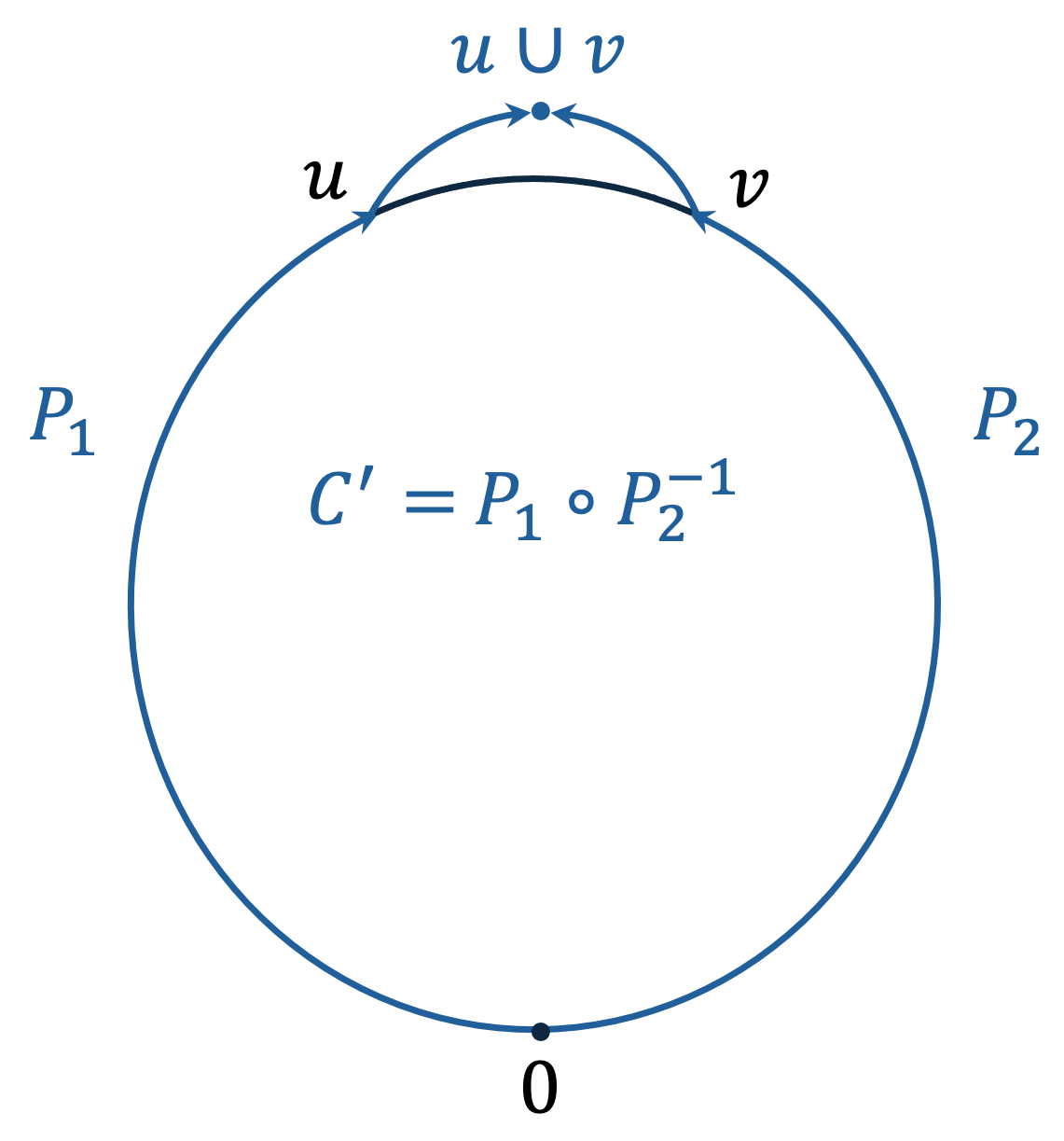}
         \caption{$C'$, $P_1$, and $P_2$}
         \label{fig:cone-cyc-1}
     \end{subfigure}
    \hfill
     \begin{subfigure}[b]{0.49\textwidth}
         \centering
\includegraphics[width=0.6\textwidth]{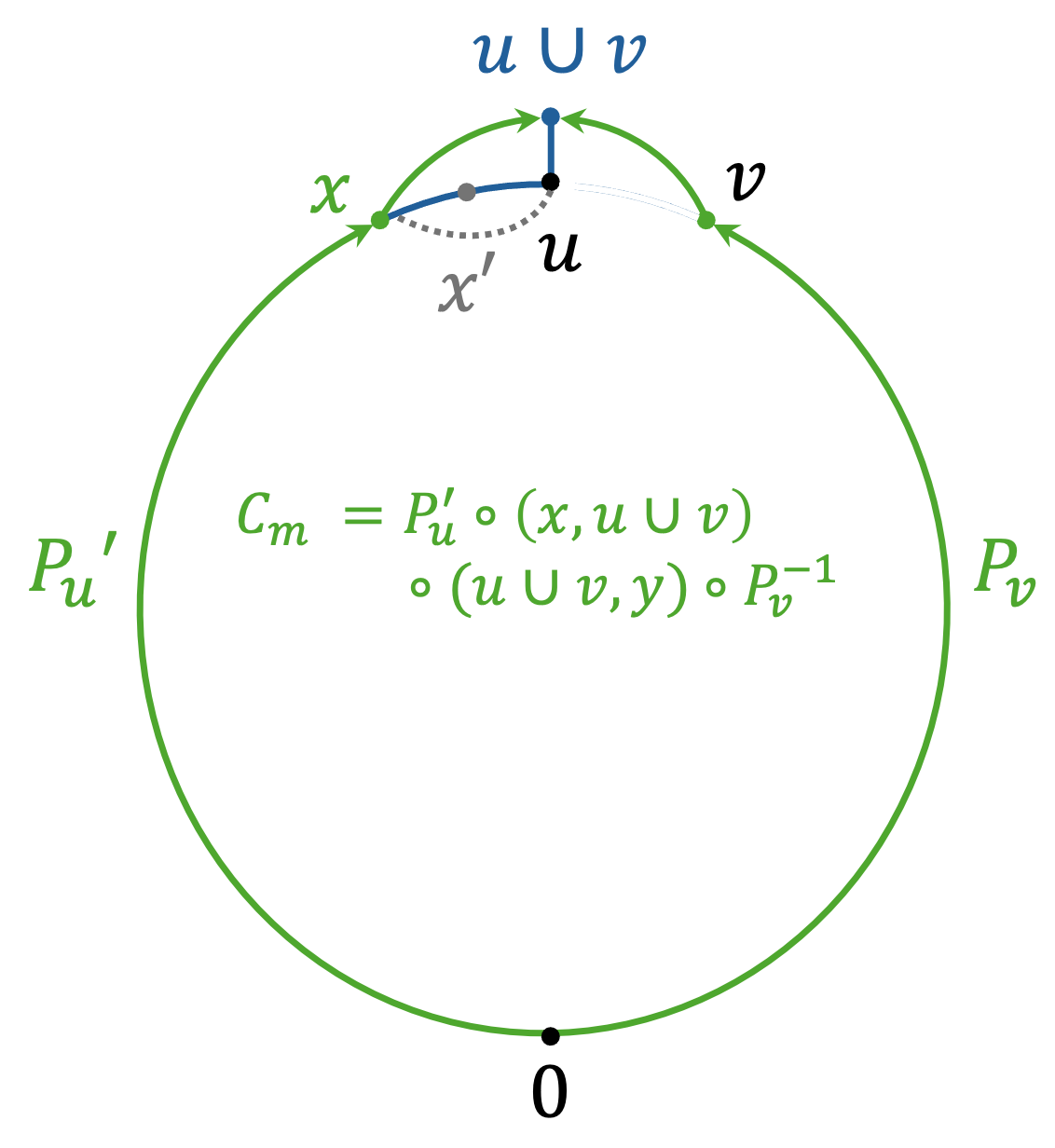}
         \caption{Minimal monotone cycle $C_m$}
         \label{fig:cone-cyc-2}
     \end{subfigure}
     \caption{Van-Kampen diagrams used in the proof of \pref{claim:X'-good-cone}}
\end{figure} 

\begin{proof}[Proof of \pref{claim:X'-good-cone}]
    Let \(\set{u,v}\) be an edge in $X'(1)$ and \(C_{uv} = P_u\circ (u,v) \circ P_v^{-1}\) be the decoding cycle. Since \(\set{u,v} \in X'(1)\) it holds that \(\abs{u + v} \leq \varepsilon n\). Thereby \(\abs{u + v} = \abs{u \cup v} - \abs{u \cap v} \geq \abs{u \cup v} - \abs{v}\) and similarly \(\abs{u + v} \geq \abs{u \cup v} - \abs{u}\). Thus the triangle \(\set{v,u,u \cup v}\) is in \(X'(2)\). 
    Define the cycle \(C'=P_u \circ (u,u\cup v,v) \circ P_{v}^{-1}\), and the monotone paths \(P_1 = P_u \circ (u,u \cup v)\) and \(P_2 = P_v \circ (v,v \cup u)\) (and \(C' = P_1 \circ P_2^{-1}\)). 
    If both paths have equal length, then we can contract the cycle $C'$ using \((16/\varepsilon - 2)\)-triangles by \pref{lem:monotone-path-contraction} as follows. We begin by constructing the following van-Kampen diagram \pref{fig:cone-cyc-1} to decompose $C_{uv}$ into $C'$ and $C'' = (u,u\cup v, v, u)$: let $H$ be the graph with vertices $V(P_u)\cup V(P_v) \cup \{u,v,u\cup v\}$ and edges $E(P_u)\dunion E(P_v) \dunion \{\set{u,v},\set{u,u\cup v}, \set{v,u\cup v}\}$, where \(V(P)\) denote  \(P\)'s vertex set and \(E(P)\) its edge set. By definition, the plane graph $H$ given in \pref{fig:cone-cyc-1} is a van-Kampen diagram with $3$ plane graph faces. 
    $H$ embeds into $X'$ as illustrated by \pref{fig:cone-cyc-1}. In particular, the edge boundaries of the $3$ faces in $H$ are mapped to $ C_{uv}, C'$, and $ C''$ in $X'$. Therefore by \pref{lem:van-kampen}, $C_{uv}$ has a \((16/\varepsilon - 1)\) contraction.
    
    If $P_1$ and $P_2$ are of different length, we do as follows (\pref{fig:cone-cyc-2}). Assume that \(P_1\) is longer. Since $\abs{u+v} \le \varepsilon n$, the length of $P_u$ and the length of $P_v$ differ by at most $2$. By construction, the prefix \(P_u'\) that is obtained by removing the last $\abs{P_u} - \abs{P_v}$ vertices and ends with a vertex \(x\) has \(\abs{x} \ge \abs{v}\). Thus $\abs{u\cup v \setminus x} \le \abs{u \cup v\setminus v} \le \varepsilon n$. In this case, the triangle \(\set{x,u,v \cup u} \in X'(2)\) can be used to decompose \(C_{uv}\). Denote the monotone cycle $C_m = P_u'\circ (x,u\cup v)\circ (u\cup v,y)\circ {P_v}^{-1}$. Again by \pref{lem:monotone-path-contraction}, $C_m$ has a \(8 \cdot (2/\varepsilon - 1) -2 \)-triangle contraction. 
    Now we add $(x,u\cup v)$ (and also $(x,u)$ if it is not in $P_u$) to the graph $H$. Now the van-Kampen diagram $(H,Id)$ has $3$ faces (or $4$ if $(x,u)$ is not in $P_u$) whose face boundaries are mapped to $C_{uv},C_m,  (x,u,u\cup v,x)$ (and potentially $(x,x',u,x)$). Since the last $1$ (or $2$) cycle has a $1$-triangle contraction, $C_{uv}$ has a \(8 \cdot (2/\varepsilon - 1)\)-triangle contraction by \pref{lem:van-kampen}.

\end{proof}

\subsubsection[Contracting X'(2) with X]{Contracting \(X'(2)\) with \(X\): proof of \pref{claim:contract-X'(2)} }\label{sec:contracting-X(2)}
We move back to consider the Johnson complex $X$ which is a subcomplex of $X'$. 

\begin{claim} \label{claim:four-cycle-in-X}
    Let \(C=(v_0,v_1,v_2,v_1',v_0)\) be \textit{any} four-cycle in \(X\). Then \(C\) has an \(8\)-triangle contraction.
\end{claim}
Note that in contrast to \pref{claim:four-cycle-in-X-prime} we do not assume monotonicity.

\begin{proof}
    Without loss of generality \(v_0=0\). In this case \(v_2\) has \(\abs{v_2} = t\) for some even \(2 \leq t \leq 2\varepsilon n\).
    We intend to find a path between $v_1$ and $v_1'$ and then apply the middle path contraction from \pref{claim:middle-path}. For this we define the following graph \(G=(V,E)\).
    \[V = N(0) \cap N(v_2)\]
    \[E = \sett{\set{u,w}}{\set{u,w,0}, \set{u,w,v_2} \in X(2)}.\]
    We note that $v_1,  v_1'\in V$, and that if there is a length-$\ell$ path from \(v_1\) to \(v_1'\) in this graph, then the four-cycle has a $2\ell$-triangle contraction by \pref{claim:middle-path}.
    
    To find such paths, we define two functions mapping tensor products of Johnson graphs to $G$.
    
    \[\psi_c : G_c := J(t,\frac{t}{2},\ceil{\frac{t}{4}}) \otimes J(n-t, \varepsilon n - \frac{t}{2}, \frac{\varepsilon}{2}n -\ceil{\frac{t}{4}}) \to G,\]
    \[\psi_f : G_f :=J(t,\frac{t}{2},\fl{\frac{t}{4}}) \otimes J(n-t, \varepsilon n - \frac{t}{2}, \frac{\varepsilon}{2}n - \fl{\frac{t}{4}}) \to G,\]
    
    where every vertex $(u,w)$ in the two tensor graphs where $u\in \binom{v_2}{t/2}$ and $w \in \binom{[n]\setminus v_2}{\varepsilon n - t/2}$ is mapped to $\psi_c((u,w)) = \psi_f((u,w)) = u\cup w$. Here we implicitly identify the $t$ elements in $J(t,\frac{t}{2},*)$ with the set $v_2$, and the $n-t$ elements in $J(n-t, \varepsilon n - \frac{t}{2}, *)$  with $[n]\setminus v_2$.
    Note that the two tensor product graphs have the same vertex space.
    Furthermore when $\frac{t}{4}$ is an integer, the two functions are identical.

    We first show that the functions give two bijections between the vertex sets: for any distinct vertices $(u,w),(u',w')$ in the tensor product graphs, the unions $u\cup w$ and $u'\cup w'$ are also distinct in $V$. Furthermore, for any \(v \in V\), \(v\) and \(v_2-v\) both have weight \(\varepsilon n\). This implies that \(\abs{v \cap v_2} = \frac{t}{2}\) and \(\abs{v \setminus v_2} = \varepsilon n - \frac{t}{2}\).  Therefore $v$ has the preimage $(v \cap v_2,  v \setminus v_2)$ under $\psi_c$.
    Vertex bijectivity follows. 

    We next show that $\psi_c$ is actually a graph homomorphism. That is to say every edge in $G_c$ is mapped to an edge in $G$.
    For every edge \(\{(u_1,w_1),(u_2,w_2)\}\) in $G_c$, we have \(\card{u_1\cap u_2} = \ceil{\frac{t}{4}}\) and \(\card{w_1\cap w_2} =\frac{\varepsilon}{2}n-\ceil{\frac{t}{4}}\). Subsequently 
    \[\psi_c((u_1, w_1)) + \psi_c((u_2, w_2)) = (u_1\cup w_1)+(u_2\cup w_2) = (u_1+u_2)\cup(w_1+w_2)  \]
    has weight \(2(\frac{t}{2}-\ceil{\frac{t}{4}}) + 2 ( \varepsilon n-\frac{t}{2}- (\frac{\varepsilon}{2}n-\ceil{\frac{t}{4}}))\varepsilon n = \varepsilon n\). Therefore $\{\psi_c((u_1, w_1)), \psi_c((u_2, w_2)\}$ is an edge in \(G\). Thus $\psi_c$ is a vertex bijective graph homomorphism. A similar argument shows that the same is true for $\psi_f$.

    We next use the following claim to find length-$4$ paths between any pair of vertices in $G_c \cup G_f$ and then transfer this property to $G$ via $\psi_c,\psi_f$.
    \begin{claim}\label{claim:2-step-Johnson}
        Let $G_c = (V_c,E_c)$ and $G_f = (V_f=V_c,E_f)$ be defined as above.
        For any two vertices $v_1 = (u,w)$ and $\beta =(u',w')$ in $V$ such that $u\neq u'$ and $w \neq w'$, there is a length two path $(v_1,(x,y),\beta)$ in $G_{cf}=(V_c,E_c\cup E_f)$ such that $\{v_1,(x,y)\} \in E_c$ and $\{(x,y),\beta\} \in E_f$.
    \end{claim}
    The claim immediately implies that for any pair of vertices $v_1,v_1'$ in the union graph $G_{cf}$, we can find an intermediate vertex $\beta$ such that the two components in $\beta$ are distinct from the components in $v_1$ and $v_1'$, and then find a length-$4$ path of the form $(v_1,\alpha,\beta,\alpha',v_1')$ in $G_{cf}$.
    
    We postpone the proof of the claim and conclude that using the vertex bijective graph homomorphisms $\psi_c$ and $\psi_f$, such paths also exist in $G$. Therefore there exists a length-$4$ path between any two vertices $v_1,v_1'$ in $G$. So the four-cycle has an $8$-triangle contraction.
\end{proof}
\begin{proof}[Proof of \pref{claim:2-step-Johnson}]
    When $\frac{t}{4}$ is an integer, $G_c = G_f$. The existence of the intermediate vertex $(x,y)$ is derived from the fact that there is a length two path between any two vertices in the Johnson graphs $J(t,\frac{t}{2},{\frac{t}{4}})$ and $J(n-t, \varepsilon n - \frac{t}{2}, \frac{\varepsilon}{2}n -{\frac{t}{4}})$.
    
    Now we focus on the case when $\frac{t}{2}$ is odd. Construct $x$ as illustrated in \pref{fig:x-in-tensor}. Take $a_x \subseteq u \setminus u'$ s.t. $\abs{a_x} = \ceil{\frac{\abs{u \setminus u'}}{2}}$, $b_x \subseteq u \cap u'$ s.t. $\abs{b_x} = \ceil{\frac{\abs{u \cap u'}}{2}}$, $c_x \subseteq u' \setminus u$ s.t. $\abs{c_x} = \ceil{\frac{\abs{u' \setminus u}}{2}}-1$, and $d_x \subseteq \ol{u\cup u'}$ s.t. $\abs{d_x} = \ceil{\frac{\abs{\ol{u\cup u'}}}{2}}$\footnote{Since $u\neq u'$, the sizes on the right-hand side of the size constraints are nonnegative. So these size constraints are satisfiable.}. Then set $x = a_x \cup b_x \cup c_x \cup d_x$. We can verify that $\abs{u\cap x} = \ceil{\frac{t}{4}}$, $\abs{x\cap u'} = \ceil{\frac{t}{4}}-1 = \fl{\frac{t}{4}}$, and  $\abs{x} = \ceil{\frac{t}{4}} + \fl{\frac{t}{4}} = \frac{t}{2}$. 

    Similarly, construct $y$ as illustrated in \pref{fig:y-in-tensor}. Take $a_y \subseteq w \setminus w'$ s.t. $\abs{a_y} = \fl{\frac{\abs{w \setminus w'}}{2}}$, $b_y \subseteq w \cap w'$ s.t. $\abs{b_y} = \fl{\frac{\abs{w \cap w'}}{2}}$, $c_y \subseteq w' \setminus w$ s.t. $\abs{c_y} = \fl{\frac{\abs{w' \setminus w}}{2}}+1$, and $d_y \subseteq \ol{w\cup w'}$ s.t. $\abs{d_y} = \fl{\frac{\abs{w\cap w'}}{2}}$. Then set $y = a_y \cup b_y \cup c_y \cup d_y$. We can verify that $\abs{w\cap y} = \frac{\varepsilon n}{2}- \ceil{\frac{t}{4}}$, $\abs{y\cap w'} = \frac{\varepsilon n}{2}-\ceil{\frac{t}{4}}+1 = \frac{\varepsilon n}{2}- \fl{\frac{t}{4}}$, and  $\abs{y} = \fl{\frac{\abs{w}}{2}} + \ceil{\frac{\abs{w'}}{2}} = \varepsilon n - \frac{t}{2}$. 

    Now by construction $\{(u,w),(x,y)\} \in E_c$ and $\{(x,y), (u',w')\} \in E_f$, and thus we conclude the proof. 

    \begin{figure}
     \centering
     \begin{subfigure}[b]{0.49\textwidth}
         \centering
    \includegraphics[width=0.6\textwidth]{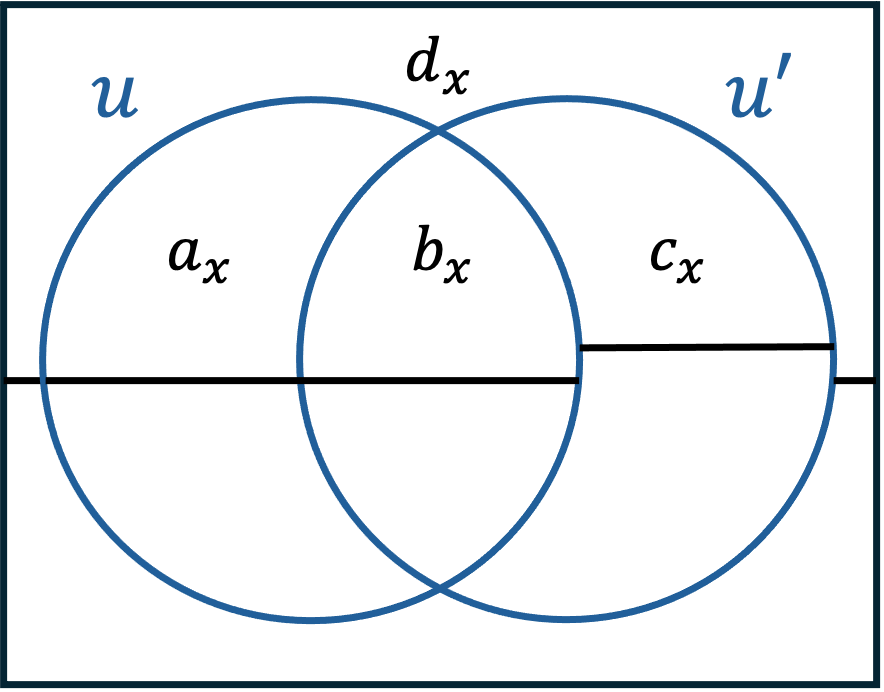}
         \caption{$x = a_x \cup b_x \cup c_x \cup d_x$}
         \label{fig:x-in-tensor}
     \end{subfigure}
     \hfill
     \begin{subfigure}[b]{0.49\textwidth}
         \centering
    \includegraphics[width=0.6\textwidth]{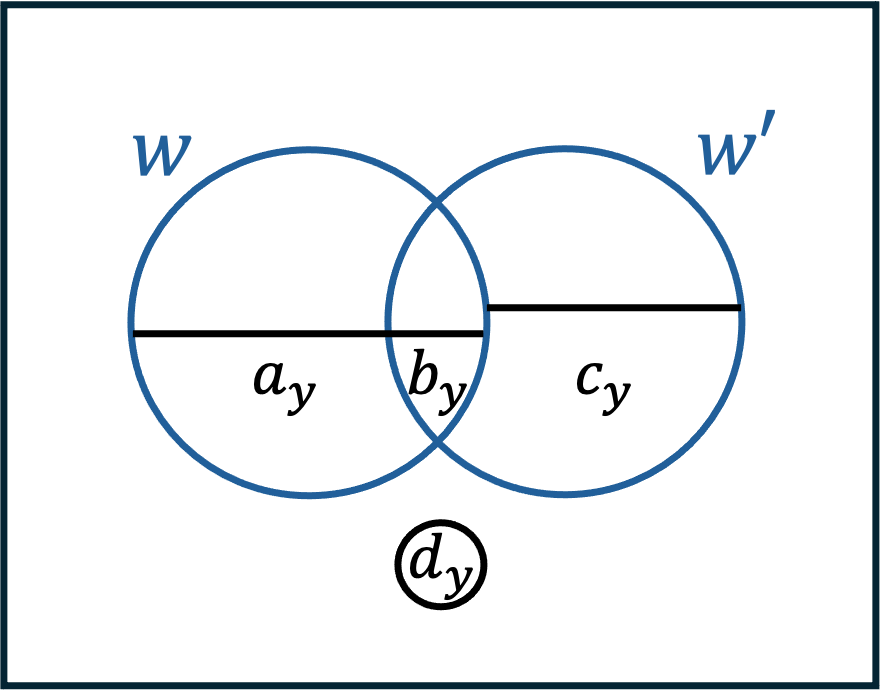}
         \caption{$y = a_y \cup b_y \cup c_y \cup d_y$}
         \label{fig:y-in-tensor}
     \end{subfigure}
     \caption{Construction of the intermediate vertex $(x,y)$}
\end{figure} 
\end{proof}

We are ready to prove the main result of this subsection
\begin{proof}[Proof of \pref{claim:contract-X'(2)}]
    Recall that by \pref{claim:X-2-sub} for every edge in \(\set{x,x+s} \in X'(1)\) we can find some \(s_1,s_2 \in S_{\varepsilon}\) such that path \(P_{x,x+s} = (x,x+s_1,x+s_1+s_2=x+s)\). So we focus on six-cycles in $X$ of the form $C =(u,a,v,b,w,c,u)$ such that $(u,v,w)$ is a triangle in $X'$ (i.e. $v-u, w-u, v-w$ all have weight at most $\varepsilon n$). Let us further assume without loss of generality that \(u=0\). Thus, \(C = (0,a,v,b,w,c,0)\) where $v, w, v-w$ all have weight at most $\varepsilon n$.

    Our plan is as follows:

    \begin{enumerate}
        \item We find vertices \(a',b',c'\) such that \((0,a',v,b',w,c',0)\) is a cycle in $X$, and such that \(a',b',c'\) are neighbors of \(0\).
        \item We decompose \(C\) to five four-cycles \(C_a,C_b,C_c,C_1,C_2\) as illustrated in \pref{fig:six-cyc}. Then by \pref{claim:four-cycle-in-X} each of the four-cycle has a $8$-triangle contraction. Furthermore, the figure gives a van-Kampen diagram with $C$ being the outer face boundary, and the four-cycles being the inner face boundaries. So applying \pref{lem:van-kampen}, we can conclude that $C$ has a $40$-triangle contraction.
    \end{enumerate}
        \begin{figure}[ht]
     \centering
     
         \begin{subfigure}[b]{0.49\textwidth}
         \centering
    \includegraphics[width=0.4\textwidth]{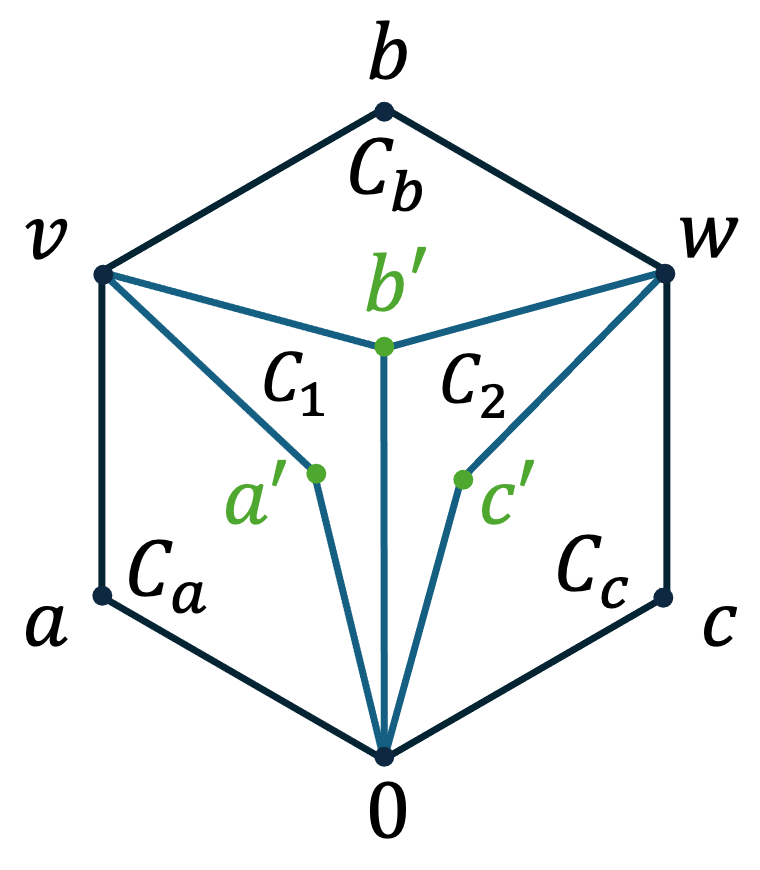}
         \caption{Decomposition of $C$}
         \label{fig:six-cyc}
     \end{subfigure}
     \hfill
     \begin{subfigure}[b]{0.49\textwidth}
         \centering
    \includegraphics[width=0.6\textwidth]{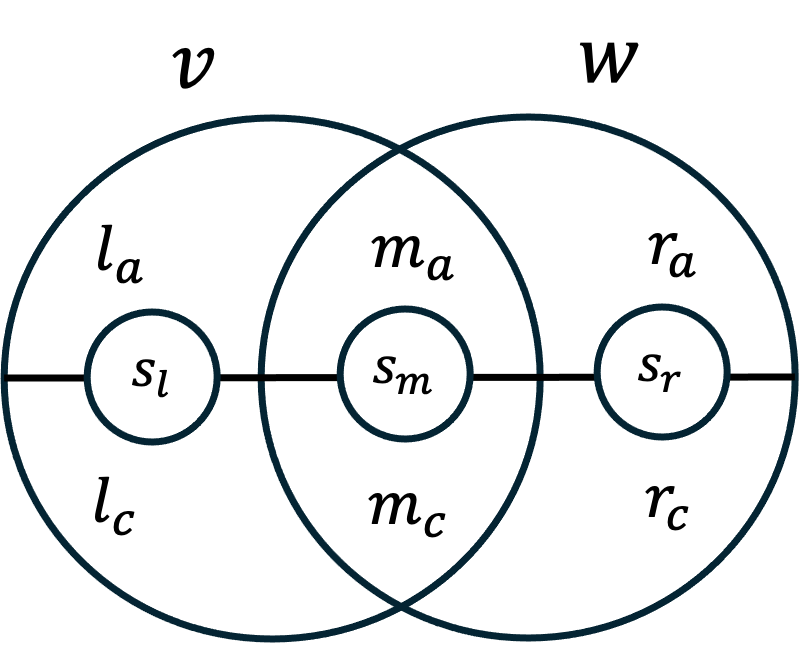}
    \vspace{1em}
         \caption{Partition of the support of $v\cup w$}
         \label{fig:vw-dp}
     \end{subfigure}
     \caption{Decomposition and partition in the proof of \pref{claim:contract-X'(2)}}
\end{figure}
Now we construct $a',b',c'$. 
The construction is easy when $ v\cap w$ has even weight. In this case $v\setminus w$ and $w\setminus v$ also have even weight. We construct the auxiliary vertices by first evenly partitioning $v\setminus w = l_a\dunion l_c$, $v\cap w = m_a\dunion m_c$
and $w\setminus v = r_a\dunion r_c$, then finding a set $x\subseteq \overline{v\cup w}$ s.t. $\card{x} = \varepsilon n - \card{v\cup w}/2$, and setting $a' = l_a + m_a+r_a + x$, $b' = l_a + m_a +r_c+x$, and $c' = l_c+m_c+r_c+x$. By construction the three vertices are adjacent to $0$ and form the cycle $(0,a',v,b',w,c',0)$ in $X$. 

The construction becomes complicated when $ v\cap w$, $v\setminus w$, and $w\setminus v$ have odd weights. In this case, we partially partition the supports of $v\cap w, v\setminus w$, and $w\setminus v$ as follows (\pref{fig:vw-dp}):

\begin{align*}
    v\setminus w &= l_a \dunion l_c \dunion s_l  &&\text{ s.t. } \begin{cases}
        \card{l_a} = \fl{\card{v\setminus w}/2}\\
        \card{l_c} = \fl{\card{v\setminus w}/2}
    \end{cases} \\
     v\cap w &= m_{a}\dunion m_{c}\dunion s_m &&\text{ s.t. } \begin{cases}
        \card{m_a} = \fl{\card{v\cap w}/2}\\
        \card{m_c} = \fl{\card{v\cap w}/2} 
    \end{cases} \\
     w\setminus v &= r_a \dunion r_c \dunion s_r &&\text{ s.t. } \begin{cases}
        \card{r_a} = \fl{\card{w\setminus v}/2}\\
        \card{r_c} = \fl{\card{w\setminus v}/2}
    \end{cases} 
\end{align*}
where $\card{s_l}=\card{s_m} = \card{s_r} =1$. 

Then we pick $a'=l_a + m_{a}+r_a+ s_m + x$ and $c'= l_c + m_{c}+r_c+s_m+x$ where $x$ is supported over $\overline{v \cup w}$ and $\card{x} = \varepsilon n - \fl{ \card{v\cup w}/2}$. Note that since $\card{v\cup w} = \card{v}+\card{w\setminus v} \le \card{v}+\card{w - v}  \le 2\varepsilon n$, the cardinality $\varepsilon n - \fl{\card{v\cup w}/2} \ge 0$ is well defined. Define $b' = l_a+m_{a}+r_c+ s_m + x$ 
. We note that $\card{v - a'} = \card{l_c+m_c+r_a+s_l+x} = \varepsilon n$. Similarly $\card{w - c'} = \card{l_c+m_a+r_a+s_r+x} = \varepsilon n$, so $a'$ and $c'$ have edges to $v$ and $w$ respectively. Furthermore $\card{ v - b'} = \card{l_c+m_c+r_c+s_l+x} = \varepsilon n$ and $\card{ w - b'} = \card{l_a + m_c+r_a+s_r+x} = \varepsilon n$. So $b'$ is connected to $v$ and $w$. By construction $a',b',c'$ are all connected to $0$ therefore the cycles $C_a,C_b,C_c,C_1,C_2$ all exist. Thus we complete the proof.
\end{proof}

\subsection[Cone diameter of the vertex links]{Cone diameter of the vertex links: proof of \pref{lem:cone-for-link-X}}
Recall that the links of every vertex in the complex are isomorphic to one another. Therefore it is enough to prove the lemma for the link of \(0\). By \pref{lem:face-decomposition}, the underlying graph of the link complex is \(J(n,\varepsilon n, \frac{\varepsilon}{2}n)\), and the $2$-faces are 
\[ X_0(2) = \sett{\set{v_0,v_1,v_2}}{\set{v_0,v_1},\set{v_1,v_2},\set{v_0,v_2} \in X_0(1) \text{ and }\abs{v_0 \cap v_1 \cap v_2} = \frac{\varepsilon}{4}n}.\]
Let us begin with the following claim.
\begin{claim} \label{claim:contracting-almost-disjoint-4-cycles-link-X}
    Let \(C = (v,u,w,u',v)\) be a \(4\)-cycle in \(J(n,\varepsilon n, \frac{\varepsilon}{2}n)\) 
    such that \(v \cap w = \emptyset \), then \(C\) has a contraction using \(4\) triangles.
\end{claim}

\begin{proof}
    We intend to use the middle path pattern from \pref{claim:middle-path}. So let us show that there exists a path \((u,z,u')\) such that both edges in the path form triangles with \(v\) and with \(w\). 
    To find $z$, we find separately \(z_1 = z \cap v\) and \(z_2 = z \cap w\) and set \(z= z_1 \dunion z_2\). To find \(z_1\) we set \(x_1=u \cap v , x_2 = u' \cap v\) and note that both of these have size \(\frac{\varepsilon}{2}n\) while $v$ has size $\varepsilon n$. We note that \(\abs{x_1 \setminus x_2} = \abs{x_2 \setminus x_1}\) and we take \(z_1\) by taking \(t=\min \set{\abs{x_1 \setminus x_2}, \frac{\varepsilon}{4}}\) coordinates from each of \(x_1 \setminus x_2, x_2 \setminus x_1\), an additional \(\frac{\varepsilon}{4}n - t\) coordinates from each of \(x_1 \cap x_2, v \setminus (x_1 \cup x_2)\). This is possible since:
    \begin{enumerate}
        \item \(\abs{x_1 \cap x_2} = \frac{\varepsilon}{2}n-t\) and
        \item  \(\abs{v \setminus (x_1 \cup x_2) } \geq \varepsilon n - \abs{x_1 \cup x_2} = \varepsilon n -\abs{x_2} - \abs{x_1 \setminus x_2} \ge \frac{\varepsilon}{2}n-t \) .
    \end{enumerate}
    Find $z_2$ analogously in $w$. Then by construction $\{u,v,z\},\{u',v,z\},\{u,w,z\},\{u',w,z\}$ are all in $X_0(2)$. So \pref{claim:middle-path} implies that $C$ has a \(4\)-triangle contraction. 
\end{proof}

\begin{proof}[Proof of \pref{lem:cone-for-link-X}]
    Without loss of generality consider the link of $0$. Let \(v_0 \in X_0\) be the base vertex. Since the diameter of the Johnson graph is \(2\), for any \(u \in X_0(0) \setminus\{v_0\}\) we arbitrarily choose a path of length $2$ from \(v_0\) to \(u\) as \(P_{u}\) (the existence of triangles ensures that we can do this for neighboring vertices as well). To bound the cone diameter, we consider any edge \(\set{u,u'} \in X_0(1)\) and its corresponding cone cycle \(C = P_{u} \circ (u,u') \circ P_{u'}^{-1}\). Let us re-annotate \(C = (v_0,v_1,v_2,v_3,v_4,v_0)\).

    Let us begin by reducing to the case where \(v_0\) is disjoint from \(v_2 \cup v_3\). Observe that \(\abs{v_i} = \varepsilon n\) and that \(\abs{v_i \setminus \bigcup_{j=0}^{i-1}v_j} \leq \abs{v_i \setminus v_{i-1}} \leq \frac{\varepsilon}{2}n\). Thus \(\abs{\bigcup_{j=0}^4 v_j} \leq 3\varepsilon n\) and by assumption that \(\varepsilon \leq \frac{1}{4}\) there exists a vertex \(z\) such that \(z\) is disjoint from $\bigcup_{j=0}^4 v_j$.

    Since the diameter of the Johnson graph is $2$, we can decompose $C$ into $5$ five-cycles $C_0,C_1,C_2,C_3, C_4$ via the van-Kampen diagram given in \pref{fig:five-cyc-1}. Then by \pref{lem:van-kampen}, if every $C_i$ can be contracted by \(m\) triangles, the original cycle $C$ has a \(5m\)-triangle contraction. So we will show next that every $C_i$ has a contraction of \(17\) triangles.

    \begin{figure}[ht]
     \centering
         \begin{subfigure}[b]{0.49\textwidth}
         \centering
    \includegraphics[width=0.4\textwidth]{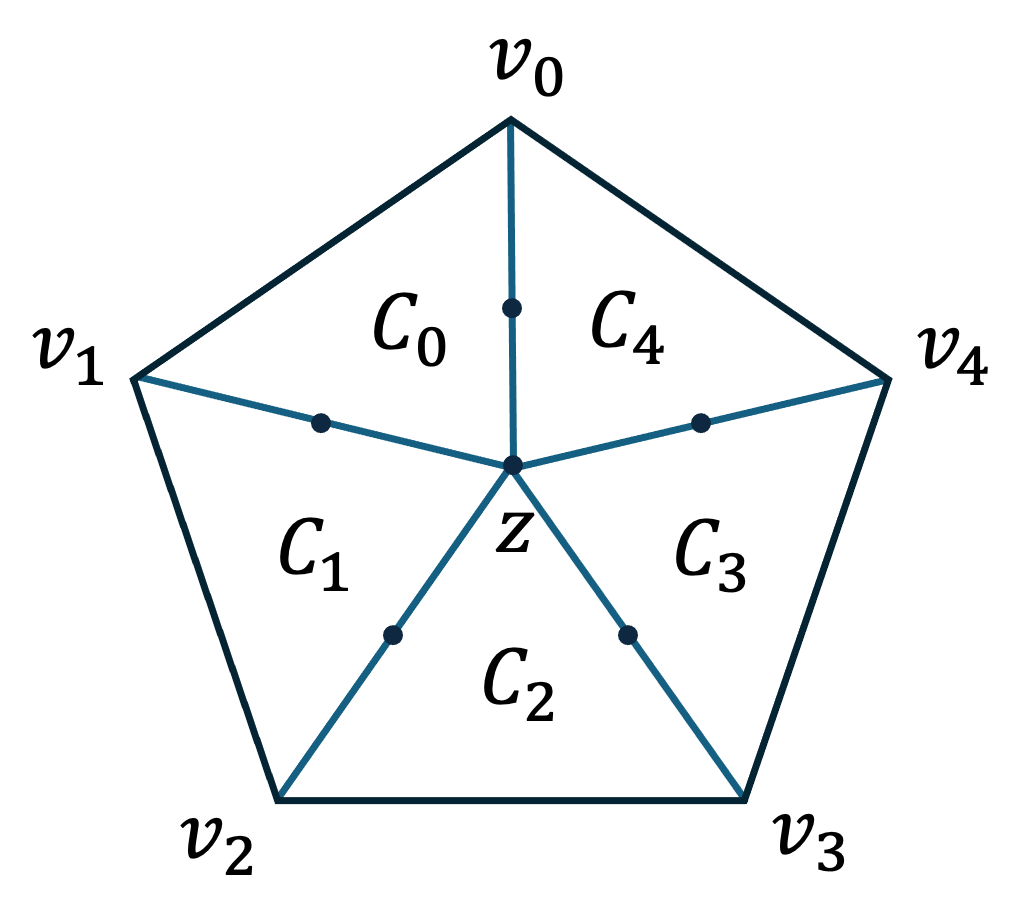}
         \caption{van-Kampen diagram of $C$ }
         \label{fig:five-cyc-1}
     \end{subfigure}
     \hfill
     \begin{subfigure}[b]{0.49\textwidth}
         \centering
    \includegraphics[width=0.4\textwidth]{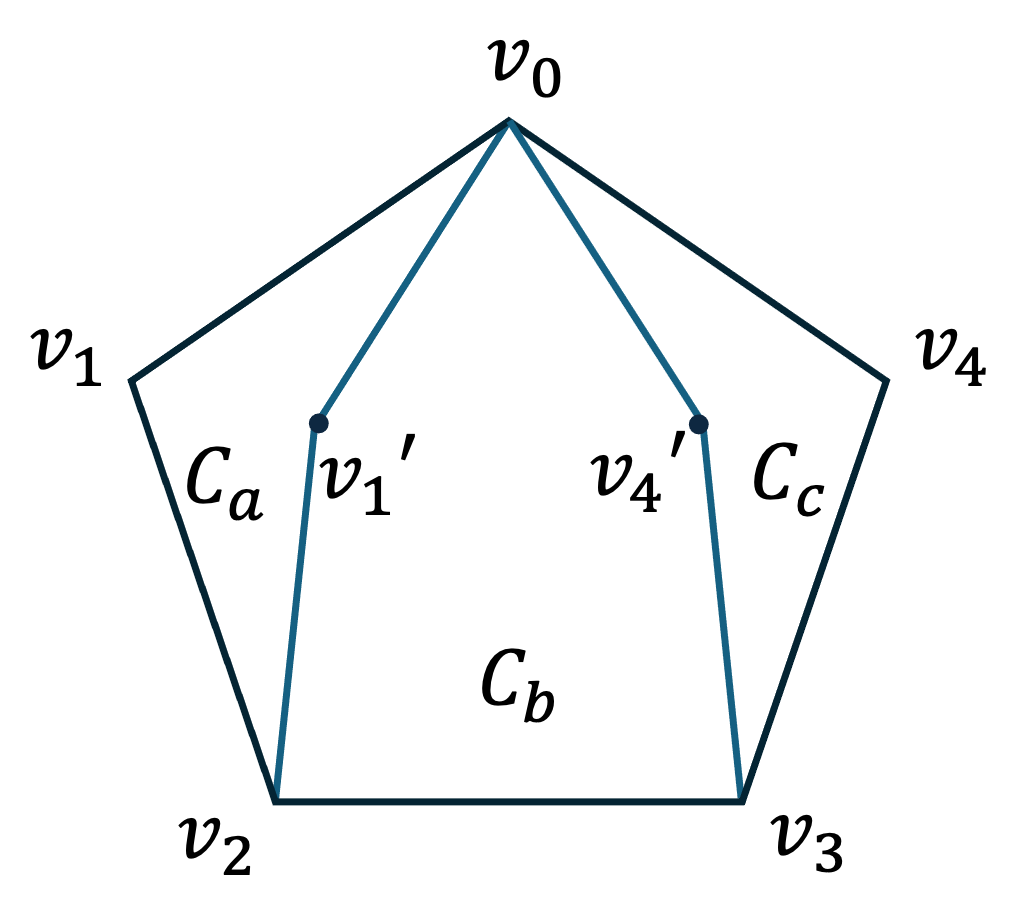}
         \caption{van-Kampen diagram of $C_i$}
         \label{fig:five-cyc-2}
     \end{subfigure}
     \caption{Van-Kampen diagrams used in the proof of \pref{lem:cone-for-link-X}}
\end{figure}

    Re-annotate \(C_i=(v_0,v_1,v_2,v_3,v_4,v_0)\) such that \(v_0\) is disjoint from \(v_2 \cup v_3\).
    We can find some \(v_1'\) and \(v_4'\) such that the following holds:
    \begin{enumerate}
        \item $v_1'$ is disjoint from $v_3\cup v_4'$.
        \item $v_4'$ is disjoint from $v_1'\cup v_2$.
        \item The path \(P'=(v_0,v_1',v_2,v_3,v_4',v_0)\) is a closed walk in \(X_0\).
    \end{enumerate}
    If such vertices exists, then $C_i$ can be decomposed into two four-cycles $C_a, C_c$ and a five-cycle $C_b$ via the van-Kampen diagram in \pref{fig:five-cyc-2}. Furthermore since $v_0\cap v_2 = v_0\cap v_3 = \emptyset$, each of the four-cycles can by contracted by $4$ triangles (\pref{claim:contracting-almost-disjoint-4-cycles-link-X}). Furthermore, we will soon show that $C_b$ has a $9$-triangle contraction using the following claim.
    \begin{claim}\label{claim:disjoint-five-cycle}
       If a five-cycle $C = (v_0,v_1',v_2,v_3,v_4',v_0)$ in $X_0$ satisfies that every vertex $v \in C$ is disjoint from the cycle vertices that are not neighbors of $v$ in $C$, then $C$ has a $9$-triangle contraction in $X_0$.
    \end{claim}
    Before proving this claim let us complete the proof of the lemma.
    Applying \pref{lem:van-kampen} again, we can conclude that $C_i$ has a $9+ 2\cdot 4 = 17$-triangle contraction. 

    We now find such $v_1'$ and $v_4'$. Partition $v_0$ into $s_1\dunion s_2$ of equal size $\frac{\varepsilon}{2}n$. Then define $v_1' = s_1  \cup (v_2\setminus v_3)$ and $v_4' = s_2 \cup (v_3\setminus v_2) $. The two sets both have size $\varepsilon n$ since $\abs{s_1} = \abs{s_2} = \abs{v_2\setminus v_3} = \abs{v_3\setminus v_2} = \frac{\varepsilon}{2}n$. It is straightforward to verify that the three properties above are satisfied. We thus complete the proof. 
\end{proof}

\begin{figure}[ht]
     \centering
     \begin{subfigure}[b]{0.49\textwidth}
         \centering
    \includegraphics[width=0.55\textwidth]{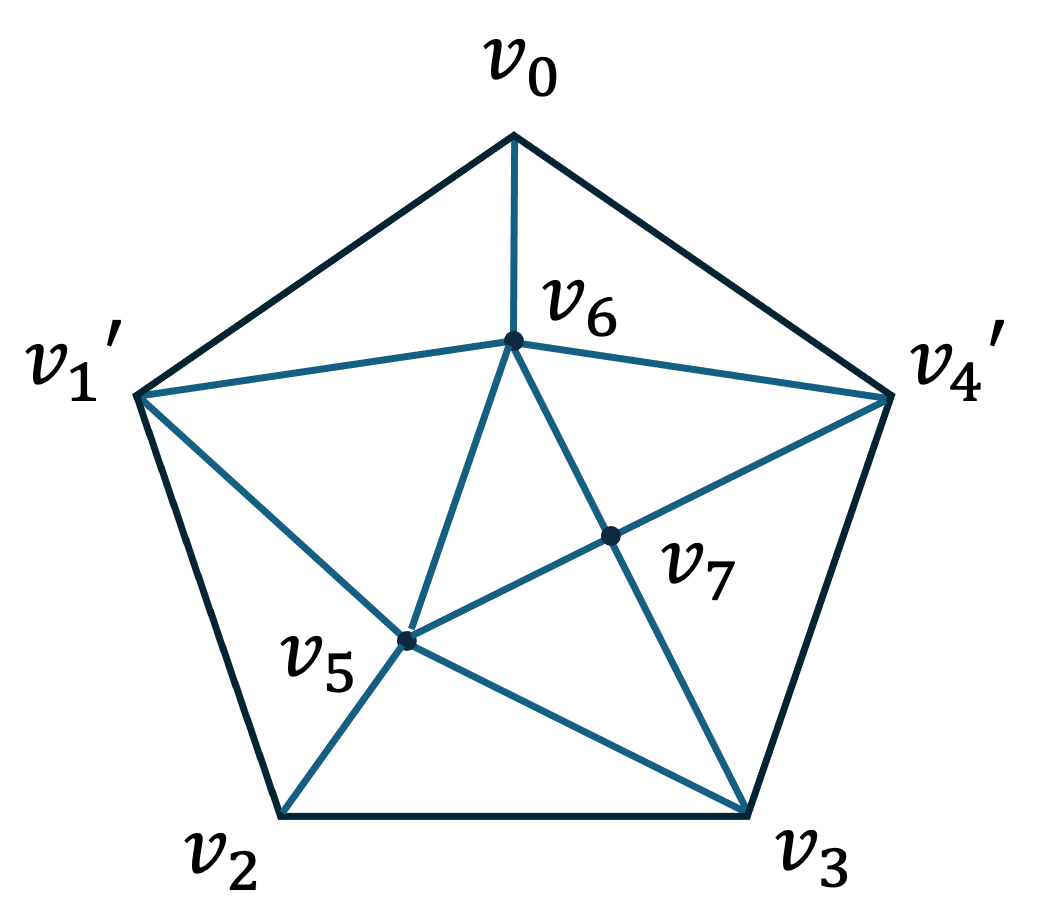}
         \caption{van-Kampen diagram of $C$}
         \label{fig:disjoint-five-cycle}
     \end{subfigure}
     \hfill
     \begin{subfigure}[b]{0.49\textwidth}
         \centering
    \includegraphics[width=0.8\textwidth]{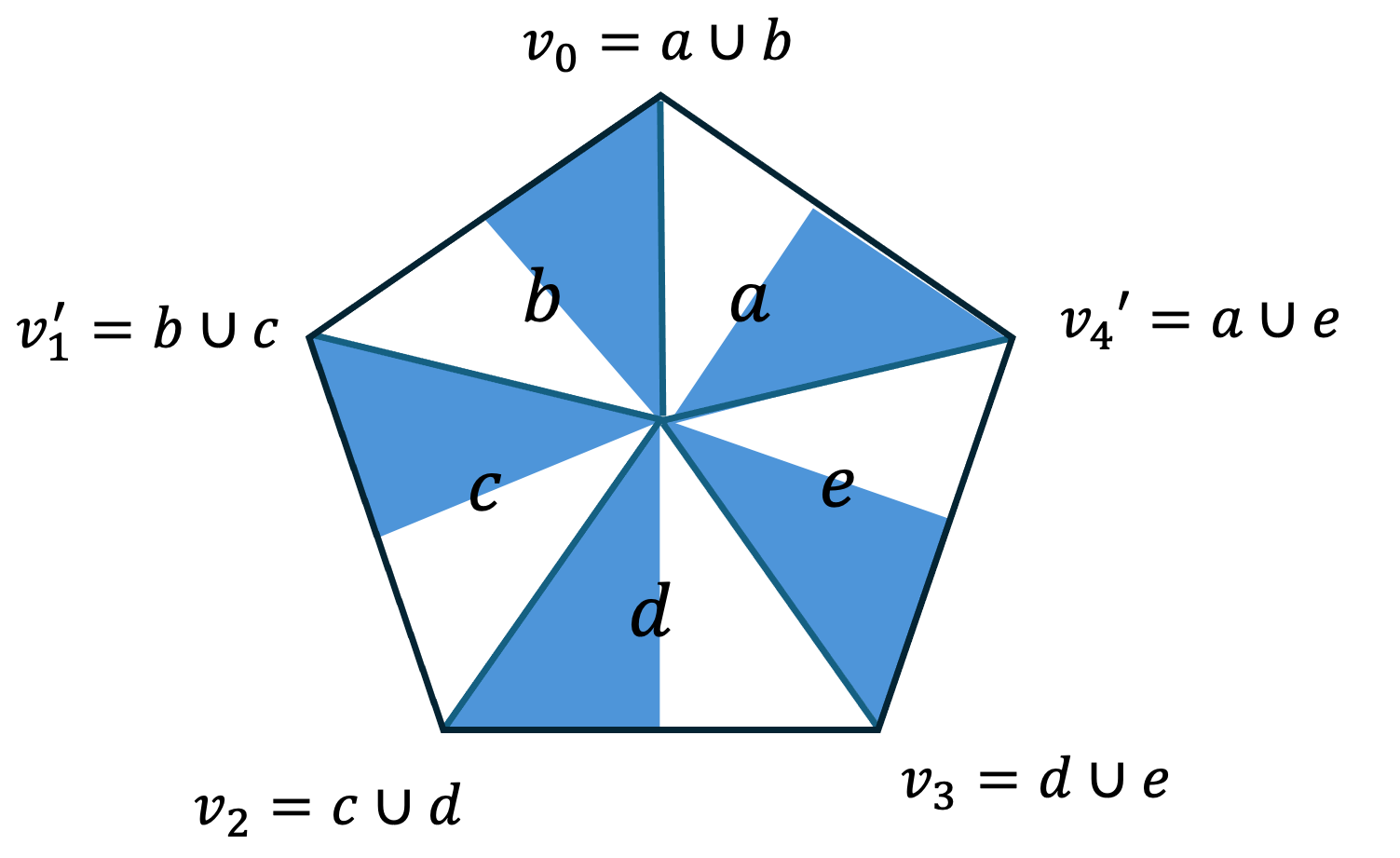}
         \caption{Decomposition of the set $v_0\cup v_1'\cup v_2\cup v_3 \cup v_4'$}
         \label{fig:five-cycle-decomp}
     \end{subfigure}
     \caption{Decomposition of $C$ in \pref{claim:disjoint-five-cycle}}
\end{figure} 

\begin{proof}[Proof of \pref{claim:disjoint-five-cycle}]
 Define sets $a,b,c,d,e \subseteq [n]$ to be $a = v_4'\cap v_0, b = v_0 \cap v_1', c=v_1'\cap v_2, d = v_2\cap v_3$ and $e= v_3\cap v_4'$. By the assumption on the vertices, the $5$ sets all have size $\frac{\varepsilon n}{2}$ and satisfy that $v_0 = a\dunion b, v_1' = b\dunion c, v_2 = c\dunion d, v_3 = d\dunion e,$ and $v_4' = e\dunion a$. 
Furthermore, for each $x \in \{a,b,c,d,e \}$ we partition it into two disjoint sets $x_1, x_2$ of size $\frac{\varepsilon n}{4}$ as illustrated in \pref{fig:five-cycle-decomp}. 

Now we can decompose $C$ into $9$ triangles as illustrated in \pref{fig:disjoint-five-cycle} by defining $v_5,v_6,$ and $v_7$ as follows: $v_5 = b_1\cup c_1\cup d_1\cup e_1$, $v_6 = a_1\cup b_1\cup c_2\cup e_1$, $v_7 = a_1\cup b_1\cup d_1\cup e_2$. We can verify that each of the  $9$ inner face in the figure has its boundary being a triangle in $X_0(2)$. Thus $C$ has a $9$-triangle contraction.
\end{proof}


\subsection[Coboundary expansion of links in the complexes of Golowich]{Coboundary expansion of links in the complexes of \cite{Golowich2023}}
\label{sec:coboundary-Golowich}
Even though most of this \pref{sec:cob-exp-johnson} is devoted to the Johnson complex, our techniques also extend to the complexes constructed by \cite{Golowich2023} that were described in \pref{sec:construction-subpoly-HDX}. In this section we prove that the link of every vertex in these complexes are coboundary expanders.
\begin{lemma} \label{lem:louis-coboundary-expander}
    Let \(\Y=\M(\mathcal{G}(\F_2^n)^{\leq d})\) for \(n > 2d\). Let \(X = \beta(\Y)\) be the basification of \(Y\). Then \(X\) has a cone with area \(459\) and is therefore a $1$-dimensional \(\frac{1}{459}\)-coboundary expander.
\end{lemma}

The proof of this lemma is just a  more complicated version of the proof of \pref{lem:cone-for-link-X}. 
Before proving the lemma, we give an explicit description of the link structure. The vertices are \(X(0) = \MP(\frac{d}{2})\), i.e.\ all matrices of a certain rank \(\frac{d}{2}\). The edges are
\[X(1) = \sett{\set{A_1,A_2}}{\exists B \in \MP(\frac{d}{4}), \; B \oplus (A_1-B) \oplus (A_2-B)}.\]
Finally, the triangles \(X(2)\) are all \(\set{A_1,A_2,A_3}\) such that there exist \(B_1,B_2,\dots,B_7 \in \MP(\frac{d}{8})\) such that the sum \(\bigoplus_{j=1}^7 B_j\) is direct, and
\[A_1 = B_1\oplus B_2 \oplus B_3 \oplus B_4,\]
\[A_2 = B_1\oplus B_2 \oplus B_5 \oplus B_6,\]
\[A_3 = B_1\oplus B_3 \oplus B_5 \oplus B_7.\]

This description is a direct consequence of the definition of admissible functions in \pref{sec:construction-subpoly-HDX}.

We follow the footsteps of the proof of \pref{lem:cone-for-link-X}. Towards this we need two claims. The first proves that there exist length \(4\) paths between any two vertices in \(X\), so that we can construct a cone where the cycles have constant length. This was immediate in the Johnson case but needs to be argued here. This claim also constructs length-$2$ paths between any two matrices that form a direct sum. we will need this later on.
    \begin{claim} \label{claim:louis-link-const-diam}
        \begin{enumerate}
            \item For any two matrices \(A,B \in X(0)\) there is a path of length \(4\) between \(A\) and \(B\).
            \item If \(A \oplus B\) then there is a path \((A,C,B)\). Moreover, there exists such a path where the middle vertex \(C = A_1 \oplus B_1\) such that the components of this sum have equal rank and satisfy \(A_1 \leq A\) and \(B_1 \leq B\)\footnote{Not all paths from \(A\) to \(B\) have this property. For instance, let \(e_1,e_2,e_3,e_4\) be four independent vectors. Let \(A=e_1 \otimes e_1 + e_2 \otimes e_2\), \(C = e_2 \otimes e_2 + e_3 \otimes e_3\) and \(B=(e_2 + e_3) \otimes e_3 + e_4 \otimes e_4\). One can verify that because \(C = (e_2+e_3) \otimes e_3 + e_2 \otimes (e_2 + e_3)\) then \((A,C,D)\) is indeed a path. However, we can directly check that there are no \(A_1 \leq A, B_1 \leq B\) such that \(C=A_1 \oplus B_1\).}.
        \end{enumerate} 
    \end{claim}
The second claim is a a variant of \pref{claim:contracting-almost-disjoint-4-cycles-link-X}. \pref{claim:contracting-almost-disjoint-4-cycles-link-X} bounds \(4\)-cycles in the Johnson complex provided that two non-neighbor vertices in the cycle are disjoint. This claim is the analog for \(X\):
\begin{claim} \label{claim:tiling-louis-special-cycles}
        Let \((Z,C,A,D,Z)\) be a \(4\)-cycle with the following properties.
    \begin{enumerate}
        \item The matrices \(Z \oplus A\) are a direct sum.
        \item We can decompose \(C=C_1 \oplus C_2\) and \(D=D_1 \oplus D_2\) such that \(C_1,D_1 \leq Z\) and \(C_2,D_2 \leq A\).\footnote{We note that there is no analogue to second property of the cycle in the statement \pref{claim:contracting-almost-disjoint-4-cycles-link-X}. The analogue is that \(C= (C_1 \cap A) \cup (C \cap Z)\) (and similarly for \(D\)). This follows from the first property in the Johnson case but we require it explicitly in this case.}
    \end{enumerate}
    Then the cycle has a tiling with \(O(1)\)-triangles.
    \end{claim}
Both of these claims are proven after the lemma itself.

\begin{proof}[Proof of \pref{lem:louis-coboundary-expander}]
    It is easy to see that the complex \(X\) has a transitive action on the triangles. Therefore by \pref{thm:group-and-cones}, it suffices to find a cone whose area is \(459\) to prove coboundary expansion.
    
    Let \(M_0\) be an arbitrary matrix we take as the root of our cone. We take arbitrary length \(4\) paths \(P_A\) from \(M_0\) to every matrix \(A \in X(0)\). This is possible by \pref{claim:louis-link-const-diam}.

    For every edge \(\set{A,A'}\), we need to find a contraction for the cycle \(P_A \circ (A,A')\circ P_{A'}^{-1}\). We note that this cycle has length \(9\). Therefore it suffices to show that every \(9\) cycle \((M_0,A_1,A_2,\dots,A_8,M_0)\) has a tiling with \(459\) triangles. We begin by reducing the problem into tilings of \(5\) cycles instead. Indeed, observe that for every edge \(\set{A_i,A_{i+1}}\) in the cycle, \(\dim(\row(A_i)+\row(A_{i+1})) = \dim(\col(A_i)+\col(A_{i+1})) = \frac{3d}{4}\). Therefore there exists some \(Z \in X(0)\) such that \(\row(Z)\) and \(\col(Z)\) are a direct sum with \(\row(A_i)+\row(A_{i+1})\) and \(\col(A_i)+\col(A_{i+1})\) respectively (In fact, a random matrix \(Z\) will satisfy this with high probability for all edges in the cycle because \(n \geq 2d\)).

    \begin{figure}[ht]
     \centering
     \begin{subfigure}[b]{0.49\textwidth}
     \centering
        \includegraphics[width=0.6\linewidth]{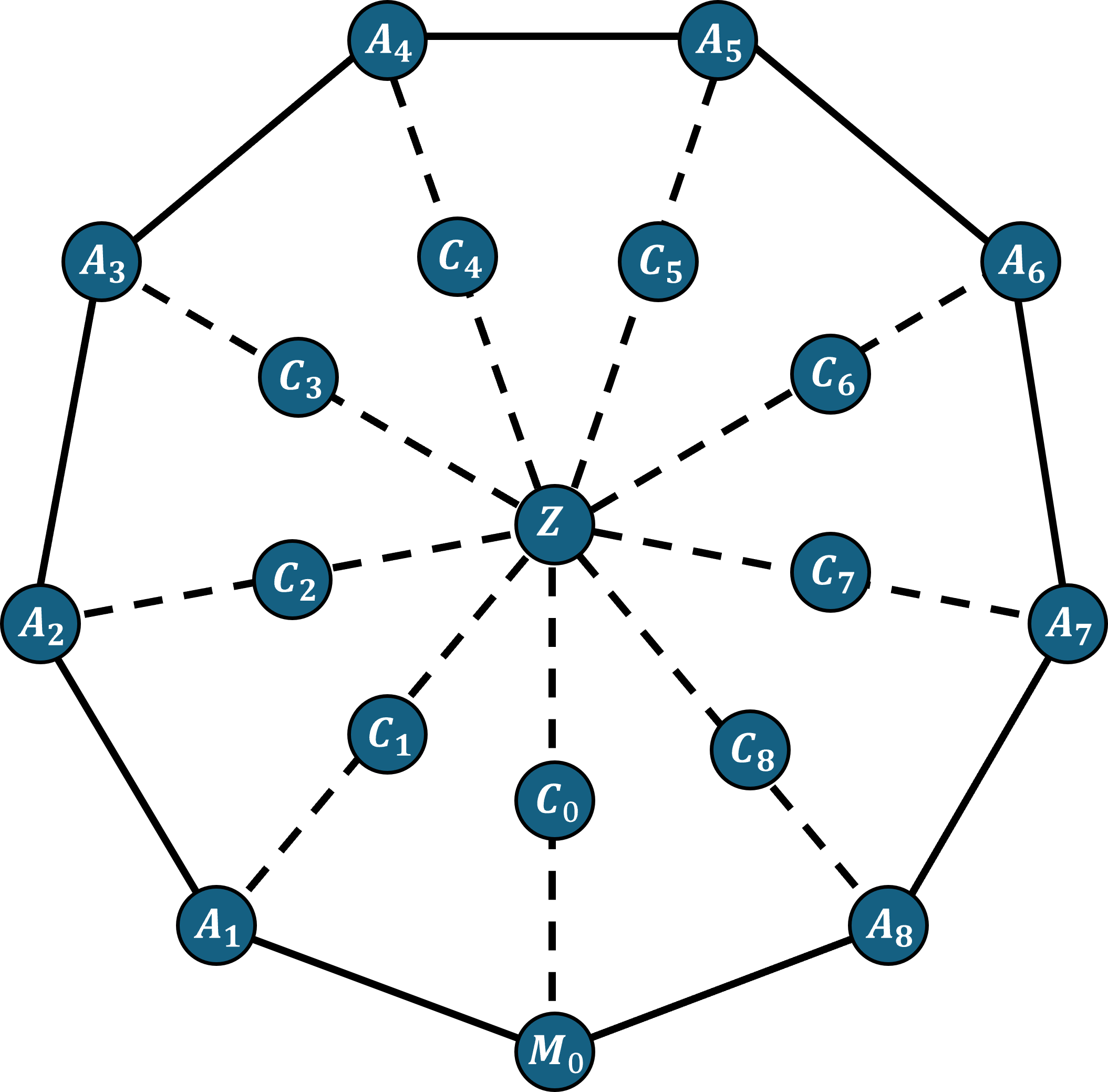}
        \caption{From a \(9\)-cycle to \(5\)-cycles}
        \label{fig:9-to-5-cycles}
     \end{subfigure}
     \hfill
     \begin{subfigure}[b]{0.49\textwidth}
        \centering
        \includegraphics[width=0.55\linewidth]{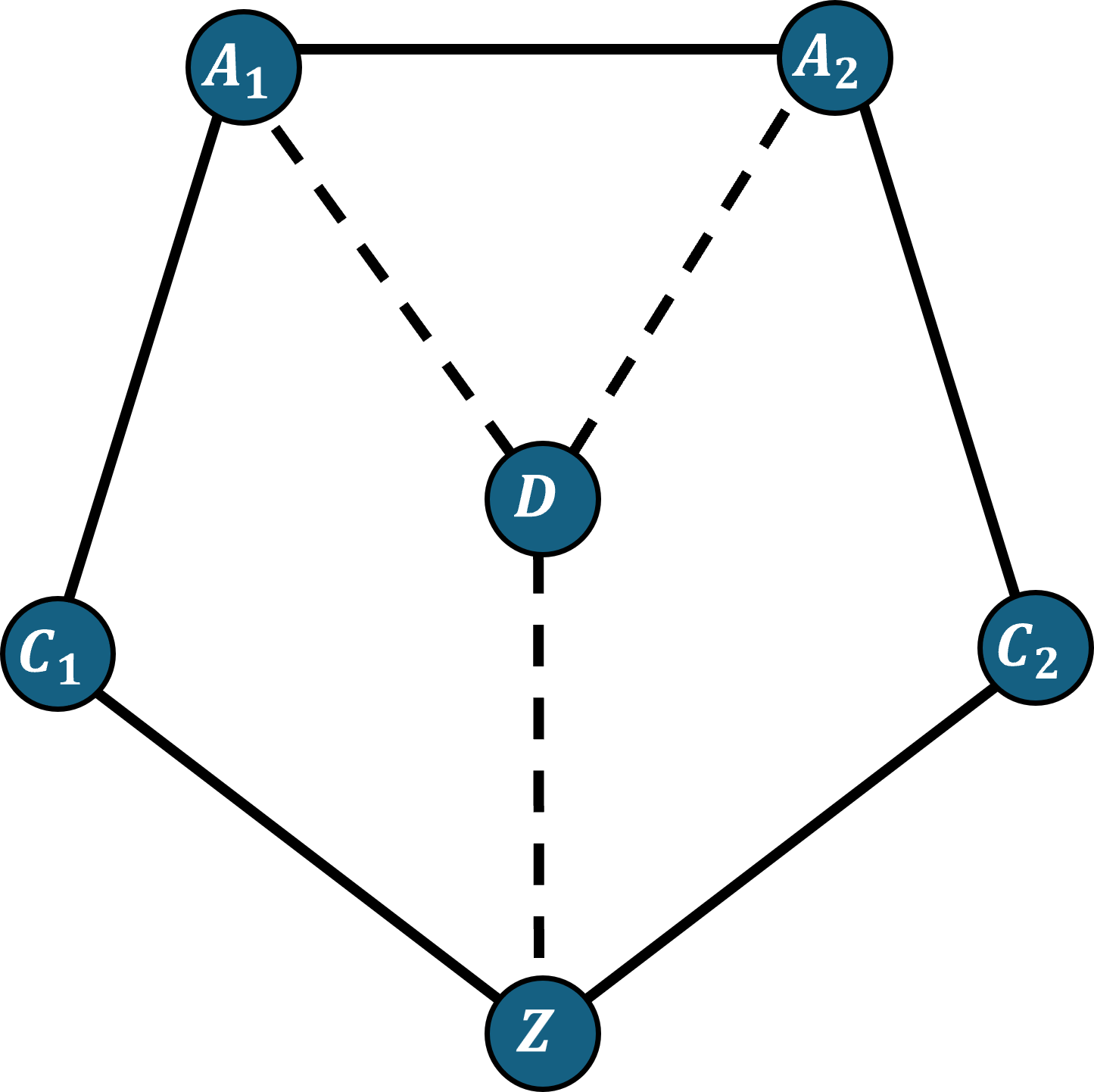}
        \caption{Middle vertex \(D\)}
        \label{fig:D-in-middle-5-cycle}         
     \end{subfigure}
     \caption{Decomposition of the cycles}
\end{figure} 
    
    By \pref{claim:louis-link-const-diam} there exists a length two path from \(Z\) to every vertex \(A_i\) in the cycle as in \pref{fig:9-to-5-cycles}. Therefore, by drawing these paths we can decompose our cycle into \(9\) cycles of length \(5\) of the form \((Z,C_i,A_i,A_{i+1},C_{i+1},Z)\) such that the row/column spaces of \(Z\) are a direct sum with \(\row(A_1)+\row(A_2)\) and \(\col(A_1)+\col(A_2)\).  This is a similar step as the one we do in \pref{fig:five-cyc-1} in the Johnson case.

    Fix a cycle as above let us denote by \(B\) the matrix such that \(B \oplus (A_1-B) \oplus (A_2-B)\). Let \(Z_1,Z_2\) be of equal rank such that \(Z = Z_1 \oplus Z_2\). Denote by \(D = Z_1 \oplus B\). We observe that \(\set{Z,D},\set{A_1,D}\) and \(\set{A_2,D}\) are all edges and therefore we can decompose this cycle into a \(3\)-cycle \((D,A_1,A_2,D)\) and two \(4\)-cycles \((Z,C_i,A_i,D,Z)\) that we need to tile (see \pref{fig:D-in-middle-5-cycle}).

    The tiling of \((D,A_1,A_2,D)\) is simple. We decompose \(Z_1,B,A_1-B\) and \(A_2-B\) into a direct sum of matrices of equal rank \(Z_1=K_1 \oplus K_2, B=L_1\oplus L_2, A_1-B=M_1 \oplus M_2\) and \(A_2-B=N_1 \oplus N_2\). Then the matrix \(E = K_1\oplus L_1 \oplus M_1 \oplus N_1\) is connected in a triangle with all three edges. so we can triangulate this three cycle using three triangles.

    As for the cycles \((Z,C_i,A_i,D,Z)\), we tile them with the claim \pref{claim:tiling-louis-special-cycles} using \(24\)-triangles. Thus we can triangulate every \(5\)-cycle using of \(51\) triangles. As every \(9\)-cycle is decomposed into \(9\) such \(5\)-cycles we get a total of \(459\) triangles, so the cone area is \(459\).
\end{proof}

\begin{proof}[Proof of \pref{claim:louis-link-const-diam}]
    First note that the first statement follows from the second, because for any \(A,B\) there exists some \(D \in X(0)\) such that \(A \oplus D\) and \(B \oplus D\) (this is because of the assumption that \(n > 2d\)). Thus we can take a \(2\)-path from \(A\) to \(D\) and then another \(2\)-path from \(D\) to \(B\). 
    
    Let us thus show that assuming that \(A \oplus B\) there is a \(2\)-path from \(A\) to \(B\). Decompose \(A=A_1\oplus A_2\) of equal rank, and similarly for \(B=B_1 \oplus B_2\). Let \(C=A_1\oplus B_1\) and take the path \((A,C,B)\). The claim follows.
\end{proof}

Up until now, this was largely similar to \pref{lem:cone-for-link-X}. Surprisingly, the proof of \pref{claim:tiling-louis-special-cycles} is more complicated then its Johnson analogue \pref{claim:contracting-almost-disjoint-4-cycles-link-X}. This stems from the fact the matrix domination poset is different from the matrix Johnson Grassmann poset, and in particular there isn't a perfect analogue to set intersection in the matrix poset (this is called a meet in poset theory).

\begin{proof}[Proof of \pref{claim:tiling-louis-special-cycles}]
    We will use the middle path pattern \pref{claim:middle-path}. We observe that \(C_1,D_1 \leq Z\) and \(D_2,D_2 \leq A\), and we can view them as vertices in \(\under^{Z}\) and \(\under^{A}\) respectively. We will use a similar strategy as in \pref{claim:contracting-almost-disjoint-4-cycles-link-X}, and use the fact that the graph \(\under^{Z}\) has short paths between every two matrices. That is,
    \begin{claim}\label{claim:short-paths-in-under-graph}
        For every \(m\) and two matrices \(C,D \in \under^m\) there is a path of length \(12\) between \(C\) and \(D\).
    \end{claim}
    This claim is proven in the end of \pref{app:short-paths-in-subposet-graph}. 
    
    Using this claim, we find a pair of \(12\)-paths from \(C_i\) to \(D_i\) in \(\under^{Z}\) and \(\under^{A}\) respectively: \((C_1,M^1,M^2,\dots,M^{11},D_1)\) and \((C_2,N^1,N^2,\dots,N^{11},D_2)\). Then we observe that the matrices \(K^i=M^i \oplus N^i\) form a path \((C,K^1,K^2,\dots,K^{11},D)\) from \(C\) to \(D\) in \(X\). Moreover, every edge in this path is in a triangle with \(A\) and with \(Z\). The claim is proven.
\end{proof}

\subsubsection{Coboundary expansion in the links and cosystolic expansion}
    In \cite{Golowich2023} it is proven that the Cayley complex whose link is \(X\) is not a coboundary expander for any constant \(\beta > 0\). It is done by proving that \(Z^1(X,\Gamma) \ne B^1(X,\Gamma)\).
    
    However, we can still prove that these complexes are \emph{cosystolic expanders}. This notion is a weakening of coboundary expansion that allows the existence of cohomology, i.e.\ for \(Z^1(X,\Gamma) \ne B^1(X,\Gamma)\):
    \begin{definition}[$1$-dimensional cosystolic expander] \label{def:def-of-cosystolic-exp}
        Let \(X\) be a \(d\)-dimensional simplicial complex for \(d \geq 2\). Let \(\beta >0\). We say that \(X\) is a $1$-dimensional \emph{\(\beta\)-cosystolic expander} if for every group \(\Gamma\), and every \(f \in C^1(X,\Gamma)\) there exists some \(h \in Z^1(X,\Gamma)\) such that
    \begin{equation} \label{eq:def-of-cosyst-exp}
        \beta \dist(f,h) \leq \wt(\coboundary f).
    \end{equation}
    \end{definition}
    The difference between this definition and the definition of coboundary expansion, is that in the definition of coboundary expansion we require \(h=\coboundary g\), or equivalently \(h \in B^1(X,\Gamma)\) whereas here we only require \(h \in Z^1(X,\Gamma)\) which is less restrictive.

    We note that cosystolic expansion is still an important property and can sometimes be used in applications such as topological overlapping property \cite{DotterrerKW2018} or to get some form of agreement testing guarantee \cite{DiksteinD2023agr}.

Kaufman, Kazhdan and Lubotzky were the first to prove that in local spectral expanders, cosystolic expansion follows from coboundary expansion of the links \cite{KaufmanKL2016} (together with local spectral expansion). Later on their seminal result was extended to all groups and higher dimensions \cite{EvraK2016,KaufmanM2018,KaufmanM2021,KaufmanM2022,DiksteinD2023cbdry}. In particular, we can use the quantitatively stronger theorem in \cite[Theorem 1.2]{DiksteinD2023cbdry}, to get cosystolic expansion from \pref{lem:louis-coboundary-expander} and the local spectral expansion.
\begin{corollary}
    Let \(\Y=\M(\mathcal{G}(\F_2^n)^{\leq d})\) for \(n > 2d\). Let \(X = \beta(\Y)\) be the basification of \(Y\). Let \(Cay(\F_2^n,S)\) be the complexes whose links are isomorphic to \(X\), as constructed in \pref{sec:construction-subpoly-HDX}. Assume that \(n\) is sufficiently large such that \(Cay(\F_2^n,S)\) is a \(10^{-4}\)-local spectral expander. Then \(Cay(\F_2^n,S)\) is also a $1$-dimensional \(10^{-4}\)-cosystolic expander.
\end{corollary}
We omit the details.
 
\section{Induced Grassmann posets} \label{sec:induced-Grassmann}
In this subsection we will present a generalization of our Johnson and Matrix Grassmann posets. After a few preliminaries, we will present a way to construct Grassmann posets based on the Hadamard encoding. We will see that such complexes have a simple description of their links. Afterwards, we will show that on some of these complexes we can decompose the links into a tensor product of simple graphs, generalizing some of the previous results in this paper. We start with some classical facts about quotient spaces and bilinear forms. 
\subsection{Preliminaries on quotient spaces and bilinear forms}
\paragraph{Quotient spaces} Let \(U \subseteq V\) be a subspace. The quotient space is a vector space \(\nicefrac{V}{U}\) whose vectors are the sets \(v+U = \sett{v+y}{y \in U}\) for every vector \(v \in V\). It is standard that addition (and scalar multiplication) is well defined on these sets. Addition is defined as \((v+U) + (x+U) \coloneqq (v+x)+U\) and the sum of the two sets does not depend on the representatives \(v \in v+U\) and \(x \in x+U\).

\paragraph{Bilinear forms}Let \(V_1,V_2\) be two vector spaces. A bilinear form is a function \(\iprod{}:V_1 \times V_2 \to \F_2\) such that \[\iprod{x_1 + x_2,y_1 + y_2} = \iprod{x_1,y_1}+\iprod{x_1,y_2}+\iprod{x_2,y_1}+\iprod{x_2,y_2}\] 
for every \(x_1,x_2 \in V_1\) and \(y_1,y_2 \in V_2\). A bilinear form is \emph{non-degenerate} if for every \(0 \ne x \in V_1\) there exists \(y \in V_2\) such that \(\iprod{x,y}=1\), and for every \(0 \ne y \in V_2\) there exists \(x \in V_1\) such that \(\iprod{x,y}=1\).

When possible, we will work with the standard bilinear form \(\iprod{}_{std}:\F_2^n \times \F_2^n \to \F_2\) given by \(\iprod{x,y}_{std} = \sum_{j=1}^n x_j y_j\) instead of arbitrary bilinear forms. The next claim allows us to move to this standard bilinear form.
\begin{claim} \label{claim:always-work-with-standard}
    Let \(\iprod{}:V_1 \times V_2 \to \F_2\) be a non-degenerate bilinear form. Then \(\dim(V_1) = \dim(V_2) =: m\) and there exists isomorphisms \(\psi_i:V_i \to \F_2^m\) such that \(\iprod{x,y} = \iprod{\psi_1(x),\psi_2(y)}_{std}\).
\end{claim}
From now on, we write \(\iprod{}\) instead of \(\iprod{}_{std}\) when it is clear from context; \pref{claim:always-work-with-standard} justifies the abuse of notation.
\begin{proof}[Proof sketch of \pref{claim:always-work-with-standard}]
    To see that \(\dim(V_1) \leq \dim(V_2)\), note that from non-degeneracy the function \(\pi:V_1 \to V_2^*\) that sends \(x \overset{\pi}{\mapsto} \iprod{x, \cdot}\) is an injective linear function. Swapping \(V_1\) and \(V_2\) shows that the dimensions are equal. 
    
    We observe that for every basis \(f_1,\dots,f_m\) of \(V_1\) there exists a \emph{dual basis} \(h_1,\dots,h_m\) of \(V_2\) such that \(\iprod{f_i, h_j} = \begin{cases}
        1 & i=j \\ 0 & i \ne j
    \end{cases}\).
    As for the isomorphisms, we fix some basis \(f_1,\dots,f_m\) of \(V_1\) and its dual basis \(h_1,\dots,h_m \in V_2\). Let \(e_1,\dots,e_m \in \F_2^m\) be the standard basis. It is easy to verify that the pair of isomorphisms \(\psi_1\) that sends \(f_i\) to \(e_i\), and \(\psi_2\) that sends \(h_i\) to \(e_i\) satisfy \(\iprod{x,y} = \iprod{\psi_1(x),\psi_2(y)}\).
\end{proof}

For a fixed bilinear form and a subspace \(U \subseteq V_1\) the \emph{orthogonal subspace} is the subspace \[U^\bot = \sett{y \in V_2}{\forall x \in U, \iprod{x,y}=0} \subseteq V_2,\]
and for a subspace \(U \subseteq V_2\), one can analogously define \(U^\bot \subseteq V_1\). 
\begin{definition}[Derived bilinear form]
    Let \(\iprod{}:V_1 \times V_2 \to \F_2\) be a non-degenerate bilinear form and let \(U \subseteq V_1\) be a subspace. The \emph{derived bilinear form} \(\iprod{}_U:\nicefrac{V_1}{U} \times U^\bot \to \F_2\) is a bilinear form given by
\[\iprod{x+U,v}_U = \iprod{x,v}\]
for any \(x+U \in \nicefrac{V_1}{U}\) and \(v \in U^\bot\).
\end{definition}
It is easy to verify that \(\iprod{x+U,v}_U\) is well defined. That is, it does not depend on the representative \(x \in x+U\). It is also straightforward to verify that this derived bilinear form is non-degenerate.

Most of the time we will abuse notation and write \(\iprod{x+U,v}\) instead of \(\iprod{x+U,v}_U\) when $U$ is clear from the context.

\subsection{The induced Grassmann poset - definition and running examples}
In this subsection we construct a type of Grassmann posets called \emph{induced Grassmann posets} \Ynote{Hadamard Grassmann posets?}. This construction generalizes the matrix Grassmann poset \cite{Golowich2023}, its sparsified version in \pref{def:matrix-poset} and the Johnson Grassmann poset in \pref{def:Johnson-complex}.

Let \(s:\F_2^d \setminus \set{0} \to \F_2^n\) be a function, not necessarily linear. The \emph{induced Hadamard encoding} of \(s\) is the \emph{linear} map \(\widehat{s}:\F_2^d \to \F_2^n\) where 
\[\widehat{s}(x) \coloneqq \sum_{v \in \F_2^d}\iprod{x,v} s(v) = \sum_{v \in \F_2^d: \iprod{x,v}=1} s(v).\]
We note that the inner product between any $x$ and the \(0\) vector is always \(0\). Thus we just need to define these functions over \(s:\F_2^d \setminus \set{0} \to \F_2^n\).

We call \(\widehat{s}\) the induced Hadamard encoding, because if \(n=2^d - 1\), we index the coordinates of \(\F_2^{n}\) by the non-zero vectors of \(\F_2^d\), and set \(s:\F_2^d \setminus \set{0} \to \F_2^{n}\) to be the function that sends $v\in\F_2^d\setminus\set{0}$ to the standard basis vector that is \(1\) on the \(v\)-th coordinate and zero elsewhere. Then \(\widehat{s}\) is the (usual) Hadamard encoding and its image is the Hadamard code. 

\begin{definition}[Induced Grassmann poset]
    Let \(d,n \geq 0\) be two integers such that \(n \geq 2^d-1\). Let \(\mathcal{S}\) 
    be a set of (non-linear) functions such that every \(s \in \mathcal{S}\) is injective and the set \( Im(s) \coloneqq \sett{s(v)}{v \in \F_2^d \setminus \set{0}}\) is independent. The \(\mathcal{S}\)-\emph{induced Grassmann poset} is the \(d\)-dimensional Grassmann poset \(Y^{\mathcal{S}}\) such that \(Y^{\mathcal{S}}(d) = \sett{Im(\widehat{s})}{s \in \mathcal{S}}\).
\end{definition}
To be explicit, for \(i < d\), an \(i\)-dimensional subspace \(W\) is in \(\Y^{\mathcal{S}}(i)\) if there exists some \(W' \in \Y^{\mathcal{S}}(d)\) containing \(W\). Equivalently, this is if and only if there exist a \(s \in \mathcal{S}\) and a subspace \(U \subseteq \F_2^d\) such that \(\widehat{s}|_U = W\). For \(i=1\) this says that \(v \in \Y^{\mathcal{S}}(1)\) if and only if \(v=\widehat{s}(x)\) for some \(s \in \mathcal{S}\) and non-zero \(x \in \F_2^d\). 

We note that if the vectors in the image of \(s\) are independent, then \(\dim(Im(\widehat{s})) = d\). We also note that different \(s,s'\) may still result in the same spaces \(Im(\widehat{s}) = Im(\widehat{s'})\).

Although this definition seems abstract, hopefully, the following examples, which were already introduced in the previous sections, will convince the reader that this definition is natural.
\begin{example}[The complete Grassmann]
    Let \(d\) and \(n \geq 2^d-1\) be positive integers. Let \(\mathcal{S}\) be \emph{all} injective functions whose image is independent. Then \(\Y^\mathcal{S} = \mathcal{G}(F_2^n)^{\leq d}\) is the \(d\)-skeleton of the complete Grassmann of \(\F_2^n\). To see this let us show that every \(d\)-dimensional subspace \(W\) has some \(s \in \mathcal{S}\) such that \(Im(\widehat{s})=W\). Let \(s' \in \mathcal{S}\) be an arbitrary function and let \(W' = Im(\widehat{s'})\). Find some \(A \in GL_n(\F_2)\) that maps \(W'\) to \(W\) and set \(s(x) \coloneqq A(s'(x))\). It is easy to verify that indeed \(s \in \mathcal{S}\) and that \(Im(\widehat{s}) = W\).
\end{example}

\begin{example}[Johnson Grassmann poset]
    Let \(d\) and \(n \geq 2^d-1\) be positive integers. Let \(E = \set{e_i}_{i=1}^n\) be the standard basis for \(\F_2^n\). Let \(\mathcal{S}_J\) be the set of injective functions from \(\F_2^d\) to \(E\). For this \(\mathcal{S}_J\), the poset \(Y^{\mathcal{S}_J}\) is the Johnson Grassmann poset.
    
    First note that a vector from \( \F_2^n\) is in \(\Y^{\mathcal{S}_J}(1)\) if and only if its Hamming weight is \(2^{d-1}\): every such vector \(w\) is equal to \(w = \widehat{s}(x) = \sum_{v \in \F_2^d} \iprod{x,v} s(v)\) for some \(s \in \mathcal{S}_J\) and \(x \ne 0\). The \(s(v)\)'s are distinct standard basis vectors, and there are exactly \(2^{d-1}\) non-zero \(\iprod{x,v}\) in the sum. Therefore the weight of \(w\) is \(2^{d-1}\).

    The \(d\)-dimensional subspaces are the Hadamard encoding on \(2^d-1\) coordinates inside \(n\). Let us be more precise. Fix a map \(s\) and for simplicity assume that \(Im(s) = \set{e_1,e_2,\dots,e_{2^d-1}}\). Let \(W = \sp (Im(s))\). By indexing \(W\) with the coordinates of \(v \in \F_2^{d} \setminus \set{0}\), then \(\widehat{s}(x) = \left ( \iprod{x,v} \right )_{v \in \F_2^{d} \setminus \set{0}} \in W\), which is the Hadamard encoding of \(x\).

    The basification of this poset is an alternative description of the links of Johnson complexes.
\end{example}

\begin{example}[Matrix poset {\cite{Golowich2023}}] \label{ex:matrices}
    Let \(d\) and \(n \geq 2^d-1\). Recall that \(\MP(1)\) is the set of all rank one \(n \times n\) matrices over \(\F_2\). Let \(\mathcal{S}_{\MP}\) be the set of functions \(s:\F_2^d \to \MP(1)\) such that the matrix \(M_s = \left (\bigoplus_{v \in \F_2^d \setminus \set{0}}s(v) \right )\) is a direct sum, or in other words \(\rank(M_s) = 2^d - 1\). One can see directly that such \(s\)'s are the admissible functions of \pref{def:admissible-func} for the special case where the ambient poset is complete Grassmann \(\mathcal{G}(\F_2^n)\). In this case, \(\Y^{\mathcal{S}_{\MP}} = \M(\mathcal{G}(\F_2^n))\).\footnote{There is a slight technical difference between this complex and the one in \cite{Golowich2023}, that follows from Golowich's use of larger fields. See \cite{Golowich2023} for more details.}
    
    To draw the analogy to the Johnson, let us illustrate that indeed a matrix \(A\) is a $1$-dimensional subspace in the $\mathcal{S}_{\MP}$-induced Grassmann poset \(Y^{\mathcal{S}_{\MP}}\) if and only if \(\rank(A) = 2^{d-1}\). This follows the same logic as in the Johnson case. If \(A \in \Y^{\mathcal{S}_{\MP}}(1)\) then \(A = \widehat{s}(x)\) for some \(s \in \mathcal{S}_{\MP}\) and \(x\ne 0\). This means that \(A\) is the sum of \(2^{d-1}\) distinct \(s(v)\)'s. As the matrix \(M_s\) is a direct sum, the sum of the \(s(v)\)'s is also a direct sum. This implies that \(\rank(A) = \sum_{x: \iprod{v,x}=1}\rank(s(v)) = 2^{d-1}\). On the other hand, suppose \(A\) has rank \(2^{d-1}\) and decomposition \(A = \sum_{j=1}^{2^{d-1}} e_j \otimes f_j\). Take an arbitrary \(x \in \F_2^d \setminus \set{0}\) and let \(R = \sett{v \in \F_2^d}{\iprod{v,x}=1}\). As \(R\) also has size \(2^{d-1}\), we can find some \(s \in \mathcal{S}\) that maps the elements in \(R\) to the \(e_j \otimes f_j\) in the decomposition of \(A\). For such an \(s\), we have \(\widehat{s}(x) = \sum_{v \in R}s(v) = A\), thus concluding \(A \in \Y^{\mathcal{S}}(1)\).
\end{example}
The construction in \cite{Golowich2023} used matrices over a larger field of characteristic \(2\) instead of \(\F_2\) itself. See \cite{Golowich2023} for the details.

The next example is the one given in \pref{def:matrix-poset}.
\begin{example}[Sparsified matrix poset]
    Let \(d>0\) and let \(Y_0\) be any \(2^d-1\)-dimensional Grassmann poset over \(\F_2^n\). We define a sparsification \(\mathcal{S}_{Y_0} \subseteq \mathcal{S}_{\MP}\). 
    \[\mathcal{S}_{Y_0} = \sett{s \in \mathcal{S}_{\MP}}{ \row(M_s),\col(M_s) \in Y_0(2^d-1)}\footnote{The set \(\mathcal{S}_{Y_0}\) was the set of \emph{admissible} functions in \pref{sec:construction-subpoly-HDX}. We note that in this case we use \(d\) as the resulting dimension of the complex, and not \(\dimloutfull\) as in \pref{sec:construction-subpoly-HDX}.}\]
    where \(M_s = \left (\bigoplus_{v \in \F_2^d \setminus \set{0}}s(v) \right )\). The one dimensional spaces \(M \in Y^{\mathcal{S}_{Y_0}}(1)\) are \(2^{d-1}\)-rank matrices such that \(\row(M), \col(M) \in Y_0(2^{d-1})\). The top-level spaces are \(W \in Y^{\mathcal{S}_{\MP}}(d)\), such that the sum of row spaces (resp. column spaces) is in \(Y_0(2^d-1)\).
\end{example}

Finally, we end this subsection with the comment that if one has a distribution \(\mu\) over \(\mathcal{S}\), this also defines a distribution over \(Y^{\mathcal{S}}(d)\) in a natural way. That is, sampling \(s \sim \mu\) and then outputting \(Im(\widehat{s}) \in Y^{\mathcal{S}}(d)\). Unless otherwise stated, we will always take the uniform distribution over \(\Y^{\mathcal{S}}(d)\).


\subsubsection{Subspaces of an induced Grassmann poset}
Fix \(i < d\). The subspaces \(Y^{\mathcal{S}}(i)\) are defined by downward closure, but we can give an alternative description of \(Y^{\mathcal{S}}(i)\) via an induced Grassmann poset with some \(\mathcal{S}_i\) that is derived from \(\mathcal{S}\). We begin by looking at our examples for intuition and then generalize.

In the examples above, the lower dimensional subspaces have some structure to them. For instance, one can verify that the \(i\)-dimensional subspaces in the Johnson Grassmann poset could be described as images of Hadamard encodings \(\widehat{s}\) where the image of \(s:\F_2^{i}\setminus\set{0} \to \F_2^n\) consists of vectors of Hamming weight \(2^{d-i}\) such that the supports of every pair \(s(v)\) and \(s(v')\) are disjoint. This generalizes the case where \(i=d\), and the image of \(s \in \mathcal{S}_J\) are distinct vectors of weight \(1\).

In a similar manner, we saw in \pref{claim:poset-i-space-description} that \(i\)-dimensional subspaces of the Matrix Grassmann poset are images of induced Hadamard encodings
\[
    Im(\widehat{s}) = \sett{\sum_{v \in \F_2^{i}}\iprod{x,v} s(v)}{x \in \F_2^{i}},
\]
where every \(s(v)\) has rank \(2^{d-i}\)\, and \(M_s = \bigoplus_{v \in \F_2^{i} \setminus \set{0}} s(v)\) has rank \(2^d-2^{d-i}\). Again, this is a natural generalization of the case where \(i=d\).

Back to the more general setup. Recall that every \(W \in Y^{\mathcal{S}}(d)\) is the image of \(\widehat{s}\) for some \(s \in \mathcal{S}\). The map \(\widehat{s}\) is always injective, therefore every \(i\)-dimensional subspace \(W' \subseteq W\) is also the image \(W' = Im(\widehat{s}|_{U})\) of some \(i\)-dimensional subspace \(U \subseteq \F_2^d\). 

Let \(V \subseteq \F_2^d\) be the \((d-i)\)-dimensional space such that \(U=V^\perp\), that is, \(U = \sett{x \in \F_2^d}{\forall y \in V, \; \iprod{x,y} = 0}\). Let \(s_V:(\nicefrac{\F_2^d}{V}) \setminus \set{0+V} \to \F_2^n\) be \(s_V(v+V) = \sum_{v' \in v+V} s(v')\). We use the bilinear form \(\iprod{}_V\) to define the function \(\widehat{s_V}:U \to \F_2^n\) as the induced Hadamard encoding, i.e.\
\[\widehat{s_V}(x) = \sum_{v+V \in \nicefrac{\F_2^d}{V}} \iprod{v+V,x}_V \cdot s_V(v+V).\]
Here we recall that $\iprod{v+V,x}_V \coloneqq \iprod{x,v}$ for $x\in U$ (and we recall that this is independent of the choice of representative \(v\)).

Let
\[\mathcal{S}_i = \sett{s_V}{s \in \mathcal{S}, V \subseteq \F_2^d, \dim(V)=d-i}.\]
We will show that for every \(s \in \mathcal{S}\), \(Im(\widehat{s}|_U) = Im(\widehat{s_V})\) and hence \(\mathcal{S}_i\) induces \(\Y^{\mathcal{S}}(i)\).

\begin{claim} \label{claim:skeleton-of-induced-Grassmann}
    For every subspace \(U \subseteq \F_2^d\), \(V = U^\bot\), function \(s:\F_2^d \to \F_2^n\), and \(x \in U\), \(\widehat{s_V}(x)=\widehat{s}(x)\). As a result, \((Y^{\mathcal{S}})^{\leq i} = Y^{\mathcal{S}_i}\).
\end{claim}
Here there is some abuse of notation since the domain of \(\widehat{s_V}\) is not \(\F_2^i\), but some arbitrary \(i\)-dimensional space \(U \subseteq \F_2^d\). However, this does not matter since we are taking the \emph{image space} of \(\widehat{s_V}\).

Before proving this claim, let us review our previous examples in light of this.

\begin{example}[Johnson Grassmann poset] \label{ex:johnson-subspaces}
    For the Johnson Grassmann poset, fix a \((d-i)\)-dimensional subspace \(V \subseteq \F_2^d\). Then a function \(s_V \in \mathcal{S}_{J,i}\) if and only if every value \(s_V(v+V)\) has Hamming weight \(2^{d-i}\) and for every \(v+V \ne u+V\) the supports of \(s_V(v+V)\) and \(s_V(u+V)\) are mutually disjoint. 

    By \pref{claim:always-work-with-standard} we replace the functions in \(\mathcal{S}_i\), with similar functions whose domain is \(\F_2^i\), without changing the resulting poset. Thus another set of functions \(\mathcal{S}_i'\) such that the induced Grassmann on \(\mathcal{S}_i'\) is also equal to \((Y^{\mathcal{S}_J})^{\leq i}\) is the set of all functions \(s:\F_2^i \setminus \set{0} \to \F_2^n\) such that the Hamming weight of every \(s(v)\) is \(2^{d-i}\), and such that if \(v \ne u\) then the supports of \(s(v)\) and \(s(u)\) are disjoint. These are exactly the admissible functions in \pref{def:Johnson-admissible-sets}. 
\end{example}

\begin{example}[Matrix poset and sparsified matrix poset]
    For the matrix example above let us fix some \(s \in \mathcal{S}_{\MP}\) and \((d-i)\)-dimensional subspace \(V \subseteq \F_2^d\). The matrix \(s_V(v+V)\) is a direct sum of \(2^{d-i}\) rank \(1\) matrices. Thus \(\rank(s_V(v+V)) = 2^{d-i}\). The matrices \(\set{s_V(v+V)}_{v+V}\) have the property that their sum \(M_{s_V} = \bigoplus_{v+V \ne 0+V} s_V(v+V)\) is a direct sum, or equivalently that \(\rank(M_{s_V}) = 2^{d}-2^{d-i}\).\footnote{Note that this is not the matrix \(M_s\) in \pref{ex:matrices}, since the matrices \(s(v)\) do not participate in this sum when \(v \in V \setminus \set{0}\).}
    
    Again by \pref{claim:always-work-with-standard}, we can give an equivalent set of functions \(\mathcal{S}_i\) --- the level \(i\) \emph{admissible} functions in \pref{sec:construction-subpoly-HDX}. These are all functions \(s:\F_2^i \setminus \set{0} \to \F_2^n\) such that \(\rank(s(v)) = 2^{d-i}\) and \(\rank(M_s)=2^d - 2^{d-i}\) where \(M_s = \sum_{v \in \F_2^i \setminus \set{0}}s(v)\).
\end{example}
We remark that the sparsified matrix induced Grassmann poset \(\mathcal{S}_{Y_0}\) also admit a similar description of the \(i\)-spaces which was explained in \pref{sec:construction-subpoly-HDX}.

\begin{proof}[Proof of \pref{claim:skeleton-of-induced-Grassmann}]
    Fix \(s\) and an $i$ dimensional subspace \(U\). Let \(V=U^\bot\) be some \(d-i\)-dimensional subspace in \(\F_2^d\). Let \(x \in U\). Note that while \(\widehat{s}(x) = \sum_{v\ne 0} \iprod{x,v}s(v)\), we can remove all elements \(v \in V\) from the sum since \(x\) is orthogonal to them. We also recall that for every \(v_1,v_2 \in v+V\), \(\iprod{x,v_1} = \iprod{x,v_2} \coloneqq \iprod{x,v+V}\) so
    \begin{align*}
        \widehat{s}(x) &= \sum_{v\notin V} \iprod{x,v}s(v) \\
        &=\sum_{v+V \ne 0+V} \iprod{x,v+V} \left(\sum_{v' \in v+V}s(v') \right)\\
        &= \sum_{v+V \ne 0+ V}\iprod{x,v+V} s_V(v+V) = \widehat{s_V}(x).
    \end{align*}
    Thus in particular for every \(s\), \(Im(\widehat{s}|_U) = Im(\widehat{s}_V)\). As \(\widehat{s}\) is injective, every subspace \(W \in Y^{\mathcal{S}_i}(i)\) is the image \(Im(\widehat{s}|_U)\) of some \(i\)-dimensional subspace \(U \subseteq \F_2^d\). The claim follows.
\end{proof}

\subsection{A decomposition for the links}
Our goal is to prove that links of some induced Grassmann posets decompose into tensor products of simple graphs, reproving and extending both \pref{prop:louis-decomposition} for the matrix Grassmann poset, \pref{lem:grassmannian-link-decomposition} for the sparsified matrix poset, and \pref{prop:link-structure} for the Johnson Grassmann poset.

For this we need to zoom in to a less general class of induced Grassmann posets, which we call \emph{rank-based} induced Grassmann posets with the \emph{direct-meet property}. Let us define these.

\subsubsection{Rank-based posets}
Let \(G \subseteq \F_2^n\) be a spanning set that does not contain \(0\). The rank of a vector \(v\) with respect to \(G\), \(\rank_G(v)\), is the minimum number of elements in \(G\) necessary to write \(v\) as their sum. 
\begin{definition}[\(k\)-admissible spanning set]
We say \(G\) is \(k\)-admissible if for every \(v\) with \(\rank_G(v) \leq k\) there exists some \(u\) such that \(\rank_G(v+u) =\rank_G(v)+\rank_G(u) =k\).
\end{definition}

\begin{example}
\begin{enumerate}
    \item For example in Johnson Grassmann posets, take \(G_{\mathcal{J}} \coloneqq \set{e_1,e_2,\dots,e_n}\) to be the standard basis. Then \(\rank_{G_{\mathcal{J}}}(v)=wt(v)\) is just its Hamming weight. This set is \(n\)-admissible. 
    \item For matrix Grassmann posets, taking \(G_{\MP}\coloneqq \MP(1)\) to be the set of rank one matrices inside \(\F_2^{n \times n}\), then \(\rank_{G_{\MP}}(M) = \rank(M)\) is its usual matrix rank. This set is also \(n\)-admissible. This second example can be extended to \(G\) being the rank one \(k\)-tensors for \(k > 2\) in \(\F_2^{n\times n \dots \times n}\) as well. 
\end{enumerate}    
\end{example}

The following notation of a direct sum generalizes the matrix poset direct sum.
\begin{definition}[Abstract direct sum]
    We say that \(w_1+w_2+\dots+w_j\) is a direct sum and write \(w_1\oplus w_2 \oplus \dots \oplus w_j\), if \(\rank_G(w_1 + w_2 + \dots +w_j) = \sum_{\ell=1}^j \rank_G(w_\ell)\). We also write \(w_1 \leq w_2\) if \(w_1 \oplus (w_2-w_1)\).
\end{definition}
For \(G_{\mathcal{J}}\) we have \(w_1 \oplus w_2\) if and only if they are disjoint (as sets) and \(w_1 \leq w_2\) if and only if \(w_1 \subseteq w_2\). For \(G_{\MP}\) the definitions of \(M_1 \oplus M_2\) and \(M_1 \leq M_2\) coincide with the direct sum and matrix domination order defined in \pref{sec:def-of-domination-relation}. It is easy to show that for any \(G\) the relation \(\leq\) is always a partial order.

\begin{definition}[Rank \(r\)-admissible functions]
    Let \(d\) be an integer. Let \(G\) be a \((r(2^{d}-1))\)-admissible set. A \emph{rank \(r\) admissible} function is \(s:\F_2^d \setminus \set{0} \to \F_2^n\) such that every \(s(v)\) has \(\rank_G(s(v)) = r\), and such that \(\bigoplus_{v \in \F_2^d \setminus \set{0}} s(v)\) is a direct sum. The rank \(r\) admissible set of functions \(\mathcal{S}_G\) is the set containing all rank \(r\) admissible functions.
\end{definition}
We comment that this set \(\mathcal{S}\) indeed gives rise to an induced Grassmann poset, because if the sum of the vectors in \(\set{s(v)}_{v \in \F_2^d \setminus \set{0}}\) is a direct sum, then this set is independent. The reason is that if there is a nonempty subset \(A \subseteq \set{s(v)}_{v \in \F_2^d \setminus \set{0}}\) such that \(\sum_{ v \in A}s(v) = 0\), then 
\[\rank_G\left (\sum_{v \in \F_2^d \setminus \set{0}} s(v) \right )= \rank_G \left(\sum_{v \in \F_2^d \setminus (A \cup \set{0})}s(v) \right) \leq \sum_{v \in \F_2^d \setminus (A \cup \set{0})}\rank_G(s(v)) < \sum_{v \in \F_2^d \setminus \set{0}}\rank_G(s(v)), \]
which shows that in such a case the sum was not a direct sum in the first place.

For the rest of this section we will always assume there is some ambient \((r(2^{d}-1))\)-admissible set \(G\), and that \(\mathcal{S} = \mathcal{S}_G\) with respect to this set.

\begin{definition}[Rank-based induced Grassmann poset]
    If \(\mathcal{S}\) is the rank \(r\) admissible set with respect to some \(G\), then \(Y^{\mathcal{S}}\) is called a \emph{rank-based} induced Grassmann poset.
\end{definition}
\begin{example}
    Both the Johnson Grassmann poset and the matrix Grassmann posets are rank-based induced Grassmann posets for \(G_{\mathcal{J}}\) and \(\mathcal{G}_{\MP}\) respectively.
\end{example}

\begin{example}[Non example]
    The \(d\)-skeleton of the complete Grassmann is \emph{not} a rank-based induced Grassmann for any \(d \geq 2\). To see this we observe that in a rank-based induced Grassmann poset $\Y^{\mathcal{S}}$, every \(w \in \Y^{\mathcal{S}}(1)\) has \(\rank_G(w) = r2^{d-1} \geq 2\). This is because \(\widehat{s}(x)\) is always a direct sum of \(2^{d-1}\) vectors of rank \(r\). However, in the complete Grassmann \(\Y(1)\) contains all non-zero vectors, including the vectors in \(G\) which have rank one.
\end{example}

\subsubsection{The direct-meet property}
Unfortunately, this is still too general for a meaningful decomposition lemma. For this we also need to define another property of \(G\) which we will call the \emph{direct-meet property}.

\begin{definition}[Meet]
    The \(\Meet\) of a set of elements \(T \subseteq \F_2^n\) is an element \(\Meet(T) \in \F_2^n\) such that \begin{enumerate}
        \item For every \(t \in T\), \(Meet(T) \leq t\) (in the abstract direct sum sense). 
        \item For every \(v \in \F_2^n\), if \(v \leq t\) for every \(t \in T\), then \(v \leq \Meet(T)\).  
    \end{enumerate}
\end{definition}
This is a standard definition in poset theory. We note that \(\Meet(T)\) does not necessarily exist, but if it exists it is always unique.

\begin{definition}[Direct-meet property]
    We say that \(G\) has the \emph{direct-meet property} if for every three elements, \(w_1,w_2,w_3\) such that \(w_1 \oplus w_2 \oplus w_3\), the meet of \(\set{w_1 \oplus w_2,w_1 \oplus w_3}\) exists and is equal to \(w_1\).
\end{definition}
We say \(\Y^{\mathcal{S}}\) has the \emph{direct-meet property} if it is a rank-based induced complex, where the generating set \(G\) has the direct-meet property.

We note that this definition is not satisfied by all spanning sets \(G\).
\begin{example}[Non example]
    Let \(G=\set{e_1,e_2,e_3,e_1+e_4,e_2+e_4,e_3+e_4} \subseteq \F_2^4\) where \(e_1,e_2,e_3,e_4\) are the standard basis. Then \(G\) does not have the direct-meet property. To see this let \(w_1=e_1,w_2=e_2,w_3=e_3\). The triple \(w_1 \oplus w_2 \oplus w_3\) is a direct sum. However, \(e_1+e_4 \leq w_1+w_2=e_1+e_2=(e_1+e_4)+(e_2+e_4)\) and \(e_1+e_4 \leq w_1+w_3=e_1+e_3=(e_1+e_4)+(e_3+e_4)\), but \(e_1+e_4 \not \leq e_1 = w_1\). Hence \(w_1\) is not the meet of \(w_1 \oplus w_2\) and \(w_1 \oplus w_3\).
\end{example}

However, both the Johnson and the matrix sets have this property.
\begin{example}
For the set \(G_{\mathcal{J}}=\set{e_1,e_2,\dots,e_n}\) which defines \(\Y^{\mathcal{J}}\), the direct-meet property holds. Indeed, the meet of any set \(T\) is the intersection \(\Meet(T) = \bigcap_{t \in T}t\), where we treat the vectors as sets. In particular, the meet of \(w_1 \oplus w_2\) and \(w_1 \oplus w_3\) is \(w_1\).    
\end{example}

\begin{example}
   A more interesting example is the low-rank decomposition case. Let \(G\) be the set of rank \(1\) matrices which defines \(\Y^{\mathcal{S}_{\MP}}\). Contrary to the Johnson case, there exists a set \(\set{M_1,M_2}\) with no meet. For example let \(e_1,e_2\) be independent, let \(M_1=e_1 \otimes e_1 + e_2 \otimes e_2\) and let \(M_2 = e_1\otimes e_1 + e_2 \otimes (e_1+e_2)\). Both \(A = e_1 \otimes e_1\) and \(B = e_2 \otimes (e_1+e_2)\) are under \(M_1\) and \(M_2\) in the matrix domination order. However, \(A \not \leq B\) and \(B \not \leq A\) and therefore there is no meet for \(\set{M_1,M_2}\). Nevertheless, the rank \(1\) matrices have the direct-meet property.

\begin{claim} \label{claim:direct-meet-matrix-poset}
    The set \(G_{\MP}\) has the direct-meet property.
\end{claim}
This claim essentially abstracts the observation in \cite[Lemma 57]{Golowich2023}. We prove it in \pref{app:direct-meet-matrices}.
\end{example}

Finally we state an elementary claim about the meet which we will use later on, extending the property from three elements to any finite number of elements.
\begin{claim} \label{claim:extended-direct-meet-property}
    Let \(G\) be a set with the direct-meet property. Let \(T=\set{w_1,w_2,\dots,w_m}\) be such that \(\bigoplus_{j=1}^m w_j\) (with respect to $G$). Let \(\widehat{T} = \set{v_1,v_2,\dots,v_r}\) be linear combinations of \(T\), i.e.\ \(v_j = \bigoplus_{\ell \in I_j}w_\ell\) for some subsets \(I_j \subseteq [m]\). Then \[\Meet(\widehat{T}) = \bigoplus_{\ell \in \bigcap_{j=1}^r I_j} w_\ell.\]
\end{claim}
The proof is just an induction on the size of \(\widehat{T}\). We omit it.
\subsubsection{The general decomposition lemma}
We move towards our decomposition lemma for links of a rank-based \(\Y^{\mathcal{S}}\) with the direct-meet property. Our components for the decomposition lemma are the following.
\begin{definition}[\(z\)-domination graph]
    Let \(z \in \F_2^n\) be of rank \(4m\). The graph \(\mathcal{T}^{z}\) has vertices
    \[V = \sett{y \leq z}{\rank(y)=2m},\]
    \[E = \sett{\set{w_1 \oplus w_2, w_1 \oplus w_3}}{\rank(w_i) = m, w_1 \oplus w_2 \oplus w_3 \leq z}.\]
    To sample an edge we sample a uniformly random decomposition \(w_1\oplus w_2 \oplus w_3 \oplus w_4 = z\), such that \(\rank(w_i) = m\) and output \(\set{y,y'}\) such that \(y=w_1 \oplus w_2, y'=w_1 \oplus w_3\). The set of edges is the support of the the above distribution.
\end{definition}
For \(\mathcal{G}_{\MP}\), this graph coincides with the subposet graphs introduced in \pref{def:subposet-graph}. For the Johnson Grassmann poset, this graph turns out to be the Johnson graph \(\mathcal{T}^z \cong J(4m,2m,m)\). This is because every vector in the graph has weight \(2m\) and is contained in a \(z\) of weight \(4m\). Having \(y=w_1 \oplus w_2\) and \(y'=w_1 \oplus w_3\) in this definition is equivalent to \(y\) and \(y'\) intersecting on half their coordinates.

We continue with our second component.

\begin{definition}[\(z\)-link graph] \label{def:abstract-link-graph}
    Let \(z \in \F_2^n\) be a vector whose rank is \(r(2^d - 2^i)\) and let \(m \leq r2^{i-2}\). The graph \(\mathcal{H}_{z,m}\) has vertices
    \[V = \sett{y \in \F^n_2}{\rank(y)=2m, y \oplus z},\]
    \[E = \sett{\set{w_1 \oplus w_2, w_1 \oplus w_3}}{\rank(w_i)=m, z \oplus w_1 \oplus w_2 \oplus w_3}.\]
    We define the (weighted) edges according to the following sampling process. To sample an edge we sample a random \(s \in \mathcal{S}\) such that \(\bigoplus_{v \in \F_2^d}s(v) \coloneqq t \geq z\). Then we uniformly sample \(z' \leq t-z\) such that \(\rank(z') = 4m\) and output a random edge in \(\mathcal{T}^{z'}\).
\end{definition}
For \(\mathcal{G}_{\MP}\), this graph coincides with the disjoint graph introduced in \pref{def:disjoint-graph}. For the Johnson Grassmann poset, this graph is also a Johnson graph, \(\mathcal{H}_{z,m} \cong J(n-|z|,2m,m)\). This is because every vector in the graph has weight \(2m\) whose support is disjoint from \(z\). The sampling process in this graph amounts to uniformly sampling a \(4m\)-weighted set \(z'\) that is disjoint from \(z\), and then taking a step inside its \(\mathcal{T}^{z'}\). This is the same as sampling (uniformly) two subsets \(y,y'\) disjoint from \(z\) that intersect on half their coordinates.

The final component we need to state the decomposition lemma is this following claim.

\begin{claim} \label{claim:uniqueness-of-image-set-ranks}
    Let \(\Y^\mathcal{S}\) be a rank-based induced poset with the direct-meet property. Then for every \(W \in \Y^{\mathcal{S}}(i)\), and \(s_1,s_2 \in \mathcal{S}_i\) such that \(Im(\widehat{s_1})=Im(\widehat{s_2}) = W\), it holds that \(Im(s_1) = Im(s_2)\).
\end{claim}
We will hence denote by \(Q_W = \set{z_1,z_2,\dots,z_{2^i-1}}\) the image set of a given \(W\). We also denote by \(z_W = \bigoplus_{j=1}^{2^i-1}z_j\) the direct sum.

This claim is a consequence of the direct-meet property. We prove it later and for now, let us proceed to the decomposition lemma.

\begin{lemma}[The decomposition lemma]\label{lem:generalized-decomposition-lemma}
    Let \(W \in \Y^{\mathcal{S}}(i)\) with \(Q_W = \set{z_j}_{j=1}^{2^i - 1}\) and \(z_w = \bigoplus_{j=1}^{2^i} z_j\) as above. Then \(\Y_W \cong \mathcal{T}^{z_1} \otimes \mathcal{T}^{z_2} \otimes \dots \otimes \mathcal{T}^{z_{2^i-1}} \otimes \mathcal{H}_{z_W,2^{i-2}}\).
\end{lemma}

We observe that this lemma applied to \(\Y^{\mathcal{S}_{\mathcal{J}}}\) implies \pref{prop:link-structure}. Applied to \(\Y^{\mathcal{S}_{\MP}}\) this lemma proves \pref{prop:louis-decomposition}. We can even generalize this lemma slightly more to capture the sparsified matrix case in \pref{lem:grassmannian-link-decomposition}. See \pref{sec:distribution-over-admissibles} for details.

In order to prove \pref{lem:generalized-decomposition-lemma}, we need to define the isomorphism. To do so, we give a combinatorial criterion for a vector \(y\) being in \(\Y_W^{\mathcal{S}}(1)\) in terms of its meets with the elements in \(Q_W\). 

\begin{claim} \label{claim:abstract-unique-decomposition}
    Let \(W \in \Y^{\mathcal{S}}(i)\) with \(Q_W = \set{z_1,z_2,\dots,z_{2^i-1}}\) as above. Let \(y \in \F_2^n\). Then \(y\) is in the link of \(W\) if and only if there exists unique \(y_1,y_2,\dots,y_{2^i}\) with the following properties:
    \begin{enumerate}
        \item \(y = \bigoplus_{j=1}^{2^i} y_j\) (that is, \(y\) is equal to the sum of all \(y_j\)'s and this sum is direct).
        \item All \(y_j\)'s have equal rank.
        \item \(y_i \leq z_i\).
    \end{enumerate}
    In this case the subspace \(W' = W + \sp(y)\), has \(Q_{W'} = \set{z_j}_{j=1}^{2^i-1} \cup \set{z_j - y_j}_{j=1}^{2^i}\).
\end{claim}
Note that items \(1\) and \(3\) imply that the rank of \(y_j\) is half the rank of \(z_j\).

Given \pref{claim:abstract-unique-decomposition}, we can define the isomorphism \(\psi: \Y_W \to \mathcal{T}^{z_1} \otimes \mathcal{T}^{z_2} \otimes \dots \otimes \mathcal{T}^{z_{2^i-1}} \otimes \mathcal{H}_{z_W,2^{i-2}}\) by \begin{equation}\label{eq:iso-for-decomp-lemma}
\psi(y)=(y_1,y_2,\dots,y_{2^i}),
\end{equation}
where the image consists of the unique \(y_j\)'s promised by the claim. This function is also a bijection by \pref{claim:abstract-unique-decomposition}.

To conclude the proof, we need to argue that an edge \(\set{y,y'}\) is sampled by sampling the different components \(y_j,y_j'\) independently from one another. For this part, we basically need the following independence claim for the distribution we used when sampling \(s \in \mathcal{S}\). Let \(P\) be the random variable obtained by sampling a function from \(\mathcal{S}\). For a set \(R \subseteq \F_2^d\), let \(P|_R\) be the random variable obtained by restricting the function only to the vectors of \(R\). Let \(z_0 \in \F_2^n\). Then \(P|_{R}\) and \(P|_{\F_2^d \setminus R}\) are independent, if we condition on the event \(\sum_{v \in R}P|_R(v) = z_0\).

\begin{claim} \label{claim:independence-given-sums}
    Let \(\mathcal{S}\) be a rank-based poset. Let \(R_1,R_2,\dots,R_m \subseteq \F_2^d \setminus \set{0}\) be a partition.  Let \(z_1,z_2,\dots,z_{m-1} \in \F_2^n\) be such that \(\rank(z_i) = |R_i|\cdot r\) and \(z_1 \oplus z_2 \oplus \dots \oplus z_{m-1}\). 

    Consider the joint random variables \(\set{P_i = s|_{R_i}}_{i=1}^m\). Then conditioned on the event that for every \(i=1,2,\dots,m-1\)
    \[\sum_{v \in R_i} P_i(v)=z_i,\]
    these random variables are independent.
\end{claim}

We prove the lemma first, then go back to show the sub-claims.
\begin{proof}[Proof of \pref{lem:generalized-decomposition-lemma}]
    Let us denote by \(\mathcal{G} = \mathcal{T}^{z_1} \otimes \mathcal{T}^{z_2} \otimes \dots \otimes \mathcal{T}^{z_{2^i-1}} \otimes \mathcal{H}_{z_W,r2^{i-2}}\). We define a map \(\psi: Y^{\mathcal{S}}_W(1) \to \mathcal{G}\) mapping \(y \mapsto (y_1,y_2,\dots,y_{2^i})\) as in \eqref{eq:iso-for-decomp-lemma}. That is, the \(y_j\)'s are as in \pref{claim:abstract-unique-decomposition}. By \pref{claim:abstract-unique-decomposition}, this map is well defined for every \(y\) and sends link vertices into the vertices of \(\mathcal{G}\). Moreover, by \pref{claim:abstract-unique-decomposition} this is also a bijection, since the map \[\psi^{-1}(y_1,y_2,\dots,y_{2^i}) \coloneqq \bigoplus_{j=1}^{2^i} y_j\]
    is indeed onto \(\bar{Y}^{\mathcal{S}}_{W}(1)\) according to the claim.

    Next we prove that a pair \(\set{y,y'}\) is an edge in \(\Y^{\mathcal{S}}_W\) if and only if \(\set{\psi(y),\psi(y')}\) is an edge in \(\mathcal{G}\) (ignoring any edge weights for now). Recall that \(\set{y,y'}\) is an edge in \(\Y^{\mathcal{S}}_W\) if and only if \(y' \in Y^{\mathcal{S}_J}_{W'}(1)\) where \(W'=W+\sp(y)\).
    By \pref{claim:abstract-unique-decomposition}, \(Q_{W'} = \set{y_j}_{j=1}^{2^i} \cup \set{z_j - y_j}_{j=1}^{2^i-1}\). 
    Let us re-index \(t_j \in Q_{W'}\)'s with \[t_{j,b} = \begin{cases}
            y_j & b=0 \\ z_j - y_j & b=1
    \end{cases}.\]
    for \((j,b) \in \left ( [2^i - 1] \times \set{0,1} \right ) \cup \set{(2^i,0)}\). 
    Moreover, by applying \pref{claim:abstract-unique-decomposition} to \(W'\) we have that \(y' \in \Y^{\mathcal{S}}_{W'}(1)\) if and only if:
    \begin{enumerate}
        \item We can decompose \(y'=\bigoplus_{(j,b)\in [2^i] \times \set{0,1} }y'_{j,b}\) such that the rank of every summand is equal. 
        \item For every \((j,b) \ne (2^i,1)\), \(y'_{j,b} \leq t_{j,b}\).
        \item The final \(y'_{2^i,1}\) is a direct sum with \(z_{W'}\) (the sum of all elements in \(Q_{W'}\)).
    \end{enumerate}

    So on the one hand, if \(\psi(y),\psi(y')\) are neighbors in the tensor product \(\mathcal{G}\), we can take \(y_{j,0}' = \Meet(y_j',y_j)\) and \(y_{j,1}' = \Meet(y_j',z_j - y_j)\) which will satisfy the above property (and is well defined by the properties of \(\under^{z_j}\) and the direct-meet property). This shows that being an edge in \(\mathcal{G}\) implies being an edge in \(\Y^{\mathcal{S}}_W\). 
    
    For the reverse direction, suppose there exists such a decomposition of \(y'\). Then for every \(z_j\) the vectors \(w_1 = y_{j,0}', w_2=y_{j,1}', w_3 = t_{j,0}-y_{j,0}'\) and \(w_4 = t_{j,1}-y_{j,1}'\) have that \(w_1\oplus w_2 = y_j'\) and \(w_1\oplus w_3 = y_j\). In addition they have equal rank and their direct sum is equal to \(z_j\). Thus \(\set{y_j,y_j'}\) is an edge in \(\mathcal{T}^{z_j}\). For the final component, a similar argument applies: We decompose \(w_1 = y'_{2^i,0}\), \(w_2 = y_{2^i}-y'_{2^i,0}\) and \(w_3 = y'_{2^i,1}\). All three vectors have rank \(2^{i-2}\) and have a direct sum. Furthermore, \(y_{2^i} = w_1\oplus w_2\) and \(y_{2^i}' = w_1 \oplus w_3\). As a result \(\set{y_{2^i},y_{2^i}'}\) is an edge in \(\mathcal{H}_{z_W,r2^{i-2}}\). Thus \(\set{\psi(y),\psi(y')}\) is an edge in \(\mathcal{G}\). 

    Finally, we note that we can sample \(\set{y,y'}\) by sampling every pair \(\set{y_j,y_j'}\) independently. Indeed, when sampling an edge in the link of \(W\), we are conditioning on sampling some \(s \in \mathcal{S}\) so that for a fixed \((d-i)\)-dimensional subspace \(V_0\), the image of \(s_{V_0}:\nicefrac{\F_2^d}{V_0} \setminus \set{0} \to \F_2^n\) is fixed. For every \(j \leq 2^i - 1\), there exists some \(v+V_0\) such that \(\set{y_j,y_j'}\) depend only on the values of \(v' \in V_0\) (this is the \(v+V_0\) such that \(y_j,y_j' \subseteq s_{V_0}(v+V_0)\)). For the final component \(\set{y_{2^i},y_{2^i}'}\), this edge only depends on \(s\) restricted to \((0+V_0) \setminus \set{0}\). By \pref{claim:independence-given-sums}, the values of \(s\) on the respective parts in each coset \(v+V_0\) are independent once all values of \(s_{V_0}\) have been fixed. We conclude that indeed \(\psi\) is an isomorphism.
\end{proof}

We conclude this subsection with the proofs of \pref{claim:uniqueness-of-image-set-ranks}, \pref{claim:abstract-unique-decomposition} and \pref{claim:independence-given-sums}.
\begin{proof}[Proof of \pref{claim:uniqueness-of-image-set-ranks}]
    We will prove something even stronger. With the notation in the claim, we show that there exists some \(A \in GL_{i}(\F_2)\) such that \(s_2(x) = s_1(Ax)\)\footnote{Note that a priori \(s_1\) and \(s_2\) could have different domains, so technically \(A\) is an isomorphism between the domain of \(s_1\) and that of \(s_2\).}. Let \(s_1,s_2 \in \mathcal{S}_i\) be such that \(Im(\widehat{s_1})=Im(\widehat{s_2})=W\) and such that without loss of generality the domain of \(s_1\) and \(s_2\) is \(\F_2^i \setminus \set{0}\). Observe that by \pref{claim:extended-direct-meet-property}, \( \Meet(\sett{\widehat{s}(y)}{y \in \F_2^i, \iprod{x,y}=1}) = s(x) \), since the subset on the left consists of direct sums of a bunch of \(s(x')\), and \(s(x)\) is the only common component in these sums of \(\widehat{s}(y)\).

    Now assume that \(Im(\widehat{s_2}) = Im(\widehat{s_1})\). This implies that there is an isomorphism \(B\) such that \(\widehat{s_1}(x)=\widehat{s_2}(Bx)\). On the other hand, by the above
    \begin{align*}
        s_1(x) &= \Meet(\sett{\widehat{s_1}(y)}{y \in \F_2^i, \iprod{x,y}=1})\\
        &=\Meet(\sett{\widehat{s_2}(By)}{y \in \F_2^i, \iprod{x,y}=1})\\
        &=\Meet(\sett{\widehat{s_2}(By)}{By \in \F_2^i, \iprod{x,B^{-1} By}=1})\\
        &\overset{y'=By}{=}\Meet(\sett{\widehat{s_2}(y')}{y' \in \F_2^i, \iprod{x,B^{-1} y'}=1})\\
        &=\Meet(\sett{\widehat{s_2}(y')}{y' \in \F_2^i, \iprod{(B^{-1})^\top x,y'}=1})\\
        &= s_2((B^{-1})^\top x).
    \end{align*}
    Set \(A = (B^{-1})^\top\) and we got \(s_1(x)=s_2(Ax)\).
\end{proof}
We note that the above proof used the intersection patterns of the Hadamard code.

\begin{proof}[Proof of \pref{claim:abstract-unique-decomposition}]
    Let \(s \in \mathcal{S}\) and \(U' \subseteq \F_2^{d}\) be such that \(Im(\widehat{s}|_{U'}) = W'\) and let \(U \subseteq U'\) be such that \(Im(\widehat{s}|_{U}) = W\). Write \(V = V' \oplus \sp \set{b}\) such that \(U'=V'^{\bot}\) and \(U=V^\bot\). We recall that every \(v+V \in \nicefrac{\F_2^d}{V}\) can be partitioned into two cosets of \(V'\), \(v+V = v+V' \dunion (v+b)+V'\) and therefore \(s_V(v+V) = s_{V'}(v+V') + s_{V'}((v+b)+V')\). By the fact that \(s_{V'} \in \mathcal{S}_i\) this sum is also direct.
    
    We first prove the only if direction. Let \(x \in U' \setminus U\) be such that \(\widehat{s}(x)=y\). By \pref{claim:skeleton-of-induced-Grassmann}, it also holds that \(\widehat{s_{V'}}(x)=y\). By the fact that \(x \in U' \setminus U\), the inner product \(\iprod{x,b} = 1\) and therefore, \(\iprod{x,v+V'} \ne \iprod{x,(v+b)+V'}\). This implies that for every \(v_0+V \ne 0+V\), exactly one of \(y_j \in \set{s_{V'}(v_0+V'), s_{V'}((v_0+b)+V')}\) satisfies \(y_j \leq y\). Indeed:
    \begin{align*}
        y = \widehat{s_{V'}}(x) &= \sum_{v+V' \not\subseteq 0+V} \iprod{x,v+V'}s_{V'}(v+V') \\
    &= \sum_{v_0 \mid v_0+V\neq 0+V} \left( \iprod{x,v_0+V'}s_{V'}(v_0+V') + \iprod{x,(v_0+b)+V'}s_{V'}(v_0+b+V')\right).
    \end{align*}
Assume that \(\iprod{x,v_0+V'}=1\) and thus \(\iprod{x,(v_0+b)+V'} = 0\) (the other case is similar). In this case, because this whole sum is a direct sum every non-zero component in the sum is dominated by $y$. In particular \(s_{V'}(v_0+V') \leq y\). Moreover, any component that does not appear in this sum is a direct sum with \(y\), and in particular \(s_{V'}(v_0+b+V') \not \leq y\).

Thus we re-index \(z_j = s_V(v+V)\) (arbitrarily) and set \(y_j\) to be the component where \(y_j \leq y\) and \(y_j \leq s_V(v+V) = z_j\). We also set \(y_{2^i} = s_{V'}(b)\) and note that \(y_{2^i}\) also participates in the linear combination of \(\widehat{s_{V'}}(x)\) because \(\iprod{x,b}=1\). All \(y_j\)'s have rank \(2^{d-i-1}\) and by the discussion above, we have that \(y = \sum_{j=1}^{2^i} y_j\). Moreover, \(y_j \leq z_j\) for every \(j \leq 2^i-1\) and \(y_{2^i}\) is a direct sum with the sum of \(z_j\)'s. Thus, \(Q_{W'} = \sett{s_{V'}(v+V')}{v+V' \ne V'} = \set{y_j}_{j=1}^{2^i} \cup \set{z_j-y_j}_{j=1}^{2^i - 1}\) by construction. Finally, we note that the when the admissible functions come from a rank-based construction that has the direct-meet property, then \(y_j = \Meet(y,s_{V}(v_0+V))\) for some \(v_0 + V\). Thus, the \(y_j\)'s are unique.

\medskip

Conversely, let us assume that the three items hold and prove that \(W' \in \Y(i+1)\). Let \(s \in \mathcal{S}_i\) be such that \(Im(\widehat{s})=W\). Recall that \(\mathcal{S}_{i+1}\) consists of all functions \(s'\) whose image contains vectors of rank \(2^{d-i}\) whose sum is a direct sum. Therefore, let us construct such a function \(s':\F_2^{i+1} \to \F_2^n\) satisfying \(Im(\widehat{s})=W'\). First, we re-index the elements \(y_j\) as \(y_v\)'s for every \(v \in \F_2^{i} \setminus \set{0}\) (but keep the notation \(y_{2^i}\) as before) so that \(y_v \leq s(v)\). For an input \((b,v)\in (\F_2 \times \F_2^{i}) \setminus \set{0}\),
    \[s'(b,v) = \begin{cases} 
    s(v)-y_v & b=0,v\ne 0 \\
    y_v & b=1, v \ne 0 \\
    y_{2^i} & b=1,v=0 \end{cases}.\]
It is easy to see that the image of \(s'\) is \(\set{y_j}_{j=1}^{2^i} \cup \set{z_j - y_j}_{j=1}^{2^i - 1}\), and these vectors form a direct sum and they all have rank \(2^{d-i-1}\). Thus we successfully construct a \(s' \in \mathcal{S}_{i+1}\). Moreover, one readily verifies that \(\widehat{s'}(1,\vec{0}) = \sum_{v \in \F_2^i} s'(1,v) = \sum_{j=1}^{2^i} y_j = y\) and that for \(U'' = \sett{(0,v)}{v \in \F_2^i}\) we have that \(Im(\widehat{s'}|_{U''}) = W\). Thus \(Im(\widehat{s'}) = W'\).
\end{proof}

\begin{proof}[Proof of \pref{claim:independence-given-sums}]
    It suffices to prove that \(P_i\) is independent of \(P_{[m] \setminus \set{i}}\) where \(P_{[m] \setminus \set{i}}\) denotes the rest of the random variables. For any \(i =1,2,\dots,m-1 \), the distribution for $P_i$ is uniform over all partial functions \(s_0:R_i \to \F_2^d\) satisfying (1) $\forall v \in R_i$ \(\rank(s_0(v))=r\), and (2) \(\bigoplus_{v \in R_i}s_0(v_i) = z_i\). This follows from the fact that for \emph{any fixed conditioning} of the form \(P_j = s|_{(\F_2^d \setminus \set{0}) \setminus R_i}\) where the image of $P_j$ sum up to \(z_j\) for all \(j \ne i\), all such functions \(s_0\) have equal probability density.

    For \(i=m\), a similar argument applies. Conditioned on \(P_1 = s|_{R_1},P_2 =s|_{R_2},\dots,P_{m-1} = s|_{R_{m-1}}\) (where the sum of each $s|_{R_i}$'s values equals \(z_i\)), the last partial function \(s|_{R_m}\) is sampled by picking a uniformly random (ordered) set of rank \(r\) elements \(t_1,t_2,\dots,t_{|R_m|}\) satisfying that \(\left (\bigoplus_{j=1}^{m-1} z_j \right ) \oplus \left ( \bigoplus_{j=1}^{|R_m|} t_j \right )\). It is obvious that the distribution of $P_m$ is independent of the specific functions assigned to \(P_1,P_2,\dots,P_{m-1}\).
\end{proof}

\subsection{The distribution over admissible functions} \label{sec:distribution-over-admissibles}
We can extend the decomposition lemma above to a slightly more general setup. In this setup we sample \(s \in \mathcal{S}\) according to a specified distribution instead of sampling it uniformly. Let \(\mathcal{S}\) be a set of admissible functions. Let \(\mu\) be a probability density function of \(\mathcal{S}\). We say that \(\mu\) is permutation invariant if for every \(s \in \mathcal{S}\) and every permutation \(\pi\) of \(\F_2^d \setminus \set{0}\),  \(\mu(s)=\mu(s \circ \pi)\). Such a $\mu$ induces the following measure over the $d$-dimensional subspaces \(\Y^{\mathcal{S}}(d)\) in the rank-based induced poset: first sample \(s \sim \mu\), and then outputs \(Im(\widehat{s})\). We call posets obtained by this process \emph{measured rank-based posets}.

The sparsified matrix poset is an example of a measured rank-based posets, where the distribution \(\mu\) of admissible functions is given by the following: \begin{enumerate}
    \item Sample \(W_1, W_2 \in \Y_0(2^d -1)\) independently.
    \item Sample \(s\) uniformly at random such that \(\row(\bigoplus_{v \in \F_2^d \setminus \set{0}} s(v)) = W_1\) and \(\col(\bigoplus_{v \in \F_2^d \setminus \set{0}} s(v)) = W_2\).
\end{enumerate} 
For general $Y_0$, this distribution $\mu$ is nonuniform, and the resulting poset is a measured rank-based poset.
Hence this more generalized definition captures the sparsified matrix poset in \pref{def:matrix-poset}.

For such posets we can prove a version of the decomposition lemma using the following variants of the graphs from \pref{def:abstract-link-graph}.
\begin{definition}[\((z,\mu)\)-link graph]
    Let \(z \in \F_2^n\) be of rank \(2^d-2^{i}\) and let \(m \leq 2^{i-2}\). The graph \(\mathcal{H}^\mu_{z}\) has vertices
    \[V = \sett{y \in \F_2}{\rank(y)=2m, y \oplus z},\]
    \[E = \sett{\set{w_1 \oplus w_2, w_1 \oplus w_3} \in \F_2}{\rank(w_i)=m, z \oplus w_1 \oplus w_2 \oplus w_3}.\]
    The edge distribution is given by: first sample a random vector \(s \sim \mu\) conditioned on \(t:=\bigoplus_{v \in \F_2^d \setminus \set{0}}s(v) \geq z\). Then uniformly sample a \(z' \leq t\) of rank $4m$ and finally output an edge in \(\mathcal{T}^{z'}\) according to the distribution in that graph.
\end{definition}

\begin{lemma}
  Let \(\Y^{\mathcal{S}}\) be a measured rank-based poset with the direct-meet property. Then for every \(W \in \Y^{\mathcal{S}}(i)\) with \(Q_W = \set{z_1,z_2,\dots,z_{2^i-1}}\) and \(z_W= \sum_{j=1}^{2^i-1} z_j\) defined as before,
  \[\Y^{\mathcal{S}}_W \cong \mathcal{T}^{z_1} \otimes \mathcal{T}^{z_2} \otimes \dots \otimes \mathcal{T}^{z_{2^i-1}} \otimes \mathcal{H}^{\mu}_{z_W,r2^{i-2}}.\]
\end{lemma}
The proof of this lemma follows the exact same argument in the proof of the decomposition lemma \pref{lem:generalized-decomposition-lemma}. Therefore we omit the proof for brevity.
\addcontentsline{toc}{section}{Bibliography}
\printbibliography
\appendix
\section{The decomposition lemma} \label{app:loc-to-glob}
\begin{proof}[Proof of \pref{lem:local-to-global}]
    Let us denote by \(A_G,A_t\) the adjacency operators of \(G,G_t\) respectively. We also denote by \(A_\B\) the \emph{bipartite} adjacency operator of \(B_\tau\). That is, for a function \(f:V \to \RR\), \(A_\B f:\T \to \RR\) is the function \(A_\B f(t) = \Ex[v \sim \mu_t]{f(v)}\). We note that by assumption, for every \(f:V \to \RR\) such that \(\Ex[v]{f(v)} = 0\), \(\snorm{A_\B f} \leq \lambda_2^2 \snorm{f}\).
    
    Let \(f:V \to \RR\) be an eigenvector of \(A_G\) of norm \(1\) and denote its eigenvalue by \(\mu\). 
    Then
    \begin{align*}
      \abs{\mu} &= \abs{\iprod{f,A_G f}} = \Ex[\set{u,v} \sim \mu_G]{f(u)f(v)} \\
      &= \Abs{\Ex[t \in \mu_\T]{\Ex[\set{u,v} \sim \mu_t]{f(u)f(v)}}} \\
      &= \Abs{\Ex[t \in \mu_\T]{\iprod{f_t,A_t f_t}}}
    \end{align*}
    where \(f_t\) is the restriction of \(f\) to the vertices of \(V_t\). Let us denote \(f_t = f_t^0\cdot \one + f_t^\perp\) such that \(f_t^0\cdot \one\) is a constant and \(f_t^\perp \perp f_t^0 \cdot \one\). Then by expansion,
      \begin{align*}
      &\leq \Abs{\Ex[t \in \mu_\T]{(f_t^0)^2 + \gamma \iprod{f_t^\perp,f_t^\perp}}} \\
      &= \Abs{\Ex[t \in \mu_\T]{(1-\gamma)(f_t^0)^2 + \gamma \iprod{f_t,f_t}}} \\
      & = (1-\gamma)\Abs{\Ex[t \in \mu_\T]{(f_t^0)^2}} + \gamma \Abs{\Ex[v \sim V]{f(v)^2}} \\
      &= (1-\gamma)\Abs{\Ex[t \in \mu_\T]{(f_t^0)^2}} + \gamma.
    \end{align*}
    Here first equality is algebra, the second is because \(\mu_G(v)f(v)^2 = \sum_{v \in t}\mu_\T(t) \mu_t(v)f(v)^2\) and the last equality is due to the fact that \(f\) has norm \(1\). Finally, we note that the constant \(f_t^0 = \Ex[v \sim \mu_t]{f(v)} = A_\B f(t)\). That is
    \[
         = (1-\gamma)\snorm{A_\B f} + \gamma \overset{(\snorm{A_\B f} \leq \lambda_2^2)}{\leq} \gamma + (1-\gamma)\lambda_2^2.
    \]
\end{proof}


\section{Intermediate subspaces of the expanding matrix Grassmann poset} \label{app:admissible-functions-structure-proof}
In this appendix we prove that the intermediate subspaces in the Grassmann matrix poset come from level \(i\)-admissible functions.
\restateclaim{claim:poset-i-space-description}

For this proof we recall the preliminaries in \pref{sec:induced-Grassmann} on quotient spaces and bilinear forms.
\begin{proof}[Proof of \pref{claim:poset-i-space-description}]
    Fix \(\dimlin\) and \(\dimloutfull = \log \dimlin\). It is enough to prove that \(W \in \sett{Im(\widehat{s})}{s \in \mathcal{S}_i}\) if and only if there exists \(W' \in \Y'(\dimloutfull)\) such that \(W' \supseteq W\) (i.e.\ if and only if \(W \in \Y'(i)\)). We start with some \(W \in \Y'(i)\) and prove that \(W \in \sett{Im(\widehat{s})}{s \in \mathcal{S}_i}\). Let \(W' \in \Y'(\dimloutfull)\) be such that \(w' \supseteq w\). Let \(s \in \mathcal{S}_d\) be such that \(Im(\widehat{s}) = W'\). The matrices in the image of \(s\) are linearly independent \footnote{Here we say matrices are linearly independent if they are linear indpendent as vectors in \(\F_2^{n^2}\).} since their sum has maximal rank. Hence \(\widehat{s}:\F_2^\dimloutfull \to W'\) is an isomorphism. From this there exists some \(i\)-dimensional subspace \(U \subseteq \F_2^{\dimloutfull}\) so that  \(W = Im(\widehat{s}|_U)\). We write \(U=V^\bot\) for some \((d-i)\)-dimensional subspace \(V\) (with respect to the standard bilinear form). Let us define \(s': (\nicefrac{\F_2^d}{V}) \setminus \set{[0]} \to \F_2^{n \times n}\) as
    \[s'(v+V) = \sum_{v' \in v+V} s(v').\]
    It is direct to check that \(s'\) is admissible, since every \(\sum_{v' \in v+V} s(v')\) is a disjoint sum of \(2^{\dimloutfull-i}\) rank one matrices, and the sum \(M_{s'} = \sum_{v+V \ne 0+V} s'(v+V)\) is also equal to
    \[M_{s'} = \sum_{v \notin V}s(v)\]
    which has rank \(2^d-2^{d-i}\) and whose row and column spaces are in \(\Y_1\) and \(\Y_2\) respectively, since it is contained in the row and column space of \(M_s = \sum_{v \ne 0}s(v)\). We can also define \(\widehat{s'}:U \to \F_2^{n \times n}\) as \[\widehat{s'}(x) = \sum_{v+V \in \nicefrac{\F_2^\dimloutfull}{V} \setminus \set{[0]}} \iprod{x,v+V} s'(v+V).\]
    Here the inner product is the derived inner product as in \pref{sec:vec-spaces-defs}. We note that 
    \(\widehat{s'}(x) = \widehat{s}(x)\) because
    \begin{align*}
        \widehat{s'}(x) &= \sum_{v+V \in (\nicefrac{\F_2^d}{V}) \setminus \set{[0]}} \iprod{x,v+V} s'(v) \\
        &= \sum_{v+V \in (\nicefrac{\F_2^d}{V}) \setminus \set{[0]}} \iprod{x,v+V} \sum_{v' \in v+V} s(v) \\
        &\overset{(*)}{=} \sum_{v+V \in (\nicefrac{\F_2^d}{V}) \setminus \set{[0]}} \sum_{v' \in v+V}\iprod{x,v'} s(v') \\
        &= \sum_{v \in \F_2^{\dimloutfull} \setminus \set{0}} \iprod{x,v'} s(v') = \widehat{s}(x)
    \end{align*}
    where \((*)\) is because the derived inner product \(\iprod{x,v+V}\) is equal to (the standard) inner product \(\iprod{x,v'}\) for every \(v' \in v+V\). Even though the function \(s'\) is defined over \(U\) which is not equal to \(\F_2^i\), and \(\widehat{s}\) is defined through the derived bilinear form (not the standard one), by \pref{claim:always-work-with-standard} we can find an admissible \(s''\) so that \(Im(\widehat{s''})=Im(\widehat{s'})\). This proves that \(w \in \sett{Im(\widehat{s})}{s \in \mathcal{S}_i}\).

    As for the converse direction, `reverse engineer' the proof above. Let \(W \in \sett{Im(\widehat{s})}{s \in \mathcal{S}_i}\), and we show that \(W \in \Y'(i)\). Let \(s \in \mathcal{S}_i\) be such that \(Im(\widehat{s}) = W\). Let \(M_s = \sum_{v \in \F_2^i \setminus \set{0}}s(v)\) and let \(W_1 \in (\Y_1)_{\row(M_s)}(\dimlin - 1), W_2 \in (\Y_2)_{\col(M_s)}(\dimlin - 1)\). These spaces exist from the purity of \(\Y_1\) and \(\Y_2\). We define \(s' \in \mathcal{S}_{\dimloutfull}\) as follows. For every \(v \in \F_2^{i} \setminus \set{0}\) we (arbitrarily) find a decomposition of \(s(v)\) into rank one matrices, and index them by \(x \in \F_2^{\dimloutfull-i}\). That is, \(s(v) = \sum_{x \in \F_2^{\dimloutfull-i}} e_x \otimes f_x\). We set \(s'(v,x) = e_x \otimes f_x\). For \(v=0 \in \F_2^i\) we define \(s'\) using \(W_1\) and \(W_2\). We find space \(W'_1\) and \(W'_2\) such that \(\row(S) \oplus W'_j = W_j\) and ordered bases \(e_1,e_2,\dots,e_{2^{\dimloutfull-i}-1}\) to \(W_1\) and \(f_1,f_2,\dots,f_{2^{\dimloutfull-i}-1}\) to \(W_2\). We (arbitrarily) re-index these bases as \(\sett{e_x}{x \in \F_2^{\dimloutfull-i} \setminus 0}\) and \(\sett{f_x}{x \in \F_2^{\dimloutfull-i} \setminus 0}\) and set \(s'(0,x) = e_x \otimes f_x\). It is a direct definition chase to verify that \(s'\) is admissible, i.e.\ that \(M_{s'} = \sum_{(v,x) \in \F_2^{\dimloutfull} \setminus \set{0}} s'(v,x)\) has rank \(2^\dimloutfull -1\) and its row and column spaces are equal to \(W_1\) and \(W_2\) respectively. Finally, the image of \(\widehat{s'}\) restricted to \(U = \sett{(x,0)}{x \in \F_2^i}\) is \(w\) since
    \begin{align*}
        \widehat{s'}(x,0) &= \sum_{(v,v') \in \F_2^i \times \F_2^{\dimloutfull - i} \setminus \set{(0,0)}} \iprod{(x,0),(v,v')} s'(v,v')\\
        &= \sum_{v \in \F_2^i \setminus \set{0}} \iprod{x,v} \sum_{v' \in \F_2^{\dimloutfull - i}}s'(v,v') \\
        &= 
        \sum_{v \in \F_2^i \setminus \set{0}} \iprod{x,v} s(v) = \widehat{s}(x).
    \end{align*}
    Thus \(W' = Im(\widehat{s'}) \in \Y'(\dimloutfull)\) and it contains \(W=Im(\widehat{s}|_U)\), so \(W \in \Y'(i)\).
\end{proof}

\section{Elementary properties of matrix poset} \label{app:removing-lower-matrix}
\subsection{Domination is a partial order}
Let us begin by showing that this relation is a partial order. We do so using this claim.

\begin{proof}[Proof of \pref{claim:trivially-intersecting spaces}]
    The implications of \((4) \Rightarrow (2)\) and \((3) \Rightarrow (2)\) are a direct consequence of the definitions. Let us see that \((1) \Rightarrow (2)\).
    \begin{align*}
        \rank(A+B) &= \dim(\row(A+B)) \\
        &\leq \dim(\row(A) + \row(B))\\ 
        &= \dim(\row(A)) + \dim(\row(B)) - \dim(\row(A) \cap \row(B)) \\
        &= \rank(A) + \rank(B) - \dim(\row(A) \cap \row(B)).
    \end{align*}
    This shows that if \(\rank(A+B) = \rank(A) + \rank(B)\) then \(\row(A) \cap \row(B) = \set{0}\) (and similarly for the column space).

    Finally, let us see that \((2) \Rightarrow (1),(3),(4)\). Assume that \(\row(A) \cap \row(B), \col(A) \cap \col(B)\) are trivial. Let us write \(r_1 = \rank(A), r_2 = \rank(B)\). Then we can write \(A = \sum_{i=1}^{r_1} e_i \otimes f_i\) and \(B = \sum_{i=1}^{r_2} g_i \otimes h_i\) and because the row and column spaces intersect trivially, we have that \(\set{e_i} \dunion \set{g_i}\) and \(\set{f_i} \dunion \set{h_i}\) are \(r_1 + r_2\) sized sets that are independent, thus proving \((4)\) (which implies \((3)\)). This also shows that \(\rank(A+B) = \dim(\col(A+B)) = \rank(A)+\rank(B)\), thus proving \((1)\).
    
\end{proof}

\begin{proof}[Proof of \pref{claim:three-term-direct-sum}]
    
    By \pref{claim:trivially-intersecting spaces}, \(X \oplus Y\) if and only if \(\col(X) \cap \col(Y)\) and \(\row(X) \cap \row(Y)\) are \(\set{0}\). So first of all note that if \(C \oplus (B \oplus A)\) then \(\row(C) \cap \row(B \oplus A) = \set{0}\) (and the same for the column space). By \pref{claim:trivially-intersecting spaces} \(\row(B \oplus A) = \row(B) \oplus \row(A)\), so in particular \(\row(C) \cap \row(B) = \set{0}\) (resp. \(\col(C) \cap \col(B) = \set{0}\)). Thus, \(C \oplus B\).
    
    Now let us show that \(A \oplus (B \oplus C)\). By \pref{claim:basic-three-spaces}, if \(\row(A) \oplus \row(B)\) and \(\row(C) \oplus (\row(A) \oplus \row(B))\) then \(\row(A) \oplus (\row(B) \oplus \row(C))\) (resp. column spaces). Thus by \pref{claim:trivially-intersecting spaces} \(A \oplus (B \oplus C)\).
\end{proof}

\begin{proof}[Proof of \pref{claim:domination-is-partial-order}]
We need to verify reflexivity, anti-symmetry and transitivity. Let \(A,B,C\) be matrices.
    \begin{enumerate}
        \item Reflexivity is clear.
        \item Anti-symmetry: assume that \(A \leq B\) and \(B\leq A\) this implies that \(\rank(B-A) = \rank(A)-\rank(B) = \rank(B)-\rank(A)=0\) so \(A=B\).
        \item Transitivity: \(C \leq B\) and \(B \leq A\). Then by \pref{claim:trivially-intersecting spaces}, \(\row(A) = \row(B) \oplus \row(A-B)\). In addition this claim shows that \(\row(B-C) \subseteq \row(B)\). Together this implies that \(\row(B-C)\) intersects trivially with \(\row(A-B)\). The same holds also for the column spaces. Thus by \pref{claim:trivially-intersecting spaces} this implies that \(\rank(A-C) = \rank(A-B) + \rank(B-C)\), i.e.\ \(A-C \geq A-B\). Thus,
        \begin{align*}
          \rank(A) &\overset{(B\leq A)}{=} \rank(B) + \rank(A-B)  \\
          &\overset{(C\leq B)}{=} (\rank(C) + \rank(C-B)) + \rank(A-B) \\
          &\overset{(A-B\leq A-C)}{=} (\rank(C) + \rank(C-B)) + (\rank(A-C)-\rank(B-C)) \\
          &= \rank(C) + \rank(A-C).
        \end{align*}
    \end{enumerate}
\end{proof}

\begin{remark}
Recall the following definition from partially ordered set theory.
\begin{definition}[Rank function]
    Let \((A,\leq)\) be a poset. A \emph{rank function} \(\rho:\MP \to \NN\) is a function with the following properties.
\begin{enumerate}
    \item If \(A > B\) then \(\rho(A) > \rho(B)\).
    \item For every \(A > B\) there exists \(A \geq C > B\) so that \(\rho(C) = \rho(B) + 1\).
\end{enumerate}
A poset is called a \emph{graded poset} if there exists a rank function for that poset.
\end{definition}
The order of matrices is also a graded poset, where the matrix \(\rank\) is the rank function. To see this, we note that if \(A \leq B\), then \(\rank(A) = \rank(B) + \rank(A-B)\) so in particular \(\rank(A) \geq \rank(B)\) and if \(A \ne B\), \(\rank(A) > \rank(B)\). Moreover, if \(A > B\) then by the fourth item of \pref{claim:trivially-intersecting spaces}, we can write \(A = B \oplus \left ( \sum_{j=1}^t e_j \otimes f_j \right )\) for some \(t > 0\). Taking \(C = B \oplus (e_j \otimes f_j)\) concludes that the \(\rank\) function (of matrices) is a rank function (of the graded poset).
\end{remark}

\subsection{Isomorphism of intervals}

We begin with this claim.
\begin{claim} \label{claim:elementary}
    \begin{enumerate}
        \item Let \(A \leq B,C\). Then \(B \leq C\) if and only if \(B-A \leq C-A\).
        \item Let \(B,C\) be such that \(B \oplus A\) and \(C \oplus A\). Then \(B \leq C\) if and only if \(B \oplus A \leq C \oplus A\).
    \end{enumerate}
\end{claim}

\begin{proof}
    For the first item,
    \begin{align*}
        B &\leq C \\
        &\Leftrightarrow rank(C) = rank(B) + rank(C-B) \\
        &\Leftrightarrow rank(C) - rank(A) = rank(B) - rank (A) + rank(C- B) \\
        &\overset{(A \leq B,C)}{\Leftrightarrow} rank (C-A) = rank(B-A) + rank((C-A) - (B-A)) \\
        &\Leftrightarrow B-A \leq C-A.
    \end{align*}

    For the second item,
    \begin{align*}
        B &\leq C \\
        &\Leftrightarrow rank(C) = rank(B) + rank(C-B) \\
        &\Leftrightarrow rank(C) + rank(A) = rank(B) + rank (A) + rank(C- B) \\
        &\Leftrightarrow rank (C \oplus A) = rank(B \oplus A) + rank((C\oplus A) - (B \oplus A)) \\
        &\Leftrightarrow B \oplus A \leq C \oplus A.
    \end{align*}
\end{proof}

\begin{proof}[Proof of \pref{prop:iso-of-posets}]
    By \pref{claim:elementary} the bijection \(\psi\) maps \(A\) such that \(A \leq M_2\) to \(A'=A- M_1\) such that \(A' \leq M_2 - M_1\). Therefore \(\psi\) is well defined as a function from \(I(M_1,M_2)\) to \(I(0,M_2-M_1)\). Obviously \(\psi\) is injective (since it is injective over the set of all matrices). Moreover, if \(M_1 \leq A_1 \leq A_2 \leq M_2\) then by \pref{claim:elementary} the order is maintained, i.e. \(A_1' = A_1 -M_1 \leq A_2 - M_1 = A_2'\). The fact that \(rank(A') = rank(A) - rank(M_1)\) follows from \(M_1 \leq A\). Thus to conclude we need to show that \(\psi\) is surjective, or in other words, if \(A' \leq M_2 - M_1\) then \(A' + M_1 \leq M_2\). Let \(A' \leq M_2 - M_1\). Then by \pref{claim:trivially-intersecting spaces}
    \(\row(A') + \row(M_2 -M_1 - A) = \row(M_2 - M_1)\) and in particular \(\row(A') \subseteq \row(M_2 - M_1)\) and similarly \(\col(A') \subseteq \col(M_2-M_1)\). By \(M_1 \leq M_2\), we have that \(\rank(M_2) = \rank(M_1) + \rank(M_2 - M_1)\) and by \pref{claim:trivially-intersecting spaces} this implies that \(\row(M_2-M_1) \cap \row(M_1) = \set{0}\) and similarly for the columns spaces. We conclude that \(\row(A') \cap \row(M_1)\) and \(\col(A') \cap \col(M_1)\) are trivial. By using \pref{claim:trivially-intersecting spaces} we have that for \(A = \psi^{-1}(A') = A' + M_1\),
    \[\rank(A) = \rank(A') + \rank(M_1)\]
    or equivalently that \(A \geq M_1\). Finally, we also have that \(A \leq M_2\) by \pref{claim:elementary} (as \(A,M_2 \geq M_1\) and \(A'=A-M_1 \leq M_2 - M_1\)). Thus \(\psi\) is surjective since \(A\) is a preimage of \(A'\) inside \(\sett{A}{M_1 \leq A \leq M_2}\).
\end{proof}

\begin{proof}[Proof of \pref{claim:iso-second-type-of-lower-loc}]
    Fix \(A\) and note that \(B \oplus A\) if and only if \(A \leq A+B\). Thus \(\phi\) is a bijection. The order preservation follows from the second item of \pref{claim:elementary}.
\end{proof}

\subsection{The matrices have the direct-meet property} \label{app:direct-meet-matrices}
In this subsection we prove \pref{claim:direct-meet-matrix-poset}.
\begin{proof}[Proof of \pref{claim:direct-meet-matrix-poset}]
    The technical statement we shall use in the proof is the following: For any matrices \(A,B \leq C\) such that \(\row(A) \subseteq \row(B)\) and \(\col(A) \subseteq \col(B)\), it holds that \(A \leq B\).

    Let us first show that this statement implies the direct meet property. Let \(M_1,M_2,M_3\) be three matrices whose sum is direct. We observe that \(\row(M_1) = \row(M_1 \oplus M_2) \cap \row(M_1 \oplus M_3)\)\footnote{It is clear that \(\row(M_1) \subseteq  \row(M_1 \oplus M_2) \cap \row(M_1 \oplus M_3)\). The other direction follows from a dimension count.} and similarly \(\col(M_1) = \col(M_1 \oplus M_2) \cap \col(M_1 \oplus M_3)\). Consider any $A$ such that \(A \leq M_1 \oplus M_2\) and \(A \leq M_1 \oplus M_3\). Observe that \(\row(A) \leq \row(M_1 \oplus M_2) \cap \row(M_1 \oplus M_3)\). Thus \(\row(A) \subseteq \row(M_1)\). The same argument applies to the respective column spaces, and so by the technical statement above \(A \leq M_1\). This implies that the set \(G_{\MP}\) has the direct meet property.

    Now we prove the statement itself. Observe that \(\row(B-A) \subseteq \row(B)\) and \(\col(B-A) \subseteq \col(B)\). Moreover, \((C-B) \oplus B\) which means that \(\row(C-B)\) (respectively \(\col(C-B)\)) intersects trivially with \(\row(B)\) (respectively \(\col(B)\)). Thus \((C-B) \oplus (B-A)\). So we get
    \[ \rank(C-A) = \rank((C-B) + (B-A)) = \rank(C-B) + \rank(B-A).\]
    On the other hand, because \(A,B \leq C\), the equality above can be rewritten as
    \[\rank(C) - \rank(A) = \rank(C) - \rank(B) + \rank(B-A)\]
    or equivalently \(\rank(B) = \rank(A) + \rank(B-A)\) i.e.\ \(A \leq B\).
\end{proof}

\section{Properties of the graphs induced by the matrix poset} \label{app:proofs-of-expansion-matrix-graphs}
In this appendix we prove \pref{lem:complete-part-analysis} and \pref{claim:expansion-no-maximal-matrix-graph} necessary to prove the expansion of the complexes constructed in \pref{sec:construction-subpoly-HDX}. We also prove \pref{claim:short-paths-in-under-graph} necessary for the coboundary expansion proven in \pref{sec:cob-exp-johnson}. To do so, we need to take a technical detour and provide some further definitions that generalize the graphs in said statements.
In \pref{sec:construction-subpoly-HDX} we defined some graphs that naturally arise from the relations \(\leq\) and \(\oplus\). To analyze them, we need to generalize these definitions.

\begin{definition}[Relation graph]
    Let \(i < j \leq n\). The relation graph \(\mathcal{R}(i,j) = (L,R,E)\) is the bipartite graph whose vertices are \(L=\MP(i), R=\MP(j)\) and \(\set{A,B} \in E\) if \(A \leq B\).
\end{definition}
This is the analogous graph to the `down-up' graph in simplicial complexes, i.e.\ the graph whose vertices are \(X(i)\) and \(X(j)\) (for some simplicial complex \(X\)) and whose edges are all pairs \(\set{s,t}\) such that \(s \subseteq t\) (see e.g.\ \cite{DinurK2017}).


Another important graph is the following `sum graph'.
\begin{definition}[Sum graph]
    Let \(i,j,k\) be integers such that \(k \leq i,j\) and \(i+j-k \leq n\). The Sum graph \(\mathcal{S}(i,j,k) = (L,R,E)\) is the bipartite graph whose vertices are \(L=\MP(i), R=\MP(j)\) and \(\set{A,B} \in E\) if there exists \(C \in \MP(k)\) such that \(C \oplus (A-C) \oplus (B-C)\) are a direct sum.
\end{definition}
When \(k=0\) we denote this graph by \(\mathcal{DS}(i,j) \coloneqq \mathcal{S}(i,j,0)\). In this case its definition simplifies: \(A \sim B\) if and only if \(A \oplus B\).

All the aforementioned graphs turn up when one analyzes the Grassmann graph defined in \pref{sec:construction-subpoly-HDX}. However, the bulk of the analysis in \pref{sec:construction-subpoly-HDX} is on \emph{links} of the Grassmann complex. There the graphs that appear are actually subgraphs of these graphs. Hence, we also make the following definitions of the graphs where we take \(\MP^{U}\) or \(\MP_{\cap D}\) as our underlying posets instead of all \(\MP\).

\begin{definition}[Upper bounded relation graph]
    Let \(i < j \leq \ell \leq n\). Let \(U \in \MP(\ell)\). The graph \(\mathcal{R}^{U}(i,j) = (L,R,E)\) is the bipartite graph whose vertices are \(L=\MP^{U}(i), R=\MP^{U}(j)\) and \(\set{A,B} \in E\) if \(A \leq B\).
\end{definition}
Note that for every \(U_1,U_2 \in \MP(\ell)\), the graphs \(\mathcal{R}^{U_1}(i,j),\mathcal{R}^{U_2}(i,j)\) are isomorphic. So when we do not care about the matrix we write \(\mathcal{R}^{\ell}(i,j)\) instead of \(\mathcal{R}^{U}(i,j)\). 

The graphs \(\mathcal{DS}^{U}(i,j)\) and \(\mathcal{S}^{U}(i,j,k)\) are defined analogously as the subgraphs of matrices in \(\MP^{U}\):
\begin{definition}[Upper bounded sum graph] \label{def:upper-bounded-sum-graph}
    Let \(i,j,k,\ell\) be integers such that \(k \leq i,j\) and \(i+j-k \leq \ell\). Let \(U \in \MP(\ell)\). The Sum graph \(\mathcal{S}^U(i,j,k) = (L,R,E)\) is the bipartite graph whose vertices are \(L=\MP^U(i), R=\MP^U(j)\) and \(\set{A,B} \in E\) if there exists \(C \in \MP(k)\) such that \(C \oplus (A-C) \oplus (B-C) \leq U\).
\end{definition}
In these cases we also write \(\mathcal{DS}^\ell(i,j), \mathcal{S}^\ell(i,j,k)\) when we do not care about \(U\).

The second type of localization is with respect to subspaces. 

\begin{definition}[Lower bounded local relation graph]
    Let \(i,j,\ell\) be integers such that \(i < j \leq n-\ell\). Let \(D \in \MP(\ell)\). The graph \(\mathcal{R}_{\cap D}(i,j)\) is the subgraph of \(\mathcal{R}(i,j)\) whose vertices are in \(\MP_{\cap D}\).
\end{definition}

\begin{definition}[Lower bounded sum graph]\label{def:lower-bounded-sum-graph}
    Let \(i,j,k,\ell\) be integers such that \(k \leq i,j\) and \(i+j-k \leq n-\ell\). Let \(D \in \MP(\ell)\). The Sum graph \(\mathcal{S}_{\cap D}(i,j,k) = (L,R,E)\) is the bipartite graph whose vertices are \(L=\MP_{\cap D}(i), R=\MP_{\cap D}(j)\) and \(\set{A,B} \in E\) if there exists \(C \in \MP(k)\) such that \(D \oplus C \oplus (A-C) \oplus (B-C)\) (is a direct sum).
\end{definition}

In the next section we prove upper bounds on the expansion of some of these graphs. In all the following claims \(\gamma \in (0,1)\) stands for universal constants which we don't calculate (and may change between claims). The main lemmas we need are these. 
\begin{lemma} \label{lem:expansion-of-S-upper-loc}~
    \(\lambda_2(\mathcal{S}^{\ell}(i,j,k)) \leq \gamma^{-(\ell-j-\max(k,i-k))}\).
\end{lemma}

\begin{lemma} \label{lem:expansion-of-S-lower-loc}
    Let \(i,j,k,\ell\) be such that \(i \leq j\) and \(2i \leq n-\ell\). Let \(D \in \MP(\ell)\). Then \(\lambda_2(\mathcal{S}_{\cap D}(i,j,k)) \leq \gamma^{-(n-\ell-j-\max(k,i-k))}\).
\end{lemma}

In order to prove these lemmas we need to analyze the other graphs as well. More specifically, we will use \pref{lem:local-to-global} to reduce the analysis of the \(\mathcal{S}\) graphs to the expansion of the \(\mathcal{R}\) and \(\mathcal{DS}\) graphs. The bounds we need for \pref{lem:expansion-of-S-upper-loc} are these.
\begin{claim} \label{claim:upper-local-ds-bound}
    Let \(i,j \leq \ell\). Then
    \(\lambda_2(\mathcal{DS}^{\ell}(i,j)) \leq i j \cdot \gamma^{-(\ell-i-j)}\).
\end{claim}

\begin{claim} \label{claim:upper-local-R-bound}
    Let \(U \in \MP(\ell)\). The graph \(\mathcal{R}^{U}(i,j)\) is isomorphic to \(\mathcal{DS}^{\ell}(i,\ell-j)\). Consequently by \pref{claim:upper-local-ds-bound}
    \(\lambda_2(\mathcal{R}^{\ell}(i,j)) \leq \gamma^{-(j-i)}\).
\end{claim}

For \pref{lem:expansion-of-S-lower-loc} we will also need this claim.
\begin{claim} \label{claim:lower-R-bound}
    Let \(i,j,\ell\) be such that \(i \leq \frac{j}{2}\) and \(n-\ell \geq j\). Let \(D \in \MP(\ell)\). Then \(\lambda_2(\mathcal{R}_{\cap D}(i,j)) \leq \gamma^{j}\).
\end{claim}

The proofs of all these claims are all quite technical, and while most of the proof ideas repeat themselves among claims, we did not find one proof from which to deduce all claims. More details are given in the next subsections. We conclude this sub section with the short proofs of \pref{lem:complete-part-analysis} and \pref{claim:expansion-no-maximal-matrix-graph}, that were necessary to prove the main theorem in \pref{sec:construction-subpoly-HDX}. They follow directly from \pref{lem:expansion-of-S-upper-loc} and \pref{lem:expansion-of-S-lower-loc}.

\begin{proof}[Proof of \pref{lem:complete-part-analysis}]
    The double cover of \(\under^m\) is \(\mathcal{S}^{4m}(2m,2m,m)\). Thus by \pref{lem:expansion-of-S-upper-loc} and \pref{obs:double-cover-expansion} \(\lambda(\under^m) = \lambda_2(\mathcal{S}^{4m}(2m,2m,m)) = \gamma^{ m}\).
\end{proof}

\begin{proof}[Proof of \pref{claim:expansion-no-maximal-matrix-graph}]
    Careful definition chasing reveals that the double cover of \(\SL(\Gr(w_1),\Gr(w_2),W)\) is (isomorphic to) \(\mathcal{S}_{\cap Z_W}(2m,2m,m)\). Indeed, assume for simplicity that \(w_1 = w_2 = \F_2^n\). Thus the vertices in the double cover of \(\SL(\Gr(w_1),\Gr(w_2),W)\) are two copies of \(\MP_{\cap Z_W}(2m)\). This is because by \pref{claim:trivially-intersecting spaces}, \(A \oplus S_W\) if and only if \(\row(A) \oplus \row(Z_W), \col(A) \oplus \col(Z_W)\) are a direct sum. Hence by \pref{obs:double-cover-expansion} it is enough to argue about the bipartite expansion of \(\mathcal{S}_{\cap Z_W}(2m,2m,m)\).
    
    We observe that \(A \sim A'\) in \(\mathcal{S}_{\cap Z_W}(2m,2m,m)\) if and only if \(A \sim A'\) in \(\SL(\Gr(W_1),\Gr(W_2),W)\). Matrices \(A\) and \(A'\) are connected in \(\SL(\Gr(w_1),\Gr(w_2),U)\) when there exists \(U \in \MP_{\cap Z_W}(4m)\) such that \(A \sim A' \in \mathcal{S}^U(2m,2m,m)\). This is equivalent to existence of \(B_1,B_2,B_3,B_4 \in \MP(m)\) such that \(B_1 \oplus B_2 \oplus B_3 \oplus B_4 \in \MP_{\cap Z_W}(4m)\) with \(A = B_1 \oplus B_2\) and \(A' = B_1 \oplus B_3\) (the sum of \(B_i\)'s is \(U\)). Observe that \(B_4\) is redundant in this definition, that is, it is enough to require that there exists \(B_1,B_2,B_3 \in \MP(m)\) such that \(B_1 \oplus B_2 \oplus B_3 \in \MP_{\cap Z_W}(3m)\) with \(A = B_1 \oplus B_2\) and \(A' = B_1 \oplus B_3\). This is because if we are given these \(B_1,B_2,B_3\) as so, we can always find a suitable \(B_4\). By writing \(B_2 = A-B_1\) and \(B_3 = A'-B_1\), we get that \(A,A'\) are connected in \(\SL(\Gr(w_1),\Gr(w_2),W)\), if and only if \(Z_W \oplus B_1 \oplus (A-B_1) \oplus (A'-B_1)\). This is exactly the definition of the edges in \(\mathcal{S}_{\cap Z_W}(2m,2m,m)\).

   Thus by \pref{lem:expansion-of-S-lower-loc} \(\lambda(\SL(\Gr(w_1),\Gr(w_2),W)) = \lambda_2(\mathcal{S}_{\cap Z_W}(2m,2m,m))\leq \gamma^{(d_0-\dim(W)-3m)} = \gamma^{m}\).    
\end{proof}

\subsection{Expansion of the upper bound local graphs}
We begin by showing how \pref{lem:expansion-of-S-upper-loc} and \pref{lem:expansion-of-S-lower-loc} reduce to the other claims by \pref{lem:bipartite-loc-to-glob}.
\begin{proof}[Proof of \pref{lem:expansion-of-S-upper-loc}]
    Fix \(U \in \MP(\ell)\) and \(i,j,k\). We intend to use \pref{lem:bipartite-loc-to-glob}. We index our components by \(\MP(k)\), where every component \(G_C\) will be the graph containing edges \(\set{A,B}\) such that \(C \oplus (A-C) \oplus (B-C) \leq U\). We claim that \(G_C \cong \mathcal{DS}^{U-C}(i-k,j-k)\), by the map \(A \mapsto A-C\). Indeed, by \pref{prop:iso-of-posets} the vertices in \(G_C\), which are \(\MP^U_{\cap C}(i)\dunion \MP^U_{\cap C}(j)\) map isomorphically to the vertices of \(\mathcal{DS}^{U-C}(i-k,j-k)\). Moreover, by \pref{claim:elementary} \(C \oplus (A-C) \oplus (B-C) \leq U\) if and only if \((A-C) \oplus (B-C) \leq U-C\). 
    
    Thus \(G_C \cong \mathcal{DS}^{U-C}(i-k,j-k)\) and by \pref{claim:upper-local-ds-bound} we conclude that \(\lambda_2(G_C) \leq \gamma^{(\ell-k-(j-k)-(i-k))}=\gamma^{ (\ell-j-(i-k))}\).
    
    The corresponding bipartite decomposition graph of (say) the right-side is the graph \(\B_{\tau,R}=(V,T,E)\) given by:
    \begin{enumerate}
        \item \(V = \MP^U(j)\).
        \item \(T=\MP^U(k)\).
        \item \(E = \sett{(A,C)}{A \geq C}\).
    \end{enumerate}
    This is precisely \(\mathcal{R}^U(j,k)\). By \pref{claim:upper-local-R-bound} this graph is an \(\gamma^{(\ell-j-k)}\)-spectral expander. Therefore by \pref{lem:bipartite-loc-to-glob}, \(\lambda_2(\mathcal{S}^{U}(i,j,k)) \leq \gamma^{(\ell-j-\max(k,i-k))}\).
\end{proof}

\begin{proof}[Proof of \pref{lem:expansion-of-S-lower-loc}]
We use \pref{lem:bipartite-loc-to-glob}. Our components will be \(G^U = \mathcal{S}^U(i,j,k)\) for all \(U \in \MP_{\cap D}(n-\ell)\). We observe that sampling an edge by sampling \(U \in \MP_{\cap D}(n-\ell)\) and then sampling an edge in \(\mathcal{S}^U(i,j,k)\), results in a uniform distribution over the edges of \(\MP_{\cap D}(i,j,k)\). By \pref{lem:expansion-of-S-upper-loc} every such component has \(\lambda_2(G^U) \leq \gamma^{(n-\ell-j-\max(k,i-k))}\).

We note that the (say) left side of the bipartite decomposition graph is the bipartite graph \(\B_{\tau,L}=(V,T,E)\) with:
\begin{enumerate}
    \item \(V = \MP_{\cap D}(i)\).
    \item \(T = \MP_{\cap D}(n-\ell)\).
    \item \(E = \sett{(B,U)}{B \leq U}\).
\end{enumerate}
This is precisely \(\mathcal{R}_{\cap D}(i,n-\ell)\). By \pref{claim:lower-R-bound}, \(\lambda_2(\mathcal{R}_{\cap D}(i,n-\ell)) \leq \gamma^{(n-\ell-i)}\leq \gamma^{(n-\ell-j-\max(k,i-k))}\) so by \pref{lem:bipartite-loc-to-glob},
\[\lambda_2(\mathcal{S}_{\cap D}(i,j,k)) \leq \gamma^{(n-\ell-j-\max(k,i-k))}.\]
\end{proof}

Let us also give a short proof for \pref{claim:upper-local-R-bound}.
\begin{proof}[Proof of \pref{claim:upper-local-R-bound}]
    Fix \(i,j,\ell\) and \(U \in \MP(\ell)\). For notational convenience let \(\mathcal{R}^U(i,j) = (L,R,E)\) and \(\mathcal{DS}^U(i,\ell-j) = (L',R',E')\). We define 
    \(\phi:L \dunion R \to L' \dunion R'\) to be
    \[\forall B \in L, \; \phi(B) = B \; \; \ve \; \; \forall A \in R, \; \phi(A) = U-A.\]
    
    Let us first show that indeed \(\phi\) bijectively maps the vertices \(L,R\) to \(L',R'\) respectively. For \(L\) this is clear. Moreover, \(A \leq U\) if and only if \(U-A \leq U\), and in this case \(\rank(A) = j\) if and only if \(\rank(U-A) = \rank(U)-\rank(A)=\ell-j\). Thus \(\phi\) is a bijection of the vertex sets in every side.

    Next we verify that \(\set{A,B} \in E\) if and only if \(\set{U-A,B} \in E'\). Indeed, \(\set{A,B}\in E\) if and only if \(B \leq A \leq U\). By definition of \(\oplus\), this is if and only if \(A=B \oplus (A-B)\) and \(U=A \oplus (U-A)\), or more compactly, \(B \oplus (A-B) \oplus (U-A) = U\). On the other hand, \(B \oplus (U-A) \leq U\) if and only if \(B \oplus (U-A) \oplus (A-B) = U\), and thus \(\set{A,B} \in E\) if and only if \(\set{U-A,B} \in E'\). The isomorphism is proven.
\end{proof}

In the next subsections we give proofs to \pref{claim:upper-local-ds-bound} and \pref{claim:lower-R-bound}, are more technical.
\subsection{Bounding the expansion of the direct sum graphs}
The proof of \pref{claim:upper-local-ds-bound} is by induction on the sum of matrix ranks, \(i+j\). We begin with the base case.
\begin{claim} \label{claim:base-case}
    \(\lambda_2(\mathcal{DS}^U(1,1))\leq O(2^{-\frac{1}{2}\ell})\).
\end{claim}

The base case follows from the fact that the walk in \(\mathcal{DS}^U(1,1)\) is close to the uniform random walk as in \pref{cor:closeness-to-uniform}.
\begin{proof}[Proof of \pref{claim:base-case}]
    Without loss of generality let us assume that \(U=I\), the identity matrix whose rank is \(\ell\). For vectors \(x,y \in \mathbb{F}_2^{\ell}\) \(\iprod{x,y} = \sum_{j=1}^\ell x_j y_j\) the standard bilinear form over \(\F_2\), which we sometimes abuse notation and call an `inner product'. Golowich observes the following
    \begin{claim}[{\cite[Lemma 66]{Golowich2023}}] \label{claim:condition-for-being-under-id}
    Let \(I\) the identity matrix of rank \(\ell\). A matrix \(A \leq I\) if and only if \(A = \sum_{i=1}^{\ell''} e_i \otimes f_i\) for some \(\ell'' \leq \ell\) and \(\set{e_i}_{i=1}^{\ell''},\set{f_i}_{i=1}^{\ell''}\) such that
    \[\iprod{e_i,f_j} = \begin{cases}
        1 & i=j \\
        0 & i \ne j
    \end{cases}.\footnote{An easy way to see the `if part', is that if one can complete \(A = \sum_{j=1}^{\ell''} e_i \otimes f_i\) to \(I = \sum_{j=1}^\ell e_i \otimes f_i\) then the dual basis \(e_i^{*}\) of \(e_i\) has that  \(e_i^* = I e_i^* = f_i\). The other implication follows by iteratively completing \(\set{e_i}_{i=1}^{\ell''},\set{f_i}_{i=1}^{\ell''}\) to bases that are dual to one another. We omit this proof as it already appeared in \cite{Golowich2023}.} \]    
    \end{claim}
    
    Thus \(\mathcal{DS}^U(1,1)\) is (isomorphic to) the bipartite graph where \(L = \sett{(e,f) \in \mathbb{F}_2^{2\ell}}{\iprod{e,f}=1}\), \(R\) is a disjoint copy of \(L\), and we connect \((e,f) \sim (e',f')\) if and only if \(\iprod{e,f'}=\iprod{e',f}=0\) (i.e.\ we move from \(e \otimes f\) to \((e,f)\)). We will use \pref{cor:closeness-to-uniform} and show that the two step walk in this graph is \(O(2^{-\ell})\)-close to the uniform graph, thus proving that \(\lambda_2(\mathcal{DS}^U(1,1)) \leq O(2^{-\frac{1}{2}\ell})\). For concreteness we will use the two step walk from \(L\) to \(L\). 
    
    Let us fix an initial vertex \((e_0,f_0)\) and let \(\mu\) be the distribution over vertices in the two step walk. Let \(J\) be the uniform distribution over vertices.
    We want to show that \(\norm{\mu - J} =\sum_{v \in V}\abs{\mu(v) - \frac{1}{|V|}} \leq O(2^{-\ell})\).

    The basic idea is this: we prove below that the probability that we traverse from \((e_0,f_0)\) to \((e_2,f_2)\) essentially only depends on whether \(e_0=e_2\), and whether \(f_0 = f_2\) (and not on \(e_2\) or \(f_2\) themselves). While this is not true \emph{per se}, it is true up to some negligible probability. Then we use the large regularity to bound the probability that \((e_2,f_2)\) has \(e_0=e_2\) or \(f_0 = f_2\) by \(O(2^{-\ell})\). This asserts that \(\mu\) is (almost) uniform on a set that contains all but a negligible fraction of the vertices, which is enough to show closeness to \(J\). Details follow.

    \medskip
    
    We record for later the following observations.
    \begin{observation}~ \label{obs:basic-facts-about-base-case}
        \begin{enumerate}
            \item If \((e,f) \in L\) then \(e,f \ne 0\).
            \item \(|L| = (2^\ell-1)2^{\ell-1}\).
            \item The graph is \(D=(2^{\ell-1}-1)2^{\ell-2}\)-regular.
        \end{enumerate}
    \end{observation}
    The first two observations are immediate, but we explain the third: fixing \((e,f)\), we count its neighbors \((e',f')\). There are \(2^{\ell-1}-1\) choices for \(e'\) since it is perpendicular to \(f\) and non-zero. Given \(e'\), the vector \(f'\) satisfies \[\iprod{e',f'}=1 \; \; \ve \; \; \iprod{e,f'}=0.\]
    Note that \(e,e'\) are linearly independent since both are non-zero and \(\iprod{e,f} \ne \iprod{e',f}\). There are \(2^{\ell-2}\) solutions to these linear equations hence the graph is \((2^{\ell-1}-1)2^{\ell-2}\)-regular. 

    \medskip
    
    Let \((e_2,f_2)\) be any vertex. The first crude bound we give on the is just by regularity. For any \((x,y) \in V\), \(\mu(x,y) \leq \frac{1}{D} = \frac{1}{(2^{\ell-1}-1)2^{\ell-2}}\). This bound is tight when \((e_0,f_0) = (e_2,f_2)\). Similarly, if either \(e_0 = e_2\) or \(f_0 = f_2\) then we also just bound \(\mu(e_0,f_0) \leq \frac{1}{(2^{\ell-1}-1)2^{\ell-2}}\); there are at most \(2^{\ell}\) such pairs.
    
    Now let us assume \(e_2 \ne e_0, f_2 \ne f_0\) and calculate more accurate upper and lower bounds on \(\mu(e_2,f_2)\) by counting how many two paths there are from \((e_0,f_0)\) to \((e_2,f_2)\), or equivalently, how many \((e_1,f_1) \in R\) are common neighbors of \((e_0,f_0),(e_2,f_2)\).
    
    Given \(e_0,e_2\), the vector \(e_1\) needs to satisfy \(\iprod{f_0,e_1}=\iprod{f_2,e_1}=0\). The two vectors \(f_0,f_2\) are linearly independent by assumption that they are both distinct and non-zero. Thus there are \(2^{\ell-2}-1\) non-zero \(e_1\)-s that satisfy \(\iprod{f_0,e_1}=\iprod{f_2,e_1}=0\). 
    
    Fixing \(e_1\) the vector \(f_1\) needs to satisfy \(\iprod{e_0,f_1}=\iprod{e_2,f_1}=0\) and \(\iprod{e_1,f_1} = 1\). If \(e_0,e_1,e_2\) are linearly independent then there are exactly \(2^{m-3}\) solutions to these equations, all of them are non-zero, and every such solution corresponds to a common neighbor \((e_1,f_1)\). If they are not linearly independent we claim that there is no solution. We note that every pair must be linearly independent: All three are non-zero vectors. Moreover, \(e_0,e_2\) are distinct vectors by assumption, and similarly \(e_1 \ne e_0\) (resp. \(e_1 \ne e_2\)) because \(\iprod{f_0,e_0} \ne \iprod{f_0,e_1}\) (resp. \(\iprod{f_2,e_2} \ne \iprod{f_2,e_1}\)). Every pair of distinct non-zero vectors are independent. Thus, if there is a dependency then \(e_1 = e_0+e_2\), and so the first two equations imply that \(\iprod{e_1,f_1} = 0\), so there is no solution. 
    
    We conclude that the number of paths from \((e_0,f_0)\) to \((e_2,f_2)\) is one of the following.
    \begin{enumerate}
        \item \((2^{\ell-2}-1)2^{\ell-3}=2^{2\ell-5}-2^{\ell-3}\) - if \(e_0 + e_2\) is not a valid solution in the first step, i.e. either \(\iprod{f_0,e_0+e_2} \ne 0\) or \(\iprod{f_2,e_0+e_2} \ne 0\).
        \item \((2^{\ell-2}-2)2^{\ell-3} = 2^{2\ell-5}-2^{\ell-2}\) otherwise.
    \end{enumerate}
    We have that for every such pair,
    \[\mu(e_2,f_2) = \frac{2^{2\ell-5}}{((2^{\ell-1}-1)2^{\ell-2})^2} - \frac{2^{\ell-b}}{((2^{\ell-1}-1)2^{\ell-2})^2},\]
    for \(b\in \set{2,3}\).
    This implies that \(\mu(e_2,f_2) = \frac{1}{|V|} + \varepsilon_1 + \varepsilon_2\) for
    \[\abs{\varepsilon_2} = \Abs{\frac{2^{\ell-b}}{((2^{\ell-1}-1)2^{\ell-2})^2}} \leq 2^{6-3\ell}\]
    and 
    \begin{align*}
        \abs{\varepsilon_1} &\leq  \Abs{\frac{2^{2\ell-4}}{((2^{\ell-1}-1)2^{\ell-2})^2} - \frac{1}{|V|}} \\
        &= \frac{2(3 \cdot 2^\ell - 5)}{(2^\ell-2)^2(2^\ell-1)2^\ell} \leq 2^{6-3\ell}
    \end{align*} 
    Thus in particular,
    \begin{align*}
        \norm{\mu - J_{(e_0,e_2)}}_1 &= \sum_{e_2 = e_0 \vee f_2=f_0} \abs{\mu(e_2,f_2) - \frac{1}{|V|}} + \sum_{e_2 \ne e_0 \ve f_2 \ne f_0} \abs{\mu(e_2,f_2) - \frac{1}{|V|}}\\
        &\leq \frac{2^\ell}{(2^{\ell-1}-1)2^{\ell-2}} + 2^{2\ell}(\varepsilon_1 + \varepsilon_2) \\
        &\leq 2^{-\ell} (2^4+2^7) \leq 2^8 2^{-\ell}.
    \end{align*}

    The claim now follows from \pref{cor:closeness-to-uniform}.
\end{proof}

Now we prove \pref{claim:upper-local-ds-bound}. As mentioned before, the proof is by induction on \(i+j\). A similar induction was done in \cite{DiksteinD2019} when analyzing the so-called \emph{swap walk}, a generalization of the Kneser graph to high dimensional expanders. In the inductive step, one uses \pref{lem:bipartite-loc-to-glob} to reduce the analysis of \((i,j)\) to that of \((i-1,j)\) (which holds true by the inductive assignment). Details follow.
\begin{proof}[Proof of \pref{claim:upper-local-ds-bound}]
    We show that if \(i > 1\) then 
    \begin{equation} \label{eq:induction-on-sum}
        \lambda_2(\mathcal{DS}^\ell(i,j)) \leq \lambda_2(\mathcal{DS}^{\ell-1}(i-1,j)) + \lambda_2(\mathcal{DS}^{\ell}(1,j)).
    \end{equation}
    By using this inequality we eventually get that
    \[\lambda_2(\mathcal{DS}^\ell(i,j)) \leq (i\cdot j)\lambda_2(\mathcal{DS}^{\ell-i-j}(1,1)) = O \left ((i\cdot j)2^{-\frac{1}{2}(\ell-i-j) }\right ),\]
    where the final equality holds by \pref{claim:base-case}.

    To prove \eqref{eq:induction-on-sum} we use \pref{lem:bipartite-loc-to-glob}. In this case our local decomposition is \(\T = \sett{G_{B}}{B \in \MP^U(1)}\). The graph \(G_B=(L_B,R_B)\) is the induced bipartite subgraph where \(L_B = \sett{A \in L}{B \leq A \leq U}\) and \(R_B = \sett{A \in R}{(A,B) \in \mathcal{DS}^{U}(1,j)}\).  \snote{The definition for \(\mathcal{DS}^{U}(i,j)\) is that left side has dimension $i$ and right side dimension $j$. The above two definitions mix this up.}\Ynote{I think I fixed it, but maybe another read could be helpful.} We claim that \(G_B \cong \mathcal{DS}^{U-B}(i-1,j)\): the isomorphism \(\phi : G_B \to \mathcal{DS}^{U-B}(i-1,j)\) maps \[A \in R_B \mapsto \phi(A)=A \; \; \ve \; \; A \in L_B \mapsto \phi(A)=A-B.\] 
    For the right sides, we note that \((A,B) \in \mathcal{DS}^{U}(1,j)\) if and only if \(A \leq U - B\) so this isomorphism maps the right side of \(G_B\) to that of \(\mathcal{DS}^{U-B}(i-1,j)\). For the left sides, \pref{prop:iso-of-posets} gives us that \(B \leq A \leq U\) if and only if \(A-B \leq U-B\). Thus \(\phi\) maps matrices \(A \in R_B\) isomorphically to matrices on the left side of \(G^0_{m-1}(U-B,a_1,a_2-1)\). Checking that \(A_1 \sim A_2\) in $\mathcal{DS}^{U-B}(i-1,j)$ if and only if \(\phi(A_1) \sim \phi(A_2)\) in \(G_B\) is direct. Therefore for every \(B\), the second largest eigenvalue \(\lambda_2(G_B) = \lambda_2(\mathcal{DS}^{\ell-1}(i-1,j))\). 

    We note that by definition the graph \(\B_{\tau,L} = \mathcal{DS}^{\ell}(1,j)\). Thus \eqref{eq:induction-on-sum} follows from \pref{lem:bipartite-loc-to-glob}.
\end{proof}

\subsection{Bounding the expansion of the relation graphs}
We start by proving the claim in the following special case.
\begin{claim} \label{claim:lower-bound-relation-i-small}
    Let \(D \in \MP(\ell)\) and let \(i,j\) be such that \(i \leq \frac{j}{3}\), then 
    \(\lambda_2(\mathcal{R}_{\cap D}(i,j)) \leq \gamma^{ j}.\)
\end{claim}

After proving \pref{claim:lower-bound-relation-i-small}, we prove \pref{claim:lower-R-bound} by using \pref{lem:bipartite-loc-to-glob} again.

\begin{proof}[Proof of \pref{claim:lower-bound-relation-i-small}]
    Fix \(D,i,j\). For this case we intend to use \pref{cor:closeness-to-uniform} on the two step walk on \(\MP_{\cap D}(i)\). We prove that 
\begin{equation} \label{eq:mu-close-to-J}
    \max_{B_1 \in V} \norm{\mu_{B_1} - J}_1 \leq \gamma''^{j}
\end{equation}
where \(\mu_{B_1}\) is the distribution of the two step walk starting from \(B_1\), and the distribution \(J\) is the uniform distribution over \(\MP_{\cap D}(i)\). Fix \(B_1 \in \MP_{\cap D}(i)\) and from now on we just write \(\mu\) instead of \(\mu_{B_1}\) for short. 

We will show that there exists a set \(V' \subseteq \MP_{\cap D}(i)\) such that both \(\Prob[B_2 \sim \mu]{B_2 \in V'}, \Prob[B_2 \sim J]{B_2 \in V'} \geq 1-p\) for a sufficiently small \(p \leq (\gamma')^{j}\). Then we will also prove that the distribution \(\mu\) conditioned on this event \(V'\) is uniform, and this will give us \eqref{eq:mu-close-to-J}. Here 
\[V' = \sett{B_2 \in \MP_{\cap D}(i)}{\row(B_2) \cap (\row(B_1) \oplus \row(D)) = \set{0},  \col(B_2)\cap (\col(B_1) \oplus \col(D)) = \set{0}}.\]

A quick calculation will show that in this case \(\norm{\mu - J}_1 \leq 2p \leq \gamma^{j}\) for some \(\gamma \in (0,1)\).

First we show that \(\Prob[B_2 \sim J]{V'} \geq 1-O(2^{2i-j})\): Let us show that the probability of choosing a uniform \(B_2 \in \MP_{\cap D}(i)\) such that \(\row(B_2) \cap (\row(B_1) \oplus \row(D)) \ne \set{0}\) is at most \(i2^{2i-j}\) (the same also holds for the column spaces). The row span of \(B_2\) is sampled uniformly as an \(i\) dimensional subspace inside a \(n\) dimensional space, conditioned on \(\row(B_2) \cap \row(D) = \{0\}\). So this probability is 
\begin{multline*}
    \cProb{\row(B_2)}{(\row(B_1) \oplus \row(D)) \cap \row(B_2) \ne \set{0}}{\row(B_2) \cap \row(D) = \set{0}}\\
    \leq \frac{\Prob[\row(B_2)]{(\row(B_1) \oplus \row(D)) \cap \row(B_2) \ne \set{0}}}{\Prob[\row(B_2)]{\row(B_2) \cap \row(D) = \set{0}}}.
\end{multline*}
The ratio is bounded by \(\frac{i2^{2i+\ell-n-1}}{1-i2^{\ell-n}} \leq O(i2^{2i +(\ell-n)})\) via standard calculations. The same probability holds for the columns spaces, therefore \(\Prob[B_2 \sim J]{V'} \geq 1-O(i2^{i +(\ell-n)}) = 1-O(\gamma^{(2i-j)})\). The last equality is by \(j \leq n-\ell\).

Next we claim that \(\Prob[B_2 \sim \mu]{V'} \geq 1-O(2^{2i-j})\). Indeed, the two step walk from \(B_1\) traverses to some \(U \in \MP_{\cap D}(j)\) (whose row space intersects trivially with \(\row(D)\)), and from \(U\) we independently choose a matrix \(B_2 \leq U\). As \(\row(U) \cap \row(D) = \set{0}\), it follows that \(\row(B_2) \cap \row(B_1) \oplus \row(D) \ne \set{0}\) if and only if \(\row(B_1) \cap \row(B_2) \ne \set{0}\). For any fixed \(U \in \MP_{\cap D}(j)\), the probability that \(\row(B_1) \cap \row(B_2) \ne \set{0}\) is at most \(O(i2^{2i-j})\) (repeating the calculation above).

\medskip

Now let us show that conditioned on \(V'\), for any two neighbors \(B_2,B_2' \in V'\) of \(B_1\) in the two step walk, the probability of sampling \(B_2\) is the same as sampling \(B_2'\). To show this it is enough to show that there exists a graph isomorphism of the original ``one-step graph'', that fixes \(B_1\) and sends \(B_2\) to \(B_2'\). Let us denote by \(GL(\F_2^n)\) the general linear group, or equivalently the group of \(n \times n\) invertible matrices. We will find \(E_1 \in GL(w_1), E_2 \in GL(w_2)\) such that:
\begin{enumerate}
    \item \(E_2 A=A E_1 = A\).
    \item \(E_2 B_1 E_1 = B_1\).
    \item \(E_2 B_2 E_1 = B_2'\).
\end{enumerate}
Assume we found such \(E_1,E_2\). We claim that \(X \mapsto E_2 X E_1\) is an isomorphism of the one-step graph. Indeed, note that if \(A \oplus B\) then as multiplying by an invertible matrix maintains rank, it follows that \(E_2 A E_1 \oplus E_2 B E_1= A \oplus E_2 B E_1\). Therefore, \(X \mapsto E_2 X E_1\) is a bijection on the vertices of the graph. It is easy to check that this multiplication also maintains the \(\leq\) relation so this is indeed an isomorphism. Finally, it is also easy to see that \((B_1,M,B_2)\) is mapped to a path \((B_1,E_2 M E_1,B_2')\), so in particular such an isomorphism implies that the probability of sampling \(B_2,B_2' \in V'\) is the same.

To find such \(E_1,E_2\) recall that \(E_2 (e \otimes f) E_1 = (E_2 e) \otimes (E_1 f)\). Thus we can decompose \(A = \sum_{t=1}^\ell e_t \otimes f_t\), \(B_1 = \sum_{t=1}^i e'_t \otimes f'_t\), \(B_2 = \sum_{t=1}^i h_t \otimes k_t\) and \(B_2' = \sum_{t=1}^i h'_t \otimes k'_t\). By assumption on the intersections of the row (and column) spaces, the sets \(\set{e_t}_{t=1}^\ell \dunion \set{e'_t}_{t=1}^i \dunion \set{h_t}_{t=1}^i\) and \(\set{e_t}_{t=1}^\ell \dunion \set{e'_t}_{t=1}^i \dunion \set{h'_t}_{t=1}^i\) are independent and have the same size. The same holds for the corresponding sets of column vectors (replacing \(e,h\) with \(f,k\)). Thus we just find any \(E_1,E_2\) such that
\[E_1 e_t = e_t, \; E_1 e'_t = e'_t, \; E_1 h_t = h'_t\]
\[E_2 f_t = f_t, \; E_1 f'_t = f'_t, \; E_1 k_t = k'_t.\]
These will have the desired properties. By the arguments above, the claim is proven.
\end{proof}
We comment that the above proof is one place where an analogous argument fails to hold for higher dimensional tensors. We heavily relied on the fact that for any fixed matrix \(U\), the row span and column span of an \(i\) rank matrix \(B \leq U\) is uniformly distributed over \(i\) dimensional subspaces inside \(\row(U)\) and \(\col(U)\), even if these marginals are not independent.

We move on to the general case.
\begin{proof}[Proof of \pref{claim:lower-R-bound}]
    Fix \(i,j,\ell\) and \(D \in \MP(\ell)\). The case where \(i \leq \frac{j}{3}\) is by \pref{claim:lower-bound-relation-i-small} so assume that \(\frac{j}{3} < i \leq \frac{j}{2}\). Towards using \pref{lem:bipartite-loc-to-glob}, we decompose \(\mathcal{R}_{\cap D}(i,j)\) to components \(\set{G_B}_{B \in \MP_{\cap D}(\frac{j}{3})}\) where \(G_B\) is the graph induced by all edges \(\set{M,C}\) such that \(B \leq C \leq M\). We claim that \(G_B\) is isomorphic to \(\mathcal{R}_{\cap (D \oplus B)}(i-\frac{j}{3},\frac{2j}{3})\). The isomorphism \(\phi: G_B \to \mathcal{R}_{\cap (D \oplus B)}(i-\frac{j}{3},\frac{2j}{3})\) is \(\phi(M) = M-B\). Let us start by showing that this maps vertices to vertices.

    Indeed, fix any \(B \in \MP_{\cap D}(\frac{j}{3})\) and let \(G_B=(L_B,R_B,E_B)\). Explicitly,
    \[L_{B} = \sett{M \in \MP_{\cap D}(i)}{M \geq B}, R_{B} = \sett{C \in \MP_{\cap D}(j)}{C \geq B}.\]
    Note that every \(M \in L_B\) is a direct sum \(M \oplus D\), and as \(B \leq M\) this implies that \((B \oplus (M-B)) \oplus D\). As we saw before, this implies that \((B \oplus D) \oplus (M-B)\), therefore \(\phi(M)=M-B \in \MP_{\cap (D \oplus B)}(i-\frac{j}{3})\). Conversely, if \(M' \in \MP_{\cap (D \oplus B)}(i-\frac{j}{3})\) then \(\phi^{-1}(M') = M' \oplus B\) has \(\rank(\phi^{-1}(M')) = i\). Moreover, as \(M' \oplus (D \oplus B)\) then \((M' \oplus B) \oplus D\), i.e. \(\phi^{-1}(M') \in \MP_{\cap D}(j)\). It is also clear that \(B \leq \phi^{-1}(M')\) hence \(\phi^{-1}(M') \in L_B\). Of course, a similar argument holds for the right side \(R_B\), then \(\phi\) is a bijection of the vertex set. Finally, by \pref{claim:elementary}, \(C \leq M\) if and only if \(C-B \leq M-B\), so this is indeed an isomorphism of graphs.

    We note that if \(i \leq \frac{j}{2}\) then \(i-\frac{j}{3} \leq \frac{1}{2} \cdot \frac{2j}{3}\) and therefore we can use \pref{claim:lower-bound-relation-i-small} and deduce that for every \(B \in \MP_{\cap D}(\frac{j}{3})\), \(\lambda_2(G_B) \leq \gamma^{j}\). Moreover, we observe that the bipartite local to global graph for the left side, \(\B_{\tau, L}\) is \(\MP_{\cap D}(\frac{j}{3},j)\), thus by \pref{claim:lower-bound-relation-i-small} this also has \(\lambda_2(\B_{\tau, L}) \leq \gamma^{j}\) and so by \pref{lem:bipartite-loc-to-glob} the lemma is proven.
\end{proof}
We comment that the requirement that \(i \leq \frac{j}{2}\) is just for convenience. Multiple uses of the same decomposition with \pref{lem:bipartite-loc-to-glob}, would yield a similar bound whenever \(i \leq (1-\varepsilon)j\) for any fixed constant \(\varepsilon > 0\). As we only need the claim for \(i=\frac{j}{2}\), we merely mention this without proof.

\subsection[Short paths in the subposet graph]{Short paths in \(\under^m\)} \label{app:short-paths-in-subposet-graph}
This subsection of the appendix is devoted to the proof of \pref{claim:short-paths-in-under-graph} which is necessary to analyze the coboundary expansion of the complexes in \cite{Golowich2023}.

We begin with the following auxiliary claim proven in the end of the subsection.
\begin{claim} \label{claim:connecting-disjoint-matrices-in-under-graph}
    Let \(Z_1\) be a rank \(4m\) matrix and let \(Z_2\) be a rank \(3m\) matrix.
    \begin{enumerate}
        \item Let \(A,B \leq Z_1\) be such that \(\rank(A)=\rank(B)=m\). Then there exists a matrix \(T\) with \(\rank(T)=m\) such that \(A \oplus T, T \oplus B \leq Z_1\).
        \item Let \(A,B \leq Z_2\) be such that \(\rank(A)=\rank(B)=m\) then there exists \(M^1,M^2,M^3\) such that
        \[A \oplus M^1, M^1 \oplus M^2, M^2 \oplus M^3, M^3 \oplus B \leq Z_2.\]
    \end{enumerate}        
\end{claim}
Phrasing this claim differently, the first item claims that any two matrices in one side of \(\mathcal{DS}^{4m}(m,m)\) have a length \(2\) path between them. The first item claims that any two matrices in one side of \(\mathcal{DS}^{3m}(m,m)\) have a length \(4\) path between them.

\begin{proof}[Proof of \pref{claim:short-paths-in-under-graph}]
    We proceed as follows. We first decompose \(C,D \in \under^Z\) into \(m\) ranked matrices such that \(C=C^1 \oplus C^2\) and \(D=D^1 \oplus D^2\). Then we find some \(T\) of rank \(m\) such that \(T \oplus C^1, T \oplus D^1 \leq Z\). This is possible due to the first item in \pref{claim:connecting-disjoint-matrices-in-under-graph}.

    Then we will prove that there exists a \(12\)-path composed of three paths of length \(4\): one from \(C^1\oplus C^2\) to \(C^1 \oplus T\), one from \(C^1 \oplus T\) to \(D^1 \oplus T\) and one from \(D^1 \oplus T\) to \(D^1 \oplus D^2\).

    Let us show how to construct the first path, the other two are constructed similarly. Note that \(C^1 \leq C^1 \oplus C^2, C^1 \oplus T\) and that they are both \(\leq Z\). Thus if we find \(M^1,M^2,M^3 \leq Z-C^1\) such that
    \[C^2 \oplus M^1, M^1 \oplus M^2, M^2 \oplus M^3, M^3 \oplus T \leq Z - C^1,\]
    Then by \pref{prop:iso-of-posets}, adding \(C^1\) to all three \(M^1,M^2,M^3\) will give us the desired path in \(\under^Z\):
    \[(C^1\oplus C^2,C^1\oplus M^1, C^1\oplus M^2, C^1\oplus M^3, C^1\oplus T).\]
    But observe that \(Z-C^1\) has rank \(3m\), and thus existence of such matrices \(M^1,M^2,M^3\) is promised by \pref{claim:connecting-disjoint-matrices-in-under-graph} so we can conclude.
\end{proof}

\begin{proof}[Proof of \pref{claim:connecting-disjoint-matrices-in-under-graph}]
    Without loss of generality let us assume that \(Z_1\) is the identity matrix of rank \(4m\) and that \(Z_2\) is the identity matrix of rank \(3m\). We wish to use the criterion for \(A \leq I\) given in \pref{claim:condition-for-being-under-id}.

    For this proof we also require the following observation which follows from elementary linear algebra. We leave its proof to the reader.
    \begin{observation}\label{obs:solution-to-non-homogenuous-linear-equations}
        Let \(W \subseteq \F_2^n\) be a subspace of dimension \(j\). Let \(E = \set{e_1,e_2,\dots,e_k}\) be \(k\) independent vectors. Assume that \(W \cap \sp(E) = \set{0}\). Then for every \(b_1,b_2,\dots,b_k \in \F_2\) there exists a solution \(x\) to the set of linear equations:
        \begin{enumerate}
            \item For every \(w \in W\), \(\iprod{x, w} = 0\), and
            \item For every \(i=1,2,\dots,k\), \(\iprod{x,e_i} = b_i\).
        \end{enumerate}
    \end{observation}

    Let us begin with the first item. Given \(A,B\) of rank \(m\) we need to find some \(T\) of tank \(m\) such that \(A \oplus T, B \oplus T \leq Z_1\). Equivalently, we will find \(T\) such that \(A \oplus T\) and \(B \oplus T\) satisfies the criterion in \pref{claim:condition-for-being-under-id}.
    
    Let us write \(A = \sum_{j=1}^{m} e_j(A) \otimes f_j(A), B=\sum_{j=1}^{m}e_j(B) \otimes f_j(B)\) where these decompositions satisfy the property in \pref{claim:condition-for-being-under-id}. Let \(W = \sett{w \in \F_2^{4m}}{\forall j=1,2,\dots,m \; \iprod{w,f_j(A)}=\iprod{w,f_j(B)}=0}\). As \(W\) is a solution space of \(2m\) linear equations over \(\F_2^{4m}\), it is a subspace of dimension \(\geq 2m\). We will select the rows of \(T\) from \(W\). 
    
    However, it is not enough to just select arbitrary rows from \(W\), since we need to satisfy the condition that if \(e_i, f_j\) participates in the decomposition of \(T\) then \(\iprod{e_i,f_j}=1\) if and only if \(i=j\). In addition, we will also require that \(f_j\) will be perpendicular to all the \(e_i(A),e_i(B)\)'s. This will define a set of linear equations which we will solve using \pref{obs:solution-to-non-homogenuous-linear-equations}. To do this we will find an \(m\)-dimensional subspace in \(W' \subseteq W\) that intersects trivially with the span of the \(e_j(A),e_j(B)\)'s (the sum of row spaces of \(A\) and \(B\)).

    We note that \(W \cap \row(A)=W \cap \row(B) = \set{0}\). Let us show this for \(\row(A)\): any non-zero vector \(x\in \row(A)\) is a sum of \(\sum_{j=1}^{m} b_j e_j(A)\) where at least on of the \(b_j\)'s are not zero. On the other hand, \(\iprod{x,f_j(A)} = b_j\), so there is at least one \(\iprod{x,f_j(A)} \ne 0\), implying that \(x \notin W\). Thus we conclude that \(\dim(W \cap (\row(A) + \row(B))) \leq m\), and in particular there exists some \(m\) dimensional subspace \(W' \subseteq W\) such that \(W'\) intersects trivially with \(\row(A) + \row(B)\). Let \(E = \set{e_1(T),\dots,e_{m}(T)}\) be a basis of \(W'\).

    Thus by \pref{obs:solution-to-non-homogenuous-linear-equations}, for every \(i=1,2,\dots,m\) there exists some \(f_i(T)\) such that 
    \begin{enumerate}
        \item For every \(j=1,2,\dots,m\), \(\iprod{e_j(A),f_i(T)} = \iprod{e_j(B),f_i(T)} = 0\) for all \(j=1,2,\dots,m\), and
        \item In addition
        \[\iprod{e_j(T),f_i(T)} = \begin{cases}
        1 & i=j \\ 0 & i \ne j
    \end{cases}.\]
    \end{enumerate} 
    By taking \(T = \sum_{j=1}^{m} e_j(T) \otimes f_j(T)\) we observe that both \(A \oplus T\) and \(B \oplus T\) satisfy \pref{claim:condition-for-being-under-id}, and therefore \(A \oplus T, B \oplus T \leq Z_1\).
    
    \medskip
    
    Let us turn to the second item. We continue with the notation \(A = \sum_{j=1}^{m} e_j(A) \otimes f_j(A), B=\sum_{j=1}^{m}e_j(B) \otimes f_j(B)\). The strategy here will be similar. Let \(W^1 = \sett{w \in \F_2^{3m}}{\forall j=1,2,\dots,k \; \iprod{w,f_j(A)}=\iprod{w,f_j(B)}=0}\). We observe that this is the solution set of \(2m\) linear equations and therefore its dimension is \(\geq m\). Let \(e_1(M^{1}),e_2(M^{1}),\dots,e_{m}(M^{1}) \in W^1\) be \(m\) independent vectors, and set \(e_j(M^3)=e_j(M^1)\). We observe that as before \(W^1 \cap \row(A) = \set{0}\) so by \pref{obs:solution-to-non-homogenuous-linear-equations}, for every \(j=1,2,\dots,m\) there exists some \(f_j(T)\) such that
    \begin{enumerate}
        \item For every \(e_i(A)\), \(\iprod{e_i(A),f_j(M^1)}=0\).
        \item For every \(i,j\), \[\iprod{e_i(M^1),f_j(M^1)} = \begin{cases}
        1 & i=j \\ 0 & i \ne j
    \end{cases}.\]
    \end{enumerate}
    We note that the linear equations for \(f_j(M^1)\)'s do not consider the \(e_j(B)\)'s. 
    
    Similarly, we find \(f_j(M^3)\) ignoring the \(e_j(A)\)'s. That is, for every \(j=1,2,\dots,m\) we find \(f_j(M^3)\) such that 
    \begin{enumerate}
        \item For every \(e_i(B)\), \(\iprod{e_i(B),f_j(M^3)}=0\).
        \item For every \(i,j\), \[\iprod{e_i(M^3),f_j(M^3)} = \begin{cases}
        1 & i=j \\ 0 & i \ne j
    \end{cases}.\]
    \end{enumerate}
    
    Then we set \(M^1 = \sum_{j=1}^{m} e_j(M^1) \otimes f_j(M^1)\) and \(M^1 = \sum_{j=1}^{m} e_j(M^3) \otimes f_j(M^3)\). It is easy to verify that as before \(A \oplus M^1, B \oplus M^3 \leq Z_2\).

    Now we find \(M^2\) such that \(M^1 \oplus M^2, M^2 \oplus M^3 \leq Z_2\). The upshot of the previous step is that now \(M^1\) and \(M^3\) have the same row space \(\sp(\set{e_1(M^1),e_2(M^1),\dots,e_{m}(M^1)})\). Thus we find \(M^2\) as follows. Let \(W^2 = \sett{w \in \F_2^{3m}}{\forall j=1,2,\dots m \; \iprod{w,f_j(A)}=\iprod{w,f_j(B)}=0}\). This subspace also has dimension \(\geq m\) so let \(e_1(M^2),e_2(M^2),\dots,e_{m}(M^2) \in W^2\) be a set of independent vectors from $W^2$.

    As before, we observe that \(W^2 \cap \row(M^1) = \set{0}\) but this time \(\row(M^1) + \row(M^3) = \row(M^1)\). Thus by \pref{obs:solution-to-non-homogenuous-linear-equations} for every \(j=1,2,\dots,m\), we can find \(f_j(M^2)\) such that 
    \begin{enumerate}
        \item For every \(i=1,2,\dots,m\), \(\iprod{e_i(M^1),f_j(M^2)}=\iprod{e_i(M^3),f_j(M^2)} = 0\), and
        \item In addition, \[\iprod{e_i(M^2),f_j(M^2)} = \begin{cases}
        1 & i=j \\ 0 & i \ne j
    \end{cases}.\]
    \end{enumerate}
    Then we set \(M^2 = \sum_{j=1}^{m} e_j(M^2) \otimes f_j(M^2)\). It is direct to verify that \(M^1 \oplus M^2, M^2 \oplus M^3 \leq Z_2\) by \pref{claim:condition-for-being-under-id}. 
\end{proof}

\section[Second largest eigenvalue of J(n,k,k/2)]{Second largest eigenvalue of $J(n,k,k/2)$}
 \label{app:proofs-of-Johnson-lambda2}
\begin{proof}[Proof of \pref{lem:Johnson-eigval-g2k}]
    Let $n= (2+\eta)k$. For the Johnson graph $J(n,k,k/2)$, its largest unnormalized eigenvalue
    $\lambda_0 = \binom{k}{k/2} \binom{n-k}{k/2}~.$
    Therefore by \pref{thm:Johnson-graph-eigenvalue} the absolute values of the normalized eigenvalues are 
    \begin{align*}
        \frac{\abs{\lambda_t}}{\lambda_0} 
        &= \frac{\abs{\sum_{i=0}^t (-1)^{t-i} \binom{k-i}{k/2-i} \binom{n-k+i-t}{k/2+i-t}\binom{t}{i}}}{\binom{k}{k/2} \binom{n-k}{k/2}}   \\
        &= \left| \sum_{i=\max(0,t-k/2)}^{\min(t,k/2)} (-1)^{t-i} \binom{t}{i} \cdot \frac{k/2\times \dots \times (k/2-i+1)}{k\times\dots\times (k-i+1)} \cdot \frac{k/2\times \dots \times (k/2+i-t+1)}{(n-k)\times\dots\times (n-k+i-t+1)}  \right| ~.\\
    \end{align*}
    Here the convention is that when we write \(\binom{a}{b}\) for \(b>a\) or \(b<0\), then the term is equal to \(0\). Thus, the terms corresponding to $i< t-k/2$ or $i>\min\set{t,k/2}$ vanish since the product of the three binomial coefficients is $0$ for these choices of $i$. We first show the statement for odd $t\in[k]$. In this case, we group terms $\binom{t}{i}$ and $\binom{t}{t-i}$ together to get 
    \begin{align*}
        &\le \sum_{i=\max(0,t-k/2)}^{\fl{t/2}} \binom{t}{i}\cdot \left|  \prod_{j=0}^{i-1} \frac{ (k/2-j)}{(k-j)} \cdot \prod_{j=0}^{t-i-1} \frac{  k/2-j}{ n-k-j}   -    \prod_{j=0}^{t-i-1}\frac{ k/2-j}{k-j} \cdot \prod_{j=0}^{i-1}\frac{  k/2-j}{n-k-j}  \right| \\
        &= \sum_{i=\max(0,t-k/2)}^{\fl{t/2}} \binom{t}{i}  \cdot    \prod_{j=0}^{i-1}\frac{ k/2-j}{k-j} \cdot \prod_{j=0}^{i-1} \frac{ k/2-j}{n-k-j}  \cdot \left| \prod_{j=i}^{t-i-1}\frac{k/2-j}{ n-k-j} - \prod_{j=i}^{t-i-1} \frac{k/2-j}{k-j} \right| ~.
    \end{align*}
    To simplify the products we first observe that for any $0\le j < k/2 $, $\frac{k/2-j}{k-j} \le \frac{k/2}{k} = \frac{1}{2}$ and $\frac{k/2-j}{n-k-j} \le \frac{k/2}{n-k} = \frac{1}{2(1+\eta)}$.
    \begin{equation}\label{eq:e1}
        \le \sum_{i=\max(0,t-k/2)}^{\fl{t/2}} \binom{t}{i}   \cdot \left(\frac{1}{2}\right)^{i}  \cdot \left(\frac{1}{2(1+\eta)}\right)^{i} \cdot \left| \prod_{j=i}^{t-i-1}\frac{k/2-j}{ n-k-j} - \prod_{j=i}^{t-i-1} \frac{k/2-j}{k-j} \right| ~.
    \end{equation}
    
    Furthermore we use the following claim to bound the term in absolute value: 

    \begin{claim}\label{claim:diff-bound}
        For any constant $\eta > 0$ and any integers $  1\le m < C$ and $0\le s \le m$,
        \[\prod_{j=s}^{m}\frac{C-j}{2C-j} - \prod_{j=s}^{m}\frac{C-j}{2(1+\eta)C-j} \le \left(\frac{C-s}{2C-s}\right)^{m-s+1} - \left(\frac{C-s}{2(1+\eta)C-s}\right)^{m-s+1}~.\]
    \end{claim}
    A proof is given later in the appendix. By the claim the expression in \pref{eq:e1} can be bounded by a binomial sum:

     \begin{align*}
         &\le \sum_{i=\max(0,t-k/2)}^{\fl{t/2}} \binom{t}{i} \cdot \left(\frac{1}{2}\right)^{i}  \cdot \left(\frac{1}{2(1+\eta)}\right)^{i} \cdot \left( \left(\frac{1}{2}\right)^{t-2i} - \left(\frac{1}{2(1+\eta)}\right)^{t-2i} \right) \\
         &= \sum_{i=\max(0,t-k/2)}^{\fl{t/2}} \left| (-1)^{i} \binom{t}{i} \cdot  \left(\frac{1}{2}\right)^{i} \cdot \left(\frac{1}{2(1+\eta)}\right)^{t-i} +  (-1)^{t-i} \binom{t}{t-i} \cdot \left(\frac{1}{2}\right)^{t-i} \cdot \left(\frac{1}{2(1+\eta)}\right)^{i} \right| \\
         &\le \left| \sum_{i=0}^t (-1)^{i} \binom{t}{i} \cdot \left(\frac{1}{2}\right)^{i} \cdot \left(\frac{1}{2(1+\eta)}\right)^{t-i} \right|\\
         &\le \left(\frac{1}{2} - \frac{1}{2(1+\eta)} \right)^t \\
         &=\left (\frac{\eta}{2(1+\eta)} \right)^t 
 \le\frac{\eta}{2(1+\eta)}~.
     \end{align*}
    Next consider the case of even $t\in[k]$. First use the binomial identity $\binom{t}{i} = \binom{t-1}{i}+ \binom{t-1}{i-1}$ to obtain the following bound
    \begin{align*}
        &\le \left|\sum_{i=\max(0,t-k/2) }^{ \min(t,k/2) } \left((-1)^{t-1-(i-1)}\binom{t-1}{i-1} - (-1)^{t-1-i} \binom{t-1}{i} \right)\cdot  \prod_{j=0}^{i-1} \frac{ (k/2-j)}{(k-j)} \cdot \prod_{j=0}^{t-i-1} \frac{  k/2-j}{ n-k-j}  \right|\\
        &=|\sum_{i=\max(0,t-1-k/2) }^{ \min(t-1,k/2)} (-1)^{t-1-i}\binom{t-1}{i}\cdot  \prod_{j=0}^{i} \frac{ (k/2-j)}{(k-j)} \cdot \prod_{j=0}^{t-i-2} \frac{  k/2-j}{ n-k-j} \\
        &~~~~~~- \mathbb{1}[t-1>k/2] \cdot   (-1)^{t-1-k/2}\binom{t-1}{k/2}\cdot  \prod_{j=0}^{k/2} \frac{ (k/2-j)}{(k-j)} \cdot \prod_{j=0}^{t-k/2-2} \frac{  k/2-j}{ n-k-j}\\
        &~~~~~~+\sum_{i=\max(0,t-1-k/2) }^{ \min(t-1,k/2)} (-1)^{t-1-i} \binom{t-1}{i} \cdot  \prod_{j=0}^{i-1} \frac{ (k/2-j)}{(k-j)} \cdot \prod_{j=0}^{t-i-1} \frac{  k/2-j}{ n-k-j} \\
        &~~~~~~- \mathbb{1}[t-1\ge k/2] \cdot   (-1)^{k/2}\binom{t-1}{k/2}\cdot  \prod_{j=0}^{t-2-k/2} \frac{ (k/2-j)}{(k-j)} \cdot \prod_{j=0}^{k/2} \frac{  k/2-j}{ n-k-j} |~.
    \end{align*}

    First note that the second term and the fourth term always evaluate to $0$. Thus we can simplify the expression as 

    \begin{align*}
        &\le\left|\sum_{i=\max(0,t-1-k/2) }^{ \min(t-1,k/2)} (-1)^{t-1-i}\binom{t-1}{i}\cdot \frac{ (k/2-i)}{(k-i)} \cdot  \prod_{j=0}^{i-1} \frac{ (k/2-j)}{(k-j)} \cdot \prod_{j=0}^{(t-1)-i-1} \frac{  k/2-j}{ n-k-j} \right| \\
        &~~~+\left|\sum_{i=\max(0,t-1-k/2) }^{ \min(t-1,k/2)} (-1)^{t-1-i} \binom{t-1}{i} \cdot \frac{  k/2-(t-i-1)}{ n-k-(t-i-1)}\cdot \prod_{j=0}^{i-1} \frac{ (k/2-j)}{(k-j)} \cdot \prod_{j=0}^{(t-1)-i-1} \frac{  k/2-j}{ n-k-j} \right|~.
    \end{align*}

    These two terms resembles the expression of $\frac{\abs{\lambda_{t-1}}}{\lambda_0}$ and can be bounded by grouping terms $\binom{t-1}{i}$ and $\binom{t-1}{t-1-i}$.
    \begin{align*}
        &\le \sum_{i=\max(0,t-1-k/2)}^{\fl{(t-1)/2}} \binom{t-1}{i}  \cdot    \prod_{j=0}^{i}\frac{ k/2-j}{k-j} \cdot \prod_{j=0}^{i-1} \frac{ k/2-j}{n-k-j}  \cdot \left|\prod_{j=i}^{(t-1)-i-1}\frac{k/2-j}{ n-k-j} - \prod_{j=i}^{(t-1)-i-1} \frac{k/2-j-1}{k-j-1} \right| \\
        &~~~+\sum_{i=\max(0,t-1-k/2)}^{\fl{(t-1)/2}} \binom{t-1}{i}  \cdot    \prod_{j=0}^{i-1}\frac{ k/2-j}{k-j} \cdot \prod_{j=0}^{i} \frac{ k/2-j}{n-k-j}  \cdot \left|\prod_{j=i}^{(t-1)-i-1}\frac{k/2-j-1}{ n-k-j-1} - \prod_{j=i}^{(t-1)-i-1} \frac{k/2-j}{k-j} \right|
    \end{align*}

    Using \pref{claim:diff-bound} again yields 
    \begin{align*}
        &\le \sum_{i=\max(0,t-1-k/2)}^{\fl{(t-1)/2}} \binom{t-1}{i}  \cdot  \left(\frac{1}{2}\right)^{i+1}\cdot \left(\frac{1}{2(1+\eta)}\right)^{i} \cdot \left(\left(\frac{1}{2}\right)^{(t-1)-2i}  - \left(\frac{1}{2(1+\eta)}\right)^{(t-1)-2i} \right) \\
        &~~~+\sum_{i=\max(0,t-1-k/2)}^{\fl{(t-1)/2}} \binom{t-1}{i}  \cdot  \left(\frac{1}{2}\right)^{i}\cdot \left(\frac{1}{2(1+\eta)}\right)^{i+1} \cdot \left(\left(\frac{1}{2}\right)^{(t-1)-2i}  - \left(\frac{1}{2(1+\eta)}\right)^{(t-1)-2i} \right)\\
        &\le \left(\frac{1}{2}+\frac{1}{2(1+\eta)}\right)\cdot \left(\frac{1}{2} - \frac{1}{2(1+\eta)} \right)^{t-1} \le\frac{\eta}{2(1+\eta)}~.
    \end{align*}
    We thus conclude that for all $t\in[k]$, $\frac{\abs{\lambda_t}}{\lambda_0}\le \frac{\eta}{2(1+\eta)}$.
\end{proof}  
\begin{proof}[Proof of \pref{claim:diff-bound}]
    We prove the statement by induction on $s$. The case of $s = m$ holds trivially. Suppose the statement holds for all $s>s^*$,  then we are going to show that the inequality also holds for $s = s^*$.
    Note that by the induction hypothesis
     \begin{equation}\label{eq:ind-hyp}
     \prod_{j=s^*+1}^{m}\frac{C-j}{2C-j} - \prod_{j=s^*+1}^{m}\frac{C-j}{2(1+\eta)C-j} \le \left(\frac{C-(s^*+1)}{2C-(s^*+1)}\right)^{m-s^*} - \left(\frac{C-(s^*+1)}{2(1+\eta)C-(s^*+1)}\right)^{m-s^*}~.
     \end{equation}
     To simplify the expressions we define
     \[X(s) = \prod_{j=s}^{m}\frac{C-j}{2C-j}~~x(s) = \frac{C-s}{2C-s}\]
     and also 
     \[Y(s) = \prod_{j=s}^{m}\frac{C-j}{2(1+\eta)C-j}~~y(s) = \frac{C-s}{2(1+\eta)C-s}.\]
     Using this notation, \eqref{eq:ind-hyp} can be rewritten as 
     \begin{equation*}
         X(s^*+1) - Y(s^*+1) \le x(s^*+1)^{m-s^*} - y(s^*+1)^{m-s^*}.
     \end{equation*}
     Rearranging the inequality we get that
     \begin{align}
         x(s^*+1)^{m-s^*} - X(s^*+1) &\ge y(s^*+1)^{m-s^*} - Y(s^*+1) \nonumber\\
         x(s^*)\cdot(x(s^*+1)^{m-s^*} - X(s^*+1)) &\ge y(s^*)\cdot(y(s^*+1)^{m-s^1} - Y(s^*+1)) \nonumber\\
          x(s^*)\cdot x(s^*+1)^{m-s^*} - X(s^*) &\ge y(s^*)\cdot y(s^*+1)^{m-s^*}- Y(s^*) \nonumber\\
          X(s^*) - Y(S^*) &\le x(s^*)\cdot x(s^*+1)^{m-s^*}  - y(s^*)\cdot y(s^*+1)^{m-s^*} \label{eq:X-Y-bound}
     \end{align}
     Now we bound this difference
     \begin{align}
         &\left(x(s^*)^{m-s^*+1} - y(s^*)^{m-s^*+1}\right) - \left(x(s^*)\cdot x(s^*+1)^{m-s^*}  - y(s^*)\cdot y(s^*+1)^{m-s^*}\right) \nonumber\\
        = & x(s^*)\cdot \left(x(s^*)^{m-s^*} - x(s^*+1)^{m-s^*}\right) - y(s^*) \cdot \left(y(s^*)^{m-s^*} - y(s^*+1)^{m-s^*}\right) \label{eq:x-y-bound}
     \end{align}
     We note that for fixed constants $\alpha, \beta > 1 $, and $\gamma > 0$, the function $p(x) = \left(\frac{\alpha}{\beta(1+x)}\right)^\gamma - \left(\frac{\alpha-1}{\beta(1+x)-1}\right)^\gamma$ is decreasing when $ x\ge 0$. If we set $\alpha = C-s^*, \beta = 2C-s^*$, and $\gamma = m-s^*$, then 
     \[x(s^*)^{m-s^*} - x(s^*+1)^{m-s^*} = p(0),\text{ while }y(s^*)^{m-s^*} - y(s^*+1)^{m-s^*} = p\left(\frac{2\eta C}{2C - s^*}\right).\]
     Therefore
     \[x(s^*)^{m-s^*} - x(s^*+1)^{m-s^*} \ge y(s^*)^{m-s^*} - y(s^*+1)^{m-s^*},\]
     and together with the fact  that $x(s^*) \ge y(s^*)$ we can deduce that \eqref{eq:x-y-bound}$\ge 0$. Equivalently,
     \[ x(s^*)^{m-s^*+1} - y(s^*)^{m-s^*+1} \ge x(s^*) \cdot x(s^*+1)^{m-s^*} -  y(s^*) \cdot y(s^*+1)^{m-s^*} \ge X(s^*) - Y(s^*),  \]
     where the second inequality holds by \eqref{eq:X-Y-bound}. Note that this inequality is exactly the claim  statement for $s = s^*$. Thus we finish the induction step and conclude the proof.
\end{proof}

\section{Sparsifying the Johnson complexes}
\label{app:sparsify-JC}

While both the Johnson complexes and the RGCs over $N = 2^{n-1}$ vertices have vertex degree $\mathsf{N}$, the Johnson complexes have much denser vertex links, or equivalently more $2$-faces. The degree of the Johnson complex $X_{\varepsilon,n}$'s vertex links is 
\[d_{link} = {\binom{\varepsilon n}{\varepsilon n/2}}\cdot {\binom{(1-\varepsilon) n}{\varepsilon n/2}} \gg N^{\varepsilon/2} 
,\]

while that of the RGCs is $\mathrm{polylog}(N)$. In fact we can sparsify the links and reduce the number of $2$-faces in Johnson complexes by keeping each $2$-face with probability $p$, and set $p$ so that the link degree becomes $\mathrm{polylog}(N)$. Then we shall use the following theorem to show that the link expansion is preserved after subsampling.  

\begin{theorem}[Theorem 1.1 \cite{ChungH07}]
    Suppose $G$ is a graph on $n$ vertices with second largest eigenvalue in absolute value $\lambda_2(G)$ and minimum degree $d_{\min}$. A random subgraph $H$ of $G$ with edge-selection probability $p$ has a second eigenvalue $\lambda_2(H)$ satisfying
 \[ \lambda_2(H) \le \lambda_2(G) + O\left( \sqrt{\frac{\log n}{pd_{\min}}} + \frac{(\log n)^{3/2}}{pd_{\min} (\log\log n)^{3/2}}\right)\]
 with probability $>1 - f(n)$ where \(f(n)\) tends to \(0\) faster than any polynomial in $\frac{1}{n}$.
\end{theorem}
From this eigenvalue preservation result we can get the following sparsification.

\begin{corollary}
    Consider a random subcomplex $X'$ of a $2$-dimensional Johnson complex $X_{\varepsilon, n}$ constructed by selecting each $2$-face i.i.d. with probability $p \ge \frac{(h(\varepsilon)\log N)^{3/2}}{d_{link}}$ where $N =  2^{n-1}$, and $d_{link}$ is the degree of $X_{\varepsilon, n}$'s vertex links. Then with high probability every vertex link of $X'$ is a $\frac{1}{2}(1-\varepsilon)(1+o_n(1))$-spectral expander.
\end{corollary}
Here we recall that \(h(x)\) is the binary entropy function.
\begin{proof}
    We first argue that almost surely $X'(1) = X_{\varepsilon,n}(1)$ so that for every $v\in X'(0)$ the vertex link $X'_v$ has the same vertex set as $(X_{\varepsilon,n})_v$. From the subsampling process we know that for any edge $\{v,v+s\} \in X_{\varepsilon,n}(1)$, the degree of $v+s$ in the link of $v$ follows a binomial distribution with
    \[\E[\textrm{degree of } v+s \textrm{ in } X'_v] 
        = p \cdot d_{link} =(h(\varepsilon)\log N)^{3/2}.\]
    By the Chernoff bound for binomial distributions we know that 
    \[\Pr[\textrm{degree of } v+s \textrm{ in } X'_v \le 0.1 (h(\varepsilon)\log N)^{3/2} ] \le e^{-0.45(h(\varepsilon)\log N)^{3/2}} = N^{-\Theta(\sqrt{\log N})}.\]

    So by a simply union bound over all pairs of $\{v,v+s\}$ we get that with probability $1 - N^{1+h(\varepsilon)} \cdot N^{-\Theta(\sqrt{\log N})} = 1- o(1)$ every link in $X'_v$ has the same vertex set as $(X_{\varepsilon,n})_v$.

    Next, recall by \pref{lem:Johnson-complex-spectral-expansion} $X_{\varepsilon,n}$'s vertex links have second eigenvalue $<\frac{1}{2}(1-\varepsilon)$. Since the choice of $p$ ensures that $\sqrt{\frac{\log N^{h(\varepsilon)}}{p\cdot d_{link}}} + \frac{(\log h(\varepsilon))^{3/2}}{p\cdot d_{link} (\log\log h(\varepsilon)N)^{3/2}} = o(1)$, each link of $X'$ is a $\frac{1}{2}(1-\varepsilon)(1+o(1))$ expander. Finally, applying an union bound over all links of $X'$ completes the proofs.
\end{proof}

\section{Proof of the van Kampen lemma}
\label{app:van-kampen}


To prove this lemma, we would like to contract \(C_0\) using the contractions of \(\set{C_i}_{i>0}\) so the proof should give us an order in which $C_i$-s are contracted. This order is given by the order in which the set \(\set{f_i}_{i > 0}\) are contracted to reduce \(\R^2 \setminus f_0\) to a point. Therefore, for every such diagram, it suffices to find a single face \(f_j\) such that we can contract to obtain a diagram of a cycle with one less interior face. The property we need from \(f_j\) is that the shared boundary of \(f_j\) and \(f_0\) is a contiguous path. This requirement is captured in the following definition.

\begin{definition}[valid face and path]
    Let \(H\) be a van Kampen diagram. We say that an interior face \(f_j\) is \emph{valid} if we can write its boundary cycle \(\tilde{C}_j = P_j \circ Q\)  such that the exterior cycle can be written as \(\tilde{C}_0 = P_j \circ Q'\) where \(P_j\) contains at least one edge and the vertices of \(Q,Q'\) are disjoint (except for the endpoints). In this case we call \(P_j\) a \emph{valid path}.
\end{definition}

\pref{lem:van-kampen} will easily follow from the following three claims.
\begin{claim} \label{claim:valid-face-exists}
    Every \(2\)-connected plane graph \(H\) has at least one valid face.
\end{claim}

\begin{claim} \label{claim:vertex-degree-in-valid-path}
    Let \(P_j = (v_0,v_1,\dots,v_m)\) be a valid path in $H$. Then for every \(1 \leq i\leq m-1\), \(\deg(v_i)=2\).
\end{claim}

\begin{claim} \label{claim:2-connectivity-remains}
    Let \(H = (V,E)\) be a van Kampen diagram, let \(P_j = (v_0,v_1,\dots,v_m)\) and let \(H_j = (V',E')\) be the induced subgraph on \(V' = V \setminus \set{v_1,v_2,\dots,v_{m-1}}\). Then \(H_j\) is \(2\)-vertex connected.
\end{claim}

Given the three claims, we can prove the van Kampen lemma as follows. 


\begin{figure}[ht]
    \centering
    \includegraphics[scale=0.3]{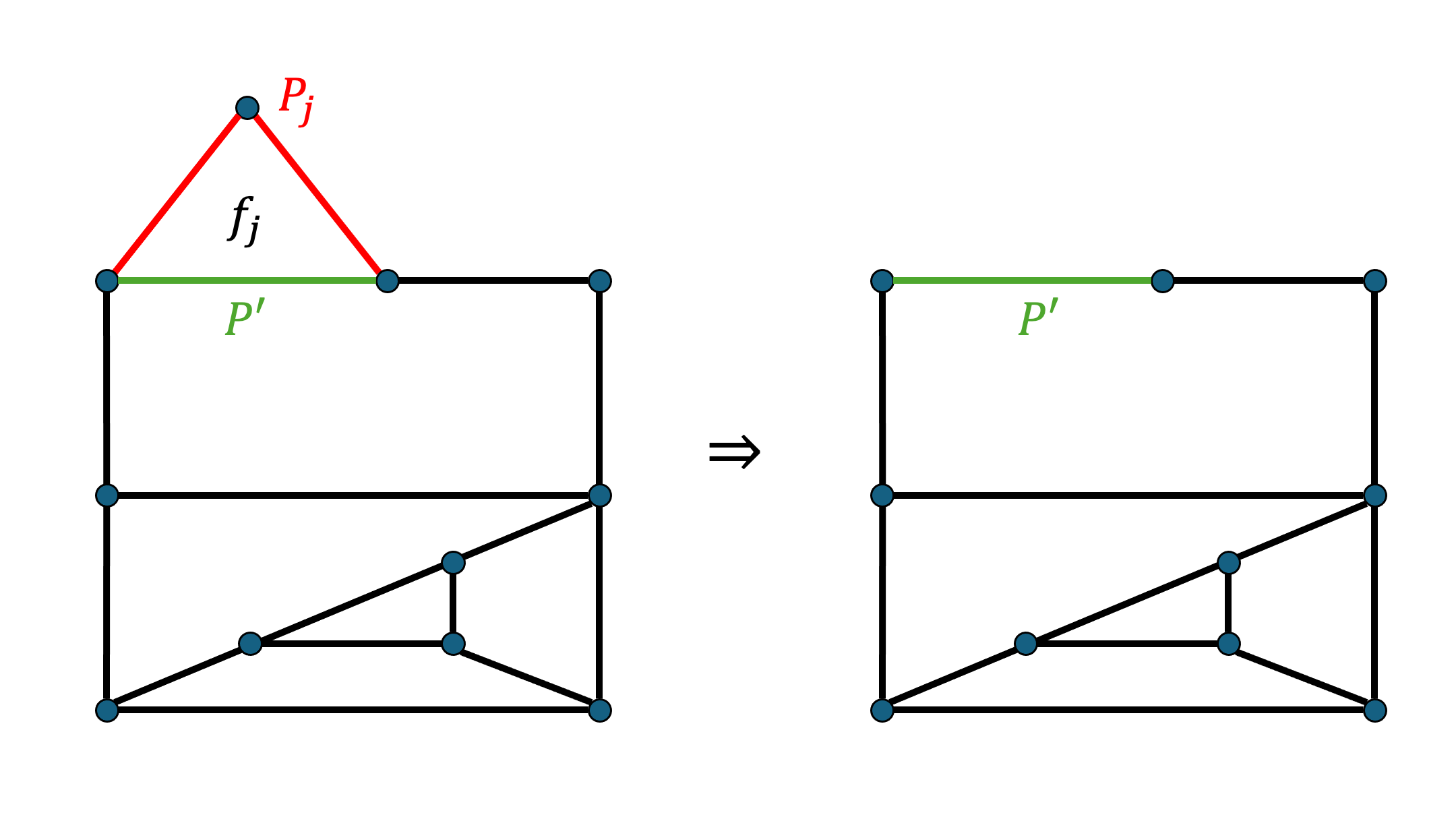}
    \caption{Contracting a face in $H$ to obtain $H_j$ in \pref{lem:van-kampen}}
    \label{fig:van Kampen-diagram-contraction}
\end{figure}

\begin{proof}[Proof of \pref{lem:van-kampen}]
    Fix \(X,C_0,H,F,\psi\) be as defined in \pref{sec:van-Kampen}. We prove this lemma by induction on the number of inner faces of $H$. If $H$ has exactly one inner face \(f_1\), then its boundary $\tilde{C}_1$ is also the boundary $\tilde{C}_0$ of the outer face \(f_0\). Then we can deduce that $C_0 = \psi(\tilde{C}_0) = \psi(\tilde{C}_1) = C_1$. So indeed there is a contraction of \(C_0\) using \(m_{0}=m_1\)-triangles.
    
    Assuming the lemma is true for all van Kampen diagrams with at most \(\ell\) inner faces, we shall prove it for van Kampen diagrams $H,F$ with \(\ell+1\) inner faces.
    By \pref{claim:valid-face-exists}, we can always find a valid face \(f_j\) in $H$ such that \(\tilde{C}_0 = P_j \circ Q\) and \(\tilde{C}_j = P_j \circ P'\) where \(P_j\) is the corresponding valid path. The decompositions are illustrated in \pref{fig:van Kampen-diagram-contraction}.

    Using \pref{claim:shift-between-paths}, we can contract \(C_0\) (the image of the outer boundary of $H$ under $\psi$) to \(\psi(Q \circ P')\) using \(m_{j}\) triangles. Correspondingly, we can construct an induced subgraph $H_j$ of $H$ by removing all edges and intermediate vertices in \(P_j\). 
    
    Note that the the outer boundary of $H_j$ is $\tilde{C}_0'=Q \circ P'$.
    Furthermore, since the intermediate vertices in $P_j$ all have degree $2$ in $H$ (by \pref{claim:vertex-degree-in-valid-path}), removing these vertices only removes the inner face $f_j$. Therefore, $H_j$ has $\ell$ inner faces.
    Finally by \pref{claim:2-connectivity-remains}, $H_j$ is still a $2$-connected plane graph. 

    Thereby $H_j, F\setminus\set{f_j}$ is a van Kampen diagram of the cycle $\psi(\tilde{C}_0')$ in $X$ with $\ell$ inner faces, and so the induction hypothesis applies.
    Thus we complete the inductive step and the lemma is proven.       
    %
\end{proof}

Next we shall prove the three supporting claims. Before doing so,  we define the cyclic ordering of vertices in a cycle which will later be used in the proofs.


\begin{definition}[Cyclic ordering]
    Let $H = (V,E)$ be a graph. 
    For a given cycle \(C\) in $H$, a cyclic ordering of the vertices in $C$ is such that two consecutive vertices in the ordering share an edge in $C$.
\end{definition}

\begin{proof}[Proof of \pref{claim:valid-face-exists}]
    We will show that either there exists a valid face, or that there exists two inner faces \(f_i \ne f_j\) and four vertices \(v_1,v_2,v_3,v_4\) in the cycle \(C_0\) such that \(v_1,v_3 \in f_i\), \(v_2,v_4 \in f_j\) (that is, on the boundaries of the respective faces), and such that in the cyclic ordering of \(C_0\), \(v_1 \leq v_2 \leq v_3 \leq v_4\). Then we will show that the latter case implies there is an embedding of a \(5\)-clique on the plane, reaching a contradiction. We mention that these points do not need to be vertices in \(H\), only points on the plane.
    
    Indeed, take a face \(f_i\) that shares an edge with \(f_0\). If it is not valid, it partitions \(C_0 \setminus \boundary f_i\) into disconnected segments (and this partition is not empty since otherwise \(f_i\) was valid). If there exists some \(f_j\) that touches two distinct segments, we can take \(v_1,v_3 \in f_i\) and \(v_2,v_4 \in f_j\) as above. Otherwise, we take some \(S \subseteq C_0 \setminus \boundary f_i\), that contains an edge. We now search for a valid face whose path is contained in \(S\), and continue as before. That is, take some \(f_{i'}\) that shares an edge in \(S\) with the outer face. If it is valid we take \(f_{i'}\) and if not we consider the partition of \(S \setminus f_{i'}\) as before. We note that this process stops because the number of edges in the segment in every step strictly decreases - and when we stop we either find a valid face or four points as above.
    
    Now let us embed the \(5\)-clique. Indeed, we fix a point on the outer face \(v_0\) in the interior of \(f_0\), and take \(\set{v_0,v_1,v_2,v_3,v_4}\) to be the vertices of the clique. The edges are drawn as follows.
    \begin{enumerate}
        \item From \(v_0\) to every other vertex we can find four non intersecting curves that only go through \(f_0\).\footnote{For this one can use the Riemann mapping theorem to isomorphically map the outer face to the unit disc and \(C_0\) to its boundary. Without loss of generality we can choose \(v_0\) so that it is mapped to the center. Then we draw straight lines from the image of \(v_0\) to the rest of the points in the boundary (and use the inverse to map this diagram back to our graph).}
        \item The edges \(\set{v_1,v_2},\set{v_2,v_3},\set{v_3,v_4},\set{v_4,v_1}\) are the curves induced by the cycle \(C_0\). These curves do not intersect the previous ones (other than in \(v_1,v_2,v_3,v_4\)) since the curves in the first item lie in the interior of \(f_0\).
        \item By assumption, the pairs \(v_1,v_3 \in f_i\) and \(v_2,v_4 \in f_j\) so we can connect these pairs by curves that lie in the interiors of \(f_i\) and \(f_j\) respectively, and in particular they do not intersect any of the other curves.
    \end{enumerate}
\end{proof}

\begin{proof}[Proof of \pref{claim:vertex-degree-in-valid-path}]
    Let \(f_j\) be the valid face. Let us denote by \(e_0,e_2\) the edges adjacent to \(v\) in \(P_0\) which is not an endpoint. Assume towards contradiction that there is another edge \(e_1\) adjacent to \(v\). Thus \(e_1\) either lies in the interior or the exterior of \(f_j\).

    If it lies in the interior, then by \(2\)-vertex connectivity there exists a path starting at \(e_1\) connecting to another point in the boundary of \(f_{j}\), and other than \(v\) and the end point, the path lies in the interior of \(f_{j}\). This is a contradiction, since any such path disconnects \(f_{j}\) to two faces. Otherwise it lies in the exterior, but the argument is the same only with \(f_0\) instead of \(f_j\). 
    %
\end{proof}

\begin{proof}[Proof of \pref{claim:2-connectivity-remains}]
%
    %
    Let \(v_0,v_1\) the end points in \(P_1\), and notice that \(v_0,v_1\) participate in a cycle in \(H\)- the cycle bounding the outer face. In particular \(H_j\) is connected. Remove any vertex \(z\) from \(H_j\). Let \(P_{x,y}\) be a path that connects two vertices \(x,y\) in \(H \setminus \set{z}\) (that may contain some vertices in \(P_j\)). Let us show that we can modify \(P_{x,y}\) to some \(P'\) in \(H_j \setminus \set{z}\). By \pref{claim:vertex-degree-in-valid-path} every vertex in \(P_j\) which we removed had degree \(2\). Therefore any path containing a vertex in \(P_j\), contains \emph{all of} \(P_j\) (or its reverse). Thus to mend \(P_{x,y}\) we just need to show that and \(v_0,v_1\) are in a simple cycle in \(H_j\). Indeed, this is true from the definition of a valid path: \(C_j = P_j \circ Q\) and \(C_0 = P_j \circ Q'\) such that all vertices in \(Q,Q'\) are disjoint besides the end points. The paths \(Q,Q'\) participate in simple cycles therefore \(Q' \circ Q^{-1}\) is a simple cycle and we conclude.
    %
\end{proof}

\section{Expansion of containment walks in the Grassmann poset} \label{app:grassmann-containment}
To stay self contained we reprove \pref{claim:loc-to-glob-grassmann} in a similar manner as in \cite{DiksteinDFH2018}.
\begin{proof}[Proof of \pref{claim:loc-to-glob-grassmann}]
    Although all the claims below apply for \(Y_w\) for any \(w \in Y\), for a simple presentation let us assume that \(w = \set{0}\) (the same idea works for arbitrary \(w\)). Let us denote by \(U_{i,j}\) the bipartite graph operator from \(\Y(i)\) to \(\Y(j)\). We note that \(U_{i,j} = U_{j-1,j} \dots U_{i+1,i+2} U_{i,i+1}\), and therefore it is sufficient to prove that \(\lambda_2(U_{i,i+1}) \leq \left (0.61 + i \lambda \right )^{\frac{1}{2}}\). We do so by induction on \(i\).

    In fact, we prove a slightly stronger statement, that \(\lambda_2(U_{i,i+1}) = \left (\sum_{\ell=1}^i \frac{1}{2^{\ell+1}-1} + i \lambda \right )^{\frac{1}{2}}\). This suffices because \(\sum_{\ell=1}^i \frac{1}{2^{\ell+1}-1} \leq \sum_{\ell=1}^\infty \frac{1}{2^{\ell+1}-1} \leq 0.61\). 
    Let \(D_{i+1,i}\) be the adjoint operator to \(U_{i,i+1}\), and denote by \(\lambda_i = \lambda(D_{i+1,i}U_{i,i+1})\) be the second largest eigenvalue of the two step walk from \(\Y(i)\) to \(\Y(i)\). To show that \(\lambda_2(U_{i,i+1}) = \left (\sum_{\ell=1}^i \frac{1}{2^{\ell+1}-1} + i \lambda \right )^{\frac{1}{2}}\) it suffices to prove that \(\lambda_i = \left (\sum_{\ell=1}^i \frac{1}{2^{\ell+1}-1} + i \lambda \right )\).

    The base case for the induction is for \(U_{1,2}\). In this case
    \[\lambda_1 = \lambda(D_{2,1} U_{1,2}) = \frac{1}{3}I + \frac{2}{3}M^+\]
    where \(M^+\) is the non-lazy version of the walk (namely, where we traverse from \(w_1\) to \(w_2\) if \(w_1 +w_2 \in \Y(2)\)). The adjacency operator \(M^+\) is the link of \(\set{0}\), and thus by \(\lambda\)-expansion we have that \(\lambda_1 \leq \frac{1}{3} + \frac{2}{3} \lambda \leq \frac{1}{2^2-1} + \lambda\).

    Assume that \(\lambda_2(U_{i,i+1}) \leq \left (\sum_{\ell=1}^i \frac{1}{2^{\ell+1}-1} + i \lambda \right )^{\frac{1}{2}}\) and consider the two step walk on \(\Y(i+1)\). As above, we note that
    \[\lambda_{i+1} = \lambda(D_{i+2,i+1}U_{i+1,i+2}) = \frac{1}{2^{i+2}-1} I + \left (1- \frac{1}{2^{i+2}-1} \right ) M_{i+1}^+\]
    where \(M^+\) is the non-lazy version of the walk. We note that the laziness probability \(\frac{1}{2^{i+2}-1}\) is because every subspace \(w \in \Y(i+2)\) has exactly \(2^{i+2}-1\) subspaces of codimension \(1\) (so this is the regularity of the vertex). Thus we will prove below that \(\lambda (M_{i+1}^+) \leq \lambda + \lambda_i \leq \left (\sum_{\ell=1}^i \frac{1}{2^{\ell+1}-1} + (i+1) \lambda \right )\) and this will complete the induction.
    
    Let us decompose the non-lazy walk \(M^+\) via \pref{lem:local-to-global}. Our components are \(\T = \set{\Y_W}_{W \in \Y(i)}\). To be more explicit, note that \(W_1 \sim W_2\) in the two step walk if and only if \(W_1+W_2 \in \Y(i+2)\). We can choose such a pair by first choosing \(W \in \Y(i)\) and then choosing an \(\set{W_1,W_2}\) in the link \(\Y_{W}\). The distribution of the link if the conditional distribution, hence taking \(\mu\) to be the \(\Y(i)\)-marginal of the distribution over flags in \(\Y\), we have that \(\T\) is a local-to-global decomposition. The local-to-global graph \(\B_\tau\) connects \(W \in \Y(i)\) to \(W \in \Y(i+1)\) - so this is just the containment graph \((\Y(i),\Y(i+1))\) which by induction is a \(\sqrt{\lambda_i}\)-expander. By \pref{lem:local-to-global},
    \(\lambda (M_{i+1}^+) \leq \lambda + \lambda_i \leq \left (\sum_{\ell=1}^i \frac{1}{2^{\ell+1}-1} + (i+1) \lambda \right )\) and we can conclude.
\end{proof}

\section{A degree lower bound} \label{app:degree-lower-bound}
    A classical result in graph theory says that a connected Cayley graph over \(\F_2^n\) has degree at least \(n\). Moreover, if the Cayley graph is a \(\lambda\)-expander then it has degree at least \(\Omega(\frac{n}{\lambda^2 \log(1/\lambda)})\) \cite{alon1992simple}. Are there sharper lower bounds on the degree if the Cayley graph is also a Cayley local spectral expander?
    
    In this section we consider the `inverse' point of view on this question. Given a graph \(G\) with \(m\) vertices, what is the maximal \(n\) such that there exists a connected Cayley complex over $\F_2^n$ whose link is isomorphic to \(G\). Using this point of view, we prove new lower bounds for Cayley local spectral expanders, that depend on properties of \(G\). Specifically, we prove that if \(G\) has no isolated vertices then \(m > 1.5 (n-1)\); if \(G\) is connected we prove that \(m > 2n-1\). Finally, if the link is a \(d\)-regular edge expander, we prove a bound of \(\Omega(d^{1/2} n)\). For all these bounds, we use the code theoretic interpretation of the links given in \cite{Golowich2023}.
    
    We begin by describing the link of a Cayley complex, and using this description we give a lower bound on the degree of the Cayley complex in terms of the rank of some matrix. This linear algebraic characterization of a degree lower bound is already implicit in \cite[Section 6]{Golowich2023}. 
    
    After that, we will lower bound the rank of the matrix, and through this we will get lower bounds on the degree of such Cayley complexes.  
    
    A link in a Cayley complex \(X\) naturally comes with labels on the edges which are induced by the generating set (i.e.\ an edge between \(x+s_1\) and \(x+s_2\) in the link of \(x\) will be labeled \(s_1+s_2\) which by assumption is a generator in the generating set). Note that if \(x+s_1,x+s_2\) is labeled \(s_1+s_2\) then there is a \(3\)-cycle in the link between \(x+s_1,x+s_2\) and \(x+s_1+s_2\) where every edge whose endpoints correspond to two generators is labeled by the third generator.

    Let \(G=(V,E)\) be a graph with vertex set \(V=\set{v_1,v_2,\dots,v_m}\) and a labeling \(\ell:E \to V\). The graph \(G\) is \emph{nice} if for every \(\set{v_i,v_j}\) that is labeled \(v_k\) there is a three cycle \(v_i,v_j,v_k\) in the graph where \(\set{v_i,v_k}\) is labeled \(v_j\) and \(\set{v_j,v_k}\) is labeled \(v_i\). We would like to understand for which \(n\) does there exist a Cayley complex \(X\) over $\F_2^n$ so that the vertex links of \(X\) are isomorphic to \(G\). That is, for which \(n\) can we find generators \(s_1,s_2,\dots,s_m \in \F_2^n\) such that the map \(\phi(s_i)=v_i\) is a graph isomorphism between the vertex link and $G$. Observe that for this to hold, whenever there is an edge \(\set{v_i,v_j} \in E\) labeled by \(v_k\), we require  that \(s_i+s_j+s_k = 0\). Thus we can define the matrix \(H_G \in \F_2^{|E| \times |V|}\) to have in every row the indicator vector of \(i,j,k\) for every \(\set{v_i,v_j} \in E\) labeled by \(v_k\). Let us show that this is a sufficient condition.

\begin{claim} \label{claim:deg-lower-bound}
    Let \(G\) be a nice graph with \(m\) vertices. Let \(S = \set{s_1,s_2,\dots,s_m} \in \F_2^n\) and let \(M_s\) be the matrix whose rows are the \(s_i\)'s. Then the following is equivalent.
    \begin{enumerate}
        \item The columns of \(M_s\) are in the kernel of \(H_G\).
        \item There exists a Cayley complex \(X\) over \(\F_2^n\) and whose link is isomorphic to \(G\) whose underlying graph is \(Cay(\F_2^n,S)\).
    \end{enumerate}
    In particular, if $X$ is connected then 
    \(m \geq n + \rank(H_G)\).
\end{claim}
This equivalence is essentially what was proven in \cite[Lemma 77]{Golowich2023}. The degree bound is a direct consequence of it. Nevertheless, we prove it in our language for clarity.

\begin{proof}
    If the first item occurs, this means that for every \(\set{v_i,v_j} \in E\) labeled by \(v_k\), \(s_i+s_j+s_k = 0\) and therefore the \(3\)-cycles \((x,x+s_i,x+s_i+s_j=x+s_k,x)\) are in \(Cay(\F_2^n,S)\) so add \(\set{x+s_i,x+s_j,x+s_k}\) into the triangle set of $X$. Then it is straightforward to verify that the link of \(X\) is isomorphic to \(G\).

    On the other hand, from the discussion above, the isomorphism to \(G\) implies that for every edge \(\set{v_i,v_j}\) in $G$ with label, their corresponding generators satisfy \(s_i+s_j+s_k = 0\). Hence the columns of $M_s$ are in \(Ker(H_G)\).

    In particular, this implies that \(\rank(M_s) \leq \dim(Ker(H_G)) = m-\rank(H_G)\). On the other hand, if the complex is connected then \(\rank(M_s) \geq n\). Combining the two we get that the degree of $X$ is lower bounded by \(m \geq n+\rank(H_G)\).
\end{proof}

\begin{corollary}
    Let \(X\) be a connected Cayley complex over \(\F_2^n\) with degree \(m\) such that the link of \(X\) contains no isolated vertices. Then \(m \geq 1.5 (n-1) \). Moreover, if the link is connected then \(m \geq 2n-1\).
\end{corollary}

\begin{proof}
    The corollary follows from the inequality \(m \geq n+\rank(H_G)\), together with the observation that \(\rank(H_G) \geq \frac{m}{3}\) which we now explain. We can greedily select independent rows from \(H_G\) such that whenever a generator \(s_i\) has not appeared in any selected rows, we take an row where \(s_i\) appears. As every newly selected row contains at most \(3\) new generators, we can find at least \(\fl{\frac{m}{3}}\) such equations. Thus,
    \(m \geq n+\fl{\frac{m}{3}}\) and the bound follows.

    If the link is connected, we can improve on this argument. At every step when we select a row, we denote by \(D \subseteq V\) the set of vertices that have appeared in the selected rows. Since the link is connected, therefore while \(D\ne V\) we can always choose a new row containing vertices from both $D$ and its complement, in which case the new row (besides the first one) introduces at most \(2\) new vertices. Thus \(\rank(H_G) \geq 1+\frac{m-3}{2} = \frac{m}{2}-\frac{1}{2}\). The bound follows.
\end{proof}

\subsection{Improved bound from expansion}
If the graph \(G\) is an (edge) expander, we can improve the naive bounds above. Let \(m\) be the degree of the Cayley complex and \(d\) be the degree of a vertex in the link. We get a degree bound of \(m = \Omega(d^{1/2})n\). In particular, if the degree goes to infinity (which is also a necessary condition to get better and better spectral expansion), this lower bound tends to infinity as well. 

\begin{lemma} \label{lem:deg-strong-lower-bound}
    Let \(X\) be a simple Cayley HDX over \(\F_2^n\) with \(m\) generators.\footnote{Simple, means that there are no loops or multi-edges. This also implies that every two labels of edges that share a vertex are distinct.} Assume that the link is \(d\)-regular and denote it by \(G=(V,E)\). Let \(c \in (0,1)\) be such that for every subset of vertices \(D \subseteq V\) in the link ,
    \[\Prob[(x,y)\sim E]{x\in D, y \notin D} \geq c\Prob[x\sim V]{x\in D} \Prob[y\sim V]{y\in V \setminus D}\]
    where the probability is over sampling an edge (oriented) in the link. Then
    \[m = \Omega(c d^{1/2} n).\]
\end{lemma}
A few remarks are in order.
\begin{remark}
    \begin{enumerate}
        \item We note that the edge expansion condition in the lemma is equivalent to one-sided $\lambda$-spectral expansion for some \(\lambda < 1\) by Cheeger's inequality. Thus whenever \(\lambda\) is bounded away from \(1\) we can treat \(c\) as some constant absorbed in the \(\Omega\) notation, and write \(m=\Omega(d^{1/2} n)\).
        \item Moreover, the well known Alon-Boppana bound implies that \(d\)-regular \(\lambda\)-one sided spectral expanders have \(d = \Omega(\lambda^{-2})\). Therefore, this bound also implies
        \[m = \Omega(\lambda^{-1} n).\]
        However, this is not as sharp as the bound given in \cite{alon1992simple}.
        \item We prove this lemma for regular graphs for simplicity. The argument could be extended for more general edge expanders.
    \end{enumerate}
\end{remark}

\begin{proof}[Proof of \pref{lem:deg-strong-lower-bound}]
    We will assume without loss of generality that \(d \geq 36\) since otherwise the claim is trivial from the bound \(m\geq n\).
    By \pref{claim:deg-lower-bound}, \(m \geq n + \rank(H_G)\) therefore it suffices to prove that 
    \begin{equation}\label{eq:rank-inequality-of-ker-mat}
    \rank(H_G) \geq \left( 1-O \left(\frac{1}{c d^{1/2}} \right) \right) m
    \end{equation}
    to prove the lemma.

    This lower bound follows from the same greedy row selection process we conducted above. We start by greedily selecting rows from \(H_G\) such that the selected rows are independent. We keep track of the vertices/variables that appear in the selected rows, and select new rows such that every new row introduces the least number of new vertices possible. We will prove that there is a way to select rows so that only \(O(\frac{1}{d^{1/2}c})\) fraction of the rows selected will introduce more than one new vertex (and the number of newly introduced vertices in one step is always \(\le 3\)). Since the process terminates after \(\rank(H_G)\) steps, the argument above implies that \(\rank(H_G)\cdot (1+ O(\frac{1}{d^{1/2}c})) \geq m\) which in turn implies \eqref{eq:rank-inequality-of-ker-mat}.
    
    For \(i=1,2,\dots,\rank(H_G)\) we denote by \(D_i\) the set of variables that are contained in one of the first \(i\) rows selected. By assumption, \(\Prob[(x,y)\sim E]{x\in D_i, y \notin D_i} \ge c \prob{x\in D_i} \prob{y\in V \setminus D_i}\). Equivalently,
    \[\cProb{(x,y)\sim E}{x \in D_i}{y \notin D_i} \geq c \prob{x\in D_i}.\]
    In particular, for every \(D_i\) there exists a vertex \(y_i \in V \setminus D_i\) such that at least \(c\prob{x\in D_i} \cdot d\) of its outgoing edges go into \(D_i\). In particular, once \(\prob{x\in D_i} \geq \frac{1}{c d^{1/2}}\) every \(y_i\) has at least \(d^{1/2}\) edges into \(D_i\).

    Hence, we select rows as follows. If \(\prob{D_i} < \frac{1}{c d^{1/2}}\) we select new rows arbitrarily that introduce as few new vertices as possible. After \(\prob{x\in D_i} \ge \frac{1}{c d^{1/2}}\), if in the \(i\)-th step every new row introduces at least two new vertices, we select a row that contains \(y_i\) (as defined above) and a vertex in \(D_i\) (such a row exists by our choice of \(y_i\)). We observe that if every new row introduces at least two new vertices, this implies that the edges connecting \(y_i\) to $D_i$ are labeled by distinct generators in \(V \setminus D_i\). In particular, after selecting \(y_i\) which adds \(y_i\) and another new vertex to \(D_{i+1}\), there are at least \(d^{1/2} - 1 \geq \frac{1}{3} d^{1/2}\) edges from \(y_i\) to \(D_{i+1}\) that are labeled by distinct vertices in \(V \setminus D_{i+1}\) (since these edge labels were not in \(D_i\) when we inserted \(y_i\)). So the next $\ge \frac{1}{3}d^{1/2}$ steps of the process add at most $1$ new vertex per step. In particular, for every such step where \(2\) vertices are inserted to \(D_i\) to get \(D_{i+1}\), there is a sequence of \(\geq \frac{1}{3}d^{1/2}\) steps each of which adds  \(\le 1\) new vertex. In conclusion, there are at most \(O(\frac{3m}{d^{1/2}} + \frac{m}{2c d^{1/2}})\)  steps that introduce more than one new vertex to \(D_i\) (out of at least \(\frac{m}{2}-\frac{1}{2}\) steps) and the lemma follows.
\end{proof}
\end{document}